\RequirePackage[T1]{fontenc}
\documentclass[12pt,twoside]{amsart}
\calclayout

\usepackage[whole]{bxcjkjatype}
\usepackage[utf8]{inputenc}
\usepackage{amsmath}
\usepackage{amssymb}
\usepackage{mathtools}
\numberwithin{equation}{section}
\usepackage{slashed}
\usepackage{braket}
\usepackage[svgnames]{xcolor}
\usepackage{colortbl}
\definecolor{lightgray}{gray}{0.9}
\definecolor{xgray}{gray}{0.8}
\colorlet{conjcolor}{blue!5!white}
\usepackage[colorlinks,citecolor=DarkGreen,linkcolor=FireBrick,urlcolor=FireBrick,linktocpage,breaklinks=true]{hyperref}
\usepackage{cite}
\usepackage[pdftex]{graphicx}
\usepackage{times}
\usepackage{courier}
\usepackage{bm}
\usepackage{subfig}
\usepackage{fnpct}

\usepackage{mathrsfs} 

\usepackage[all]{xy}
\usepackage{enumitem}
\if0
\usepackage{mdframed}

\renewenvironment{figure}[1][]{
  \begin{originalfigure}[#1]
    \begin{mdframed}[linecolor=black!0,backgroundcolor=black!1]
}{
    \end{mdframed}
  \end{originalfigure}
}
\fi
\theoremstyle{plain}
\newtheorem{defn}[equation]{Definition}

\newtheorem{thm}[equation]{Theorem}

\newtheorem{prop}[equation]{Proposition}

\newtheorem{fact}[equation]{Fact}
\def\PhysicsFact{Physics Assumption}

\newtheorem{fact?}[equation]{Fact?}
\newtheorem{cor}[equation]{Corollary}

\newtheorem{lem}[equation]{Lemma}

\newtheorem{conj}[equation]{Conjecture}
\newtheorem{proposal}[equation]{Proposal}

\theoremstyle{remark}
\newtheorem{rem}[equation]{Remark}
\newtheorem{ex}[equation]{Example}
\newtheorem{const}[equation]{Construction}
\newtheorem{physremark}[equation]{Physics Remark}

\usepackage[many]{tcolorbox}

\tcbset{
  breakable,
  colback=conjcolor,
  boxrule=-1pt,
  boxsep=1pt,
  left=2pt,right=2pt,top=2pt,bottom=2pt,
  oversize=2pt,
  sharp corners,
  before skip=\topsep,
  after skip=\topsep,
}
\tcolorboxenvironment{proofc}{}
\tcolorboxenvironment{defnc}{}
\tcolorboxenvironment{propc}{}
\tcolorboxenvironment{thmc}{}
\tcolorboxenvironment{corc}{}

\let\oldendrem\endrem
\def\endrem{\hfill {$\lrcorner$} \oldendrem}
\let\oldendconst\endconst
\def\endconst{\hfill {$\lrcorner$} \oldendconst}
\let\oldendphysremark\endphysremark
\def\endphysremark{\hfill {$\lrcorner$} \oldendphysremark}

\def\Aut{\mathop{\mathrm{Aut}}}

\def\bC{\mathbb{C}}

\def\bR{\mathbb{R}}
\def\bQ{\mathbb{Q}}
\def\bZ{\mathbb{Z}}
\def\cS{\mathcal{S}}

\def\Hom{\mathrm{Hom}}
\def\Ext{\mathrm{Ext}}
\def\tr{\mathop{\mathrm{tr}}}
\def\ko{\mathrm{ko}}
\def\H{\mathrm{H}}
\def\K{\mathrm{KU}}
\def\KU{\mathrm{KU}}
\def\KO{\mathrm{KO}}

\def\MF{\mathrm{MF}}
\def\TMF{\mathrm{TMF}}
\def\Tmf{\mathrm{Tmf}}
\def\tmf{\mathrm{tmf}}
\def\mf{\mathrm{mf}}
\def\pt{\mathrm{pt}}
\def\Cl{\mathrm{Cl}}

\def\BO{BO}
\def\Z{\mathbb{Z}}
\def\MString{\mathrm{MString}}
\def\MSpin{\mathrm{MSpin}}

\def\Q{\mathbb{Q}}
\def\R{\mathbb{R}}
\def\ch{\mathrm{ch}}

\def\Spin{Spin}

\def\SU{SU}
\def\Nequals#1{$\mathcal{N}{=}#1$}

\def\spin{\text{spin}}
\def\stri{\text{string}}
\def\diff{\text{diff}}
\let\oldsigma\sigma
\let\sigmasymbol\sigma
\def\Wit{\oldsigma}
\def\WitVec{\mathrm{Wit}}
\def\sigma{\Phi}
\def\Witt{\mathop{\mathrm{Witt}}\nolimits}
\def\SQFT{\mathrm{SQFT}}

\def\Ker{\mathop{\mathrm{Ker}}}

\def\index{\mathop{\mathrm{index}}\nolimits}
\def\hocolim{\mathop{\mathrm{hocolim}}}
\def\del{\partial}
\def\cw{\mathrm{cw}}
\def\hBord{h\mathrm{Bord}}
\def\clo{\mathrm{clo}}

\def\top{\text{top}}
\def\an{\text{an}}
\def\geom{\text{geom}}
\def\id{\mathrm{id}}

\def\ABS{\mathrm{ABS}}

\def\tor{\text{torsion}}

\def\JF{\mathrm{JF}}
\def\ind{\mathrm{ind}}
\def\wt{\mathrm{wt}}
\def\res{\mathrm{res}}

\def\SO{SO}
\def\Semispin{SemiSpin}
\def\RO{\mathrm{RO}}
\def\moduli{\mathcal{M}}
\def\cL{\mathcal{L}}
\def\Pic{\mathrm{Pic}}
\def\cE{\mathcal{E}}
\def\tr{\mathrm{tr}}

\def\V{\mathbb{V}}
\def\W{\mathbb{W}}
\def\cM{\mathcal{M}}
\def\cH{\mathcal{H}}
\def\cS{\mathcal{S}}
\def\cI{\mathcal{I}}

\newcommand{\mayuko}[1]{\textcolor{red}{(Mayuko:#1)}}

\def\lpar{(\!(}
\def\rpar{)\!)}
\let\oldtext\text
\def\text#1{\oldtext{\upshape\mdseries #1}}

\DeclareSymbolFont{yhlargesymbols}{OMX}{yhex}{m}{n} \DeclareMathAccent{\reallywidehat}{\mathord}{yhlargesymbols}{"62}

\DeclareFontFamily{U}{mathx}{}
\DeclareFontShape{U}{mathx}{m}{n}{<-> mathx10}{}
\DeclareSymbolFont{mathx}{U}{mathx}{m}{n}
\DeclareMathAccent{\widehat}{0}{mathx}{"70}
\DeclareMathAccent{\widecheck}{0}{mathx}{"71}

\let\hat\widehat
\let\tilde\widetilde
\let\check\widecheck
\let\breve\tilde

\let\oldwidehat\widehat
\def\widehat#1{\oldwidehat{#1}{}}

\def\paragraph#1{

\medskip\noindent \textit{#1} --- }

\let\epsilonelem\epsilon
\def\epsilon{\ahoaho}
\def\varepsilon{\ahoaho}
\def\varsigma{\ahoaho}

\begin{document}

\title[Anderson duality of TMF and its differential manifestations]{Anderson duality of %\\[.2em]
topological modular forms\\[.2em]
and its differential-geometric manifestations}
\author{Yuji Tachikawa}
\author{Mayuko Yamashita}
\date{v1: May, 2023; v2: May, 2024, v3: June, 2025}
\address{Kavli Institute for the Physics and Mathematics of the Universe \textsc{(wpi)},
  The University of Tokyo,
  5-1-5 Kashiwanoha, Kashiwa, Chiba, 277-8583,
  Japan
}
\email{yuji.tachikawa@ipmu.jp}
\address{Perimeter Institute for Theoretical Physics, 31 Caroline Street North,
Waterloo, Ontario, Canada, N2L 2Y5
	}
	\email{myamashita@perimeterinstitute.ca}
\def\funding{The research of YT is  supported 
in part  by WPI Initiative, MEXT, Japan through IPMU, the University of Tokyo,
and in part by Grant-in-Aid for JSPS KAKENHI Grant Number 17H04837 and 24K06883.
The work of MY is supported by Grant-in-Aid for JSPS KAKENHI Grant Number 20K14307 and JST CREST program JPMJCR18T6.}
\thanks{\emph{Acknowledgments}: 
The authors thank N. Kawazumi for his insightful question after a seminar given by MY on \cite{TachikawaYamashita}, concerning the interpretation of the secondary invariant associated to the vanishing of the primary invariant studied in \cite{TachikawaYamashita}.
Trying to answer his question was a major source of inspiration for our work.
The authors also thank K. Yonekura for the collaboration at the early stage of this work,
and Y. Moriwaki for helpful discussions on vertex operator (super)algebras.
The appendix~\ref{app:sanath-proof}, giving an alternative proof of one of our theorems, was kindly provided by Sanath Devalapurkar.
\funding
}

\begin{abstract}
\if0
We construct a morphism of spectra from $\KO((q))/\TMF$ to $\Sigma^{-20}I_\Z \TMF$,
which we show to be an equivalence and to implement the Anderson self-duality of $\TMF$.
This morphism is then used to define another morphism from $\TMF$ to $\Sigma^{-20}I_\Z(\MSpin/\MString)$,
which induces a differential geometric pairing
and captures not only the invariant of Bunke and Naumann
but also a finer %, `height-two' 
invariant which detects subtle Anderson dual pairs of elements of $\pi_\bullet\TMF$.
Our analysis leads to conjectures concerning certain self-dual vertex operator superalgebras and some specific torsion classes in $\pi_\bullet\TMF$.
\fi
We construct and study a morphism of spectra implementing the Anderson duality of 
topological modular forms ($\TMF$).
Its differential version will then be introduced,  allowing us to pair
elements of $\pi_d\TMF$ with spin manifolds 
whose boundaries are equipped with string structure.
A few negative-degree elements of $\pi_d\TMF$ will then be constructed
using the theory of $\RO(G)$-graded $\TMF$,
and will be identified  using the differential pairing.
We also discuss a conjecture relating vertex operator algebras
and negative-degree elements of $\pi_d\TMF$,
underlying much of the discussions of this paper.
%This  paper is written as an article in mathematics, but much of the discussions in it was originally motivated by a study in physics.
The paper ends with a separate appendix for physicists, 
in which the contents  of the paper are summarized and translated into their language.
%In physics terms, our result allows us to compute the discrete part of the Green-Schwarz coupling of the $B$-field in a couple of subtle hitherto-unexplored cases.

\end{abstract}
\maketitle

\tableofcontents

%\newpage

\section{Introduction}

\subsection{Prelude to the introduction}
Topological modular forms (TMF) are an object in algebraic topology
which combines the ring of weakly-holomorphic integral modular forms 
and  stable homotopy groups of spheres in an intricate fashion \cite{Hopkins2002,TMFBook,BrunerRognes}.
Topological modular forms are, in a certain precise sense, 
the most universal of all the elliptic cohomology theories,
They are known to arise as the natural next step after K-theory and KO-theory from the perspective of the chromatic homotopy theory,
and share various properties with them.
For example, just as K-theory is 2-periodic and KO-theory is 8-periodic,
topological modular forms are $24^2=576$-periodic.
Another shared feature is that spin manifolds have KO-theory orientation,
while string manifolds have TMF orientation \cite{AHR10},
where spin structures and string structures are tangential structures of
manifolds obtained by killing the homotopy groups of orthogonal groups successively. 
\if0
Moreover,  there is a morphism of spectra \begin{equation}
\TMF\to \KO((q))\label{quux}
\end{equation} relating the two. 
\mayuko{The $\KO((q))$ is the Tate K-theory, which is not just a $\KO$-theory with formal power series but correctly regarded as elliptic cohomology for the Tate curves. (For example, the spin orientation of $\KO((q))$ is not just $ABS((q))$. )So I believe it is misleading to bring up $\KO((q))$ at this stage. }
\fi
In this sense, compared to K and KO theory, we can consider $\TMF$ as a better approximation of the sphere spectum, which is the central object in the stable homotopy theory. 
% and more detailed version of K and KO theory.
Furthermore, the formal power series with coefficients in $\KO$, $\KO((q))$, is known to arise
as the elliptic cohomology associated to the moduli of the Tate curve, and as such  it receives a morphism from $\TMF$:
\begin{equation}
\Phi\colon\TMF\to \KO((q)).\label{quux}
\end{equation}

One significant difference still exists, however.
Whereas K theory and KO theory have  geometric constructions via vector bundles,
we do not yet have a corresponding `geometric' construction of topological modular forms,
which have so far been  studied purely within homotopy theory.
There is a program proposed by Stolz and Teichner \cite{StolzTeichner1,StolzTeichner2}
to construct topological modular forms via
two-dimensional supersymmetric quantum field theories (SQFTs).
But its implementation is still in infancy, due in many respects to the fact that
our mathematical understanding of SQFTs is very meager. 

With all that said, it is already possible to take inspirations from 
this program and from the developments on the side of physics,
and transfer them into the realm of pure mathematics,
so that we can increase our understanding of topological modular forms.
In this paper we are going to do exactly that.

More concretely, using notations which will be introduced more fully later,
we do the following.
%\begin{itemize}
%\item 
We first define  a morphism of spectra 
\begin{equation}
\KO((q))/\TMF \to \Sigma^{-20} I_\bZ \TMF,\label{qux}
\end{equation}
where $I_\bZ E$ 
is the Anderson dual of a spectrum $E$.
We will show that this morphism is an equivalence, 
and that it is to be regarded as the $\TMF$ version of the 
known Anderson self-duality of $\Tmf$ originally found in  \cite{Sto1,Sto2}.
This will allow us to reproduce the results of \cite{BunkeNaumann} in a straightforward manner.
%\item
Furthermore, our formulation \eqref{qux} 
of the Anderson duality of topological modular forms
will lead to a differential pairing \begin{equation}
\widehat\TMF^{-d} \times \reallywidehat{\MSpin/\MString}_{-d-21} \to \bR/\bZ,
\label{diffpairgenintro}
\end{equation}
which is to be thought of as a TMF analogue 
of the eta invariant of Dirac operators on spin manifolds.

%\item 
Our considerations will then allow us to give an explicit, concrete 
construction of some special torsion elements 
in the kernel of the homomorphism \begin{equation}
\pi_d\TMF\to \pi_d\KO((q)).
\end{equation} induced from \eqref{quux}.
These are the most subtle type of torsions in $\pi_* \TMF$,
and are known to capture in positive degrees  a number of nontrivial elements of 
the stable homotopy groups of spheres \cite{Hopkins2002}.
The situation in the negative degrees was much less clear, however.
In this paper, among others,
we construct an element in $\pi_{-28}\TMF$,
using in an essential way  
the theory of $\RO(G)$-graded topological modular forms developed in 
 \cite{GepnerMeier,lin2024topologicalellipticgenerai,BauerMeierTJF}.
We then use the differential pairing \eqref{diffpairgenintro}
 to show that this element is the generator of the kernel of the homomorphism $\Phi$ on $d=-28$,
 which is known to be $\bZ/2\bZ$ from the spectral sequence computation.

%\end{itemize}
\begin{physremark}
The morphism \eqref{qux}, the differential pairing \eqref{diffpairgenintro},
and the negative-degree elements we construct in $\pi_{d}\TMF$ are all inspired
from considerations in physics, as we will mention along the way.
But it is to be emphasized that all the discussions in this paper
are mathematically rigorous. 
Any background knowledge in physics 
is neither required nor expected, although
a partial translation of the contents of this paper for  physicists 
is provided in Appendix~\ref{sec:introphys}.
\end{physremark}

\if0
The main theme of our paper is the Anderson duality of topological modular forms and its differential geometric manifestations involving string manifolds.
What we will describe can be considered as a generalization of the simpler case of the Anderson duality of K-theory and its differential geometric manifestations as realized by the eta invariants on spin manifolds.
Therefore we would like to begin by  recalling these concepts in this simpler setting of K-theory   before delving in to the much more interesting but complicated case of  topological modular forms.%\footnote{%
%We start this paper with an introduction aimed for mathematicians.
%Before starting our paper, we note that 

\begin{physremark}
As mentioned above,
 many of the contents of this paper were originally motivated by physics.
As such, in the paper, we sometimes have remarks from physics perspectives.
That said, skipping all these remarks should not cause any problem in the understanding of the rest of the paper, which is written in the standard and rigorous mathematical style.
A partial translation of the contents of this paper for  physicists 
is provided in Appendix~\ref{sec:introphys}.%}
%but physicists interested in this paper are encouraged to read also this introduction.
\end{physremark}
\fi

Now that a general introduction was given,
a more detailed preview of the rest of the paper will be presented below.
We will start in Sec.~\ref{subsec:duality-KO} by discussing 
the Anderson self-duality of KO-theory and its differential geometric manifestations 
as realized by the eta invariants on spin manifolds.
Then in Sec.~\ref{subsec:duality-TMF},
we will outline our construction of the crucial morphism \eqref{qux}
which implements the Anderson duality of topological modular forms,
and also discuss its differential version.
In Sec.~\ref{subsec:intro-appl}, 
we then discuss various applications of this Anderson duality of topological modular forms.
Among others, we explain the relation to the invariant of Bunke and Naumann,
and also the construction and the determination of special negative-degree elements of $\pi_d\TMF$.
The next section, Sec.~\ref{subsec:intro-ver}, will be to place this paper
in the larger context of the Stolz-Teichner program, by formulating a precise conjecture
relating vertex operator algebras and negative-degree elements of $\pi_d\TMF$.
The main part of our paper then finally begins
after we present the overall organization of the paper in Sec.~\ref{subsec:intro-org}
and the notations and the conventions in Sec.~\ref{subsec:nota-conv}.
Let us now move on to the detailed introduction of the paper.

\subsection{Anderson duality and the induced pairings of $\KO$}
\label{subsec:duality-KO}
The Anderson duality $I_\bZ$  is a natural duality operation in stable homotopy theory.
Various important spectra are known to be self-dual:
$\H\bZ \simeq I_\bZ \H\bZ$,
$\K \simeq I_\bZ \K$,
and $\KO\simeq \Sigma^{4} I_\bZ \KO$.
When we take the homotopy groups, the Anderson duality provides non-torsion and torsion pairings, both of which are perfect.
We take the case of \begin{equation}
\zeta_\KO: \KO\xrightarrow{\sim} \Sigma^4 I_\bZ \KO \label{zetaKO}
\end{equation}
as an example. 

%\alert
{ 
We start with a basic lemma whose easy proof will be provided later:
\begin{lem}[=Lemma~\ref{lemma:basic}]
\label{basiclemmaintro}
For a ring spectrum $R$ and an $R$-module spectrum $M$,
an $R$-module homomorphism $\zeta: M\to \Sigma^{-s} I_\bZ R$
from $M$ to the Anderson dual of $R$ 
is specified by an element $\zeta \in \pi_s I_\Z M$, 
which we denote by the same symbol by a slight abuse of notation.
\end{lem}
In our case, then, $\zeta_\KO$ is specified by an element 
$\zeta_\KO\in \pi_{-4} I_\Z \KO$.
To understand $\pi_{-4} I_\Z \KO$,
we use the exact sequence
\begin{equation}
\label{seqseq}
0\to \Ext(\pi_{-d+3}\KO,\Z)\to \pi_{d-4} I_\Z\KO \to \Hom(\pi_{-d+4} \KO,\bZ) \to 0.
\end{equation}
We also need nonzero homotopy groups of $\KO$,
which are given by 
$\pi_0\KO\simeq \bZ$, $\pi_1\KO\simeq \bZ/2$, $\pi_2\KO\simeq \bZ/2$,
$\pi_4\KO\simeq 2\bZ$.
Here,  we included a factor of $2$ in the last isomorphism 
to preserve the graded ring structure of $\pi_\bullet \KO$.
From \eqref{seqseq} for the case $d=0$ and $\pi_3\KO\simeq 0$, 
we see that   \begin{equation}
\zeta_\KO\in \pi_{-4}I_\bZ\KO \simeq \Hom(\pi_4\KO,\bZ) \simeq \Hom(2\bZ,\bZ)
\end{equation}  which is given by  \begin{equation}
\zeta_\KO = \frac12\cdot - \in \Hom(2\bZ,\bZ).
\end{equation}

The Anderson self-duality  \eqref{zetaKO} then gives rise to 
the non-torsion pairing \begin{align}
\langle -,-\rangle_{\zeta_\KO}&\colon \pi_d \KO \times \pi_{-d+4} \KO \to \bZ
\end{align}
giving the perfect pairing between $\pi_0\KO\simeq \bZ$ and $\pi_4\KO\simeq 2\bZ$, and
\begin{align}
( -,- )_{\zeta_\KO} &\colon (\pi_d \KO)_\tor \times (\pi_{-d+3} \KO)_\tor \to \bQ/\bZ
\end{align}
giving  the perfect pairing between $\pi_1\KO\simeq \bZ/2$ and $\pi_2 \KO\simeq \bZ/2$.}

We can also combine the morphism \eqref{zetaKO}  with the Anderson dual of  the Atiyah-Bott-Shapiro orientation $\ABS: \MSpin\to \KO$ to have the morphism \begin{equation}
\zeta_\spin= I_\bZ\ABS\circ\zeta_\KO \colon \KO \to \Sigma^{4}I_\bZ\MSpin,
\end{equation}
which in turn induces the pairings \begin{align}
\langle -,-\rangle_{\zeta_\spin}&\colon \pi_d \KO \times \pi_{-d+4} \MSpin \to \bZ,
\label{first}\\
( -,- )_{\zeta_\spin} &\colon (\pi_d \KO)_\tor \times (\pi_{-d+3} \MSpin)_\tor \to \bQ/\bZ.
\label{second}
\end{align}
Furthermore, $\KO$ and $\MSpin$  have  differential refinements $\widehat\KO$ and $\widehat\MSpin$, respectively, so that we can subsume the two pairings \eqref{first}, \eqref{second}
into a single  differential pairing \begin{equation}
(-,-)_{\widehat\zeta_\spin}\colon \widehat\KO^{-d} \times \widehat\MSpin_{-d+3} \to \bR/\bZ,
\end{equation}
which is  essentially given by the eta invariant of Atiyah, Patodi and Singer.

\subsection{Anderson duality and the induced pairings of $\TMF$}
\label{subsec:duality-TMF}
What we aim to achieve in this paper is to generalize this story  to topological modular forms.
Topological modular forms come in three versions, $\TMF$, $\Tmf$ and $\tmf$,
as explained in the standard textbooks such as \cite{TMFBook,BrunerRognes}.
Rationally, $(\pi_{2\bullet}\TMF)_\bQ\simeq(\MF_\bullet )_\bQ$
and $(\pi_{2\bullet}\tmf)_\bQ\simeq(\mf_\bullet )_\bQ$
are the rings of weakly-holomorphic modular forms and of modular forms, respectively,
but integrally, homotopy groups of topological modular forms also contain interesting torsional information.

The behavior of topological modular forms under the Anderson duality was first found for $\Tmf$ in \cite{Sto1,Sto2} in the form \begin{equation}
\Tmf \simeq \Sigma^{-21} I_\bZ\Tmf,
\label{eq_Tmf_duality_intro}
\end{equation} and was later reformulated in terms of $\tmf$ in \cite{BrunerRognes}.
Our first main result in this paper is its reformulation in terms of $\TMF$.

To state it, we need to use the cofiber sequence  \begin{equation}
\TMF \xrightarrow{\Phi} \KO((q)) \xrightarrow{C \Phi} \KO((q))/\TMF
\label{cofib}
\end{equation}  where $\Phi: \TMF\to \KO((q))$ is 
the morphism corresponding to the restriction to the Tate moduli in the definition of $\TMF$
as the global section of a sheaf of spectra over the moduli stack of elliptic curves,
as constructed in \cite{AHR10,HillLawson}.
A fundamental  result of this paper is
 \begin{thm}[=Theorem~\ref{main}]\label{main-intro}
 We have  an equivalence of $\TMF$-module spectra \begin{equation}
 \label{selfduality2}
 \alpha_{\KO((q))/\TMF} : \KO((q))/\TMF \xrightarrow{\sim} \Sigma^{-20} I_\bZ \TMF.
 \end{equation}
 \end{thm}
 
We now want to describe $\alpha_{\KO((q))/\TMF}$ more explicitly.
Using Lemma~\ref{basiclemmaintro},
this is done by describing $\alpha_{\KO((q))/\TMF}$ as an element of $\pi_{20}I_\Z \KO((q))/\TMF$.

In order to do so, let us first consider a $\KO((q))$-module morphism \begin{equation}
\alpha_{\KO((q))}\colon \KO((q)) \to  \Sigma^{-20} I_\bZ \KO((q))
\end{equation}
specified by an element 
\begin{equation}
    \alpha_{\KO((q))} \in \pi_{20} I_\Z \KO((q))\simeq \Hom(\pi_{-20}\KO((q)), \Z) \simeq {\Hom(2\Z((q)), \Z)},
\end{equation}
which we choose to be\begin{align}
\label{eq:choice-alpha}
   { \alpha_{\KO((q))} = \frac{1}{2}\Delta(q) \cdot - \big|_{q^0} \in \mathrm{Hom}(2\Z((q)), \Z).}
\end{align}
Here $\Delta(q)$ is the modular discriminant
and $f(q)\big|_{q^0}$ denotes the constant term of a $q$-series $f(q)$.

\begin{physremark}%\footnote{%
There are many ways to motivate why we choose this particular morphism $\alpha_{\KO((q))}$ in \eqref{eq:choice-alpha}.
One is to say that this works in the end.
Another motivation, which was the original reason why the authors considered it, was that this morphism appeared in the characterization of the anomaly of heterotic string theories in the authors' previous study \cite{TachikawaYamashita}.
\label{foot:phys}
\end{physremark}

Then \begin{equation}
I_\bZ\Phi ( \alpha_{\KO((q))} ) \in \pi_{20} I_\bZ \TMF \label{eq:vanishing}
\end{equation} turns out to vanish, which allows us to find a lift \begin{equation}
\alpha_{\KO((q))/\TMF} \in \pi_{20} I_\bZ \KO((q))/\TMF
\end{equation} under the long exact sequence of homotopy groups
associated to the Anderson dual of the cofiber sequence \eqref{cofib}, \begin{equation}
\pi_{20} I_\bZ \KO((q))/\TMF
\xrightarrow{I_\bZ C\sigma}
\pi_{20} I_\bZ \KO((q))
\xrightarrow{I_\bZ \sigma}
\pi_{20} I_\bZ \TMF.
\end{equation}
This lift $\alpha_{\KO((q))/\TMF} $, which turns out to be unique,  gives the $\TMF$-module morphism in Theorem~\ref{main-intro}.
It is not immediate that this morphism is actually an equivalence;
the detailed construction of this morphism will be in Sec.~\ref{subsec:secondary} and the fact that it is an equivalence will be proved in two distinct ways in Appendices~\ref{app:original-proof} and \ref{app:sanath-proof}.\footnote{%
The proof given in Appendix~\ref{app:original-proof} is based on the Anderson duality of $\tmf$ and is rather computational.
Another proof given in Appendix~\ref{app:sanath-proof} is more conceptual and is based on the $\Tmf$ self-duality \eqref{eq_Tmf_duality_intro},
and was kindly provided by Sanath Devalapurkar.
}

\begin{physremark}
As discussed in the authors' previous work \cite{TachikawaYamashita},
the anomaly of heterotic string theories is characterized by  $I_\bZ\Phi\circ\alpha_{\KO((q))}$,
 whose vanishing \eqref{eq:vanishing} guarantees 
that   heterotic string theories are always anomaly-free.
The authors then considered the secondary morphism $\alpha_{\KO((q))/\TMF}$,
which can be naturally defined from this vanishing, 
in order to describe the Green-Schwarz couplings of general heterotic string theories.
It was a major surprise even to the authors
 that the secondary morphism $\alpha_{\KO((q))/\TMF}$ thus defined
  gave the Anderson duality \eqref{selfduality2} of $\TMF$.
\end{physremark}
%This choice was motivated by a physics consideration as detailed in \cite[Sec.~2.2]{TachikawaYamashita}.

As already mentioned for the case of the Anderson dual pairing of $\KO$, 
the existence of the Anderson duality \eqref{selfduality2} induces two pairings \begin{align}
\langle-,-\rangle_{\alpha_{\KO((q))/\TMF}} &\colon
\pi_d \TMF \times \pi_{-d-20} \KO((q))/\TMF \to \bZ,
\label{nontorsion-pairing}\\
(-,-)_{\alpha_{\KO((q))/\TMF}} &\colon
(\pi_d \TMF)_\tor \times (\pi_{-d-21} \KO((q))/\TMF)_\tor  \to \bQ/\bZ
\label{torsion-pairing}
\end{align}
which are perfect under suitable topologies on the groups involved, after dividing by the torsion subgroups in the case of \eqref{nontorsion-pairing}.

Thanks to the detailed data contained in \cite{BrunerRognes},
we know the structure of  the Abelian groups $\pi_\bullet \TMF$, $\pi_\bullet \KO((q))/\TMF$ 
and the pairings between them very explicitly.
These pairings associated to the Anderson duality  \eqref{selfduality2}
are explicitly tabulated in Appendix~\ref{app:duality} for the convenience of the readers.

The discussions so far are almost purely within homotopy theory, 
except possibly our motivation behind the choice of our crucial element $\alpha_{\KO((q))}$ in \eqref{eq:choice-alpha},
as explained in Physics Remark~\ref{foot:phys}.
We now connect these  considerations in homotopy theory to differential geometry of string manifolds.
As we will see, this allows us to compute the pairings of some concrete elements of $\pi_\bullet \TMF$ of negative degree
and some explicit torsion classes of string manifolds,
manifesting the Anderson duality of $\TMF$ in a differential geometric setting.

\begin{physremark}
These results in pure mathematics were heavily motivated by the study of branes in heterotic string theories in \cite{KaidiOhmoriTachikawaYonekura,Kaidi:2024cbx},
and in turn have direct applications in uncovering the previously unknown features of heterotic string theories.
\end{physremark}

For this purpose, we 
utilize the sigma orientations $\Wit_\text{string}: \MString \to \TMF$
and $\Wit_\text{spin}: \MSpin \to \KO((q))$, which were constructed in \cite{AHR10}
so that the following diagram commutes:
\begin{equation}
\vcenter{\xymatrix{
\MString \ar[d]^-{\Wit_\text{string}} \ar[r]^-{\iota} &
\MSpin \ar[d]^-{\Wit_\text{spin}} \\
\TMF \ar[r]^-{\Phi}& 
\KO((q))
}},
\end{equation}
where $\MString$ and $\MSpin$ are the Thom spectra for the string and spin manifolds, respectively.
We extend the two rows in the following way:
\begin{equation}
\vcenter{\xymatrix{
\MString \ar[d]^-{\Wit_\text{string}} \ar[r]^-{\iota} &
\MSpin \ar[d]^-{\Wit_\text{spin}} \ar[r]^-{C\iota} &
\MSpin/\MString \ar[d]^-{\Wit_\text{spin/string}} \\
\TMF \ar[r]^-{\sigma}& 
\KO((q))  \ar[r]^-{C\sigma} & 
\KO((q))/\TMF
}}.
\end{equation}
Here,  $\MSpin/\MString$ is defined to be the cofiber of the forgetful morphism $\iota$,
and is the Thom spectrum of the relative spin/string manifolds, whose homotopy groups 
give bordism groups of spin manifolds with string-structured boundaries.
The morphism $\Wit_\text{spin/string}$ was then chosen so that the diagram above commutes.
By composing it with \eqref{selfduality2}, we have the morphism \begin{equation}
\alpha_\text{spin/string} \colon \TMF\to \Sigma^{-20}I_\bZ \MSpin/\MString,
\end{equation}
which in turn gives the following more geometric versions of
the two pairings \eqref{nontorsion-pairing} and \eqref{torsion-pairing}:
\begin{align}
\langle-,-\rangle_{\alpha_\text{spin/string}} &\colon
\pi_d \TMF \times \pi_{-d-20} \MSpin/\MString \to \bZ,
\label{nontorsion-pairing-geom}\\
(-,-)_{\alpha_\text{spin/string}} &\colon
(\pi_d \TMF)_\tor \times (\pi_{-d-21} \MSpin/\MString)_\tor \to \bQ/\bZ.
\label{torsion-pairing-geom}
\end{align}
As  was also the case for $\KO$, 
these two pairings are  subsumed in a single differential pairing \begin{equation}
(-,-)_{\widehat\alpha_\text{spin/string}}  \colon\widehat\TMF^{-d} \times \reallywidehat{\MSpin/\MString}_{-d-21} \to \bR/\bZ,
\label{differential-pairing}
\end{equation}
whose properties we can and will explore.

\begin{physremark}
The physics interest of this pairing lies in the fact that it provides the Green-Schwarz couplings for a given string worldsheet theories specified by the element in $\widehat \TMF^{-d}$.
Some more details of the physics significance will be explained in Appendix~\ref{sec:introphys}.
\end{physremark}

As our construction is somewhat involved, we have an Appendix~\ref{app:toy}
where the same techniques are applied to a much simpler setting,
with the correspondence given below:
\[
\begin{array}{|c||c|c|c|c|}
\hline
\text{main text}  & \TMF & \KO((q))  & \MString & \MSpin\\
\hline
\text{App.~\ref{app:toy}}  & \KO & \K & \MSpin & \MSpin^c  \\
\hline
\end{array}.
\]

\subsection{Applications}
\label{subsec:intro-appl}
We will study the consequences of our  pairings induced from the Anderson duality of $\TMF$,
from various points of view.

\subsubsection{Determination of the invariants of Bunke and Naumann}
The first application is an alternative computation of the invariant $b^\tmf$ of Bunke and Naumann \cite{BunkeNaumann}.
This invariant was defined in \cite{BunkeNaumann} for elements $x$ in the kernel of $\sigma \colon \pi_{4k-1}\TMF\to \pi_{4k-1}\KO((q))$ by lifting them to $\pi_{4k}\KO((q))/\TMF$
and taking their rationalization. 
The rationalization of the pairing $\langle-,-\rangle_{\alpha_{\KO((q))/\TMF}}$ pairs them with the rationalization of $\pi_{-4k-20}\TMF$,
which will allow us to completely determine this invariant.
We will also see that the torsion pairing $(-,-)_{\alpha_{\KO((q))/\TMF}}$ provides torsion invariants of $\TMF$ finer than that given by Bunke and Naumann.

Bunke and Naumann also had a differential geometric version of their invariant, which they denoted by $b^\geom$, 
defined on the kernel of $\pi_{4k-1}\MString\to\pi_{4k-1}\MSpin$.
We will see that this invariant can be naturally thought  of as part of the differential pairing \eqref{differential-pairing}.

\subsubsection{Other examples of the differential pairing}

Another example is the following. 
The kernel of $\sigma: \pi_{-31}\TMF \to \pi_{-31}\KO((q))$ is known to be $\bZ/2$,
whose  generator we denote by $x$.
The torsion pairing \eqref{torsion-pairing} pairs it with torsion elements of $\pi_{10}\MSpin/\MString$.
%\begin{equation}
%\alpha_\text{spin/string}(x,-): (\pi_{10}\MSpin/\MString)_\text{torsion} \to \bZ/2,
%\label{xpairing}
%\end{equation}
%as we recalled above.
Denoting by $\nu$ a generator of $\pi_3\MString=\bZ/24$,
the element
$\nu^3\in \pi_9 \MString$ clearly vanishes when sent by $\iota$ to $\pi_9\MSpin$,
and therefore can be lifted to a class $y$ in $\pi_{10}\MSpin/\MString$.
We will see that the pairing $(x,y)_{\alpha_\text{spin/string}}$ %\eqref{torsion-pairing-geom} 
 is nontrivial if and only if $y$ is a lift of $\nu^3$.
%using our analysis and the Anderson duality of $\tmf$ studied in \cite[Chapter 10]{BrunerRognes}.

\subsubsection{Identification of a $\TMF$ class coming from a power operation}
One more application of our general result is that it sometimes allows us 
to identify elements of $\pi_\bullet \TMF$
by performing differential geometric computations of the Anderson-dual pairing.
More concretely, 
we combine the authors' previous works \cite{TachikawaYamashita,Yamashita:2021cao,YamashitaAndersondualPart2} with the results of this paper to have differential geometric understandings of the pairing $\alpha_{\spin/\stri}$. 
This enables us to understand torsion elements coming from a power operation on $\TMF$, which had been understood only by spectral sequence computations in the literature, by differential geometric methods. 

The concrete case we study in this paper is the following.
Let $e_8 \in \pi_{-16}\TMF$ be the unique element 
whose image by $\sigma$ is  $c_4/\Delta$.\footnote{%
One rationale for this notation $e_8$ is that $c_4$ is the lattice theta function for the lattice $E_8$;
further motivations will be given amply below.}
The $H_\infty$-ring structure of $\TMF$ allows us to define its square $(e_8)^2$
as an element of $\TMF^{32}(B\Sigma_2)$, 
where $\Sigma_2$ is the symmetric group on two points
representing the exchange of two factors of $e_8$.
Let $f:S^1\to B\Sigma_2$ be the classifying map of a nontrivial $\Sigma_2$ bundle over $S^1$,
and we let \begin{equation}
e_8^{(2)} :=  f^* ((e_8)^2)/[S^1] \in \TMF^{31}(\pt)=\pi_{-31}\TMF,
\end{equation} where $[S^1] \in \TMF_1(S^1)$ is the fundamental class with the bounding string structure.
This procedure is a power operation called a cup-$1$ product, as we will see in Remark \ref{rem_cup1}, and denoted by $\cup_1 \colon \pi_{2\ell}\TMF \to \pi_{4\ell+1}\TMF$.

This element $e_8^{(2)}$ is of some physics interest,
and it is fairly straightforward to check that it vanishes when sent to $\pi_{-31}\KO((q))$ via $\sigma$.
As we already mentioned, the kernel of $\sigma:\pi_{-31}\TMF\to\pi_{-31}\KO((q))$ is $\bZ/2$, whose generator we denoted by $x$ above.
Then we can ask the question: is $e_8^{(2)}$ zero or $x$?
We will answer this question in two completely distinct methods.

The first method is the following. As $e_8^{(2)}$ is obtained by a power operation in $\TMF$,
we can determine it to be $x$, using the detailed data provided in \cite{BrunerRognes} 
concerning power operations as seen in the Adams spectral sequence for $\pi_\bullet\TMF$.
This we will do in Sec.~\ref{app:power}.

The second method is presented %uses our general result concerning $\alpha_\text{spin/string}$,
%which allows us to come to the same conclusion in a completely different manner.
%This is done 
in Sec.~\ref{sec:computation},
using the differential version of the pairing \eqref{torsion-pairing-geom} above.
There, we compute   $(e^{(2)}_8,y)_{\alpha_\text{spin/string}}$  for a class $y\in \pi_{10}\MSpin/\MString$
lifting $\nu^3$ and confirm it to be nonzero,
thus showing indeed that $e^{(8)}_2$ is the generator $x$.

As described above, the first method is purely homotopy-theoretical,
while the second method uses the differential geometric pairing.
Two computations proceed in two  totally different ways,
and the eventual consistency of the results is highly nontrivial.
To the authors this was one of the most exciting and satisfying aspects of this work.

\begin{physremark}
In addition, these computations determine previously undetermined couplings of heterotic string theory of a very subtle type,
known as discrete Green-Schwarz couplings. For more details, see Appendix~\ref{sec:introphys}.
\end{physremark}

We note that our differential computation is performed using the following proposition involving the genuinely equivariant refinement of $\TMF$ developed by Gepner-Meier \cite{GepnerMeier}:
\begin{prop}
\label{prop:above}
The element $e_8$ has a natural lift to an $\RO(SU(2))$-graded genuinely equivariant $\TMF$-cohomology element, \[
\check e_8 \in \TMF_{SU(2)}^{16 + \overline{V}_{SU(2)}}(\pt)
\]
where $\overline V_{SU(2)} = V_{SU(2)} - 4\underline \R \in \RO(SU(2))$ is the virtual representation of $SU(2)$ associated to the fundamental representation $V_{SU(2)}$ of $SU(2)$. 
% Appendix~\ref{app:equivariant}.  
% where $\tau$ is a map $BSU(2)\to K(\Z,4)$ whose homotopy class gives 
%  $-c_2\in \H^4(BSU(2);\Z)$.
\end{prop}
% Here, we use the fact that we can regard $\tau$ as  specifying a twist of 
% equivariant $\TMF$ via the canonical map $K(\Z, 4) \to BO \langle 0, \ldots, 4 \rangle \to BGL_1(\TMF)$ 
% as explained in \cite{ABG}.
We will show Proposition~\ref{prop:above} in Appendix~\ref{app:equivariant}. 
% using the method of \cite{lin2024topologicalellipticgenerai}.
$\RO(G)$-grading of $\TMF$ is conjectured to be a part of genuinely $G$-equivariant twist of $\TMF$ parametrized by $[BG, BO\langle 0, 1, 2, 3, 4 \rangle]$, which is yet to be constructed. 
The proposition above is a lesser version of the following conjecture:
\begin{conj}
\label{conj:above}
The element $\check e_8$ further lifts to an $E_8$-equivariant twisted element \[
\breve e_8\in \TMF^{16+\breve\tau}_{E_8}(\pt),
\]
where $\breve\tau \colon BE_8 \to K(\Z, 4)$ is a map
which pulls back to $\tau$ above via the standard inclusion $SU(2)\subset E_8$ using a simple root.
\end{conj}
This conjecture has a good physics motivation, as we will explain below.
% We are in no position to prove this conjecture, however, 
% as the general theory of twisted equivariant $\TMF$ is so far only available for classical groups.

\subsection{Vertex operator algebras and $\TMF$}
\label{subsec:intro-ver}
Now we would like to explain the physics behind Proposition~\ref{prop:above} and Conjecture~\ref{conj:above} above,
and how it leads to our conjectures concerning vertex operator algebras and $\TMF$ classes.
The Stolz-Teichner proposal \cite{StolzTeichner1,StolzTeichner2} posits that the deformation classes of two-dimensional \Nequals{(0,1)} supersymmetric quantum field theories (SQFTs) determine $\TMF$ classes.
Among such SQFTs, two  simpler classes of theories are (i)
purely left-moving modular-invariant spin conformal field theories (CFTs) and (ii)
purely right-moving modular-invariant superconformal field theories (SCFTs).
The $\TMF$ classes coming from (ii) were the subject of a paper \cite{Gaiotto:2018ypj} and are automatically in $\pi_{d>0}\TMF$.
Our paper concerns the classes coming from (i), which are in contrast in $\pi_{d<0}\TMF$.

Mathematically, purely left-moving modular-invariant spin CFTs are believed to be described by \emph{self-dual} (or equivalently \emph{holomorphic}) vertex operator superalgebras (VOSAs), whose central charge $c$ is necessarily of the form $n/2$.
A slightly generalized form of the Stolz-Teichner proposal then has the following form,
which will be given more fully as Conjecture~\ref{conj:affinevoa}:
\begin{conj}
A holomorphic VOSA $\V$ of central charge $n/2$  containing an affine algebra $\hat{\mathfrak{g}}_k$
should give a class  $[\V]\in \TMF^{n+ k\tau}_G(\pt)$,
where $G$ is the simply-connected compact Lie group of type $\mathfrak{g}$
and {$\tau \colon BG \to K(\Z, 4)$ is
the map in \cite[Example 5.1.5]{FSSdifferentialtwistedstring}
which represents a generator of $\H^4(BG; \Z) \simeq \Z$. }
Furthermore,  $\sigma([\V])\in \KO^n_G((q))(\pt)$ can be computed using the theory of VOSA.%
\footnote{%
In the formulation of this conjecture, we use genuine $G$-equivariant $\TMF$ with twists,
whose theory has not been adequately developed yet for general $G$.
In the rest of the paper, we only use the $\RO(G)$-graded equivariant $\TMF$
about which we already have sound mathematical foundations 
for which no issue remains, 
see e.g.~\cite{GepnerMeier,BauerMeierTJF,lin2024topologicalellipticgenerai}. 
\label{footnote_genuine}}
\end{conj}

One of the simplest nontrivial VOSA is the one given by the affine algebra $(\widehat {\mathfrak{e}_8})_1$ of $c=8$,
which should give a class  $\breve e_8\in \TMF^{16+\tau}_{E_8}(\pt)$,
which is the element  in the conjecture above.
Not only that, there is a physics argument saying that the element $e_8^{(2)}$ above,
reformulated physically in the language of SQFTs,
can be continuously deformed to a VOSA based on the affine algebra $(\widehat {\mathfrak{e}_8})_2$ of $c=31/2$.
The general conjecture stated above then implies that the element $e_8^{(2)}\in \TMF^{31}(\pt)$
should have a genuine $E_8$-equivariant lift $\check e_8^{(2)}\in \TMF^{31+ 2\tau}_{E_8}(\pt)$
together with its Borel equivariant version $\hat e_8^{(2)}\in \TMF^{31+ 2\tau}(BE_8)$.
Combined, we now have a conjectural description of the generator of $\Ker(\sigma: \pi_{-31}\TMF\to\pi_{-31}\KO((q))) = \bZ/2$
in terms of a VOSA.

\subsection{Organization of the paper}
\label{subsec:intro-org}
The rest of the paper is organized as follows:
\begin{itemize}
\item
In Sec.~\ref{sec:secondary}, we define two secondary morphisms $\alpha_{\KO((q))/\TMF}$ and $\alpha_\text{spin/string}$.
We also explain that the former gives the Anderson duality of $\TMF$.
%this is one of the main results of the paper.
We give two versions of its proof, in Appendices~\ref{app:original-proof} and \ref{app:sanath-proof}, respectively.
\item
In Sec.~\ref{sec:pairings}, we study  pairings induced by our secondary morphisms
and provide a number of computations of the pairings, in both torsion and non-torsion cases.
Along the way, we compute the invariant of Bunke and Naumann,
and relate it to the differential pairings.
We also see that the Anderson dual pairing of $\TMF$ has a differential geometric manifestation.
\item
In Sec.~\ref{sec:diff}, we define two $\TMF$ classes $x_{-d} \in \pi_{-d}\TMF$  for $d=31$ and $d=28$,
where $x_{-31}$ is the element we denoted by $e_8^{(2)}$ in this introduction.
They are both in $A_{-d}:=\Ker(\sigma: \pi_{-d}\TMF\to\pi_{-d}\KO((q)))$.
We prove three statements  independently of each other:
i) that $x_{-31}$ generates $A_{-31}\simeq \Z/2$, by performing a computation of the 
power operation in the Adams spectral sequence;
ii) that $x_{-28}$ generates $A_{-28}\simeq \Z/2$, 
by computing the Anderson duality pairing differential-geometrically;
and iii) that we have $x_{-28}=x_{-31}\nu$, using an explicit twisted string bordism.
The statements ii) and iii)  combined will provide an alternative, differential-geometric proof of the statement i) which is obtained purely homotopy-theoretically.
\end{itemize}
We also have a number of Appendices:
\begin{itemize}
\item
In Appendix~\ref{app:VOA}, we formulate conjectures concerning VOSAs and $\TMF$ classes.
Among others, we describe the VOSAs which should be behind $x_{-31}$ and $x_{-28}$.
\item
In Appendix~\ref{app:proofs}, we provide two different proofs of the Anderson duality of $\TMF$, Theorem~\ref{main-intro}.
The first proof uses the Anderson duality of $\tmf$ as formulated by \cite{BrunerRognes}.
The second proof uses the Anderson self-duality of $\Tmf$ and was kindly provided by Sanath Devalapurkar.
The first proof is   computational but is more explicit,
while the second proof is conceptual but is more abstract.
%In Appendix~\ref{app:original-proof}, we give a proof of our fundamental theorem, Theorem~\ref{main-intro},
%in terms of the Anderson duality of $\tmf$. % as formulated by \cite{BrunerRognes}. 
%The proof is rather computational but is more explicit.
%\item
%In Appendix~\ref{app:sanath-proof}, we give another proof of our fundamental theorem, Theorem~\ref{main-intro}, in terms of the Anderson self-duality of $\Tmf$. The proof is more conceptual, and was provided by Sanath Devalapurkar.
\item
In Appendix~\ref{app:TMF}, we summarize the structure of $\pi_\bullet\TMF$ as Abelian groups.
In particular we describe the important torsion groups $A_\bullet$,
which is the kernel of $\sigma:\pi_\bullet\TMF\to \pi_\bullet \KO((q))$.
We also describe the effects of the Anderson duality of $\TMF$ on its homotopy groups concretely.
\item
In Appendix~\ref{app:equivariant}, we present a minimal amount of the theory $\RO(G)$ graded topological modular forms,
and prove Proposition~\ref{prop:above} and other results concerning equivariant lifts of certain elements of $\pi_\bullet\TMF$.
They are used in Sec.~\ref{sec:diff}.
%\item
%In Appendix~\ref{app:application}, we use our results on the Anderson duality 
%to give a more conceptual proof of the fact that the multiplication by $e_8$ annihilates the kernel of $\iota:\pi_\bullet \MString\to\pi_\bullet \MSpin$.
%This partially answers a question posed in \cite{TYY}.
\item
In Appendix~\ref{app:diff_bordism}, we provide and develop necessary backgrounds on differential bordism groups and their Anderson duals.
\item
In Appendix~\ref{app:toy},  we discuss a simpler analogue of the study given in the main part of the paper, 
where the roles played by $\MString$, $\MSpin$, $\TMF$, and $\KO((q))$
are replaced by $\MSpin$, $\MSpin^c$, $\KO$, $\K$, respectively.
This allows us to obtain the Anderson self-duality of $\KO$ in terms of the Anderson self-duality of $\K$.
We also study various induced pairings in this setting. 
\item
Finally in Appendix~\ref{sec:introphys}, we give a brief description
of the bulk of the contents of this paper for physicists.
%Mathematicians are not expected to understand this,
%which can be skipped in its entirety, without causing problems in the understanding of the rest of the paper.
%The other sections are all written in the mathematical style.
\end{itemize}

We note that we sometimes refer to the authors' previous work \cite{TachikawaYamashita}. 
When we explicitly refer to the equation numbers and the theorem numbers there, we always refer to the published version, or equivalently the version 2 on the arXiv,  which contains significant improvements in the exposition compared to the original version 1.
That said, familiarities in the details of \cite{TachikawaYamashita} are not assumed in this paper,
and the crucial points of \cite{TachikawaYamashita} are repeated in this article,
so that the readers can understand this paper without reading \cite{TachikawaYamashita}.

\subsection{Notations and conventions}
\label{subsec:nota-conv}
\begin{itemize}
    \item We use notations $\Z/n := \Z/n\Z$ \emph{except} in the only section written for physicists, Appendix~\ref{sec:introphys},
    where $\bZ_n:=\bZ/n\Z$ is used instead. 
     We denote the $p$-adic completion of $\Z$ by$\Z^{\wedge}_p$.
%      only appears in the proof of Proposition \ref{prop_duality_TMF}.
    \item For a spectrum $E$, we use the notations $E^{-d}$ and $E_{d}$ for $\pi_d E$, depending on the context. 
    The $p$-adic completion of $E$ is denoted by $E^{\wedge}_p$. 
    \item For a spectrum $E$, its Anderson dual spectrum is denoted by $I_\Z E$. 
    \item Given a morphism  $\alpha:E\to \Sigma^{-s}I_\bZ F$ for two spectra $E$ and $F$,
we denote the associated $\bZ$-valued and $\bQ/\bZ$-valued pairings by  \begin{equation}
\langle -,-\rangle_\alpha : \pi_d E\times  \pi_{-d-s} F \to \bZ
\label{Zpairing}
\end{equation}  and \begin{equation}
(-,-)_\alpha : (\pi_d E)_\text{torsion}\times  (\pi_{-d-s-1} F)_\text{torsion} \to \bQ/\bZ.
\label{Q/Zpairing}
\end{equation}
There are also occasions where we use the notations
 $\alpha\langle-,-\rangle$ and $\alpha(-,-)$ 
instead of $\langle-,-\rangle_\alpha$ and $(-,-)_\alpha$
for notational clarity.

Assuming that both $E$-\emph{cohomology} and $F$-\emph{homology} have differential versions
and that the differential refinement $\hat\alpha:\hat E\to \Sigma^{-s} \widehat{I_\bZ F}$ of $\alpha$ is given,
we denote the differential pairing by  \begin{equation}\label{diffpairing}
 (-,-)_{\hat\alpha}: \hat E^{-d} \times \hat F_{-d-s-1} \to \bR/\bZ
\end{equation} which subsumes both \eqref{Zpairing} and \eqref{Q/Zpairing} in the sense detailed in Sec.~\ref{subsec_diff_pairing}.    
We also denote this differential pairing by $\hat\alpha(-,-)$.
\item For a topological space $X$, we denote the associated pointed space by $X_+ := X \sqcup \pt$. 
     \item We use the formalism of twisted (co)homology developed in \cite{ABG}. 
     For a topological space $X$ and an $E_\infty$ ring spectrum $R$, a map $\tau \colon X \to BGL_1(R)$ specifies a twist. 
     We denote by $R \wedge_{\tau} X_+$ the associated $R$-module Thom spectrum, whose homotopy groups are the twisted $R$-homology groups, $R_{n+\tau}(X) := \pi_n(R\wedge_{\tau} X_+)$. 
     Moreover, given an $R$-module spectrum $E$, the associated Thom spectrum is also denoted by $E \wedge_{\tau}X_+ := E \wedge_R R \wedge_{\tau} X_+$, which represents the twisted $E$-homology groups, $E_{n+\tau}(X) := \pi_{n}(E \wedge_{\tau} X_+)$. 
     Analogously, we denote the twisted cohomology groups by $R^{n+\tau}(X)$ and $E^{n+\tau}(X)$.
\item We use the notation \begin{equation}
\mf=\bZ[c_4,c_6,\Delta]/(c_4^3-c_6^2-1728\Delta),\qquad
\MF=\mf[\Delta^{-1}]
\end{equation}
for the ring of integral modular forms and the ring of weakly-holomorphic integral modular forms, respectively.
It seems more common in the literature on $\TMF$ to use $\MF$ to denote what we denote by $\mf$,
but we believe our usage is more consistent, since we naturally have homomorphisms $\pi_{2\bullet}\tmf\to \mf_{\bullet}$ 
and $\pi_{2\bullet}\TMF\to\MF_{\bullet}$ in our notation.

We typically identity elements of $\mf$ and $\MF$ with their $q$-expansions {so that we have an inclusion $\MF_k \hookrightarrow \Z((q))$ for each even $k$.}
We define $-\big|_{q^0}$ to be the constant term of a $q$-series.
{
    \item We denote the complex Bott element by $\beta \in \pi_2 \K$. When we identify $\pi_{2n} \K \simeq \Z$, we always take the sign so that $\beta^n$ maps to $1$. 
 We always identify 
    \begin{align}
        \pi_{8k+4}\KO &\simeq 2\Z \subset \Z, \label{eq_KO_convention}\\
        \pi_{8k}\KO &\simeq \Z,
    \end{align}
    via $\pi_{4n}\KO \xhookrightarrow{c}\pi_{4n}\K \simeq \Z$, where $c$ denotes the complexification,
    \emph{except} in Sec.~\ref{subsubsec_BM_tmf}.
    The reason for this exceptional convention in Sec.~\ref{subsubsec_BM_tmf} is explained in Remark \ref{rem_exceptional_convention}. 
    }
    
     \item We denote by $X\in \pi_{576}\tmf$ the periodicity element.
The notation for the Bott element $B\in \pi_8\tmf$ follows that of \cite{BrunerRognes}.
As explained there, there are two elements $B$ and $\tilde B=B+\epsilonelem$ in $\pi_8\tmf$
whose modular form image is $c_4$, distinguished by their Adams filtration: 
$B$ is in filtration 4, while $B+\epsilonelem$ is in filtration 3.
    \item We use the following names for the morphisms between spectra, whose descriptions are given in the main text:
    \begin{equation}\label{diag_notation}
\vcenter{\xymatrix{
\MString \ar[d]^-{\Wit_\stri} \ar[r]^-{\iota} &
\MSpin \ar[d]^-{\Wit_\spin} \ar[r]^-{C\iota} &
\MSpin/\MString \ar[d]^-{\Wit_{\stri/\spin}} \\
\TMF \ar[r]^-{\sigma}& 
\KO((q))  \ar[r]^-{C\sigma} & 
\KO((q))/\TMF \\
\TMF \ar@{=}[u] \ar[r]^-{\phi} &
\TMF[B^{-1}] \ar[u]_-{p} \ar[r]^-{C\phi} &
\TMF/B^\infty \ar[u]_-{\varrho} 
}}.
\end{equation}

    \item We often consider the Abelian group \begin{equation}
    A_\bullet := \ker(\sigma: \pi_\bullet \TMF \to \pi_\bullet \KO((q))).
\end{equation}   This group equals \begin{equation}
    \Gamma_B \pi_\bullet \TMF = \ker(\phi: \pi_\bullet \TMF \to \pi_\bullet \TMF[B^{-1}])),
\end{equation}
the subgroup of $B$-power torsion elements,
as we will recall in Appendix~\ref{app:TMF}.
\item We often refer to the generator $\nu$ of  $\pi_3\mathbb{S}\simeq\pi_3\MString\simeq \pi_3\TMF\simeq \Z/24$,
given by $S^3$ equipped with the Lie group framing.
\item For a manifold $X$ and a graded $\R$-vector space $V^\bullet$, we denote by $\Omega^d(X; V^\bullet)$ the $\R$-vector space of differential forms on $X$ with coefficient in $V^\bullet$ with {\it total} degree $d$. 
The subspace of closed forms are denoted by $\Omega_\clo^*(X; V^\bullet)$. We write $\Omega^*(X):= \Omega^*(X; \R) $. 
    \item We heavily use the language of generalized differential cohomology in the latter half of the paper, starting from Sec.~\ref{subsec_diff_pairing} and onwards.
    We use the axiomatic framework given in \cite{BunkeSchickUniqueness}
    and adopt the conventional notations: a differential cohomology refining $E$ is typically denoted by $\widehat{E}^*$, and has structure maps 
    \begin{align}
    \vcenter{\xymatrix{
    \Omega^{*-1}(X ; E_\R^\bullet) / \mathrm{im}(d) \ar[r]^-{a} &\widehat{E}^*(X) \ar[d]^-{R} \ar[r]^-{I}& E^*(X) \\
    &\Omega_\clo^*(X; E_\R^\bullet)&
    }}.
    \end{align}
    See also Appendix~\ref{app:diff_bordism}.
\end{itemize}
\section{Secondary morphisms and the Anderson duality}
\label{sec:secondary}
%Here we establish the existence of two secondary morphisms 
% $\alpha_\text{spin/string}$, $\alpha_{\KO((q))/\TMF}$ in \eqref{eq:secondary-morphisms-definitions}.
% We also mention that $\alpha_{\KO((q))/\TMF}$ is a version of the Anderson self-duality of $\TMF$.
% which  will be given in two distinct ways in Appendices~\ref{app:original-proof} and \ref{app:sanath-proof}.

 \subsection{The secondary morphism $\alpha_{\KO((q))/\TMF}$}
 \label{subsec:secondary}

 We start with a basic lemma, Lemma~\ref{lemma:basic}, which is simple but so fundamental that it will be implicitly used repeatedly below.
 We need to start with a couple of definitions to state the lemma.

\begin{defn}
 \label{defn:lor}
     Given two spectra $E$ and $F$ and a morphism $\alpha \colon E \to \Sigma^{-s}I_\Z F$, we denote by $\alpha^\lor$ the following composition:
     \begin{align*}
     \alpha^\lor \colon F \to I_\Z (I_\Z F) \xrightarrow{I_\Z \alpha} \Sigma^{-s}I_\Z E, 
     \end{align*}
     where the first arrow is induced by the generator $\id_{I\Z}$ of $[S^0, F(I\Z, I\Z)] \simeq [I\Z, I\Z] \simeq \Z$. 
 \end{defn}

 \begin{rem}
 In this paper we often deal with spectra such as $\TMF$ and $\KO((q))$ whose homotopy groups at fixed degrees are not finitely generated.
 In such cases we have a canonical morphism $E\to I_\bZ I_\bZ E$ but not an equivalence $E \xrightarrow{\sim} I_\bZ I_\bZ E$.
 This subtlety requires us to distinguish $I_\bZ \alpha$ and $\alpha^\vee$.
 \end{rem}

 \begin{defn}
 \label{defn:mor}
 Given a ring spectrum $R$, an $R$-module spectrum $M$ and a morphism $\alpha \colon M \to \Sigma^{-d}I_\Z$,
 the composition \begin{equation}
 M\wedge R \xrightarrow{\text{mult}} M \xrightarrow{\alpha} \Sigma^{-d}I\bZ 
 \label{fact}
 \end{equation} determines two $R$-module morphisms \begin{align}
 \text{mor}(\alpha)&\colon M\to \Sigma^{-d}I_\bZ R, &
 \text{mor}(\alpha)^\vee&\colon R\to \Sigma^{-d}I_\bZ M.
 \end{align}
 \end{defn}

 \begin{lem}
 \label{lemma:basic}
 Any $R$-module morphism 
 \begin{align*}
      M &\to  \Sigma^{-d} I_\bZ  R 
 \end{align*} 
 is of the form $\text{mor}(\alpha)$ for a map $\alpha \colon M \to \Sigma^{-d}I_\Z$,
 and any $R$-module morphism 
 \begin{align*}
      R &\to  \Sigma^{-d} I_\bZ  M 
 \end{align*} 
 is of the form $\text{mor}(\alpha)^\vee$ for a map $\alpha\colon M \to \Sigma^{-d}I_\Z$.
 \end{lem}

 \begin{proof}
 From the definition of the Anderson dual, 
 any such morphism is determined by a morphism $M\wedge R \to \Sigma^{-d}I\bZ$.
 Due to our assumption that it is an $R$-module morphism,  it can be factored as in \eqref{fact}.
 \end{proof}

\begin{rem}
 \label{rem:mor}
 In the following, we simply write $\text{mor}(\alpha)$ by $\alpha$,
 and $\text{mor}(\alpha)^\vee$ by $\alpha^\vee$.
 \end{rem}

To proceed, we need three facts about modular forms and topological modular forms:
\begin{fact}
\label{fact:mf}
For all $f(q)\in (\MF_2)_\bQ$, we have $f(q)\big|_{q^0}=0$,
where the notation $\cdots \big|_{q^0}$ means the extraction of the constant term from a series in $q$.
\end{fact}
\begin{fact}
\label{fact:tmf21}
$\pi_{-21}\TMF=0$.
\end{fact}
\begin{fact}
\label{fact:tmf22}
The morphism $\pi_{-22}\TMF\to \pi_{-22}\KO((q))$ is injective.
\end{fact}

Fact~\ref{fact:mf} is a classic result in the theory of modular forms. For a proof, see e.g.~\cite[Lemma 3.14]{TachikawaYamashita}.
Fact~\ref{fact:tmf21} and Fact~\ref{fact:tmf22} follow from the properties of the homotopy groups of $\TMF$ summarized in Appendix~\ref{app:TMF}.

Let us now consider the $\KO((q))$-module morphism $\alpha_{\KO((q))}: \KO((q))\to \Sigma^{-20}I_\bZ\KO((q))$.
Due to the Lemma \ref{lemma:basic}, its homotopy class is specified by an element in $\pi_{20}I_\bZ \KO((q))$:
\begin{defn}
\label{eq_def_alpha_spin}
We define the element $\alpha_{\KO((q))}\in\pi_{20}I_\Z\KO((q))$ via
\[
    \alpha_{\KO((q))} = {\frac12\Delta(q) \cdot - \big|_{q^0} \in \mathrm{Hom}(2\Z((q)), \Z)}
\]
under the isomorphism 
\[
\pi_{20} I_\Z \KO((q)) \simeq \Hom(\pi_{-20}\KO((q)), \Z) \simeq {\Hom(2\Z((q)), \Z)}.
\]
\end{defn}
This choice of $\alpha_{\KO((q))}$ was motivated by a string theoretic consideration, see \cite[Sec.~2 and Lemma~3.8]{TachikawaYamashita} for details.
The main mathematical result of the authors' previous work \cite{TachikawaYamashita} was the following:
\begin{thm}[=Theorem 3.17 of \cite{TachikawaYamashita}]
\label{oldmain}
The element $I_\bZ \sigma (\alpha_{\KO((q))}) \in \pi_{20}I_\bZ \TMF$ is zero.
\end{thm}
\begin{proof}
By Fact~\ref{fact:tmf21}, we have $\pi_{-21}\TMF=0$. Therefore we have the isomorphism \begin{equation}
 \pi_{20}I_\bZ \TMF=\Hom(\pi_{-20}\TMF,\Z).
\end{equation}
Therefore, we only have to show the vanishing of $I_\bZ \sigma (\alpha_{\KO((q))} ) $ 
after rationalization, which takes values in $\Hom((\MF_{-10})_\Q,\Q)$.
From Definition~\ref{eq_def_alpha_spin}, we only have to show that \begin{equation}
\Delta(q) f(q) \big|_{q^0} =0
\end{equation} for any $f(q)\in (\MF_{-10})_\Q$.
This follows by recalling that $\Delta(q)\in \MF_{12}$ and Fact~\ref{fact:mf}.
\end{proof}

Define the $\TMF$-module spectrum $\KO((q))/\TMF$ as the cofiber
\begin{align}\label{eq_def_F}
    \TMF \stackrel{\sigma}{\to} \KO((q)) \stackrel{C\sigma}{\to} \KO((q))/\TMF. 
\end{align}
We then have the following proposition:
\begin{prop}%\label{lem_alpha_F}
There is a unique elelemt $\alpha_{\KO((q))/\TMF} \in \pi_{20} I_\Z  \KO((q))/\TMF$ which satisfies
\begin{align}
    I_\Z C\Phi (\alpha_{\KO((q))/\TMF} ) = \alpha_{\KO((q))} .  
\end{align}

% lift of $ \alpha_{\KO((q))} \in \pi_{20} I_\Z \KO((q))$ to  
% $\pi_{20} I_\Z  \KO((q))/\TMF$,
% which we denote by  $\alpha_{ \KO((q))/\TMF}$. 
\end{prop}
\begin{proof}
The long exact sequence for the homotopy fibration Anderson dual to \eqref{eq_def_F} gives
\begin{align}
    \pi_{21}I_\bZ \KO((q)) \xrightarrow{I_\Z \Phi} \pi_{21}I_\bZ \TMF \to \pi_{20} I_\Z  \KO((q))/\TMF \xrightarrow{I_\Z C\Phi} \pi_{20}I_\bZ \KO((q)) \xrightarrow{I_\Z \Phi}  \pi_{20}I_\bZ \TMF. 
\end{align}
From Theorem~\ref{oldmain} we know that $ \alpha_{\KO((q))} \in \pi_{20} I_\Z \KO((q))$ is in the kernel of the last map. Thus there is a lift. 
To show the uniqueness, it is enough to show that the first map is surjective. 

For this we notice that
$\pi_{21}I_\bZ \TMF\simeq \Ext(\pi_{-22}\TMF , \Z) $ 
because $\pi_{-21}\TMF = 0$ by Fact~\ref{fact:tmf21}.
Similarly, we have 
$\pi_{21}I_\bZ \KO((q)) \simeq \Ext(\pi_{-22}\KO((q)), \Z)$.
Now the map $\sigma:\pi_{-22}\TMF \to \pi_{-22}\KO((q)) $ is injective by Fact~\ref{fact:tmf22}.
The morphism  \begin{equation}
\mathrm{Ext}(\pi_{-22}\KO((q)) , \Z) \to \mathrm{Ext}(\pi_{-22}\TMF , \Z)
\end{equation}
is then a surjection, and we get the desired result. 
\end{proof}

\begin{prop}
\label{prop:KO/TMF}
We have \begin{equation}
\pi_{-21}\KO((q))/\TMF=0
\end{equation}
and therefore 
\begin{equation}
\pi_{20} I_\Z  \KO((q))/\TMF
\simeq\Hom(\pi_{-20}\KO((q))/\TMF, \Z).
\label{eq:lem_alpha_F}
\end{equation}
\end{prop}
\begin{proof}
Consider the  exact sequence \begin{equation}
\pi_{-21}\KO((q)) \to \pi_{-21}\KO((q))/\TMF \to \pi_{-22}\TMF \to \pi_{-22} \KO((q)).
\end{equation}
Then the statement follows because the first term is zero and the last arrow is injective by Fact~\ref{fact:tmf22}.
\end{proof}

\begin{prop}
\label{prop:KO/TMF20}
We have \begin{equation}
\pi_{-20}\KO((q))/\TMF \simeq {2\bZ((q))/2\MF_{-10}}.
\end{equation}
\end{prop}
\begin{proof}
We study the exact sequence \begin{equation}
\pi_{-20}\TMF\to \pi_{-20}\KO((q)) \to \pi_{-20}\KO((q))/\TMF \to \pi_{-21}\TMF.
\end{equation}
From Fact~\ref{fact:imageZ}, the image of the first arrow is $2\MF_{-10} \subset \pi_{-20}\KO((q)) \simeq 2\bZ((q))$.
From Fact~\ref{fact:tmf21}, the last term is zero.
The statement then follows.
\end{proof}

\begin{prop}
\label{rem:rat}
The element \begin{equation}
\alpha_{\KO((q))/\TMF} \in \pi_{20} I_\bZ (\KO((q))/\TMF) \simeq \Hom(\pi_{-20}\KO((q))/\TMF,\bZ)
\end{equation}
is given by \begin{equation}
\label{alpha-KO/TMF}
\begin{aligned}
{2\bZ((q))/2\MF_{-10}}  & \to \bZ, \\
[f(q)]  &\mapsto {\frac{1}{2}\Delta(q) f(q) \big|_{q^0}}.
\end{aligned}
\end{equation}
\end{prop}
\begin{proof}
This follows from the form of $\alpha_{\KO((q))}$ given in Definition~\ref{eq_def_alpha_spin}
and Proposition~\ref{prop:KO/TMF20}.
That this homomorphism is well-defined follows from the fact that elements in $\MF_{-10}$ multiplied by $\Delta$ is in $\MF_2$, whose constant term in the $q$-expansion is always zero by Fact~\ref{fact:mf}.
\end{proof}

\subsection{Anderson duality of topological modular forms and $\alpha_{\KO((q))/\TMF}$}
\label{sec:xxx}

%We constructed above a $\TMF$-module homomorphism $
%    \alpha_{\KO((q))/\TMF} \colon \KO((q))/\TMF \to \Sigma^{-20} I_\Z \TMF
%$.

We just defined the element $\alpha_{\KO((q))/\TMF}$.
This  induces $\TMF$-module homomorphisms
\begin{align}\label{eq_alpha_F_TMF}
    \alpha_{\KO((q))/\TMF} &\colon \KO((q))/\TMF \to \Sigma^{-20} I_\Z \TMF, \\
    \alpha^{\vee}_{\KO((q))/\TMF} &\colon \TMF \to \Sigma^{-20} I_\Z \KO((q))/\TMF. 
\end{align}
A fundamental result in this paper is that the morphism  $\alpha_{\KO((q))/\TMF}$  \eqref{eq_alpha_F_TMF}
 is actually an equivalence:
\begin{thm}\label{main}
 The secondary morphism   is a $\TMF$-module equivalence:
    \begin{align}\label{eq_main_thm}
        \alpha_{\KO((q))/\TMF} \colon \KO((q))/\TMF \stackrel{\sim}{\to} \Sigma^{-20}I_\Z \TMF.
    \end{align}
  \end{thm}
We give two separate proofs of this theorem,
one given by the authors of this paper originally in the first version of this paper,
and another more conceptual version provided by Sanath Devalapurkar.
As they are both somewhat technical, they are given in Appendix~\ref{app:original-proof}
and in Appendix~\ref{app:sanath-proof}, respectively.

In this main part, we give a very brief outline of the proofs. % and explain along the way that  this equivalence should be thought of as expressing the Anderson self-duality of $\TMF$.
Originally, the Anderson duality of topological modular forms was formulated and proved by Stojanoska \cite{Sto1,Sto2} as the Anderson self-duality of $\Tmf$ as follows:
\begin{fact}[\cite{Sto1,Sto2}]
    We have a $\Tmf$-module equivalence
    \begin{align}
        \alpha_{\Tmf} \colon \Tmf \simeq \Sigma^{-21} I_\bZ \Tmf. 
    \end{align}
\end{fact}
This was later reformulated in terms of $\tmf$ by \cite{BrunerRognes}. 
Following their notations, let $B \in \pi_{8}\tmf \subset \pi_8 \TMF$ be the Bott element,
whose modular form image is $c_4$.\footnote{%
There are two elements $B$ and $B+\epsilonelem$ whose image is $c_4$, where $\epsilonelem$ is of order 2.
They are distinguished by their Adams filtrations. 
}
Let $X \in \pi_{576}\tmf$ be the periodicity element, whose modular form images is $\Delta^{24}$.
Now,  given a ring spectrum $R$ and an element $x\in \pi_\bullet R$, 
we define the spectrum $R/(x^\infty)$ by the cofiber sequence \begin{equation}
R\to R[x^{-1} ] \to R/(x^\infty),
\end{equation}  and $R/(x^\infty,y^\infty):=(R/(x^\infty))/(y^\infty)$.
The Anderson duality in terms of $\tmf$ is then:
\begin{fact}[\cite{BrunerRognes}]
\begin{align}
\label{eq:aaaaaa}
\alpha_{\tmf'}\colon \tmf' \simeq \Sigma^{-20}I_{\Z} \tmf,\qquad  \text{where}\quad \tmf' := \tmf / (B^\infty, X^\infty).
\end{align}
\end{fact}

The approach in Appendix~\ref{app:original-proof} is as follows.
We consider $\TMF/B^\infty$ defined by  the cofiber sequence \begin{equation}
\xymatrix{
\TMF  \ar[r]^-{\phi} &
\TMF[B^{-1}]  \ar[r]^-{C\phi} &
\TMF/B^\infty \\
}.
\end{equation} 
It is then fairly straightforward to construct from the morphism $\alpha_{\tmf'}$ in \eqref{eq:aaaaaa} a morphism \begin{equation}
     \alpha_{\TMF/B^\infty} :\TMF/B^\infty \to \Sigma^{-20} I_\bZ \TMF
\end{equation} which turns out to be almost an equivalence in the sense that \begin{equation}
\pi_\bullet \TMF/B^\infty \simeq N'_\bullet\otimes \bZ[X^{-1},X],\qquad
\pi_\bullet \Sigma^{-20}I_\bZ\TMF \simeq N'_\bullet\otimes \bZ[X^{-1},X]].
\end{equation} for a certain $\bZ[B]$-module $N'_\bullet$.

We now use the following commutative diagram  \begin{align}
\vcenter{\xymatrix{
\TMF \ar@{=}[d] \ar[r]^-{\phi} &
\TMF[B^{-1}] \ar[d]^-{p} \ar[r]^-{C\phi} &
\TMF/B^\infty \ar[d]^-{\varrho} \\
\TMF \ar[r]^-{\sigma}& 
\KO((q))  \ar[r]^-{C\sigma} & 
\KO((q))/\TMF
}},
\label{TMFd}
\end{align} 
where $p$ is induced by the invertibility of $\sigma(B)=c_4$ in $\pi_\bullet \KO((q))$,
and $\varrho$ is filled to make the diagram commute.
This diagram allows us to show that \begin{equation}
\pi_\bullet\KO((q))/\TMF \simeq N'_\bullet\otimes \bZ[X^{-1},X]].
\end{equation}
Combined with the fact that we have a morphism $\alpha_{\KO((q))/\TMF}\colon \KO((q))/\TMF\to \Sigma^{-20}I_\bZ \TMF$, 
we can conclude that $\alpha_{\KO((q))/\TMF}$ is indeed an equivalence.
% \begin{equation}
%\alpha_{\KO((q))/\TMF} \colon \KO((q))/\TMF \stackrel{\sim}{\to} \Sigma^{-20}I_\Z \TMF.
%\end{equation} 
%All the details will be provided in Appendix~\ref{app:original-proof}.

Another proof, provided in Appendix~\ref{app:sanath-proof}, is more conceptual and is based on the $\Tmf$ version of the Anderson duality statement.
We will start by establishing the homotopy pullback square
    \begin{align}
        \vcenter{\xymatrix{
\Tmf \ar[r] \ar[d]^-{\sigma'} & \TMF \ar[d]^{\sigma} \\
\KO[[q]] \ar[r] & \KO((q))
        }}.
    \end{align}
This can then be used to show that the following diagram commutes:
    \begin{align}
        \vcenter{\xymatrix{
         \Tmf \ar[r] \ar[dd]_-{\simeq }^-{\alpha_{\Tmf}} 
         & \KO[[q]] \ar[d]_-{\simeq}^-{\alpha_{\KO((q))/\KO[[q]]}} \ar[r] 
         & \KO[[q]]/\Tmf \ar[d]^-{\simeq} \\
         & \Sigma^{-20}I_\Z \left(\KO((q))/\KO[[q]]\right)  \ar[d]_-{\simeq} 
         & \KO((q))/\TMF \ar[d]^-{\alpha_{\KO((q))/\TMF}}  \\
         \Sigma^{-21}I_\Z \Tmf \ar[r] & \Sigma^{-20}I_\Z \left(\TMF/\tmf\right) \ar[r] &\Sigma^{-20}I_\Z \TMF 
        }},
    \end{align}
where $\alpha_\Tmf$ on the left is the $\Tmf$ Anderson self-duality
and $\alpha_{\KO((q))/\KO[[q]]}$ in the middle is constructed from the $\KO$ Anderson self-duality.
This will allow us to conclude that $\alpha_{\KO((q))/\TMF}$ on the right is also an equivalence.

\subsection{The secondary morphism $\alpha_\text{spin/string}$}

We now consider more geometric versions of the morphism $\alpha_{\KO((q))/\TMF}$ introduced in Sec.~\ref{subsec:secondary}.
This will allow us to explore the consequence of the Anderson duality of $\TMF$ in a differential geometric setting later in this paper.

We start by recalling that there are   sigma orientations $\Wit_\text{string}$ and
$\Wit_\text{spin}$,  constructed in \cite{AHR10},
fitting in the following commutative diagram
\begin{equation}
\vcenter{\xymatrix{
\MString \ar[d]^-{\Wit_\text{string}} \ar[r]^-{\iota} &
\MSpin \ar[d]^-{\Wit_\text{spin}} \ar[r]^-{C\iota} &
\MSpin/\MString \ar[d]^-{\Wit_\text{spin/string}} \\
\TMF \ar[r]^-{\sigma}& 
\KO((q))  \ar[r]^-{C\sigma} & 
\KO((q))/\TMF
}}\label{sigma-cd}
\end{equation}
where the morphism $\Wit_\text{spin/string}$ is induced from the commutativity of the left square. 
Here we remind the reader that $\MSpin/\MString$ is the Thom spectrum for the relative spin/string bordism,
so that a compact spin manifold $N$ of dimension $d$ whose boundary $M=\partial N$ is equipped with a compatible string structure determines a bordism class $[N,M]\in \pi_d\MSpin/\MString$.

We now define our geometric morphisms:
\begin{defn}[{$\alpha_{\spin}$ and $\alpha_{\spin/\stri}$}]\label{def_alpha_rel}
    We define two morphisms 
    $\alpha_{\spin}$ and % \colon \TMF \to \Sigma^{-20}I_\Z \MSpin$     and
    $\alpha_{\spin/\stri }$ via  %\colon \TMF \to \Sigma^{-20}I_\Z \MSpin/\MString$ to be 
    the following compositions:
    \begin{align}
        \alpha_{\spin } &\colon \TMF \xrightarrow{ \alpha^{\vee}_{\KO((q))}} \Sigma^{-20}I_\Z \KO((q)) \xrightarrow{I_\Z \Wit_\spin} \Sigma^{-20}I_\Z \MSpin,\\ 
        \alpha_{\spin/\stri } &\colon \TMF \xrightarrow{ \alpha^{\vee}_{\KO((q))/\TMF}} \Sigma^{-20}I_\Z \KO((q))/\TMF \\
        & \qquad\qquad\xrightarrow{I_\Z \Wit_{\spin/\stri}} \Sigma^{-20}I_\Z \MSpin/\MString. \nonumber
    \end{align}
\end{defn}

\begin{physremark}
    Actually, it is possible to define $\alpha_{\spin/\stri}$ directly from the morphism $\alpha_{\spin}: \TMF \to \Sigma^{-20}I_\Z \MSpin$, without using $\KO((q))/\TMF$, but instead using the uniqueness result given in the following proposition. 
   This  proposition is not logically necessary for the rest of the paper, but we include it because this alternative way of defining $\alpha_{\spin/\stri}$ is  physically more natural than Definition \ref{def_alpha_rel}. 
\end{physremark}

    \begin{prop}
\label{lem:a_spin/string}
The morphism $\alpha_{\spin/\stri}$ in Definition \ref{def_alpha_rel} is uniquely, up to homotopy, characterized as an $\MString$-module map which fits into the homotopy commutative diagram of $\MString$-modules,
    \begin{align}\label{diag_alpha_rel_uniqueness}
        \vcenter{\hbox{\xymatrix{
        \TMF \ar[d]_-{\alpha_\text{spin/string}} \ar[r]^-{ \sigma}& \KO((q)) \ar[d]_-{\alpha_\spin}\\
    \Sigma^{-20}I_\Z \MSpin/\MString \ar[r]_-{I_\bZ C\iota}& \Sigma^{-20} I_\Z \MSpin
        }}}
    \end{align}
\end{prop}

\begin{proof}
The commutativity of the diagram \eqref{diag_alpha_rel_uniqueness} follows directly from the definition of $\alpha_{\spin/\stri}$. 
To show the uniqueness, we use the exact sequence, 
    \begin{equation}
    \label{eq_proof_secondary_lift}
    \begin{aligned}
   & [\TMF, \Sigma^{-21}I_\Z \MSpin]_{\MString} \to [\TMF, \Sigma^{-21}I_\Z \MString]_{\MString} \\
        & \to [\TMF, \Sigma^{-20}I_\Z \MSpin/\MString]_{\MString} 
    \end{aligned}
\end{equation}
where $[M,N]_R$ for two $R$-modules $M,N$ stands for the Abelian group
of $R$-morphisms from $M$ to $N$.
It is enough to show that the first map of \eqref{eq_proof_secondary_lift} is surjective. 
We prove a stronger claim that the composition \begin{equation}
\begin{gathered}
[\KO((q)), \Sigma^{-21}I_\Z \MSpin]_{\MSpin} 
 \to [\KO((q)), \Sigma^{-21}I_\Z \MSpin]_{\MString} \\
 \to [\TMF, \Sigma^{-21}I_\Z \MSpin]_{\MString} 
 \to [\TMF, \Sigma^{-21}I_\Z \MString]_{\MString} 
\end{gathered}
\end{equation} is already a surjection.
For this, we note
\begin{equation}
    [\TMF, \Sigma^{-21}I_\Z \MString]_{\MString} %&\simeq [\MString, \Sigma^{-21}I_\Z \TMF]_{\MString} \\
    \simeq\pi_{21}I_\Z \TMF \simeq \mathrm{Ext}(\pi_{-22}\TMF , \Z) 
\end{equation}
where the last isomorphism follows by $\pi_{-21}\TMF = 0$, which is Fact~\ref{fact:tmf21}.
Similarly, we have \begin{equation}
    [\KO((q)), \Sigma^{-21}I_\Z \MSpin]_{\MSpin} 
    \simeq\pi_{21}I_\Z \KO((q)) \simeq \mathrm{Ext}(\pi_{-22}\KO((q)) , \Z) .
\end{equation}
Then, the surjectivity we are after follows from the injectivity of $\pi_{-22}\TMF\to \pi_{-22}\KO((q))$,
which is Fact~\ref{fact:tmf22}.
\end{proof}

Let us pause here to note that the realification\footnote{%
Rationalization is enough, but we choose to use realification, which is more in line with differential refinements we use later.} 
of the sigma orientations have explicit expressions in terms of characteristic classes.
First, we can express $\Wit_\text{spin}$ applied to spin manifolds $N$ of dimension $d=4k$ as \begin{equation}
\Wit_\text{spin} ([N])
 = \int_N  \hat A(TN) \ch(\WitVec(TN)))  \label{WittenGen}
\end{equation}
where 
$\hat A$ is the A-roof genus, $\ch$ is the Chern character, and we define \begin{equation}
\WitVec(V) = \frac{\eta(q)^{d}}{q^{d/24}}  \bigotimes_{\ell\ge 1}
\bigoplus_{k\ge 0} q^{\ell k} \mathrm{Sym}^k V.
\label{WitVec}
\end{equation}
for a real bundle $V$ of dimension $d$, where $\eta(q)$ is the Dedekind eta function.
This formula also determines $ \sigma\circ \Wit_\text{string}$ in view of the commutative diagram \eqref{sigma-cd}.

The expression \eqref{WittenGen} goes back to \cite{Witten:1986bf},
where it was found that it defines an integral modular form when $\frac12p_1(N)\in \H^4(N;\bZ)$ vanishes rationally;
this observation was one source of the developments leading to the construction of topological modular forms.
Correspondingly, we have the factorization \begin{equation}
\widehat{A}(T)\ch (\mathop{\WitVec}(T) ) = \frac12p_1(T) \Witt(T) \label{factorized-characteristic-class}
\end{equation} where $T$ is the universal tangent bundle over $\MSpin$
and $\Witt$ is a certain characteristic class.

From this information we can find the realified expression for $\Wit_\text{spin/string}$:
\begin{prop}
\label{prop_varpi_integration}
For $[N, M] \in \pi_{4k}(\MSpin/\MString)$, we have
    \begin{multline}
    (\Wit_{\spin/\stri})_\R ([N, M] ) =\\
     \int_N \cw_{g_N}(\hat A(TN) \ch(\WitVec(TN))) - \int_M H \wedge \cw_{g_M}(\Witt(TM)) \mod (\MF_{2k})_\R.
\end{multline}
where $\cw_{g}$ is the Chern-Weil construction using the metric $g$
and $H$ is the 3-form satisfying $dH=p_1(T)/2$ which is part of the string structure on $M$.
\end{prop}
\begin{proof}
This is contained in the proof of \cite[Theorem 4.2]{BunkeNaumann}. 
To give an outline,
the factorization \eqref{factorized-characteristic-class} and the fact that the string structure on $M$ provides a trivialization of $p_1/2$ on $M$ 
means that $\hat A(TN)\ch(\WitVec(TN))$ defines a relative cohomology class in $\H^{4k}(N,M;\mathbb{R})$,
and then $(\Wit_\text{spin/string})_\bR([N,M])$ is its pairing with respect to the fundamental class.
The equation above is simply its expression in terms of differential forms.
\end{proof}

\noindent These facts will be used to evaluate the pairing induced by $\alpha_\text{spin/string}$ in the following sections.

\section{Pairings induced by secondary morphisms}
\label{sec:pairings}

Let us move on to the study of the pairing induced by our secondary morphisms on the homotopy groups of $\TMF$, $\KO((q))/\TMF$ and $\MSpin/\MString$.

\subsection{Induced non-torsion pairing and the invariant of Bunke and Naumann}

We start by considering the non-torsion part of the pairing 
and study its relation to the invariant of Bunke and Naumann.
The non-torsion pairing is defined as follows:
\begin{defn}
\label{def_int_pairing}
Given a morphism of spectra $\alpha:E\to \Sigma^{-s} I_\bZ F$
and $x\in \pi_d E$,
$\alpha(x)\in \pi_{d+s}I_\bZ F$ determines an element
$\alpha_x \in \Hom(\pi_{-d-s}F,\Z)$.
Using this, we define the pairing
\[
\langle -,-\rangle_\alpha : \pi_d E\times  \pi_{-d-s} F \to \bZ
\]  by
\[
\langle x,y\rangle_\alpha := \alpha_x(y).
\]
\end{defn}

%It directly follows from the definition of $I_\Z$ that the non-torsion pairing factors through the corresponding pairing after realification. 
The realification $\alpha_\R \colon E_\R \to \Sigma^{-s}(I_\Z F)_\R \simeq \Sigma^{-s}I_\R F_\R$ induces the pairing
\begin{align}\label{eq_R_pairing}
    \alpha_\R \colon \pi_d E_\R \times \pi_{-d-s}F_\R \to \R.  
\end{align}
The pairing of Definition~\ref{def_int_pairing} factors through the pairing \eqref{eq_R_pairing} in the following sense: for any $x \in \pi_d E$ and $y \in \pi_{-d-s}F$ with realifications $x_\R \in \pi_d E_\R$ and $y_\R\in \pi_{-d-s}F_\R$, we have the following equiality in $\R$, 
\begin{align}
    \langle x, y \rangle_\alpha = \alpha_\R(x_\R, y_\R).
\end{align}
We will often describe such a situation by a diagram, which is in this case
\begin{align}
    \vcenter{\xymatrixcolsep{0pc}
    \xymatrix{
    \pi_d E \ar[d]^-{\otimes \R}&\times& \pi_{-d-s}F\ar[rrrrrr]^-{ \langle - , -\rangle_\alpha}  \ar[d]^-{\otimes \R}&&&&&& \Z \ar@{^{(}-_>}[d] \\
    \pi_d E_\R &\times& \pi_{-d-s}F_\R\ar[rrrrrr]^-{ \alpha_\R}  &&&&&& \R 
    }}.
\end{align}
In such a case we say that the two pairings are {\it compatible}. 

Typically, in differential geometric situations, the realified pairing \eqref{eq_R_pairing} has a formula written in terms of characteristic forms, which makes it easy to compute explicitly. 
From that point of view, the important point of the pairing of Definition~\ref{def_int_pairing} is in its {\it integrality}; 
the fact that $\alpha_\R$ comes from  $\alpha$ guarantees the integrality of the pairings between elements coming from $\pi_d E$ and $\pi_{-d-s}F$. 
This often leads us to nontrivial divisibility results, as we will see in various examples below, see e.g.~Remark \ref{rem_integrality}.

\subsubsection{An example}\label{sec_example_free}
Let us begin by warming up ourselves by considering the pairings concerning $\pi_{-24}\TMF$.
Recall that $\sigma:\pi_{-24}\TMF\to \pi_{-24}\KO((q))$ is injective,
and the image of $\sigma$ is generated over $\bZ$ by $24/\Delta$ and $J^k/\Delta$ for $k\ge 1$, where $J=c_4^3/\Delta$ is the modular $J$-function.
For more details, see Appendix~\ref{app:TMF} and the references therein.

\paragraph{The pairing of $\pi_{-24}\TMF$ with $\pi_4\KO((q))/\TMF$}
We first consider \begin{equation}
\alpha_{\KO((q))/\TMF}^\vee: \pi_{-24}\TMF\to \pi_4\KO((q))/\TMF,
\end{equation}
 which  induces a $\bZ$-valued pairing \begin{equation}
\langle -,- \rangle_{\alpha_{\KO((q))/\TMF}} : \pi_{-24}\TMF \times \pi_4\KO((q))/\TMF \to \bZ.
\end{equation}
Here  we have a short exact sequence \begin{equation}
0\to  \pi_4\KO((q))/\sigma(\pi_4\TMF) \to \pi_4\KO((q))/\TMF\to A_3 \to 0
\label{seq-above}
\end{equation}
where $A_3=\Ker(\sigma:\pi_3\TMF\to \pi_3\KO((q)))=\bZ/24$.
We note that \begin{equation}
\pi_4(\KO((q)))/\sigma(\pi_4\TMF) \simeq {2\bZ((q))/2\MF_2}
\end{equation}
 contains elements whose modular form images are {$[(2c_6/c_4)J^{-\ell}]$}.
The induced pairing can be computed using the rationalized form of $\alpha_{\KO((q))/\TMF}$ given in Remark \ref{rem:rat}. We find 
\begin{equation}
\langle J^k/\Delta , {[(2c_6/c_4)J^{-\ell}]} \rangle_{(\alpha_{\KO((q))/\TMF})_\R}=
\begin{cases}
0 & (k<\ell), \\
1 & (k=\ell),\\
\in \bZ & (k>\ell).
\end{cases} \label{kl}
\end{equation} 
Note that the modular form images of $\pi_{-24}\TMF$ contain $n/\Delta$ for $n\in \bZ$ only when $24|n$.
This means that the pairing when restricted to $\pi_{-24}\TMF\times \pi_4\KO((q))/\sigma(\pi_4\TMF)$ is not perfect by a factor of $24$,
which nicely matches with the fact that $A_3=\bZ/24$ in the sequence \eqref{seq-above} above.

\paragraph{The pairing of $\pi_{-24}\TMF$ with $\pi_4\MSpin/\MString$}
Let us next consider the composition 
\begin{align}
    \pi_{-24}\TMF \xrightarrow{\alpha_{\spin/\stri}} \pi_{-4} I_\bZ \MSpin/\MString \to \Hom(\pi_4  \MSpin/\MString,\bZ)
\end{align} which defines the pairing \begin{equation}
\langle-,-\rangle_{\alpha_{\spin/\stri}}: \pi_{-24}\TMF \times \pi_4 \MSpin/\MString\to \bZ.
\label{geompair}
\end{equation}
% $\pi_{-24}\TMF$ in $\pi_{-4} I_\bZ \MSpin/\MString$,
% which determines a homomorphism $\Hom(\pi_4  \MSpin/\MString,\bZ)$.
We compute this pairing  after realification 
using the following formula, which directly follows from the definition of $\alpha_{\spin/\stri}$,  Proposition \ref{rem:rat},
and Proposition~\ref{prop_varpi_integration}:
\begin{prop}\label{claim_int_pairing_formula}
    The pairing after the realification, 
\begin{align}
    (\alpha_{\spin/\stri})_\R \colon \pi_{d} \TMF_\R \times \pi_{-d-20} \MSpin/\MString_\R \to \R,
\end{align}
is nontrivial only if $d \equiv 0 \pmod 4$, and in that case is given by
\begin{align}
    (\alpha_{\spin/\stri})_\R\left(x, y\right) = \frac12 \cdot \Delta\cdot \Phi(x) \cdot (\Wit_{\spin/\stri})_\R(y)\Bigr|_{q^0} 
\end{align}
for $x \in \pi_d \TMF_\R$ and $y \in \pi_{-d-20}\MSpin/\MString_\R$. 
Here we are regarding $\Phi(x) \in \pi_{-d}\KO((q))_\bR \simeq \R((q))$, and 
$(\Wit_{\spin/\stri})_\R(y) \in \R((q))/(\MF_{(-d-20)/2})_{\R}$ was described in Proposition~\ref{prop_varpi_integration} above.
\end{prop}
\noindent More detailed discussions of this proposition and related matters will be given in Sec.~\ref{subsubsec_diff_pairing_BunkeNaumann}.

Now, the geometry behind the short exact sequence \begin{equation}
0\to \pi_4 \MSpin \to \pi_4(\MSpin/\MString) \to \pi_3\MString \to 0
\end{equation} which is \begin{equation}
0\to \bZ \to \bZ \to \bZ/24\to 0
\end{equation} was studied in \cite[Appendix D]{TachikawaYamashita}.
There, it was recalled that the generator of $\pi_4 \MSpin = \bZ$ is the K3 surface and
$\pi_4 (\MSpin/\MString) = \bZ$ is generated by $[B^4,\nu]$, 
where we put the string structure on $S^3=\partial B^4$ so that
it comes from the generator $\nu\in \pi_3\MString=\pi_3\mathbb{S}=\bZ/24$.

The essential information in the computation of the pairing \eqref{geompair} is that \begin{equation}
\int_{[B^4,\nu]} \frac{p_1}{2}=1,
\end{equation} from which we conclude by applying the formula in Proposition~\ref{prop_varpi_integration} and Proposition \ref{claim_int_pairing_formula} that 
\begin{equation}
\langle J^k/\Delta, [B^4,\nu]\rangle_{(\alpha_{\spin/\stri})_\R}
= J^k \cdot \frac{1}{24}E_2 \Bigm|_{q^0},
\label{eq:d=4pairing}
\end{equation}
where \begin{equation}
E_2(q)=1-24\sum_{n\ge 1} \sigmasymbol_1(n) q^n
\end{equation} is the quasi-modular Eisenstein series of degree 2.
Let us  check the integrality of the pairing \eqref{geompair} directly,
using its realification \eqref{eq:d=4pairing}:
\begin{itemize}
\item For $k=0$, this is guaranteed by the fact that $c/\Delta \in \sigma(\pi_{-24}\TMF)$ if and only if $c$ is a multiple of 24.
%a fact first announced in \cite[Theorem~4.6]{Hopkins2002}.
\item For $k\ge 1$, the only possibly non-integer terms in \eqref{eq:d=4pairing} is 
$\frac{1}{24}J^k\big|_{q^0}$.
The integrality of this number can be shown in the following manner.
We use \cite[Theorem 12.1]{Borcherds},
which says that $\Theta_\Lambda(q)/\Delta^k|_{q^0}$ for 
the lattice theta function $\Theta_\Lambda(q)$ constructed from an even self-dual lattice $\Lambda$ of rank $8k$ is a multiple of 24.
Now we apply this to $\Lambda=\Lambda_{E8}^{\oplus k}$,
for which $\Theta_\Lambda(q)=J(q)^k$.
\end{itemize}

\begin{rem}\label{rem_integrality}
    Conversely, the result above $k=0$ can be regarded as obstructing the existence of $c/\Delta \in \Phi(\pi_{-24}\TMF)$ with $c \not\equiv 0\pmod {24}$ by the integrality of the pairing. 
    This observation is crucially used in the paper \cite{johnsonfreyd2024576fold}  of the second author and Theo Johnson-Freyd to give a differential geometric, and accordingly physical, explanation of the $576$-periodicity of $\TMF$. 
\end{rem}

\subsubsection{The relation to the invariant $b$ of Bunke and Naumann}\label{subsubsec_BM_tmf}

The discussion of the pairing induced by $\pi_{-24}\TMF$ and our secondary morphisms
can be put in a broader context by considering the invariant $b$ introduced by Bunke and Naumann in \cite{BunkeNaumann}.\footnote{%
In fact, four invariants, $b^\text{an}$, $b^\text{geom}$, $b^\text{top}$ and $b^\text{tmf}$ were introduced there. 
Their index theorem shows that they all agree when the respective domains are appropriately restricted, so we only discuss $b:=b^\text{tmf}$ in this section.
We will come back to $b^\text{an}$, $b^\text{geom}$ and $b^\text{top}$ when we discuss the differential pairing in Subsubsection \ref{subsubsec_diff_pairing_BunkeNaumann}.}

\begin{rem}%[Important!]
\label{rem_exceptional_convention}
    Throughout this Sec.~\ref{subsubsec_BM_tmf},
we use the isomorphism $\pi_{8k+4}\KO \simeq \bZ$
instead of our choice $\pi_{8k+4}\KO\simeq 2\bZ$ in the rest of this paper,
to reduce the clutter and also to follow the notation of \cite{BunkeNaumann}.
\end{rem}

Let us define $A_d$ to be the kernel of the map $\sigma:\pi_d \TMF\to \pi_d \KO((q))$.
We collect basic information on $A_d$ in Appendix~\ref{app:TMF}.
We now consider the long exact sequence \begin{equation}
\cdots \to \pi_{d}\KO((q))\stackrel{C\sigma}{\to} \pi_{d}\KO((q))/\TMF \to \pi_{d-1} \TMF \to \pi_{d-1} \KO((q)) \to \cdots.
\end{equation}
An element $x\in A_{d-1}$ can be lifted to $\overline{x}\in \pi_{d}\KO((q))/\TMF$,
which is determined up to  the addition of $(C\sigma)(\pi_{d}\KO((q)))$.
We now assume $d=4\ell$.
Then  \begin{equation}
\overline{x}_\bQ  \in (\pi_{d}\KO((q))/\TMF)_\bQ \simeq \bQ((q))/(\MF_{2\ell} )_\bQ
\end{equation}
is determined up to the addition of $(C\sigma)(\pi_d\KO((q)))\simeq \bZ((q))$.
Therefore the following is well-defined:
\begin{defn}
We define the Bunke-Naumann invariant of $x\in A_{4\ell-1}$ by the formula $$
b(x) := [ \overline{x}_\bQ ]\in  \frac{\bQ((q))}{(\MF_{2\ell})_\bQ + \bZ((q))}.
$$
\end{defn}

We are now going to compute the Bunke-Naumann invariants completely.
The invariant  was originally determined in a different manner in \cite{BunkeNaumann}.
Recall that $\MF_{2\ell}$ has a $\bZ$-basis consisting of \begin{equation}
S_{2\ell}^{\ge 0}:=\{ c_4^i c_6^j \Delta^k \mid 
4i+6j+12k=2\ell,\  
i\ge 0,\ 
j\in \{0,1\},\ 
k\in \bZ \}.
\end{equation} Let us analogously define \begin{equation}
S_{2\ell}^{<0} :=\{ c_4^i c_6^j \Delta^k \mid 
4i+6j+12k=2\ell,\  
i< 0,\ 
j\in \{0,1\},\ 
k\in \bZ \}.
\end{equation}
We further let \begin{equation}
S_{2\ell}:=S_{2\ell}^{<0} \cup S_{2\ell}^{\ge 0}.
\end{equation}

\begin{lem}
$[S_{2\ell}^{<0}]$ forms a $\bZ$-basis of $\bZ((q))/\MF_{2\ell}$.
\end{lem}
\begin{proof}
Given $\ell$, the element  $c_4^i c_6^j \Delta^k \in S_{2\ell}$ 
is uniquely fixed by $k\in \bZ$, exists for all $k\in \bZ$,
and has the $q$-expansion $1q^k+\cdots$.
Therefore $S_{2\ell}$ forms a $\bZ$-basis of $\bZ((q))$.
As $S_{2\ell}^{\ge 0}$ forms a $\bZ$-basis of $\MF_{2\ell}$,
the lemma follows.
\end{proof}

We now introduce a pairing \begin{equation}
\langle -,-\rangle : \MF_{-2\ell-10} \times \bZ((q))/\MF_{2\ell} \to \bZ
\end{equation} given by \begin{equation}\label{eq_pairing_MF}
\langle f(q), [g(q)] \rangle := \Delta(q) f(q) g(q) \Bigm|_{q^0}.
\end{equation}
This is well-defined thanks to Fact~\ref{fact:mf}. %since the constant term of the $q$-expansion of elements in $\MF_2$ is zero, 

\begin{prop}
\label{prop:3.22}
The rationalization of the pairing $\langle-,-\rangle$ above equals the rationalization of the pairing \begin{equation}
\langle -,-\rangle_{\alpha_{\KO((q))/\TMF}}:  \pi_{-4\ell-20} \TMF\times \pi_{4\ell}\KO((q))/\TMF \to \bZ.
\end{equation}
\end{prop}
\begin{proof}
This follows easily from the explicit form of the rationalized version of $\alpha_{\KO((q))/\TMF}$, see  \eqref{alpha-KO/TMF}.
The absence of the factor $1/2$ from \eqref{eq_pairing_MF}
compared to \eqref{alpha-KO/TMF}
is due to our choice $\pi_{8k+4}\KO\simeq \bZ$ in this subsubsection
and $\pi_{8k+4}\KO\simeq 2\bZ$ in the rest of the paper (Remark \ref{rem_exceptional_convention}).
\end{proof}

We now consider an involution \begin{equation}
S_{-2\ell-10} \ni c_4^i c_6^j \Delta^k \mapsto 
c_4^{-1-i} c_6^{1-j} \Delta^{-1-k} \in S_{2\ell}
\end{equation} which induces a bijection between $S_{-2\ell-10}^{\ge 0}$ and $S_{2\ell}^{<0}$.
A short computation reveals \begin{equation}
\langle c_4^i c_6^j \Delta^k, [c_4^{-1-i'} c_6^{1-j'} \Delta^{-1-k'} ]\rangle
= \begin{cases}
0 & (k> k'), \\
1 & (k=k'),\\
\in \bZ & (k< k').
\end{cases}
\end{equation}
Therefore, $S_{-2\ell-10}^{\ge 0}$ and $[S_{2\ell}^{<0}]$ form a dual basis for $\MF_{-2\ell-10}$ and $\bZ((q))/\MF_{2\ell}$ under our pairing.

Our pairing induces \begin{equation}
\langle -,-\rangle: (\MF_{-2\ell-10})_\bQ \times \frac{\bQ((q))}{(\MF_{2\ell})_\bQ + \bZ((q))} \to \bQ/\bZ.
\end{equation}

The following lemma is clear from our discussion: 
\begin{lem}
The homomorphisms \begin{align*}
\bZ((q))/\MF_{2\ell} \ni x &\mapsto  (f\mapsto \langle f , x\rangle )\in \prod_{f} \bZ
,\\
 \frac{\bQ((q))}{(\MF_{2\ell})_\bQ + \bZ((q))}\ni x &\mapsto (f\mapsto \langle f , x\rangle )\in \prod_{f} \bQ/\bZ
\end{align*}
 are  isomorphisms of Abelian groups, where $f$ on the right hand side runs over elements of $S^{\ge 0}_{-2\ell-10}$.
\end{lem}

Let us then make the following definition:
\begin{defn}
\label{defn:individual}
We call 
$\langle f,b(x)\rangle$ for each $f\in S^{\ge 0}_{-2\ell-10}$ 
as the individual Bunke-Naumann invariant of $x\in A_{4\ell-1}$.
The entirety of $\langle f,b(x)\rangle$ determines $b(x)$ due to the proposition above.
\end{defn}
We see however:
\begin{prop}
\label{prop:individual}
The individual Bunke-Naumann invariant $\langle f,b(x)\rangle$ is zero unless $f=\Delta^k$.
\end{prop}
\begin{proof}
Let $f=c_4^i c_6^j \Delta^k$ with $i> 0$.
Pick a lift $\overline x\in \pi_{4\ell} \KO((q))/\TMF$.
We now note that  $f$ is in the image of $\sigma:\pi_{-4\ell-20} \TMF\to \pi_{-4\ell-20} \KO((q))$,
see e.g.~Fact~\ref{fact:imageZ}.
We then have $\langle f,\overline x\rangle \in \bZ$ from Proposition~\ref{prop:3.22} above,
from which we conclude $\langle f,b(x)\rangle = \langle f,[\overline x]\rangle = 0$.
\end{proof}

We have an immediate corollary:
\begin{cor}
$b(x)$ is zero unless $x\in A_{24(-1-k)+3}$,
and is uniquely fixed by $\langle \Delta^k, b(x)\rangle$ if $x\in A_{24(-1-k)+3}$.
\end{cor}

We now have a final proposition determining the Bunke-Naumann invariant completely:
\begin{prop}
For $x\in A_{24(-1-k)+3}$, the homomorphism $$
A_{24(-1-k)+3} \ni x \mapsto \langle \Delta^k,b(x)\rangle \in \{n/\gcd(24,k)\}_{n \in \Z}  \subset \Q/\bZ
$$ is an isomorphism of Abelian groups.
\end{prop}
\begin{proof}
Lift $x$ to an element $\overline{x}$  in $\pi_{24(-1-k)+4}\KO((q))/\TMF$.
%As $x$ is a $B$-power torsion from Proposition~\ref{kernel=Btorsion},
%this element can be  lifted to 
%\begin{equation}
%\tilde x\in \pi_{24(-1-k)+4}\TMF/B^\infty.
%\end{equation}
%Let us define $\overline{x}$ be its image in $\pi_{24(-1-k)+4}\KO((q))/\TMF$ by $\varrho$.
The element $\gcd(24,k)\Delta^k$ is a generator of the image of $\sigma$, and therefore 
$\langle \gcd(24,k)\Delta^k, \overline x\rangle \in \bZ$.
Therefore $\langle \Delta^k,b(x)\rangle=n/\gcd(24,k)$ for some $n$.
The perfectness of the Anderson dual pairing between $\pi_{24k}\TMF$ and $\pi_{24(-1-k)+4}\KO((q))/\TMF$
implies that there should be an element $x_1\in A_{24(-1-k)+4}$ such that 
$\langle \gcd(24,k)\Delta^k, \overline {x_1}\rangle =1$,
meaning that  $\langle \Delta^k,b(x_1)\rangle=1/\gcd(24,k)$.
Therefore $x\mapsto  \langle \Delta^k,b(x)\rangle $ gives  an isomorphism.
\end{proof}

\begin{rem}
Recently a spectrum $\KO_\MF$, which captures the $\KO$-based invariants
and the Bunke-Naumann invariants of $\TMF$, % but not the last one,
was introduced and analyzed in great detail in \cite{berwickevans2023field}.
There, it was also shown that the space of  `SQFTs defined on cylinders and tori' has a natural morphism to  $\KO_\MF$,
making a significant step in the understanding of the proposal of Stolz and Teichner \cite{StolzTeichner1,StolzTeichner2}.
Our analysis in the rest of this article will amply show that physics of SQFT and heterotic string theory knows 
even subtler aspects of $\TMF$.
\end{rem}

\subsection{Induced torsion pairing}\label{subsec_torsion_pairing}

Let us start by introducing a notation.

\begin{defn}\label{def_torsion_pairing}
Given a morphism of spectra $\alpha: E\to \Sigma^{-s} I_\bZ F$
and a torsion element $x\in \pi_d E$,
consider $\alpha(x)\in \pi_{d+s} I_\bZ F$.
Since $\alpha(x)$ is also torsion, we have 
\[
    \alpha(x) \in \Hom(\pi_{-d-s-1} F, \Q/\Z) / \Hom(\pi_{-d-s-1} F, \Q).
\]
Restricted to $(\pi_{-d-s-1}F)_\text{torsion}$ we get a well-defined homomorphism $$
\alpha_x : (\pi_{-d-s-1} F)_\text{torsion}\to \bQ/\bZ
$$
We then define $$
\alpha(-,-) : (\pi_d E)_\text{torsion} \times (\pi_{-d-s-1} F)_\text{torsion} \to \Q/\bZ
$$ via $ \alpha(x,y) :=\alpha_x (y)$.
\end{defn}

We have now successfully determined the Bunke-Naumann invariant.
Note that it is a torsion invariant defined on a subgroup of $\pi_{d}\TMF$
for $d=4\ell-1$.
As such, it also has a description using the torsion part of the pairing induced by $\alpha_{\KO((q))/\TMF}^\vee: \pi_{d}\TMF\to \pi_{d+20}I_\bZ\KO((q))/\TMF$,
i.e.~as a part of the pairing\begin{equation}
\alpha_{\KO((q))/\TMF}(-,-) :
(\pi_{d}\TMF)_\text{torsion}\times
(\pi_{-d-21} \KO((q))/\TMF)_\text{torsion} \to \Q/\Z.
\label{TMFKOpair}
\end{equation} 
Note that we have a long exact sequence \begin{multline}
\cdots \to \pi_{-d-21} \TMF \xrightarrow{\sigma}
\pi_{-d-21} \KO((q))  \\
\to 
\pi_{-d-21} \KO((q))/\TMF \to  \\
\pi_{-d-22} \TMF \xrightarrow{\sigma}
\pi_{-d-22} \KO((q)) \to\cdots,
\end{multline}
and therefore there are three types of torsions in $\pi_{-d-21}\KO((q))/\TMF$:
\begin{itemize}
\item The first is when $-d-21\equiv  1$ or $2$ mod $8$,
so that $\pi_{-d-21}\KO((q))\simeq \bZ/2((q))$.
This is a torsion already detected in $\KO$.
\item The second is when $-d-21\equiv 0$ mod $24$,
so that $\pi_{-d-21}\KO((q))\simeq \bZ((q))$ 
but the cokernel of $\sigma$ can contain torsion elements. 
This gave rise to the  Bunke-Naumann invariant.
\item Finally, the third comes  from the lifts of kernels of $\sigma$ in $\pi_{-d-22} \TMF$,
which is the most subtle and is our next topic.
\end{itemize}

So, let us consider the pairing \eqref{TMFKOpair} with an additional assumption that $x\in A_{d}$, i.e.~$x$ is in the kernel of $\sigma:\pi_{d}\TMF\to \pi_{d}\KO((q))$.
In the proofs of the following propositions we will repeatedly refer to the following structure morphisms:
\begin{equation}
\vcenter{\xymatrix{
%\TMF[B^{-1}] \ar[d]^-{p} \ar[r]^-{\varepsilon} &
%\TMF/B^\infty \ar[d]^-{\varrho} \\
\TMF \ar[r]^-{\sigma} & \KO((q))  \ar[r]^-{C\sigma} & 
\KO((q))/\TMF \ar[r]^-{\partial} & \Sigma\TMF
}}.
\end{equation}
We use the same symbols for the homomorphisms on the homotopy groups induced by them.

\begin{lem}
\label{prop:torsionlift}
An element $x\in A_{d}$ can be lifted to a torsion element in $\pi_{d+1}\KO((q))/\TMF$
unless  $d\equiv 3$ mod $24$.
\end{lem}

\begin{proof}
The only possible problem is when 
$\pi_{d+1}\KO((q))/\TMF$ contains non-torsion elements, i.e.~when $d+1=4\ell$.
So let us take $x\in A_{4\ell-1}$ and assume $d=4\ell-1 \not\equiv 3$ mod $24$.
Pick an arbitrary lift $\overline x \in \pi_{4\ell}\KO((q))/\TMF$.
As $d\not\equiv 3$ mod $24$, its Bunke-Naumann invariant vanishes, $b(x)=0$.
This means that one can find an element $y\in \pi_{4\ell}\KO((q))$ such that the rationalization of $y$ and $\overline x$ agrees.
Then $\overline x-(C\sigma)(y)\in \pi_{4\ell}\KO((q))/\TMF$ is a torsion lift of $x$.
\end{proof}

\begin{lem}
\label{prop:welldef}
Assume $d\not\equiv 3,-1$ mod $24$, and let $y \in A_{-d-22}$.
%We can  assume $d\not\equiv -1$ mod $24$ without loss of generality,
%since $A_{24k-1}$ is empty, see Appendix~\ref{app:TMF}.
We can then pick its lift $\overline y \in (\pi_{-d-21} \KO((q))/\TMF)_\text{torsion}$ using Lemma~\ref{prop:torsionlift}, since $-d-21\not\equiv 3$ mod $24$.
The value $\alpha_{\KO((q))/\TMF}( x,\overline y)$ for $x\in A_{d}$ and $y\in A_{-d-22}$ does not depend on the lift $\overline y$.
\end{lem}

\begin{proof}
Let $\overline y$ and $\overline y'$ be two lifts.
Then $\overline y-\overline y'=(C\sigma)(z)$ for a $z\in \pi_{-d-21} \KO((q))$.
Here $z$ is itself a torsion element.
This is  because a non-torsion element in $\pi_{-d-21}\KO((q))$ can map to 
a torsion element only when
$\sigma : \pi_{-d-21}\TMF\to\pi_{-d-21}\KO((q))$ maps a primitive element to a non-primitive element,
which happens only when $-d-21\equiv 0$ mod $24$.
Now that we established that $z$ is torsion, we can use the naturality of the pairing to 
conclude that $\alpha_{\KO((q))/\TMF}( x,(C\sigma)(z))
=\alpha_\text{spin}(\sigma(x),z)=0$.
\end{proof}

This allows us to make the following definition:
\begin{defn}
We define $$
\lpar -,-\rpar : A_{d} \times A_{-d-22} \to \bQ/\bZ
$$ for $d\not\equiv 3,-1$ mod $24$ via $$
\lpar x,y\rpar = \alpha_{\KO((q))/\TMF}( x,\overline y),
$$ 
where $\overline y$ is a lift of $y\in A_{-d-22}$ in $(\pi_{-d-21} \KO((q))/\TMF)_\text{torsion}$ using Lemma~\ref{prop:torsionlift}. 
This is well-defined thanks to Lemma~\ref{prop:welldef}.
\end{defn}

\begin{prop}
\label{pont-dual}
The pairing $\lpar -,-\rpar $ is a perfect Pontryagin pairing between
$A_{d}$ and $A_{-d-22}$ for $d\not\equiv3,-1$ mod $24$.
\end{prop}

\begin{proof}

From our  Theorem~\ref{main}, we know that $\alpha_{\KO((q))/\TMF}$ gives the isomorphism \begin{equation}
(\pi_{-d-21} \KO((q))/\TMF)_\text{torsion} \simeq \Hom( (\pi_{d}\TMF)_\text{torsion} , \Q/\Z).
\end{equation}
Therefore, for any nonzero $x\in A_{d}$, 
there is an element $z\in \pi_{-d-21}\KO((q))/\TMF$
for which $\alpha_{\KO((q))/\TMF}(x,z)$ is nonzero.
Letting $y=\partial(z)\in A_{-d-22}$, we see that $\lpar x,y\rpar$ is nonzero.
Therefore, the map $A_{d}\to \Hom(A_{-d-22},\Q/\Z)$ induced by $\lpar-,-\rpar$ is injective.
In particular, $|A_{d}| \le |A_{-d-22}|$.
Exchanging the role of $d$ and $-d-22$, we obtain an inequality in the opposite direction,
showing that $|A_{d}|=|A_{-d-22}|$ and that the pairing $\lpar-,-\rpar$ is perfect.
\end{proof}

\begin{rem}
\label{rem:AApairing}
The explicit form of this pairing can be found by using the information in \cite{BrunerRognes}. 
We can work at each prime. 
There is no $p$-torsions in $A_d$ for $p\ge 5$, so what matters is $p=2$ and $p=3$.

At prime 2, it was shown in \cite[Theorem 10.29 and Table 1]{BrunerRognes} that the Anderson duality induces a perfect pairing 
\begin{equation}
\Theta N_{170-d} \times \Theta N_{d} \to \Q/\Z,
\end{equation}
where $\Theta N_d \subset \pi_d \tmf$
is the subset of all $B$-power-torsion elements whose degree  $d$ satisfies $0\le d < 192$ and $d\neq 24k+3$.
This gives a perfect pairing at prime 2 between
$A_{-192\ell-22-d}= M^{-\ell-1} \Theta N_{170-d}$
and $A_{192\ell+d}=M^{\ell} \Theta N$
where $M\in \pi_{192}\tmf$ is the periodicity element at prime 2.

At prime 3, it was shown in \cite[Theorem 13.25]{BrunerRognes} that
the Anderson duality induces a perfect pairing \begin{equation}
\Theta N_{50-\bullet} \times \Theta N_\bullet \to \Q/\Z
\end{equation} 
where $\Theta N_\bullet$ is the $B$-power-torsion subgroup of elements whose degree $d$ satisfies $0\le d < 72$.
This gives a perfect pairing at prime 3 between 
$A_{-72\ell-22-d}= H^{-\ell-1} \Theta N_{50-d}$
and $A_{72\ell+d}=H^{\ell} \Theta N$
where $H\in \pi_{72}\tmf$ is the periodicity element at prime 3.
\end{rem}

The pairing just determined can be used to perform some computations of $\alpha_\text{spin/string}$.
Recall that Definition~\ref{def_torsion_pairing} gives us the pairing 
\begin{equation}
\alpha_{\spin/\stri}( -,-): (\pi_{-d-22} \TMF)_\text{torsion} \times  (\pi_{d+1}\MSpin/\MString)_\text{torsion} \to \Q/\Z,
\end{equation}
where we switched the role of $d$ and $-d-22$ with respect to \eqref{TMFKOpair} for convenience.
% $\alpha_\text{spin/string}(x)\in \pi_{-d-1}I_\bZ \MSpin/\MString$ induces
%a torsion pairing between $A_{-d-22}$ and the torsion subgroup of 
%$\pi_{d+1}I_\bZ \MSpin/\MString$.
We then have the following:
\begin{prop}
\label{prop:geometric-p}
Take $x\in A_{-d-22}$. 
Suppose a string manifold $Y$ of dimension d
is spin null bordant, i.e.~there is a spin manifold $Z$ with $\partial Z=Y$.
We then have a class $[Z,Y]\in \pi_{d+1}\MSpin/\MString$.
When $d\not\equiv 3,-1$ mod $24$,
we have $$
\alpha_\text{spin/string}(x , [Z,Y]) = \lpar x, y\rpar,
$$
where $y=\Wit_\stri([Y])\in A_d\subset \pi_d\TMF$.
\end{prop}

\begin{proof}
The image $\overline y$ of $[Z,Y]$ in $\pi_{d+1}\KO((q))/\TMF$ 
lifts $y$. Therefore, $
\alpha_\text{spin/string}(x,[Z,Y]) = \alpha_{\KO((q))/\TMF}(x,\overline y)
= \lpar x, y\rpar
$.%, where we used Proposition~\ref{prop:relation-two-morphisms} for the first equality.
\end{proof}

\begin{rem}

Concretely, at prime $2$, take $(x,y)=(\nu_6\kappa/M,\nu^3)$,
$(\eta\nu_6\kappa/M,\epsilonelem)$,
$(\nu\nu_6\kappa/M,\nu^2)$ for $d=9,8,6$, respectively.
Here we follow the notations of \cite{BrunerRognes};
 in particular,
 $\nu\in\pi_3\mathbb{S}$ 
 and $\epsilonelem\in \pi_8\mathbb{S}$ are
  the group manifolds of $SU(2)$ and  $SU(3)$ with Lie group framing, and 
 they generate $A_d$ and $A_{-d-22}$, respectively,
see \cite[Table 10.1]{BrunerRognes}.
 Therefore $\alpha_\text{spin/string}(x,-)$ detects
 the relative bordism classes lifting $y$.

At prime $3$, take $(x,y)=(\beta^4/H,\beta)$, where $\beta\in \pi_{10}\mathbb{S}$
is the $Sp(2)$ with Lie group framing.
Again they generate $A_{-32}$ and $A_{10}$, see \cite[Sec.~13.6]{BrunerRognes}.
Therefore $\alpha_\text{spin/string}(\beta^4/H,-)$ detects the relative bordism classes lifting $\beta$.
\end{rem}

\subsection{Induced differential pairing}\label{subsec_diff_pairing}

Actually, the pairings discussed above %in Subsection \ref{subsec_torsion_pairing} 
are parts of the {\it differential pairing} which we now introduce. 
For the background materials on differentia (co)homology theories, see Appendix~\ref{app:diff_bordism}.

\subsubsection{General theory}\label{subsubsection_general_diff_paiaing}
In the general setting of Definitions \ref{def_int_pairing} and \ref{def_torsion_pairing} where we have a morphism of spectra $\alpha: E\to \Sigma^{-s} I_\bZ F$ and an integer $d$, suppose that
\begin{enumerate}
    \item A differential extension $\widehat{F}_\bullet$ of the $F$-\emph{homology} in the sense of \cite{YamashitaDifferentialIE} is given. 
    \item A differential extension $\widehat{E}^\bullet$ of $E$-\emph{cohomology} and a differential refinement $\widehat{\alpha} \colon \widehat{E}^{-d} \to (\widehat{I_\Z F})^{-d-s}$ of $\alpha \colon E \to \Sigma^{-s}I_\Z F$ are given.
\end{enumerate}
Then, using the pairing \eqref{eq_IE_diff_pairing} we get the following {\it differential pairing}, 
\begin{align}\label{eq_generalized_differential_pairing}
    \widehat{\alpha}\colon \widehat{E}^{-d} \times \widehat{F}_{-d-s-1} \to \R/\Z. 
\end{align}
% The pairing \eqref{eq_generalized_differential_pairing} generalizes the pairing \eqref{eq_diff_pairing} previously constructed in the sense that, in the presense of the condition (2), we have the equality
% \begin{align}
%     \widehat{\alpha}_{\eqref{eq_generalized_differential_pairing}}(\widehat{x}, \widehat{y}) = \widehat{\alpha}_{\eqref{eq_diff_pairing}}(I(\widehat{x}), \widehat{y})
% \end{align}
% for all $\widehat{x} \in \widehat{E}^{-d}$ with the topological class $I(\widehat{x}) \in \pi_d E$ and $\widehat{y} \in \widehat{F}_{-d-s-1}$. 

The differential pairing is related to topological pairings discussed earlier as follows. 
First, by Lemma \ref{lem_pairing_compatibility} and Lemma \ref{lem_pairing_compatibility_2}, we get the following compatibility with $\R$-valued pairing. 
\begin{lem}\label{lem_diff_generalized_pairing_compatibility}
    In the following diagram, 
    \begin{align}\label{diag_diff_gen_pairing_compatibility}
\vcenter{\xymatrixcolsep{0pc}
    \xymatrix{
    E_\R^{-d-1} \ar[d]^-{a}&\times& (F_{-d-s-1})_\R\ar[rrrrrr]^-{ \alpha_\R}  &&&&&& \R \ar[d]^-{\text{mod} \ \Z} \\
    \widehat{E}^{-d} \ar[d]^-{R}&\times& \widehat{F}_{-d-s-1}\ar[u]^-{R}\ar[rrrrrr]^-{ \widehat{\alpha}}  &&&&&& \R/\Z \\
    E_\R^{-d} &\times& (F_\R)_{-d-s}\ar[rrrrrr]^-{\alpha_\R} \ar[u]^-{a}&&&&&& \R \ar[u]_-{\text{mod} \ \Z}
    }},
\end{align}
three pairings are compatible, i.e., for any element $x \in E_\R^{-d-1}$ and $\hat{y} \in \widehat{F}_{-d-s-1}$ we have
\begin{align}
    \widehat{\alpha}(a(x), \widehat{y})=\alpha_\R(x, R(\widehat{y})) \pmod \Z, 
\end{align}
and for any element $\widehat{x} \in \widehat{E}^{-d}$ and $y \in (F_\R)_{-d-s}$ we have
\begin{align}
    \widehat{\alpha}(\widehat{x}, a(y)) = \alpha_\R (R(\widehat{x}), y) \pmod \Z. 
\end{align}
\end{lem}

Next, we discuss the relation with the torsion pairing $\alpha \colon (\pi_d E)_\text{torsion} \times (\pi_{-d-s-1} F)_\text{torsion} \to \Q/\bZ$ in Definition \ref{def_torsion_pairing}.  
Note that any torsion element $x \in (\pi_d E)_{\text{torsion}}$ can be lifted to an element $\widehat{x} \in \widehat{E}^{-d}$ with trivial curvature, $R(\widehat{x}) = 0 \in {E}^{-d}_\R$. 
Conversely, if an element $\widehat{x} \in \widehat{E}^{-d}$ satisfies $R(\widehat{x})=0$, then the topological class is torsion, $I(\widehat{x}) \in (\pi_d E)_{\text{torsion}}$. 

\begin{lem}\label{lem_diff_pairing_torsion}
Let $\widehat{x} \in \widehat{E}^{-d} $ 
lift a torsion element $x := I(\widehat{x}) \in  (\pi_{d} E)_\text{torsion}$,
further satisfying $R(\widehat{x}) = 0 \in {E}^{-d}_\R$.
Let  $\widehat{y} \in \widehat{F}_{-d-s-1} $
lift a torsion element $y := I(\widehat{y}) \in  (\pi_{-d-s-1} F)_\text{torsion}$.
Then we have
\[
    \widehat{\alpha}(\widehat{x}, \widehat{y}) = \alpha(x, y).
\]
\end{lem}

\begin{proof}
    We recall the definition of $\alpha$. A torsion element $x \in (\pi_d E)_\text{torsion}$ maps to a torsion element $\alpha(x) \in (\pi_{d+s} I_\Z F)_\text{torsion} \subset (I_{\R/\Z} F)^{-d-s-1}/H^{-d-s-1}(F; \R)$. 
    We chose a lift $\widetilde{\alpha(x)} \in (I_{\R/\Z} F)^{-d-s-1}$ of $\alpha(x)$ and defined
    \begin{align}\label{eq_proof_pairing_compatibility}
        \alpha(x, y) := \langle \widetilde{\alpha(x)} , y \rangle_{I_{\R/\Z}}, 
    \end{align}
    where the right hand side uses the pairing in the top row of the diagram
    \begin{align}\label{diag_proof_pairing_compatibility}
\vcenter{\xymatrixcolsep{0pc}
    \xymatrix{
   (I_{\R/\Z}F)^{n-1}\ar[d]^-{\iota_{\text{flat}}}  &\times& \pi_{n-1}F \ar[rrrrrr]^-{\langle \cdot, \cdot\rangle_{I_{\R/\Z}}} &&&&&& \R/\Z \ar@{=}[d]\\
    (\widehat{I_\Z F})^n &\times& \widehat{F}_{n-1} \ar[u]^-{I}\ar[rrrrrr]^-{\langle \cdot, \cdot \rangle_{\widehat{I_\Z}}}  &&&&&& \R/\Z , \\
    }}
\end{align}
where the reader should refer to \eqref{diag_exact_diff_IE} for $\iota_\text{flat}$.
The right hand side of \eqref{eq_proof_pairing_compatibility} does not depend on the lift exactly because $y$ is torsion. 
By Lemma \ref{lem_pairing_compatibility}, the two pairings in the diagram \eqref{diag_proof_pairing_compatibility} are compatible so that we have
\begin{align}
    \alpha(x, y) = \left\langle \widetilde{\alpha(x)} , y \right\rangle_{I_{\R/\Z}} = \left\langle \iota_{\text{flat}}\left(\widetilde{\alpha(x)} \right), \widehat{y} \right\rangle_{\widehat{I_\Z}}
    =\widehat{\alpha}(\widehat{x}, \widehat{y}). 
\end{align}
Here the last equation used the assumption $R(\widehat{x}) = 0$, which implies that the lift $\widehat{x}$ is taken to lie in the image of $\iota_{\text{flat}}$. 
This completes the proof.
\end{proof}

We also consider the following variant of the differential pairing. 
We still assume the condition (1) in the beginning of this subsubsection. 
Instead of assuming (2), we now assume the following. 
\begin{enumerate}
    \item[(3)] $H^{-d-s-1}(F; \R) = 0$. 
\end{enumerate}
This condition (3)  and the axiom of generalized differential cohomology imply that the forgetful map gives an isomorphism
\begin{align}\label{eq_diff_IF_forgetful}
    I \colon (\widehat{I_\Z F})^{-d-s} \simeq (I_\Z F)^{-d-s} = \pi_{d+s}(I_\bZ F). 
\end{align}
Thus the differential pairing in \eqref{eq_generalized_differential_pairing} reduces to the following pairing, which we still call as the differential pairing, 
\begin{align}\label{eq_diff_pairing}
    \widehat{\alpha}\colon \pi_d E \times \widehat{F}_{-d-s-1} \to \R/\Z. 
\end{align}
% The differential pairing is related to topological pairings discussed earlier as follows. 
% First, by the bottom compatibility in Lemma \ref{lem_pairing_compatibility}, we get the following compatibility with $\R$-valued pairing. 
% \begin{lem}\label{lem_diff_pairing_R_compatibility}
% In the following diagram, 
%     \begin{align}
% \vcenter{\xymatrixcolsep{0pc}
%     \xymatrix{
%     \pi_d E \ar@{=}[d]&\otimes& \widehat{F}_{-d-s-1}\ar[rrrrrr]^-{ \widehat{\alpha}}  &&&&&& \R/\Z \\
%     \pi_d E &\otimes& \pi_{-d-s}F_\R\ar[rrrrrr]^-{\alpha_\R} \ar[u]^-{a}&&&&&& \R \ar[u]_-{\text{mod} \ \Z}, 
%     }}
% \end{align}
% two pairings are compatible, i.e., for any $x \in \pi_d E$ and $\eta \in \pi_{-d-s}F_\R$, we have
% \begin{align}
%     \widehat{\alpha}(x, a(\eta)) \equiv \alpha_\R(x_\R, \eta) \pmod \Z. 
% \end{align}
% Here $x_\R \in \pi_d E_\R$ is the realification of $x$.  
% \end{lem}

\subsubsection{The differential pairing associated to $\alpha_{\spin/\stri}$ and the invariants of Bunke and Naumann}\label{subsubsec_diff_pairing_BunkeNaumann}

In this subsubsection we apply the general theory of differential pairing above to the case $\alpha = \alpha_{\spin/\stri} \colon \TMF \to \Sigma^{-20}I_\Z \MSpin/\MString$, and explain its relation with the invariants $b^\text{an}$, $b^\text{geom}$ and $b^\text{top}$ of Bunke and Naumann \cite{BunkeNaumann}.%, as well as the conjectural relation with eta invariants for spin manifolds with string boundaries. 

In this case we have $E=\TMF$, $F = \MSpin/\MString$ and $s=20$. 
For the condition (1) on the differential extension of $\MSpin/\MString$-homology, we use the {\it relative differential spin/string bordism groups} $\reallywidehat{\MSpin/\MString}_*$, which are explained in detail in Appendix \ref{app_subsec_diff_bordism}. 
Roughly speaking, a typical element in $\reallywidehat{\MSpin/\MString}_n$ is represented by a {\it relative differential spin/string-cycle} of dimension $n$, consisting of data $(N_n, M_{n-1}, g_N^\spin, g_M^{\stri})$, where $N$ is a compact $n$-dimensional manifold with boundary $\del N = M$, $g_N^{\spin}$ is a differential spin structure on $N$ and $g_M^\stri$ is a differential string structure on $M$ which lifts $g_N^{\spin}|_M$. 

%For $d \equiv -1$ mod $4$ the condition (2) is not satisfied, so we use the generalized version  of the pairing.
The differential pairing \eqref{eq_generalized_differential_pairing} in this case
becomes\footnote{
For the definition of the differential extension of $\TMF$, we use the {\it multiplicative} extension with $S^1$-integration, which uniquely exists by the result of \cite{BunkeSchickUniqueness}. 
}
\begin{align}\label{eq_diff_pairing_MSpin/MString}
    \widehat{\alpha}_{\spin/\stri} \colon  \widehat{\TMF}^{-d} \times \reallywidehat{\MSpin/\MString}_{-d-21} \to \R/\Z. 
\end{align}

First we investigate the easy parts of pairing where we can lift the values to $\R$, using the compatibility in Lemma \ref{lem_diff_generalized_pairing_compatibility}. 
We focus on the case $d =4k-1$ for some integer $k$. 
In this case the upper half of the diagram \eqref{diag_diff_gen_pairing_compatibility} becomes nontrivial and takes the form
\begin{align}%\label{diag_diff_gen_pairing_compatibility}
\vcenter{\xymatrixcolsep{0pc}
    \xymatrix{
    \TMF_\R^{-4k} \simeq (\MF_{2k})_\R \ar[d]^-{a}&\times& \pi_{-4k-20}\MSpin/\MString\ar[rrrrrr]^-{ (\alpha_{\spin/\stri})_\R}  &&&&&& \R \ar[d]^-{\text{mod} \ \Z} \\
    \widehat{\TMF}^{-(4k-1)} &\times& \reallywidehat{\MSpin/\MString}_{-4k-20}\ar[u]^-{I}\ar[rrrrrr]^-{ \widehat{\alpha}_{\spin/\stri}}  &&&&&& \R/\Z, 
    }}
\end{align}
We have the following formula for the differential pairings whose first variable comes from $(\MF_{2k})_\R$. 
Recall the map $\Wit_{\spin/\stri} \colon \MSpin/\MString \to \KO((q))/\TMF$, whose realification takes the form
\begin{align}\label{eq_Wit_rel_realification}
   (\Wit_{\spin/\stri})_\R \colon \pi_{4\ell}\MSpin/\MString \to \pi_{4\ell}(\KO((q))/\TMF)_\R \simeq \frac{\R((q))}{(\MF_{2\ell})_{\R}}. 
\end{align}

\begin{prop}[{The \emph{easy} part\footnote{We have the exact sequence $\TMF_\R^{-4k} \xrightarrow{a} \widehat{\TMF}^{-(4k-1)} \xrightarrow{I} \TMF^{-(4k-1)}$.
This means that, for a fixed integer $k$, the formula in Proposition \ref{prop_diff_pairing_easy_-1mod4} gives the whole information of $\widehat{\alpha}_{\spin/\stri}$ if and only if $\pi_{4k-1}\TMF = 0$. } of the pairing for $d \equiv -1 \pmod 4$}]\label{prop_diff_pairing_easy_-1mod4}
Let $k$ be an integer. 
For each $f(q) \in (\MF_{2k})_\R \simeq \TMF_\R^{-4k}$ and a differential spin/string cycle $(N, M, g_N^\spin, g_M^\stri)$ of dimension $(-4k-20)$, we have
\begin{align}\label{eq_diff_pairing_easy_-1mod4}
    \widehat{\alpha}_{\spin/\stri}\left(a(f), [N, M, g_N^\spin, g_M^\stri]\right) 
    ={\frac12}\Delta(q)f(q)(\Wit_{\spin/\stri})_\R ([N, M] )\Bigr|_{q^0} \pmod \Z. 
\end{align}
\end{prop}

\begin{proof}
    This is the direct consequence of the compatibility of the pairing in Lemma \ref{lem_diff_generalized_pairing_compatibility} for the diagram \eqref{diag_diff_gen_pairing_compatibility} and the definition of $\alpha_{\spin/\stri}$. 
\end{proof}

Now we explain the relation with the invariants $b^\text{an}$, $b^\text{geom}$ and $b^\text{top}$ of Bunke and Naumann. 
These invariants are defined on 
\begin{equation}
A_{4\ell-1}^\stri := \ker(\iota\colon\pi_{4\ell-1}\MString \to \pi_{4\ell-1}\MSpin)
\end{equation}
as
\begin{align}
    b^\an, b^\geom, b^\top \colon A^\stri_{4\ell-1} \to 
    {\frac{(\pi_{4\ell}\ko[[q]])_\bQ}{(\mf_{2\ell})_\bQ + \pi_{4\ell}\ko[[q]]} }
    \simeq
    \frac{\Q[[q]]}{(\mf_{2\ell})_\bQ + {\kappa_{4\ell}}\bZ[[q]]} 
\end{align} where {\begin{equation}
\kappa_{4\ell} = \begin{cases}
1 & \text{if $4\ell=0$ mod $8$, }\\
2 & \text{if $4\ell=4$ mod $8$. }
\end{cases}
\end{equation}}%
They are defined by three different ways: $b^\an$ is given in terms of eta invariants on string manifolds $M_d$ with $[M_d] \in A_{4\ell-1}^\stri$, $b^\geom$ is given by taking a spin-null bordism $N_{4\ell}$ of $M_{4\ell - 1}$ and integrating characteristic forms, and $b^\top$ is given by a homotopy-theoretic construction. 
Recall also the invariant  $b^\tmf$ which we studied in Sec.~\ref{subsubsec_BM_tmf},
which was defined instead on 
\begin{equation}
A_{4\ell-1}=\ker (\sigma \colon \pi_{4\ell-1} \TMF \to \pi_{4\ell-1}\KO((q))).
\end{equation}
One of the main results of \cite{BunkeNaumann} was the following equalities, 
\begin{align}
    b^\an = b^\geom = b^\top = b^\tmf \circ \Wit_\stri,
\end{align}
where we already discussed $b^\tmf$ in Sec.~\ref{subsubsec_BM_tmf}.\footnote{%
Note that in Sec.~\ref{subsubsec_BM_tmf} and in \cite{BunkeNaumann},
the isomorphism $\pi_{8k+4}\KO\simeq \bZ$ was used instead of $\pi_{8k+4}\KO\simeq 2\bZ$.
This accounts for the appearance of $\kappa_{4\ell}$ in this section.
}

Here we would like to make some comments on  $b := b^\an = b^\geom = b^\top$ 
defined on $A^\stri_{4\ell-1}$.
We begin with the following lemma.
\begin{lem}\label{lem_b_top}
For  $[N, M] \in \pi_{4\ell}(\MSpin/\MString)$, 
we have
\begin{equation}
(\Wit_{\spin/\stri})_\R([N, M]) \in \R[[q]]/(\mf_{2\ell})_\R \subset \R((q))/(\MF_{2\ell})_\R
\end{equation}
and we have
    \begin{align}\label{eq_b_varpi}
        b([M]) = (\Wit_{\spin/\stri})_\R ([N, M] ) \mod {\kappa_{4\ell}}\Z[[q]]. 
    \end{align}
\end{lem}
\begin{proof}
The first statement follows by the factorization 
\begin{equation}
\MSpin/\MString \to \ko[[q]]/\tmf \to \KO((q))/\TMF
\end{equation}
of $\Wit_{\spin/\stri}$.
    The equation \eqref{eq_b_varpi} then follows by the fact that the definition of $b^\top$ in \cite[Section 4.1]{BunkeNaumann} is exactly the same as the definition of  our $(\Wit_{\spin/\stri})_\R$ mod ${\kappa_{4\ell}}\Z[[q]]$ given in Proposition~\ref{prop_varpi_integration}.
%    The proof of \eqref{eq_varpi_integration} is contained in the proof of $b^\geom=b^\top$ in \cite[Theorem 4.2]{BunkeNaumann}.
\end{proof}

The combination of Proposition \ref{prop_diff_pairing_easy_-1mod4} and Lemma \ref{lem_b_top} allows us to write some parts of the differential pairing in terms of Bunke-Naumann invariants, as follows. 
Note that we cannot replace $(\Wit_{\spin/\stri})_\R([N, M])$ with $b([M])$ in the formula \eqref{eq_diff_pairing_easy_-1mod4} for general $f \in (\MF_{2k})_\R$. 
However this replacement is well-defined when $f \in \kappa_{4k}\MF_{2k}$. 
On the other hand, the $\R/\Z$-valued pairing is trivial when $f$ is in the image of $\pi_{4k}\TMF \to \kappa_{4k}\MF_{2k}$ because we have the following exact sequence
\begin{align}
    \TMF^{-4k} \to \TMF_\R^{-4k} \xrightarrow{a} \widehat{\TMF}^{-(4k-1)}. 
\end{align}
Thus we get the following proposition,
which geometrizes the combination of Definition~\ref{defn:individual} and Proposition~\ref{prop:individual}:
\begin{prop}
The pairing $\widehat{\alpha}_{\spin/\stri}$ for $d = 4k-1$ restricts to a pairing
    \[
    \widehat{\alpha}_{\spin/\stri} \colon \mathrm{coker}\left(\pi_{4k}\TMF \to {\kappa_{4k}}\MF_{2k} \right) \times \pi_{-4k-20}\MSpin/\MString \to \R/\Z, 
\]
and it is given in terms of the Bunke-Naumann invariant as
\[
    \widehat{\alpha}_{\spin/\stri}([f], [N, M]) = {\frac{1}{2}}\Delta(q)f(q)b([M])\big|_{q^0}. 
\]
\end{prop}

\begin{rem}
This differential-geometric expression of the Bunke-Naumann invariant, 
which uses \eqref{eq_b_varpi} and Proposition~\eqref{prop_varpi_integration}, and then restricted to $-4k-20=4$, i.e.~for three-dimensional string manifolds $M$,
was originally found independently by \cite{Redden_2011} alongside with \cite{BunkeNaumann}.
Its interpretation via a TQFT for geometric string structure was recently given in \cite{fiorenza2023integrals}.
\end{rem}

%\textcolor{red}{Logically, the content of Subsubsection \ref{subsubsec_BM_tmf} should be explained here, as the case $F=\KO((q))/\TMF$...} 

\subsubsection{The differential pairings  in terms of the differential pushforwards}
Now we move on to the more difficult parts of the differential pairings. %, where the classes involved are possibly topologically nontrivial. 
We give a general formula in Proposition \ref{prop_diff_pairing_generalformula}, and explain a conjectural relation with eta invariants for spin manifolds with string boundaries. 

The special feature of our setting is that
the pairing induced by $\alpha_{\spin/\stri}$ 
comes from combining the composition
\begin{align}\label{eq_rel_genus_module}
    \TMF \wedge (\MSpin/\MString) \xrightarrow{\id \wedge \Wit_{\spin/\stri}}
    \TMF \wedge (\KO((q))/\TMF) \xrightarrow{\text{multi}}
    \KO((q))/\TMF, 
\end{align}
and the element $\alpha_{\KO((q))/\TMF} \in \pi_{-20}I_\Z \KO((q))/\TMF$. 
For such a transformation, we can apply the result of \cite{YamashitaAndersondualPart2} to understand the pairing in terms of {\it differential pushforwards}. 

At the topological level in general, given a map of spectra of the form $E \wedge M\mathcal{S} \to F$, we get the {\it topological pushforward map} for each closed $\mathcal{S}$-manifold $(N_n, g_N^{\mathcal{S}})$ of dimension $n$ as
\begin{align}\label{eq_top_push_boundaryless}
    p_*^{N} \colon E^\bullet(N) \to F^{\bullet-n}. 
\end{align}
This is a slight generalization, explained e.g. in \cite[Section 6.1]{Yamashita:2021cao}, of the more familiar case of pushforwards associated to multiplicative genera. 
Here we use the relative version of the pushforward. 
Given a map of spectra of the form $E \wedge M\mathcal{S}/M\mathcal{S}' \to F$, for each compact relative $\mathcal{S}/\mathcal{S}'$-manifold $(N_n, M_{n-1}, g_N^\mathcal{S}, g_M^{\mathcal{S}'})$, we get the topological pushforward map as
\begin{align}
    p_*^{(N, M)} \colon E^\bullet(N) \to F^{\bullet-n}. 
\end{align}
The definition is given by the straightforward generalization of the case without boundaries using the relative Pontryagin-Thom construction\footnote{
We outline the argument as follows. 
% The cohomology theory represented by  $M\mathcal{S}/M\mathcal{S}'$ has a Thom element for each vector with $\mathcal{S}$-structure, as well as a choice of null-homotopy of the Thom element for each vector bundle with $\mathcal{S}'$-structure, in a tautological way.  
Given a compact manifold $N_n$ with boundary $M$, we choose an embedding to $\R^{k-1} \times \R_{\ge 0}$ preserving the boundary, for a $k \gg 0$. The relative $\mathcal{S}/\mathcal{S}'$-structure on $(N, M)$ equips the normal bundle $\nu$ with a $\mathcal{S}$-structure, as well as a lift to $\mathcal{S}'$-structure on the restriction $\nu|_M$ to the boundary.  
Thus we obtain the relative Thom class $\theta \in (M\mathcal{S}/M\mathcal{S}')^{k-n}(\mathrm{Thom}(\nu)/\mathrm{Thom}(\nu|_M))$ in a tautological way. 
Now we can follow the usual procedure: given an element in $E^\bullet(N)$, multiply it with $\theta$ using $E \wedge (M\mathcal{S}/M\mathcal{S}') \to F$ to obtain an element in $F^{\bullet+k-n}(\mathrm{Thom}(\nu)/\mathrm{Thom}(\nu|_M))$, push forward to the one-point compactification of $ \R^{k}$ and use the suspension to get an element in $F^{\bullet-n}$. 
}. 
We are interested in the case \eqref{eq_rel_genus_module}. Given a compact relative spin/string manifold $(N_n, M_{n-1})$ we get 
\begin{align}\label{eq_push_top}
    p^{(N, M)}_{*} \colon \TMF^\bullet(N) \to (\KO((q))/\TMF)^{\bullet-n}. 
\end{align}
There is a theory on the differential refinement of the pushforward map, explained e.g., in \cite[Appendix A]{YamashitaAndersondualPart2} for the case without boundary. This is also easily generalized to relative versions. 
Given a relative differential spin/string cycle $(N_n, M_{n-1}, g_N^\spin, g_M^{\stri})$ of dimension $n$ (or more generally an element in $\reallywidehat{\MSpin/\MString}_{n}$), we get the {\it differential pushforward map},
\begin{align}\label{eq_push_diff}
    \widehat{p}_*^{(N, M)} \colon \widehat{\TMF}{}^\bullet(N) \to (\reallywidehat{\KO((q))/\TMF})^{\bullet-n}.
\end{align}
Note that the pushforward maps \eqref{eq_push_top} and \eqref{eq_push_diff} depend on the relative (differential) spin/string structures, but we omit the reference to them from the notation for brevity.

\begin{rem}\label{rem_push_KO}
    The reader should have in mind the analogy to the case of the multiplicative genus $\mathrm{ABS} \colon \MSpin \to \KO$. 
    The topological pushforward map in this case is defined for closed spin manifolds $N_n$ of dimension $n$ and given by the $\KO^*$-valued index of twisted spin Dirac operators, 
    \begin{align}
        p^N_* = \mathrm{Ind}_N \colon \KO^\bullet(N) \to \KO^{\bullet-n}. 
    \end{align}
    The differential pushforward is defined for closed differential spin manifold $(N_n, g_N^\spin)$ as
    \begin{align}
        \widehat{p}^N_* \colon \widehat{\KO}{}^\bullet(N) \to \widehat{\KO}{}^{\bullet-n}. 
    \end{align}
    In particular when $n \equiv 3$ mod $4$, we have $\widehat{\KO}^{-n} \simeq \R/{\kappa_{n+1}}\Z$ and the differential pushforward is given by reduced eta invariants %(divided further by $2$ in the case of $n \equiv 3$ mod $8$)
     of twisted spin Dirac operators, which is the real version of the work \cite{FreedLott}. 
\end{rem}
\begin{rem}\label{rem_push_KO((q))}
    A straightforward generalization of Remark \ref{rem_push_KO} above allows us to understand the differential pushforward refining the multiplicative genus $\Wit_\spin \colon \MSpin \to \KO((q))$. 
    For each closed differential spin manifold $(N_n, g_N^\spin)$ we have the map
    \begin{align}
         \widehat{p}^N_* \colon \widehat{\KO}((q))^d(N) \to \widehat{\KO}((q))^{d-n}. 
    \end{align}
    In the case where $d \equiv 0 \pmod 8$ and $n \equiv 3 \pmod 4$, we represent each element in $\widehat{\KO}((q))^d(N)$ as a formal power series $V $ of real vector bundles with connections over $N$.
    Then we have
    \begin{align}\label{eq_push_KO((q))}
       { \widehat{p}^N_*([V]) = \eta(\WitVec(TN) \otimes V) \pmod{\Z((q))} \in \widehat{\KO}((q))^{d-n} \simeq \R((q))/\kappa_{n+1}\Z((q)), }
    \end{align}
    where $\eta(\WitVec(TN) \otimes V)$ is the reduced eta invariant of twisted Dirac operator with twist given by formal power series of vector bundles $\WitVec(TN) \otimes V$,
    where $\WitVec$ was defined in \eqref{WitVec}.
\end{rem}

\begin{lem}
    We have a canonical isomorphism
    \begin{align}\label{eq_hat_KO/TMF_21}
        (\reallywidehat{\KO((q))/\TMF})^{21} \simeq \frac{\R((q))}{(\MF_{-10})_\R + {2}\Z((q))}. 
    \end{align}
\end{lem}
\begin{proof}
    This follows from the exact sequence in the axiom of differential cohomology, 
    \begin{multline*}
        (\KO((q))/\TMF)^{20}\to(\KO((q))/\TMF)_\R^{20} \\
        \xrightarrow{a} (\reallywidehat{\KO((q))/\TMF})^{21} \xrightarrow{I} (\KO((q))/\TMF)^{21} .
    \end{multline*}
    The first term was determined in Proposition~\ref{prop:KO/TMF20},
while    the last term is shown to be zero thanks to Proposition~\ref{prop:KO/TMF}.
    Then, the cokernel of the first map is canonically identified with the right hand side of \eqref{eq_hat_KO/TMF_21}. 
\end{proof}

Thus, given a relative differential spin/string cycle $(N_{-d-21}, M_{-d-22}, g_N^\spin, g_M^\stri)$ of dimension $n=-d-21$, the differential pushforward \eqref{eq_push_diff} becomes
\begin{align}\label{eq_push_21}
    \widehat{p}^{(N, M)}_* \colon \widehat{\TMF}^{-d}(N) \to \frac{\R((q))}{(\MF_{-10})_\R + {2}\Z((q))}. 
\end{align}
Further, the map
\begin{align}\label{eq_diff_pairing_R/Z}
    \frac{\R((q))}{(\MF_{-10})_\R + {2}\Z((q))} \to \R/\Z, \quad [f(q)] \mapsto {\frac12}\Delta(q) f(q)\Bigm|_{q^0} \pmod \Z. 
\end{align}
is well-defined thanks to Fact~\ref{fact:mf}.
In terms of these maps, we can express the differential pairing \eqref{eq_diff_pairing_MSpin/MString} as follows. 

\begin{prop}\label{prop_diff_pairing_generalformula}
For each $\widehat{x} \in \widehat{\TMF}^{-d}$ and $[N, M, g_N^\spin, g_M^\stri] \in \reallywidehat{\MSpin/\MString}_{-d-21}$, we have
\[
    \widehat{\alpha}_{\spin/\stri}\left(\widehat{x}, [N, M, g_N^\spin, g_M^\stri]\right) 
    = {\frac{1}{2}} \Delta(q) \cdot \left(\widehat{p}^{(N, M)}_*\circ (p^N)^*(x) \right)\Bigm|_{q^0} \pmod \Z, 
\]
i.e., the left hand side is the image of $\widehat{x}$ under the composition, 
\[
   \widehat{\TMF}^{-d} \xrightarrow{(p^N)^*} \widehat{\TMF}^{-d}(N) \xrightarrow{\widehat{p}^{(N, M)}_*}\frac{\R((q))}{(\MF_{-10})_\R +{2} \Z((q))}
   \xrightarrow{{\frac{1}{2}} \Delta(q) \cdot (-)\bigm|_{q^0}} \R/\Z. 
\]
Here the first arrow is the pullback along $p^N \colon N \to \pt$. 
\end{prop}

\begin{proof}
    The proof is a straightforward generalization of the main result in \cite{YamashitaAndersondualPart2}, 
    to the case of relative bordisms,
    by noting the fact that the element $\alpha_{\KO((q))/\TMF} \in \pi_{20} I_\Z  \KO((q))/\TMF
=\Hom(\pi_{-20}\KO((q))/\TMF, \Z)$ is given by the formula corresponding to \eqref{eq_diff_pairing_R/Z}. 
\end{proof}

Thus, to understand the differential pairing \eqref{eq_diff_pairing_MSpin/MString}, we are left to understand the differential pushforward \eqref{eq_push_21}. 
Again let us start from easier cases. 
We can understand the differential pushforward $\widehat{p}^{(N, M)}$ for $M = \varnothing$ by the following commutative diagram, 
\begin{align}\label{diag_push_boundary}
    \vcenter{\xymatrix{
\widehat{\TMF}{}^d(N) \ar[r]^-{\widehat{p}^{(N, \varnothing)}} \ar[d]^-{\widehat{\sigma}} & \reallywidehat{\KO((q))/\TMF}{}^{d-n} \\
    \widehat{\KO}((q))^d(N) \ar[r]^-{\widehat{p}^N_{\KO((q))}} & \widehat{\KO}((q))^{d-n} \ar[u]^-{\widehat{C\sigma}}
    }},
\end{align}
Here $\widehat{p}^N_{\KO((q))}$ is the differential pushforward in $\widehat{\KO}((q))$ for differential spin manifolds explained in Remark \ref{rem_push_KO((q))}. 
Using the formula \eqref{eq_push_KO((q))}, we see that the differential pairing in the case $d \equiv 0 \pmod 8$ and $M = \varnothing$ is given in terms of eta invariants as follows. 

\begin{prop}
Let $k$ be an integer. 
    For each $\widehat{x} \in \widehat{\TMF}{}^{-8k}$ and 
    \begin{equation}
    [N, \varnothing, g_N^\spin] \in \reallywidehat{\MSpin/\MString}_{-8k-21},
\end{equation}
     we have
    \begin{align}\label{eq_diff_pairing_without_boundary}
        \widehat{\alpha}_{\spin/\stri}\left(\widehat{x}, [N, \varnothing, g_N^\spin]\right) 
    = \left.  \frac{1}{2}\Delta(q) \sigma(I(\widehat{x})) \eta(\WitVec(TN)) \right|_{q^0} \pmod \Z,
    \end{align}
    Here $\sigma(I(\widehat{x}))$ is the image of $\widehat{x}$ under the composition 
    \begin{align}
        \widehat{\TMF}^{-8k} \xrightarrow{I} \pi_{8k}\TMF \xrightarrow{\sigma} \pi_{8k}\KO((q)) \simeq \Z((q)). 
    \end{align}
    In particular the value of the pairing \eqref{eq_diff_pairing_without_boundary} only depends on the image of $\widehat{x}$ in $\MF_{4k}$. 
\end{prop}

The differential pushforward in \eqref{eq_push_21} in the presence of nontrivial boundary should be defined so that it generalize the formula in terms of eta invariants in the boundaryless case. 
It is natural to expect that we have a formula in terms of eta invariants of Dirac operators {\it with boundary}. 
By the work of Dai and Freed \cite{DaiFreed}, for differential spin manifolds with spin boundary $(N_{8\ell+3}, M_{8\ell+2}, g_N^\spin, g_M^\spin)$, the reduced eta invariants of the Dirac operators with the Atiyah-Patodi-Singer boundary condition are canonically defined as elements in the Pfaffian line associated to $(M, g_M^\spin)$, 
\begin{align}
    \exp\left(\sqrt{-1}\pi \cdot \eta(N_{8\ell+3}, M_{8\ell+2}, g_N^\spin, g_M^\spin)\right) \in \mathrm{Pf}(M, g_M^\spin), 
\end{align}
The Pfaffian line does not  admit a canonical trivialization in general, so we cannot talk about the reduced eta invariants as an element in $\R/\Z$ in this case. 

The parallel story is true for the Witten genus version of the eta invariant appearing in the formula in \eqref{eq_push_KO((q))}. 
The reduced Witten-eta invariant is defined as an element 
\begin{align}
    \exp\left(\sqrt{-1}\pi\cdot \eta \left(\WitVec(N_{8\ell+3}, M_{8\ell+2}, g_N^\spin, g_M^\spin)\right)\right) \in \mathrm{Pf}(\WitVec(M, g_M^\spin)), 
\end{align}
where $\mathrm{Pf}(\WitVec(M, g_M^\spin))$ is the Witten genus version of Pfaffian lines, a formal power series of complex lines. 
However, we conjecture that differential string structure on the boundary solves the problem:   

\begin{conj}\label{conj_diff_push_wit_eta}
    For each closed differential spin manifold $(M_{8\ell + 2}, g_M^\spin)$ of dimension $8 \ell + 2$, a lift of $g_M^\spin$ to a differential string structure $g_M^\stri$ gives a trivialization of $\mathrm{Pf}(\WitVec(M, g_M^\spin))$ up to multiplications of $\exp(i(\MF_{4\ell + 2})_\R)$, which makes the fractional reduced eta invariant a well-defined element 
    \[
       {\eta\left(\WitVec(N_{8\ell + 3}, M_{8\ell + 2}, g_N^\spin, g_M^\stri) \right) }\in \frac{\R((q))}{(\MF_{4\ell + 2})_\R + {2} \Z((q))}. 
    \]
    In terms of this, the differential pushforward \eqref{eq_push_21} for $d = 8k$ and elements $\widehat{x} \in \widehat{\TMF}{}^d$ is given by the formula
    \[
        \widehat{p}^{(N, M)}((p^N)^*(\widehat{x})) =  \frac{1}{2}\sigma(I(\widehat{x}))\eta\left(\WitVec(N, M, g_N^\spin, g_M^\stri) \right) \in \frac{\R((q))}{(\MF_{-10})_\R + {2}\Z((q))}. 
    \]
\end{conj}

Assuming Conjecture \ref{conj_diff_push_wit_eta}, by Proposition \ref{prop_diff_pairing_generalformula} we get the following formula for the differential pairing for the case $d \equiv 0 \pmod 8$. 
\begin{cor}[{of Conjecture \ref{conj_diff_push_wit_eta}}]
    Let $d = 8k$ for an integer $k$.
    For each $\widehat{x} \in \widehat{\TMF}^{-d}$ and $[N, M, g_N^\spin, g_M^\stri] \in \reallywidehat{\MSpin/\MString}_{-d-21}$,
    we then have
    \[
        \widehat{\alpha}_{\spin/\stri}\left(\widehat{x}, [N, M, g_N^\spin, g_M^\stri]\right) =\left.  \frac{1}{2}\Delta(q) \sigma(I(\widehat{x})) \eta\left(\WitVec(N, M, g_N^\spin, g_M^\stri) \right) \right|_{q^0}. 
    \]
\end{cor}

\section{On two TMF classes of string-theoretical interest}
\label{sec:diff}
In this section we employ our differential-geometric pairing and other techniques to identify certain $\TMF$ classes of string-theoretical interest.
These classes play an essential role in the formulation of the conjecture relating $\TMF$ classes and vertex operators algebras given in Appendix~\ref{app:VOA}.
To explain the $\TMF$ classes we study, let us start by recalling
a standard mathematical way to construct  $\TMF$ classes.
Take
a string manifold $M$ of dimension $d$.
This determines a class $[M]\in\pi_d\MString$,
and we can form $\Wit_\text{string}([M])\in \pi_d\TMF$,
where $\Wit_\text{string}:\MString\to\TMF$ is the string orientation.

A more general construction we use is the following:
\begin{const}
\label{const}
Suppose we have an element $u\in \TMF^{k+t}(BG)$,
where {$t \colon BG \to K(\Z, 4)$ is a map specifying a twist of $\TMF$ via the composition with $K(\Z, 4) \to BGL_1(\MString) \to BGL_1(\TMF)$; for twists of $\TMF$, see \cite{ABG}. }
We start from a manifold $M$ of dimension $n$
equipped with a $G$-bundle whose classifying map is $f:M\to BG$.
We assume that a string structure twisted by 
{$f^*(t) := f \circ t \colon M \to K(\Z, 4)$} is given on $M$;
see Appendix~\ref{app_subsec_diff_string} for the definition of twisted (differential) string structure.
This determines the fundamental class $[M]\in \MString_{n+f^*(t)}(M)$,
which then gives, via $\Wit_\text{string}$,
the fundamental class $[M]\in\TMF_{n+f^*(t)}(M)$
which we denote by the same symbol by a slight abuse of notation.
The pull-back $f^*(u)\in \TMF^{k+f^*(t)}(M)$ can now be paired against it,
defining a class $f^*(u)/[M] \in \TMF^{k-n}(\pt)$.
\end{const}

\begin{physremark}
This construction is ubiquitous in heterotic string theory, for which the class $u\in \TMF^{k+t}(BG)$ for $k=32$ represents the left-moving current algebra theory with $G$ symmetry.
Then, $M$ is the  $n$-dimensional internal space,
$f:M\to BG$ is the gauge field, and $f^*(u)/[M]$ is the $\TMF$ class
corresponding to the worldsheet theory for the compactification on this internal manifold and the gauge field.
\end{physremark}

In this section we consider two elements of this type, $x_{-d}\in \pi_{-d}\TMF$ for $d=31$ and $28$:
\begin{itemize}
\item 
For $x_{-31}$, we take an element $a\in \TMF^{32}(B\Sigma_2)$
where $\Sigma_2$ is a symmetric group on two points,
pull it back to $S^1$ via a map $\mu: S^1\to \Sigma_2$,
and take the slant product.
The element $a$ is a lift of the element $(e_8)^2\in \TMF^{32}(\pt)$,
where $e_8$ is the unique element specified by its modular form image $c_4/\Delta$,
and the $\Sigma_2$ action on $a$ exchanges the two factors of $e_8$.
\item
For $x_{-28}$, we start from a lift  $\hat e_8\in\TMF^{16+\tau}(BSU(2))$,
which we construct in Appendix~\ref{app:equivariant},
where $\tau \colon BSU(2) \to K(\Z, 4)$ is the map %in \cite[Example 5.1.5]{FSSdifferentialtwistedstring} 
which represents $-c_2=\H^4(BSU(2); \Z)$.
Then, we consider $\hat e_8\hat e_8'\in \TMF^{32+\tau+\tau'}(B(SU(2)\times SU(2)))$,
where primes are used to refer to the second factor of $SU(2)$.
We pull it back to $S^4$ via a map $f:S^4\to SU(2)\times SU(2)$,
and take the slant product to define $x_{-28}$.
\end{itemize}

We will see that each of $x_{-d}$ for $d=31$ and $28$ is the generator of the kernel $A_{-d}\simeq\Z/2$ of $\sigma:\pi_{-d}\TMF\to\pi_{-d}\KO((q))$.
The methods to be used are completely different, however:
for $x_{-31}$ we use its construction as a power operation in $\TMF$ to study it in the Adams spectral sequence,
and for $x_{-28}$ we use the differential-geometric Anderson dual pairing we discussed in the previous section.

Furthermore, we will show
that there is 
both  $a$ and $\hat e_8 \hat e_8'$ can be lifted to a single element 
\begin{equation}
\tilde a \in \TMF^{32+\tilde \tau}(B\tilde G)
\quad
\text{where} 
\quad
\tilde G=(SU(2)\times SU(2))\rtimes \Sigma_2.
\end{equation}
Together with a twisted string bordism we explicitly describe,
this will allow us to derive $x_{-28}=x_{-31}\nu$, where $\nu$ is a generator of $\pi_3\TMF$.
It is known that the generators of $A_{-31}$ and $A_{-28}$ are similarly related by the multiplication by $\nu$ \cite[Table~10.1]{BrunerRognes}.
Therefore, the three statements, 
i) $x_{-31}$ generates $A_{-31}$,
ii) $x_{-28}$ generates $A_{-28}$, and
iii) $x_{-28}=x_{-31} \nu$ 
are related in such a way that if any two of them are proved, the last one automatically follows.
We emphasize that we prove the three statements i), ii) and iii) independently.
See Table~\ref{table:info} for a summary of the information presented so far.

\begin{table}[h]
\[
\qquad\qquad\begin{array}{ccccc}
&\multicolumn{3}{c}{
 \tilde a \in \TMF^{32+\tilde \tau}(B((SU(2)\times SU(2))\rtimes \Sigma_2))
 }
   \\
 &\swarrow && \searrow \\
\multicolumn{2}{c}{a \in \TMF^{32}(B\Sigma_2)} &&
\multicolumn{2}{c}{\hat e_8\hat e_8' \in\TMF^{32+\tau+\tau'}(B(SU(2)\times SU(2)))} \\
& \searrow && \swarrow \\
&\multicolumn{3}{c}{
(e_8)^2 \in \TMF^{32}(\pt) }
\end{array}
\]
\bigskip
\[
\begin{array}{|c||c|c|c||c|c|}
\hline
\text{class} & u & M & G & \text{generator of} & \text{method of determination} \\
\hline
 x_{-31}  & a & S^1 & \Sigma_2 & A_{-31}& \text{power operation in Adams S.S.} \\
 x_{-28} & \hat e_8\hat e_8' & S^4 & E_8\times E_8 &A_{-28}& \text{differential Anderson pairing} \\
\hline
\end{array}
\]
\caption{Basic information on the two $\TMF$ classes $x_{-31}$ and $x_{-28}$ we study,
where $u,M,G$ refer to Construction~\ref{const}.
We will also prove $x_{-28} = x_{-31}\nu$.
\label{table:info}}
\end{table}

\begin{physremark}
The classes $x_{-31}$  and $x_{-28}$ represent, under the Stolz-Teichner proposal,
the worldsheet superconformal field theory for the angular parts 
of the non-supersymmetric 7- and 4-branes discussed in \cite{KaidiOhmoriTachikawaYonekura,Kaidi:2024cbx}, respectively.
The desire to understand these branes, and in particular to compute their Green-Schwarz couplings, 
was a major part of the motivation to study these two classes in $\TMF$, 
and more broadly to develop the mathematical results  in this paper. 
\end{physremark}

The rest of this section is organized as follows.
\begin{itemize}
\item
In Sec.~\ref{subsec:31}, we introduce the class $x_{-31}\in \pi_{-31}\TMF$,
using a class $a\in\TMF^{32}(B\Sigma_2)$ and $M=S^1$ in the construction above.
We show that $x_{-31}$ is in the kernel $A_{-31}$ of $\sigma: \pi_{-31}\TMF\to \pi_{-31}\KO((q))$,
and state that $x_{-31}$ is in fact the generator of $A_{-31}\simeq \bZ/2$,
whose proof we postpone until Sec.~\ref{app:power}.
\item
In Sec.~\ref{subsec:28}, we introduce a class $x_{-28}\in \pi_{-28}\TMF$.
This will be done by starting from a class $\hat e_8\hat e_8'\in \TMF^{32+\tau+\tau'}(B(SU(2)\times SU(2)))$ and $M=S^4$ in the construction above.
We show that $x_{-28}$ is in the kernel $A_{-28}$ of $\sigma: \pi_{-28}\TMF\to \pi_{-28}\KO((q))$,
and state that $x_{-28}$ is in fact the generator of $A_{-28}\simeq \bZ/2$,
whose proof we postpone until Sec.~\ref{sec:computation}.
\item
In Sec.~\ref{subsec:preliminary}, we show that $x_{-28} = x_{-31} \nu$.
This will be done by lifting both of $a$ and $\hat e_8 \hat e_8'$ to a single element 
$\tilde a \in \TMF^{32+\tilde \tau}(B(SU(2)\times SU(2))\rtimes \Sigma_2)$
and by considering an explicit twisted string bordism relating them.
\item
In Sec.~\ref{app:power}, we prove that $x_{-31}$ generates $A_{-31}$ by studying the behavior of a power operation in $\TMF$ in the Adams spectral sequence.
\item
In Sec.~\ref{sec:computation}, we prove that $x_{-28}$ generates $A_{-28}$ by computing the differential-geometric pairing of $x_{-28}$ against the generator $\nu^2$ of the Pontryagin dual $A_{6}$ of $A_{-28}$.
\end{itemize}

\begin{rem}
The main results of Sections \ref{subsec:preliminary}, \ref{app:power} and \ref{sec:computation}
are derived completely independently of each other,
and these three subsections can be read independently.
See Figure~\ref{fig:logic} on the logical relationship among the subsections of this section.
\end{rem}
\begin{figure}[h]
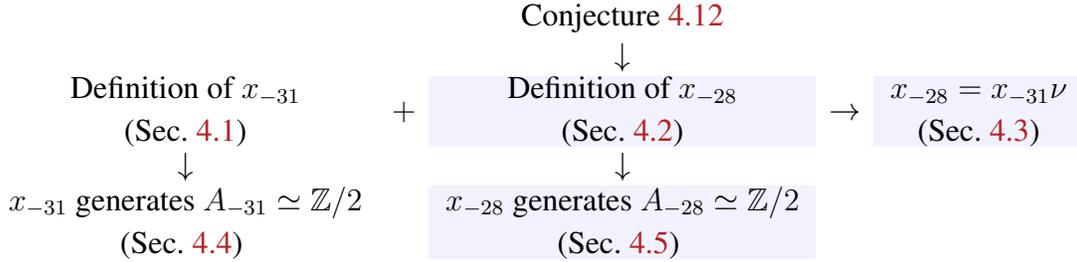

\[
\begin{array}{ccccc}
\displaystyle{
\text{Definition~of $x_{-31}$}
\atop
\text{(Sec.~\ref{subsec:31})} 
}
& + & 
\displaystyle{
\text{Definition~of $x_{-28}$}
\atop
\text{(Sec.~\ref{subsec:28})} 
}
& \to &
\displaystyle{
\text{$x_{-28}=x_{-31}\nu$}
\atop
\text{(Sec.~\ref{subsec:preliminary})}
} \\
\downarrow &&\downarrow \\
\displaystyle{
\text{$x_{-31}$ generates $A_{-31}\simeq\Z/2$}
\atop
\text{(Sec.~\ref{app:power}) }
}
&&
\displaystyle{
\text{$x_{-28}$ generates $A_{-28}\simeq\Z/2$}
\atop
\text{(Sec.~\ref{sec:computation}) }
}
\end{array}
\]
\caption{Logical relationship among subsections of this section.
The main statements of  Sections~\ref{subsec:preliminary},
\ref{app:power} and \ref{sec:computation}
are such that if any two of them are proved, the last one automatically follows.
\label{fig:logic}}
\end{figure}

\subsection{The class $x_{-31}\in \pi_{-31}\TMF$}
\label{subsec:31}
Our first class is given by a power operation applied on the following element $e_8$:
\begin{defn}
\label{defn:e8}
The class $e_8\in \pi_{-16}\TMF$  is the unique class 
specified by the condition that its image $\sigma(e_8)\in \pi_{-16}\KO((q))$ is given by $c_4/\Delta$.
\end{defn}

Recall that given an element $x\in \KO^{8n}(\pt)$ specified in terms of a vector space $V$ via $x=[V]$,
its square $x^2 \in \KO^{16n}(\pt)$ can be canonically lifted to a $\Sigma_2$-equivariant element 
$a\in \KO^{16n}(B\Sigma_2)$, where $\Sigma_2$ is the symmetric group acting on two points,
exchanging two factors of $V$ in $x^2=[V\otimes V]$.
As explained in \cite[Sec.~11.1]{BrunerRognes} and as we review in Sec.~\ref{app:power},
this process can be defined for any $E_\infty$ ring spectrum $E$,
defining an operation \begin{equation}
P_2 : E^{n}(\pt) \to E^{2n+n\rho}(B\Sigma_2),
\end{equation} where $\rho$ is the twist specified by the composition $B\Sigma_2 \to BO \to BGL_1 S \to BGL_1 E $, where the first map classifies the defining representation of $\Sigma_2$. 
We let $\mu:S^1\to B\Sigma_2$ be the classifying map of the M\"obius bundle.
Since the pullback of $2\rho$ to $S^1$ is trivial, 
we get a power operation \begin{equation}
\mu^* \circ P_2 : E^n(\pt)\to E^{2n}(S^1)
\end{equation} when $n$ is even.
We now apply this construction when $E=\TMF$ and $n=16$.
In this case $P_2$ in fact sends $\TMF^{16}(\pt)$ to $\TMF^{32+16\rho}(B\Sigma_2)=\TMF^{32}(B\Sigma_2)$,
since $\rho$ is of order $8$ in $BGL_1\TMF$.

\begin{defn}
\label{defn:31}
We let $a=P_2(e_8) \in \TMF^{32}(B\Sigma_2)$ as above,
canonically lifting the class $(e_8)^2\in \TMF^{32}(\pt)$ and equipping it with an action of $\Sigma_2$ exchanging two factors.
We then define \[
x_{-31} := \mu^*(a)/[S^1] \in \pi_{-31}\TMF,
\]
where we equip $S^1$ with the trivial string structure.
\end{defn}

We  note the following: 
\begin{prop}
\label{prop:image-in-KO-31}
 $x_{-31}$ is in the kernel $A_{-31}$ of  $\sigma:\pi_{-31}\TMF\to\pi_{-31}\KO((q))$.
\end{prop}
\begin{proof}
The image $\sigma(x_{-31})$ can be obtained by performing the construction of Definition~\ref{defn:31} against $\sigma(e_8)=c_4/\Delta\in \pi_{-16}\KO((q))\simeq \bZ((q))$.
Let $V:=\bigoplus_{n\ge -1} V_n q^n$ be a real vector space representing $c_4/\Delta$
and let $P\to S^1$ be the nontrivial $\Sigma_2$ bundle over $S^1$.
Then we have \begin{equation}
\sigma(x_{-31})=\index_{S^1} V^{\otimes 2} \times_{\Sigma_2} (P\to  S^1)
\label{eq:mod2ind31}
\end{equation}
where $\index_{S^1}(E)$ is the mod-2 index of the real bundle $E$ on $S^1$ with trivial spin structure
and $\Sigma_2$ acts on $V^{\otimes 2}$ by exchanging two factors.

The expression \eqref{eq:mod2ind31} can be seen to vanish by a direct computation.
Another method to see its vanishing is as follows. 
Note that the leading term in \eqref{eq:mod2ind31} is $(\index_{S^1} \bR^{\otimes 2}\times_{\Sigma_2} (P\to  S^1) )q^{-2} = 0 q^{-2}$ 
and therefore the pole of $\sigma(x_{-31})$ can be at most of order 1.
Now, the image $\sigma(\pi_{-31}\TMF)\subset \pi_{-31}\KO((q)) \simeq \bZ/2((q))$ is known to be 
generated by $\eta (c_4/\Delta)^2 J^k$ for $k\ge 0$, see Appendix~\ref{app:TMF}.
Therefore any nontrivial element in $\sigma(\pi_{-31}\TMF)$ has a pole at least of order 2.
This means that $\sigma(x_{-31})$ is actually zero.
\end{proof}

According to the data gathered in Appendix~\ref{app:TMF},
we have  $A_{-31}\simeq \bZ/2$. 
In fact, as will be  proved in Sec.~\ref{app:power},
\begin{thm}
\label{thm:31}
$x_{-31}$ is the generator of $A_{-31}\simeq \bZ/2$,
\end{thm}
The proof is based on the fact that 
the definition of the element $x_{-31}$ in Definition~\ref{defn:31}
 is an example of power operations in $\TMF$ described in \cite[Chapter 11.1]{BrunerRognes}.
This allows us to compute it using the data provided there concerning power operations in the Adams spectral sequence.

\subsection{The class $x_{-28}\in\pi_{-28}\TMF$}
\label{subsec:28}
Our second element is in $\pi_{-28}\TMF$. To describe its construction, 
we need to use genuinely $SU(2)$-equivariant $\TMF$.
As details will be given in Appendix~\ref{app:equivariant},
we will be brief here.

Recall that we have a natural ring homomorphism $\pi_{\bullet}\TMF \to \MF_{\bullet/2}$.
This was compatible with the homomorphism $\pi_{4n}\TMF \to \pi_{4n}\KO((q))$
and the $q$-expansion of the modular forms.
When extended to $\RO(SU(2))$-graded genuinely equivariant $\TMF$, we have \begin{equation}
e_{SU(2)}: \TMF^{m+k\overline V_{SU(2)}}_{SU(2)}(\pt) \to \JF^{\ind = 2k}_{ \wt = \frac{m}{2}}, \label{toJF}
\end{equation}
where
$\overline V_{SU(2)} = V_{SU(2)} - 4\underline \R \in \RO(SU(2))$ is the virtual representation of $SU(2)$ associated to the fundamental representation $V_{SU(2)}$ of $SU(2)$, and
$\JF^{\ind =i}_{ \wt=w}$ is the index $i$, weight $w$ part of the ring of integral Jacobi forms,
for which we follow the notation of  \cite{Gritsenko2020382}.
This map $e_{SU(2)}$ is compatible with the homomorphism \begin{equation}
\TMF_{SU(2)}^{2n+k\overline{V}_{SU(2)}}(\pt) \to \KO^{2n}_{SU(2)}((q))(\pt) \subset \KU^{2n}_{U(1)}((q))(\pt) \simeq \bZ[y ,y ^{-1}]((q))
\end{equation}
when Jacobi forms are expanded using the standard variables $q=e^{2\pi i\tau}$ and $y =e^{2\pi i z}$,
where $y \in \KU^0_{U(1)}(\pt)$ now stands for the defining one-dimensional complex representation of $U(1)$.

In this section, we use the following fact, which will be proved in Appendix~\ref{app:equivariant}:
\begin{prop}\label{prop_TEJF2}
	The class $e_8\in \pi_{-16}\TMF$ has a preferred lift to an element in $RO(SU(2))$-graded genuinely equivariant $\TMF$,
	\begin{equation}
		\check e_8 \in \TMF^{16+\overline V_{SU(2)}}_{SU(2)}(\pt)
	\end{equation}
	which has an image 
	\begin{align}\label{eq_E41Delta}
		e_{SU(2)} (\check e_8) = E_{4, 1} / \Delta, 
	\end{align}
	where the notation $E_{4, 1} = c_4 + O(z^2) \in \JF^{\ind = 1}_{\wt = 8}$ follows \cite{Gritsenko2020382}. 
	%The even Jacobi form \eqref{eq_E41Delta} is the character \mayuko{Is it the correct terminology?} of the unique integrable representation of the affine Lie algebra $(\mathfrak{e}_8)_1$. 
\end{prop}

The map from genuine fixed points to homotopy fixed points gives the map from $\RO(SU(2))$-graded $\TMF$ to the corresponding Borel equivariant twisted $\TMF$, 
\begin{align}
    \zeta \colon \TMF^{16+\overline V_{SU(2)}}_{SU(2)}(\pt) \to \TMF^{16+\tau}(BSU(2)). 
\end{align}
This map is the analogue of the Atiyah-Segal completion map. 

\begin{defn}
\label{defn:hat-e8}
We define $$\hat e_8 := \zeta (\check e_8) \in \TMF^{16+\tau}(BSU(2))$$.
\end{defn}

Consider the tangent bundle $TS^4$ of $S^4$ and fix an orientation. 
Its structure group uniquely lifts to $\Spin(4) \simeq SU(2)\times SU(2)$. 
Let $f:S^4\to BSU(2)\times BSU(2)$ be its classifying map, and
let $c_2$ and $c_2'$ be the standard generators of $\H^4(BSU(2)\times BSU(2);\bZ)\simeq \bZ\oplus \Z$.
We use the convention so that $\int_{S^4} f^*(c_2)=+1$ and $\int_{S^4} f^*(c_2')=-1$.
Recall that our map $\tau:BSU(2)\to K(\Z,4)$ represents $-c_2$.
%Recall that $(SU(2)\times E_7)/\{\pm1\} \subset E_8$ 
%and use it to define a homomorphism $SU(2)\to E_8$.
%With this we can regard $f$ as a map $f:S^4\to BE_8 \times BE_8$.

\begin{defn}
\label{defn:28}
Let us use $f$ to pull-back $\hat e_8 \hat e_8' \in \TMF^{32+\tau+\tau'}(B(SU(2) \times SU(2)))$,
where the prime is added to distinguish two factors.
The resulting pull back is an element \[
f^*(\hat e_8\hat e_8') \in \TMF^{32+f^*(\tau+\tau')}(S^4)=\TMF^{32}(S^4)
\]
where we used $f^*(\tau+\tau')=-f^*(c_2)-f(c_2')=0$.
We then define \[
x_{-28} := f^*(\hat e_8\hat e_8')/[S^4] \in \pi_{-28}\TMF.
\]
\end{defn}

\begin{rem}
%Our convention is that the string structure on $M$ twisted by {$t \colon M \to K(\Z, 4)$ is a nullhomotopy between $\frac12p_1(TM)$ and $t$. } 
%The generator {$[\tau] = 1 \in \Z \simeq \H^4(BE_8;\bZ)$} is chosen so that it pulls back to $\frac12p_1$ of $\H^4(B\Spin(16);\bZ)$.
%This then implies that under the homomorphism $SU(2)\to E_8$ determined by a root, 
%{$[\tau]$} pulls back to $-c_2$ of $\H^4(BSU(2);\bZ)$.
%\end{rem}
%\begin{remc
It is a good exercise for the reader to compute the rationalized image
of $\hat e_8\hat e_8'$ under $\alpha_\text{spin}\circ \sigma$ in $(I_\bZ\MSpin)^{12}(BSU(2)\times BSU(2))_\bQ$.
We have\begin{equation}
\alpha_\text{spin}\circ \sigma(\hat e_8 \hat e_8')_\Q  
= (\frac{p_1}2-\tau-\tau')
\frac{1}{48}
\left(p_2-3(\frac{p_1}{2})^2+2 p_1(\tau+\tau') - 8 (\tau^2 -\tau \tau' +\tau'{}^2)\right)
\end{equation}
which indeed vanishes when sent to $(I_\bZ \MString)^{12+\tau+\tau'}(BSU(2)\times BSU(2))_\bQ$.
This is a modern mathematical rephrasing of the fundamental discovery of Green and Schwarz in \cite{Green:1984sg}
of the perturbative anomaly cancellation of $E_8\times E_8$ heterotic string theory.
\end{rem}

\begin{prop}
\label{prop:image-in-KO-28}
 $x_{-28}$ is in the kernel $A_{-28}$ of  $\sigma:\pi_{-28}\TMF\to\pi_{-28}\KO((q))$.
\end{prop}
\begin{proof}
We have \begin{equation}
\sigma(x_{-28})= f^*(\sigma(\hat e_8) \sigma(\hat e_8'))/[S^4] = 0 q^{-2} + 0 q^{-1} + \cdots.
\end{equation}
Now, the image of $\sigma$ in $\pi_{-28}\KO((q))$ is generated by $(c_4c_6/\Delta) J^k$ for $k\ge 0$, see Appendix~\ref{app:TMF}.
These generators have nonzero poles at least of order 1,
and therefore $\sigma(x_{-28})=0$.
\end{proof}

According to the data gathered in Appendix~\ref{app:TMF},
the subgroup $A_{-28}$ is $\bZ/2$. 
In fact, we will show in Sec.~\ref{sec:computation} that:
\begin{thm}
\label{thm:28}
$x_{-28}$ is the generator of $A_{-28} \simeq \bZ/2$.
\end{thm}

This will be proved by computing the pairing of $x_{-28}$ 
against its Anderson dual element $\nu^2 \in A_6 \simeq \Z/2$ in a differential geometric manner, i.e.~by showing the following:
\begin{prop}
\label{prop:computation}
We have $\alpha_\text{spin/string}(x_{-28},y)=1/2$,
where  $y\in \pi_7\MSpin/\MString$ is a class lifting $\nu^2 \in \pi_6\MString$.
\end{prop}
Note that Proposition~\ref{prop:computation} immediately implies Theorem~\ref{thm:28} because of Proposition~\ref{pont-dual} and Proposition~\ref{prop:geometric-p}.

\subsection{Relation between $x_{-31}$ and $x_{-28}$}
\label{subsec:preliminary}
The two classes $x_{-31}$ and $x_{-28}$ we introduced are related in the following way.
Note that both $x_{-31}$ and $x_{-28}$ are defined by pulling back lifts of $(e_8)^2\in \pi_{-32}\TMF$,
in one case to $\TMF^{32}(B\Sigma_2)$ and in the other case to $\TMF^{32+\tau+\tau'}(B(SU(2) \times SU(2)))$, and then taking the slant product. 
The maps $\mu: S^1\to B\Sigma_2$ and $f:S^4\to B(SU(2)\times SU(2))$
can be both thought of as maps $\mu: S^1\to B\tilde G$ and $f:S^4\to B\tilde G$
where $\tilde G=(SU(2)\times SU(2))\rtimes \Sigma_2$.
%It is easy to see that there is a twist $\tilde\tau\in \H^4(B\tilde G;\bZ)$ 
%such that its restriction on $\Sigma_2$ and on $E_8\times E_8$ is $0$ and $\tau+\tau'$ respectively.
Recall we are using a canonical lift $\hat{e}_8 \in \TMF^{16+\tau}(BSU(2))$ of $e_8$ provided by Definition \ref{defn:hat-e8}. 
The exterior product $\hat{e}_8  \hat{e}'_8 \in \TMF^{32+\tau+\tau'}(B(SU(2) \times SU(2)))$ is invariant under the $\Sigma_2$-action, and lifts canonically to an element $a \in \TMF^{32+\tilde\tau}(B\tilde G)$ as before,
where $\tilde\tau$ lifts $\tau+\tau'$.
Note that the trivial string structures on $S^{1}$ and $S^4$ are actually $(\tilde G,\tilde\tau)$-twisted string structures because both $\mu^*(\tilde\tau)$ and $f^*(\tilde\tau)$ are  trivial.
Summarizing the consideration above, we have:
\begin{prop}
We have $(\tilde G,\tilde\tau)$-twisted string bordism classes \begin{equation}
[S^1,\mu] \in \MString_{1+\tilde\tau}(B\tilde G), \qquad
[S^4,f]\in \MString_{4+\tilde\tau}(B\tilde G)
\label{eq:geom-classes}
\end{equation} such that \begin{equation}
x_{-31} = a/[S^1,\mu] , \qquad x_{-28} = a/[S^4,f].
\label{eq:our-classes}
\end{equation}
\label{prop:classes}
\end{prop}
Regarding the classes $[S^1,\mu]$ and $[S^4,f]$ in \eqref{eq:geom-classes}, we have% the following proposition, whose geometric proof we provide in Sec~\ref{app:fun-with-bordism}:
\begin{prop}
\label{prop:fun-with-bordism}
The twisted string bordism classes $[S^1,\mu] \times \nu $ and $ [S^4,f] $ in $\MString_{4+\tilde\tau}(B\tilde G)$ 
are equal,
where % $[S^1,\mu]$ and $[S^4,f]$ are from \eqref{eq:geom-classes} and
 $\nu$ is the standard generator of  $\MString_3(\pt)\simeq \bZ/24$. %specified by $[S^3, H_1]$ with $\int_{S^3} H_1 = 1$. 
\end{prop}
With this proposition, which will be proved later in this subsection,
it is easy to see the following proposition relating $x_{-31}$ and $x_{-28}$:
\begin{prop}
\label{prop:relation}
$x_{-31} \nu = x_{-28} \in \pi_{-28}\TMF$.
\end{prop}
\begin{proof}
Combine Proposition~\ref{prop:classes},
Proposition~\ref{prop:fun-with-bordism},
and the compatibility of the slant product with the multiplication.
\end{proof}

\begin{rem}
This proposition is compatible with Theorem~\ref{thm:31} and Theorem~\ref{thm:28}.
Indeed, according to \cite[Table 10.1]{BrunerRognes},
$A_{-31}$ and $A_{-28}$ are both $\bZ/2$ and
are generated by the elements $\nu_6\kappa/M$  and $\nu\nu_6\kappa/M$, respectively.
\end{rem}
\begin{rem}
The explicit bordism proving Proposition~\ref{prop:fun-with-bordism} presented below
was found during the collaboration which led  to \cite{KaidiOhmoriTachikawaYonekura,Kaidi:2024cbx}; the authors thank  J. Kaidi, K. Ohmori and K. Yonekura for allowing the authors to present it here.
We also note that a different indirect proof was given in \cite[Remark 2.34]{debray2023bordism},
which used the fact that the natural map from the bordism group of  $(E_8\times E_8,\tau+\tau')$-twisted string manifolds to that of $((E_8\times E_8)\rtimes \Sigma_2,\tilde\tau)$-twisted string manifolds is a surjection when the dimension is four.
In the same paper \cite{debray2023bordism}, the $((E_8\times E_8)\rtimes \Sigma_2,\tilde\tau)$-twisted string bordism groups for dimensions up to eleven were studied in detail.
\end{rem}

\begin{proof}[Proof of Proposition~\ref{prop:fun-with-bordism}]
A string manifold representing the generator $\nu\in \pi_3\TMF\simeq \bZ/24$
is given by a round $S^3$ equipped with the 3-form $H_1$, 
which is a part of the data of the differential string manifold, given by $\int_{S^3} H_1=1$.
Our aim is then to establish a bordism between $(S^1,\mu)\times (S^3, H_1)$ and $(S^4,f)$
in the category of string manifolds equipped with a map to $B\tilde G$,
where $\tilde G=(SU(2)\times SU(2)) \rtimes \Sigma_2$,
with the twist of the string structure specified by {a map $\tilde{\tau} \colon B\tilde G \to K(\Z, 4)$}. 
Here 
\begin{itemize}
\item {$\tilde{\tau} \colon  B\tilde G \to K(\Z, 4)$ is the map obtained by lifting the composition
\begin{align}
    BSU(2) \times BSU(2) \xrightarrow{\tau \times \tau'} K(\Z, 4) \times K(\Z, 4) \xrightarrow{+} K(\Z, 4)
\end{align}
to the homotopy orbit $(BSU(2) \times BSU(2)) \times_{\Sigma_2} E\Sigma_2 \simeq B\tilde{G}$. 
Here, in the first arrow, both $\tau$ and $\tau'$ refer to the same map which we have denoted $\tau$ so far, but here we use $\tau'$ to distinguish the two factors. 
}

% {The homotopy class $[\tilde \tau] \in \H^4(B\tilde G; \Z)$} is such that its pullback to $\H^4(B(E_8\times E_8);\bZ)\simeq\bZ\oplus \bZ$
% is given by the sum {$[\tau]\oplus[\tau']$} of the generators,
\item $\mu:S^1\to B\tilde G$ is obtained from the nontrivial $\Sigma_2$ bundle classified by 
$\mu:S^1\to B\Sigma_2$ by the inclusion $\Sigma_2 \hookrightarrow \tilde G$, and 
\item  $f:S^4\to B\tilde G$ is obtained from the classifying map $f:S^4\to BSO(4)$
which can be uniquely lifted to $f:S^4\to B\Spin(4) \simeq B(SU(2)\times SU(2))\hookrightarrow B\tilde G$.
We fix the isomorphism $\Spin(4) \simeq SU(2) \times SU(2)$ so that $\int_{S^4}\tau=+1$, $\int_{S^4}\tau'=-1$. We have $\int_{S^4}(\tau+\tau')=0$.
\end{itemize}

Let us now consider the following $(\tilde G,\tilde\tau)$-twisted string manifolds.
We start from $[-\pi,\pi] \times S^3$.
We pick a point $p\in S^3$.
For $x_1, x_2, \cdots, x_k \in (-\pi,\pi)$ and a sequence of integers $n_1, n_2, \cdots, n_k$,
let us denote by $P(x_1,n_1;x_2,n_2;\ldots)$
an $E_8$ bundle with connection $\nabla$ over $[-\pi,\pi]\times S^3$
such that the $E_8$ curvature is supported on the interior of non-overlapping balls
$B_i$ around $(x_i,  p)$ 
so that $\int_{B_i} \cw_\nabla(\tau)=n_i$.
We now consider an $E_8\times E_8$ bundle over $[-\pi,\pi]\times S^3$
given by the fiber product of 
$P(x_1,n_1;\ldots)$ and $P(y_1, m_1;\ldots)$.
\iffalse\sout{We denote the two degree-4 characteristic classes of the $E_8\times E_8$ bundle 
by $\tau$ and $\tau'$.
As a $\tilde G$-bundle, the characteristic class $\tilde\tau$ is given by $\tau+ \tau'$.} \textcolor{blue}{MY: Is it ok to erase the two sentences, since it is rather complicated to replace in the language of maps to $K(\Z, 4)$?}\fi

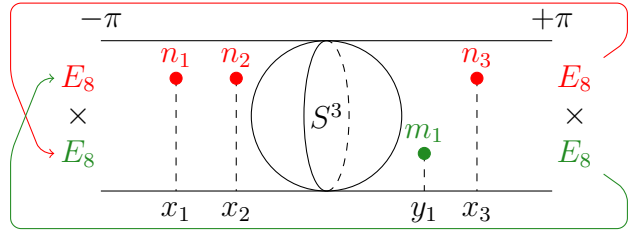
\begin{figure}
\centering
\[
\begin{tikzpicture}
\draw (-3,-1) -- (+3,-1) ;
\draw (-3,+1) node[above]{$-\pi$} -- (+3,+1) node[above] {$+\pi$} ;

\draw (0,0) ellipse (1 and 1);
\draw[dashed] (0,-1) arc[x radius=.3, y radius =1,start angle=-90,end angle=90];
\draw (0,1) arc[x radius=.3, y radius =1,start angle=90,end angle=270];
\node (S3) at (0,0) {$S^3$};

\node[above,red] (n1) at (-2,.5) {$n_1$};
\node[below] (x1) at (-2,-1) {$x_1$};
\draw[dashed] (n1) -- (x1);
\draw[red,fill=red]  (-2,.5) ellipse (.08 and .08);

\node[above,red] (n2) at (-1.2,.5) {$n_2$};
\node[below] (x2) at (-1.2,-1) {$x_2$};
\draw[dashed] (n2) -- (x2);
\draw[red,fill=red]  (-1.2,.5) ellipse (.08 and .08);

\node[above,red] (n3) at (2,.5) {$n_3$};
\node[below] (x3) at (2,-1) {$x_3$};
\draw[dashed] (n3) -- (x3);
\draw[red,fill=red]  (2,.5) ellipse (.08 and .08);

\node[above,ForestGreen] (m1) at (1.3,-.5) {$m_1$};
\node[below] (y1) at (1.3,-1) {$y_1$};
\draw[dashed] (m1) -- (y1);
\draw[ForestGreen,fill=ForestGreen]  (1.3,-.5) ellipse (.08 and .08);

\node[red] (AL) at (-3.3,.5) {$E_8$};
\node at (-3.3,0) {$\times$};
\node[ForestGreen] (BL) at (-3.3,-.5) {$E_8$};
\node[red] (AR) at (3.3,.5) {$E_8$};
\node at (3.3,0) {$\times$};
\node[ForestGreen] (BR) at (3.3,-.5) {$E_8$};

\draw[red,-<,rounded corners] 
(AR) -- (4,1) -- (4,1.5) -- (-4.2,1.5)-- (-4.2,.5)-- (-3.8,-.5)-- (BL) ;
\draw[ForestGreen,-<,rounded corners] 
(BR) -- (4,-1) -- (4,-1.5) -- (-4.2,-1.5)-- (-4.2,-.5) -- (-3.8,.5)-- (AL) ;

\end{tikzpicture}
\]
\caption{The $\tilde G$-bundle $X(x_1,n_1;x_2,n_2;x_3,n_3|y_1,m_1)$. 
We start from two $E_8$ principal bundles on $S^3\times [0,1]$, depicted with different colors,
which are trivial except within balls around finite number of points, depicted by blobs.
The characteristic numbers $\int_B \cw_\nabla(\tau)$ or $\int_B\cw_\nabla(\tau')$ at each ball are given by $n_i$ and $m_i$.
We then glue two ends of the segment $[0,1]$ by exchanging the two $E_8$ factors.
\label{fig:X}}
\end{figure}

We now make it into a $\tilde G$-bundle with connection 
over a spin manifold $S^1\times S^3$ by identifying the two ends of $[-\pi,\pi]$ 
while exchanging two $E_8$ fibers.
Let us denote the resulting $\tilde G$-bundle over $S^1\times S^3$
by $ % \begin{equation}
X(x_1,n_1;\ldots, | y_1, m_1;\ldots).
$ % \end{equation}
See Fig.~\ref{fig:X} for a drawing.

\begin{figure}
\centering
\begin{align*}
(S^4,f) \# 
\begin{tikzpicture}[baseline=(0),scale=.6]
\node (0) at (0,-.2) {};
\draw (-2,1) node[above] {$-\pi$} --(2,1) node[above] {$\pi$};
\draw (-2,-1)   --(2,-1) ;
\end{tikzpicture}
% X(|) 
&  \sim_\text{bord} 
\begin{tikzpicture}[baseline=(0),scale=.6]
\node (0) at (0,-.2) {};
\draw (-2,1) node[above] {$-\pi$} --(2,1) node[above] {$\pi$};
\draw (-2,-1)   --(2,-1) ;
\node[left,red] (r1) at (0,.5) {$+1$};
\node[right,ForestGreen] (g1) at (0,-.5) {$-1$};
\node[below] (a) at (0,-1) {$0$};
\draw[dashed] (0,.5)--(0,-1);
\draw[red,fill=red] (0,.5) ellipse (.16 and .16);
\draw[ForestGreen,fill=ForestGreen] (0,-.5) ellipse (.16 and .16);
\end{tikzpicture}
\qquad X(0,1|0,-1)
\\
&\sim_\text{bord} 
\begin{tikzpicture}[baseline=(0),scale=.6]
\path[fill=cyan!20!white] (0,1) rectangle (1,-1);
\node (0) at (0,-.2) {};
\draw (-2,1) node[above] {$-\pi$} --(2,1) node[above] {$\pi$};
\draw (-2,-1)   --(2,-1) ;
\node[left,red] (r1) at (0,.5) {$+1$};
\node[below]  at (0,-1) {$0$};
\draw[dashed] (0,.5)--(0,-1);
\draw[red,fill=red] (0,.5) ellipse (.16 and .16);
\node[right,ForestGreen] (g1) at (1,-.5) {$-1$};
\node[below]  at (1,-1) {$a$};
\draw[dashed] (1,-.5)--(1,-1);
\draw[ForestGreen,fill=ForestGreen] (1,-.5) ellipse (.16 and .16);
\end{tikzpicture}
\qquad X(0,1|a,-1)
\\
&\sim_\text{bord}
\begin{tikzpicture}[baseline=(0),scale=.6]
\path[fill=cyan!20!white] (0,1) rectangle (2,-1);
\path[fill=cyan!20!white] (-2,1) rectangle (-1.2,-1);
\node (0) at (0,-.2) {};
\draw (-2,1) node[above] {$-\pi$} --(2,1) node[above] {$\pi$};
\draw (-2,-1)   --(2,-1) ;
\node[left,red] (r1) at (0,.5) {$+1$};
\node[below]  at (0,-1) {$0$};
\draw[dashed] (0,.5)--(0,-1);
\draw[red,fill=red] (0,.5) ellipse (.16 and .16);
\node[right,red,xshift=-4] (g1) at (-1,-.5) {$-1$};
\node[below]  at (-1.2,-1) {$a$};
\draw[dashed] (-1.2,-.5)--(-1.2,-1);
\draw[red,fill=red] (-1.2,-.5) ellipse (.16 and .16);
\end{tikzpicture}
\qquad X(a,-1,0,1|)
\\
&\sim_\text{bord} 
\begin{tikzpicture}[baseline=(0),scale=.6]
\path[fill=cyan!20!white] (-2,1) rectangle (2,-1);
\node (0) at (0,-.2) {};
\draw (-2,1) node[above] {$-\pi$} --(2,1) node[above] {$\pi$};
\draw (-2,-1)   --(2,-1) ;
%\node[left,red] (r1) at (0,.5) {$+1$};
%\node[right,red] (g1) at (0,-.5) {$-1$};
%\node[below] (a) at (0,-1) {$0$};
%\draw[dashed] (0,.5)--(0,-1);
%\draw[red,fill=red] (0,.5) ellipse (.16 and .16);
%\draw[red,fill=red] (0,-.5) ellipse (.16 and .16);
\end{tikzpicture}
\quad = (S^1,\mu) \times (S^3,H_1)
\end{align*}
\caption{The bordism from $(S^4,f)$ to $(S^1,\mu)\times (S^3, H_1)$. 
We follow the convention of Fig.~\ref{fig:X} for the drawing.
In the unshaded region and in the shaded region, we have $\int_{S^3} H=0$ and $\int_{S^3} H=1$,
respectively.
\label{fig:bordism}}
\end{figure}
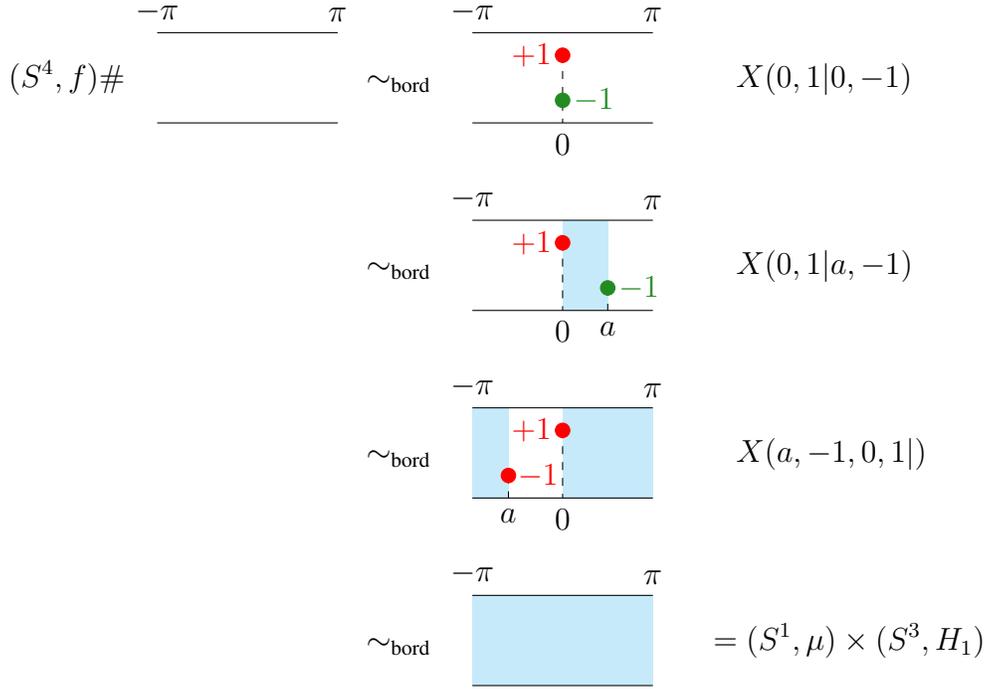

There are two choices of spin structure; we pick a bounding one.
The only obstruction to equip it with a $(\tilde G,\tilde\tau)$-twisted differential string structure is $[\tilde\tau ] \in \H^4(S^1\times S^3;\bZ)\simeq \bZ$,
where $[\tilde \tau]$ is $\sum n_i  + \sum m_j$ times the generator.
Assuming that $\sum n_i + \sum m_j=0$, then,
we can choose a $(\tilde G,\tilde \tau)$-twisted string structure,
whose data include a 3-form $H$ satisfying $dH=\cw_\nabla(\tilde\tau)$. 
Since $\H^3(S^1 \times S^3; \Z) \simeq \Z$, the data of the 3-form $H$ is sufficient to specify the $(\tilde G,\tilde\tau)$-twisted differential string structure on the $\tilde G$ bundle $X(x_1, n_1;\ldots, | y_1, m_1;\ldots)$ over $S^1 \times S^3$ up to isomorphism. 
This allows us to denote the resulting $(\tilde G,\tilde\tau)$-twisted string manifold by 
$(X(x_1, n_1;\ldots, | y_1, m_1;\ldots), H)$.

Let us move on to the construction of the bordism between $(S^1,\mu)\times (S^3, H_1)$
and $(S^4,f)$.
We note that $(S^1,\mu)\times (S^3, H_1)$ is $(X(|),H_1)$ 
where we abuse the notation to denote the pullback of $H_1 \in \Omega^3(S^3)$ by the same symbol. 
%We gradually modify $[S^4,f]$ so that it becomes $(X(|),H_1)$.
The following discussion of the bordism is drawn also as a sequence of figures in Fig.~\ref{fig:bordism},
which might be of some help to some of the readers.

We first note that $(X(|),0)$ is null bordant.
We pick a small ball around $(0, p) \in S^1\times S^3$
and glue $(S^4,f)$ there. 
The result is $(X(0,+1|0,-1),0)$.
This means that \begin{equation}
(S^4,f) \sim_{\mathrm{bord}} (S^4,f) \# (X(|),0) \sim_{\mathrm{bord} }(X(0,+1|0,-1),0),
\end{equation}
where we denote the bordism relation by $\sim_{\mathrm{bord}}$.

$X(0,+1|0,-1)$ can be continuously deformed to $X(0,+1|a,-1)$ for $0<a<\pi$.
The solution to $dH=\cw_\nabla(\tilde \tau)$ is now deformed so that 
\begin{equation}
\int_{x \times S^3} H = \begin{cases}
1 & (\epsilonelem<x<a-\epsilonelem),\\
0 & (\text{otherwise}).
\end{cases}
\end{equation}

We can now let $a$ `move from $x=+\pi$ to $x=-\pi$'.
Due to the exchange of two $E_8$ factors, we now have 
$(X(a,-1;0,1|),H)$ where $-\pi < a < 0$ and the solution to $dH=\cw_\nabla(\tilde \tau)$ is given by \begin{equation}
\int_{x \times S^3} H = \begin{cases}
0 & (-a+\epsilonelem<x<-\epsilonelem),\\
1 & (\text{otherwise}).
\end{cases}
\end{equation}
As the final step, we can continuously deform $X(a,-1;0,1|)$ to $X(|)$,
where the solution to $dH=\tilde \tau$ now
satisfies $\int_{x\times S^3}H=1$ for every $x$. This shows \begin{equation}
(X(0,+1|0,-1),0)\sim_{\mathrm{bord}} (X(|),H_1).
\end{equation}
This completes the proof.
\end{proof}

\subsection{Determination of $x_{-31}$}
\label{app:power}
Here we provide the proof of Theorem~\ref{thm:31},
i.e.~we prove $x_{-31}$ is the generator of $A_{-31}$.
This is done using the Adams spectral sequence and examining the power operation there.

First let us describe the details of the power operation used in defining  $x_{-31}$ in Definition \ref{defn:31}. 
We use the following form of the operations. 
Let $E$ be an $E_\infty$ ring spectrum. 
The data of $E_\infty$-structure includes the following refinement of the multiplication map, 
\begin{align}
    \xi_2 \colon E\Sigma_{2+} \wedge_{\Sigma_2}(E \wedge E) \to E. 
\end{align}
Thus an element $x \in \pi_n E$ represented by a map $x \colon \Sigma^n \mathbb{S} \to E$ induces a map
\begin{align}
    E\Sigma_{2+} \wedge_{\Sigma_2}(\Sigma^n \mathbb{S} \wedge \Sigma^n \mathbb{S}) \xrightarrow{x \wedge x} E\Sigma_{2+} \wedge_{\Sigma_2}(E \wedge E) \xrightarrow{\xi_2} E. 
\end{align}
This defines a map
\begin{align}\label{eq_P2_general}
    P_2 \colon \pi_n E \to E^0(E\Sigma_{2+} \wedge_{\Sigma_2}(\Sigma^n \mathbb{S} \wedge \Sigma^n \mathbb{S})). 
\end{align}
$E\Sigma_2 \wedge_{\Sigma_2}(\Sigma^n \mathbb{S} \wedge \Sigma^n \mathbb{S})$ is understood as a virtual spherical fibration of rank $2n$ over $B\Sigma_2$. 
So in general the codomain of the map \eqref{eq_P2_general} is understood as an $E$-cohomology of $B\Sigma_2$ with twisted coefficient, whose twist is given by $n\rho$, 
where $\rho$ is specified by the composition $B\Sigma_2 \to BO \to BGL_1 \mathbb{S} \to BGL_1 E $, 
such that the first map classifies the defining representation of $\Sigma_2$. 

Denote the map classifying the M\"obius bundle $\widetilde{S^1} \to S^1$ by $\mu \colon S^1 \to B\Sigma_2$. 
The composition with \eqref{eq_P2_general} gives the map
\begin{align}
    \mu \circ P_2 \colon \pi_n E \to E^0(\widetilde{S^1}_+\wedge_{\Sigma_2}(\Sigma^n \mathbb{S} \wedge \Sigma^n \mathbb{S})), 
\end{align}
Note that if $n$ is even, the spherical bundle $\widetilde{S^1}_+\wedge_{\Sigma_2}(\Sigma^n \mathbb{S} \wedge \Sigma^n \mathbb{S})$ over 
$S^1$ is trivial. Thus we get
\begin{align}\label{eq_P2_thom}
    \cup_1 \colon \pi_n E \xrightarrow{\mu \circ P_2} E^0(\widetilde{S^1}_+\wedge_{\Sigma_2}(\Sigma^n S \wedge \Sigma^n S)) \simeq E^{-2n}(S^1) \xrightarrow{\int_{S^1}} \pi_{2n+1}E ,\quad n \mbox{ : even}. 
\end{align}
In other words, the trivialization of the spherical bundle over $S^1$ gives a nontrivial element in $\pi_{2n+1}(E\Sigma_{2+} \wedge_{\Sigma_2}(\Sigma^n \mathbb{S} \wedge \Sigma^n \mathbb{S}))$ coming from $\pi_{2n+1}(\widetilde{S^1}_+\wedge_{\Sigma_2}(\Sigma^n \mathbb{S} \wedge \Sigma^n \mathbb{S}))$, and the operation $\cup_1$ is the composition of \eqref{eq_P2_general} and this element. 

\begin{rem}\label{rem_cup1}
    The notation $\cup_1$ is justified by the fact that the element in $\pi_{2n+1}(E\Sigma_{2+} \wedge_{\Sigma_2}(\Sigma^n \mathbb{S} \wedge \Sigma^n \mathbb{S})$ above maps to the generator of $H_{2n+1}(E\Sigma_{2+} \wedge_{\Sigma_2}(\Sigma^n \mathbb{S} \wedge \Sigma^n \mathbb{S}); \Z/2) \simeq \Z/2$ under the mod two Hurewicz map, see \cite[Remark 11.7]{BrunerRognes}. 
    As explained there, this notation is potentially ambiguous since the Hurewicz map is not injective in this case. 
    We make a particular choice of the element $\cup_1 \in \pi_{2n+1}(E\Sigma_{2+} \wedge_{\Sigma_2}(\Sigma^n \mathbb{S} \wedge \Sigma^n \mathbb{S}))$ by the above procedure, but it turns out that the result of this subsection still holds if we use other elements with the same Hurewicz image. 
\end{rem}

We are interested in the case of $E = \TMF$ and $n=-16$, 
\begin{align}
    \cup_1 \colon \pi_{-16}\TMF \to \pi_{-31}\TMF. 
\end{align}
Definition \ref{defn:31} says that
\begin{align}\label{eq_def_31}
    x_{-31} := \cup_1(e_8). 
\end{align}

In this subsection, we compute the result of this power operation using the Adams spectral sequence.
Recall that the Adams spectral sequences are given in terms of $H\mathbb{F}_2$, whose power operation is the Steenrod operations. 
These operations induce the operations in the spectral sequences. 
For an $E_\infty$-ring spectrum $E$, the power operations on $E$ and the Steenrod operations on the Adams spectral sequence for $E$ are known to be compatible \cite[Theorem 11.13]{BrunerRognes}. 
We apply this to $E = \tmf$ and use the computation of Steenrod operations given in \cite{BrunerRognes} to compute the right hand side of \eqref{eq_def_31}. 

Since we are interested in detecting $\Z/2$-torsion, we take $2$-adic completion in the rest of this subsection. 
In order to use the computation of \cite{BrunerRognes}, we first need to pass from $\TMF^{\wedge}_2$ to $\tmf^{\wedge}_2$.
\begin{lem}
    We have
    \[
        \cup_1(e_8) \cdot M^2 = \cup_1(e_8M)
    \]
    in $(\TMF^\wedge_2)_{2 \cdot 192 -31}$, where $M \in \pi_{192}\TMF^\wedge_2$ is the periodicity element. 
\end{lem}
\begin{proof}
    In the general setting, the map \eqref{eq_P2_general} satisfies the Cartan formula, 
    \begin{align}
        P_2(xy) = P_2(x)P_2(y). 
    \end{align}
    This implies $\cup_1(xy)=\cup_1(x) \cdot y^2 + x^2 \cdot \cup_1(y)$. 
    Thus it is enough to show that $\cup_1 (M) = 0$. 
%    For this, we lift $M$ to $M \in \pi_{192}\tmf^{\wedge}_2$ and apply \cite[Throrem 11.13]{BrunerRognes} as we do for $B_7 \in \pi_{176}\tmf_2^{\wedge}$ below. 
   % The result is quite simple in this case, since we just get zero. 
   But this is immediate since $\pi_{2\cdot 192-1}\tmf_2^\wedge$ is zero.
\end{proof}

The element $e_8M \in \pi_{192-16=176}\TMF^{\wedge}_2$ is the image of $B_7 \in \pi_{176}\tmf^{\wedge}_2$ in the notation of \cite{BrunerRognes}. 
Since the map $\tmf \to \TMF$ is an $E_\infty$-map, we are left to compute $\cup_1(B_7)$. 

\begin{proof}[Proof of Theorem \ref{thm:31}]
We use the notations of \cite{BrunerRognes}. 
The element $B_7$ is detected by $\delta \omega_2^3 \in E_2^{31, 207}$ in the $E_2$-page of Adams spectral sequence of $\tmf$, where $\delta \in E_{2}^{7, 39}$ and $\omega_2 \in E_2^{8, 56}$. 
Recall that the Steenrod operations on $E_2^{s, t} = \mathrm{Ext}_{\mathcal{A}(2)}(\mathbb{F}_2, \mathbb{F}_2)$ are the maps
\begin{align}
    Sq^i \colon E_2^{s, t} \to E_2^{(s+i), 2t}, 
\end{align}
which are nontrivial only for $0 \le i \le s$. 
The total Steenrod operation
\begin{align}
    Sq=Sq^s + Sq^{s-1} + \cdots Sq^0
\end{align}
satisfies the Cartan formula $Sq(xy) = Sq(x)Sq(y)$.

By \cite[Theorem 11.13]{BrunerRognes}, the operation $\cup_1$ is compatible with the corresponding operation in the Adams spectral sequence. We see that the element $\cup_1(B_7)$ is detected by the element $Sq^{31-1}(\delta \omega_2^3) \in E_2^{61, 414}$ up to higher filtration. 
The result of Steenrod operations was given in \cite[Theorem 1.20]{BrunerRognes} as 
\begin{align}
    Sq^i(\delta) &= \begin{cases}
        h_0e_0\omega_2 & (i = 6), \\
        h_2\beta \omega_2 & (i=5), \\
        0 & \mbox{otherwise}, 
    \end{cases} \\
    Sq^i(\omega_2) &= \begin{cases}
        \omega_2^2 & (i = 8), \\
        0 & \mbox{otherwise}. 
    \end{cases}
\end{align}
By the Cartan formula we get 
\begin{align}\label{eq_proof_power_ss_1}
    Sq^{30}(\delta \omega_2^3) = h_0e_0\omega_2^7. 
\end{align}
We can see that this element detects the generator 
\begin{equation}
\nu_6 \kappa M \in \ker\left(\pi_{161+192}\tmf_2^{\wedge} \to \pi_{161+192}\ko[[q]]^{\wedge}_2\right) = \Z/2,
\end{equation}
by e.g.~using the relation $h_2d_0 = h_0e_0$, rewriting $h_0e_0 \omega_2^7 = h_2\omega_2^3d_0\omega_2^4$ and 
then using the fact that $h_2\omega_2^3$, $d_0$ and $\omega_2^4$ detect $\nu_6$, $\kappa$ and $M$, respectively. 
Thus we see that $\cup_1(B_7) = \nu_6 \kappa M + x$ with some element $x \in \pi_{161+192}\tmf$ detected by an element with higher Adams filtration. 
But we already know that $\cup_1(B_7)$ is in the kernel of $\pi_{161+192}\tmf \to \pi_{161+192}\ko[[q]]$ by Proposition \ref{prop:image-in-KO-31}, and also $\nu_6 \kappa M$ is in the kernel by the results in \cite{BrunerRognes}.
Moreover, any nontrivial element in $\pi_{161+192}\tmf$ detected by an element with Adams filtration higher than $h_0e_0\omega_2^7$ is not $B$-power torsion, so its image in $\pi_{161+192}\ko[[q]]$ is nontrivial. 
Thus we get $x=0$ so that
\begin{align}
    \cup_1(B_7) = \nu_6 \kappa M. 
\end{align} 
This completes the proof of the claim on $x_{-31}$. 
\end{proof}

\subsection{Determination of $x_{-28}$}
\label{sec:computation}

In this section, we prove Proposition \ref{prop:computation}, i.e.~
we compute the pairing $\alpha_\text{spin/string}(x_{-28},y)$ for a class $y\in\pi_7\MSpin/\MString$ lifting $\nu^2\in\pi_6\MString$
and confirm it to be nonzero.
This is done to prove Theorem~\ref{thm:28}, which says that $x_{-28}$ is the generator of $A_{-28}\simeq \bZ/2$.
This proof is independent of the result in Subsection~\ref{app:power}. 

Our demonstration of Proposition~\ref{prop:computation} is based on the following general consideration.
Let $x\in \pi_{-d} \TMF$ be a torsion element
and $[N,M] \in \pi_{d-21} \MSpin/\MString$ be given by a string manifold $M$ 
which is a boundary of a spin manifold $N$.
We choose compatible differential string structure and differential spin structure on $N$ and $M$ respectively. 

Assume that there is a group $G$ and a twist {$t \colon BG \to K(\Z, 4)$}
such that $x$ lifts to a possibly-non-torsion element $\hat x \in \TMF^{d+t}(BG)$ 
and $M$ is null-bordant in $\MString_{d-22+t}(BG)$,
i.e.~there is a manifold $\tilde N$ with a twisted differential string structure 
given in terms of a $G$-bundle $P\to N'$ with a connection $\nabla$ 
such that $(P, \nabla)$ is trivial around $M=\partial N'$.
In particular, $N'$ is equipped with a $3$-form $H \in \Omega^3(N')$ which satisfies $dH=p_1(T N')/2 - \mathrm{cw}_\nabla(t)$,
where $\mathrm{cw}_\nabla(t)$ denotes the closed $4$-form
obtained by the Chern-Weil construction.
We now form a closed spin manifold $\tilde N=\overline{N'} \sqcup N$.
We extend the $G$-bundle $P$ with connection $\nabla$ trivially on $N$ and denote the resulting bundle and connection on $\tilde{N}$ by the same symbols. 

Recall that the rationalization of $\alpha_\text{spin}(\hat x)$ is given by \begin{equation}
X:=\frac12\Delta (q) \ch(U)\hat A(T) \ch(\WitVec(T)) \Bigm|_{q^0}.
\end{equation} where $U=\sigma(\hat x)  \in \KO((q))^d(BG)$.
Here, $T$ is  the universal bundle over $B\Spin$ which pulls back to the tangent bundle by the classifying map of a spin manifold,
and $\WitVec(V)$ for a real vector bundle was defined in \eqref{WitVec}.
For more details, see~\cite[Sec.~2.2.4]{TachikawaYamashita}.

As $\alpha_\text{string}(\hat x)$ vanishes, we know that $X$ factorizes as \begin{equation}
X=\left(\frac{p_1(T)}2 - f^*(t)\right) Y, 
\end{equation}
where $f \colon \tilde{N} \to BG$ is the classifying map for $P$. 
Finally, define $V\in \KO^d(\tilde N)$ via \begin{equation}
\label{defV}
V: = \Delta(q) f^*(U) \otimes \WitVec (T\tilde N\oplus \bR) \Bigm|_{q^0}.
\end{equation}

\begin{prop}
\label{prop:pairing}
Let $d\equiv 4 \pmod 8$. Under these assumptions, we have \begin{equation}
\label{eq:pairing}
\alpha_\text{spin/string}(x,[N, M])
= \frac12\eta(V) - \int_{N'} H\wedge \cw(Y) \pmod \Z
\end{equation}
where $\eta$ is the reduced eta invariant and $\cw(Y) \in \Omega_\clo^*(N')$ is the characteristic form obtained by applying Chen-Weil construction to $Y$ with respect to the differential spin structure on $N'$.
\end{prop}

\begin{proof}
Consider the following diagram, 
\begin{align}\label{diag_lift_BG_pairing}
    \vcenter{\xymatrixcolsep{0pc}
    \xymatrix{
   {\TMF}^d  &\times& \widehat{\MSpin/\MString}_{d-21}\ar[d]^-{i_*} \ar[rrrrrr]^-{\widehat{\alpha}_\text{spin/string}} &&&&&& \R/\Z \ar@{=}[d]\\
    {\TMF}^{d+t}(BG) \ar[d]^-{\sigma}\ar[u]^-{i^*}&\times& (\reallywidehat{\MSpin/\MString) \wedge_{t}BG_+})_{d-21} \ar[rrrrrr]^-{\widehat{\alpha}_\text{spin/string}}  &&&&&& \R/\Z \ar@{=}[d] \\
    {\KO((q))}^d(BG) &\times& (\reallywidehat{\MSpin \wedge BG_+})_{d-21} \ar[rrrrrr]^-{\widehat{\alpha}_\spin} \ar[u]^-{j}&&&&&& \R/\Z 
    }}
\end{align}
where we fix an inclusion $i \colon \pt \to BG$. 
The horizontal arrows are the {\it differential pairings} introduced in Subsection \ref{subsec_diff_pairing}\footnote{
The middle horizontal arrow is associated to \[
\alpha_{\spin/\stri} \colon \TMF \wedge_t BG_+ \to \Sigma^{-21}I_{\Z} \left(\MSpin/\MString) \wedge_{t}BG_+ \right),
\] and the bottom horizontal arrow is associated to \[
\alpha_{\spin} \colon \KO((q))\wedge BG_+ \to \Sigma^{-21}I_{\Z} \left(\MSpin/\MString) \wedge BG_+ \right).
\] 
All of them satisfy the conditions (1) and (3) in Subsubsection \ref{subsubsection_general_diff_paiaing}, so the differential pairing is reduced to the form \eqref{eq_diff_pairing}. 
\label{footnote_pairing}}. The second column consists of {\it differential bordism groups}, which are explained in Appendix \ref{app:diff_bordism} below. 
The differential spin/string structure on $(N, M)$ 
and 
the differential spin structure and the $G$-connection
determine the classes
\begin{equation}
[N, M]_\diff \in (\reallywidehat{\MSpin/\MString})_{d-20},\qquad
[\tilde{N}, P, \nabla]_\diff \in (\reallywidehat{\MSpin \wedge BG_+})_{d-20}.
\end{equation}
Here and below in this proof, 
we do not include the reference to the differential spin/string structures in notations for brevity.

We then have
\begin{align}
\alpha_\text{spin/string}(x,[N, M]) =
     \widehat{\alpha}_\text{spin/string}(x,[N, M]_\diff)
=\widehat{\alpha}_\text{spin/string}(\hat{x},i_* [N, M]_\diff), 
\end{align}
where the first equality was discussed in Sec.~\ref{subsubsection_general_diff_paiaing} and the second equality follows from the functoriality of the pairing $\widehat{\alpha}_\text{spin/string}$ applied to the first and second rows of \eqref{diag_lift_BG_pairing}. 
We have $i_*[N, M]_\diff + [\overline{N'}, M, P, \nabla]_\diff = j[\tilde{N}, P, \nabla]_\diff$ in $(\reallywidehat{(\MSpin/\MString) \wedge_t BG_+})_{d-20}$. Thus we have
\begin{align}
    \widehat{\alpha}_\text{spin/string}(\hat{x},i_* [N, M]_\diff)
   & = \widehat{\alpha}_\text{spin/string}(\hat{x},j[\tilde{N}, P, \nabla]_\diff) - \widehat{\alpha}_\text{spin/string}(\hat{x},[\overline{N'}, \overline{M}, P, \nabla]_\diff) \\
   &= \widehat{\alpha}_\spin(\sigma(\hat{x}), [\tilde{N}, P, \nabla]_\diff) + \widehat{\alpha}_\text{spin/string}(\hat{x},[{N'}, M, P, \nabla]_\diff), \label{eq_lift_BG_pairing_1}
\end{align}
where on the second equation we used the functoriality of the pairing in the second and third rows of \eqref{diag_lift_BG_pairing}. 

For the first term of \eqref{eq_lift_BG_pairing_1}, from the form of $\alpha_\spin$ given in Definition~\ref{eq_def_alpha_spin},  we can apply the result of \cite{YamashitaAndersondualPart2} to get the formula of the differential pairing $\widehat{\alpha}_\spin$ in terms of the pushforward in differential $\KO((q)) \wedge BG_+$ as we did in Subsubsection \ref{subsubsec_diff_pairing_BunkeNaumann}. 
By the formula \eqref{eq_push_KO((q))} for the pushforward, we get
\begin{align}\label{eq_lift_BG_pairing_2}
    \widehat{\alpha}_\spin(\sigma(\hat{x}), [\tilde{N}, P, \nabla]_\diff) = \frac{1}{2} \eta(V) \pmod \Z. 
\end{align}

For the second term of \eqref{eq_lift_BG_pairing_1}, note that the element $[N', M, P, \nabla]_\diff$ maps to zero in the topological group $\MSpin/\MString_{d-20+t}(BG)$. 
Explicitly, we use the following  $(G,t)$-twisted $\MSpin/\MString$-bordism from $\varnothing$ to $(N', M, P, \nabla)$. 
The underlying manifold with corners (more precisely, $\langle 2\rangle$-manifold in the sense of \cite{Janich1968}, also recalled in \cite[Section 2.3]{Yamashita:2021cao}) is $W :=(N' \times [0, 1])/\sim$, where the quotient is taken by the relation $(m, t) \sim (m, t')$ for each $m \in M$ and $t, t' \in [0, 1]$. 
In the notation of \cite[Section 2.3]{Yamashita:2021cao}, the $\langle 2\rangle$-manifold structure on $W$ is given by the decomposition of the boundary $\del_0 W = N' \times \{0\}$, $\del_1 W = N' \times \{1\}$ along the corner $\del_{01}W=\del_0 W \cap \del_1 W =M$, see Fig.~\ref{fig:corner}.
We have the quotient map $\pi \colon W \to N'$. 
We equip $W$ with the pullback $G$-bundle with connection $\pi^*(P, \nabla)$. 
The triple $(W, \del_0 W, \pi^*(P, \nabla))$ gives a differential twisted $\MSpin/\MString$-null bordism of $(N', M, P, \nabla) \simeq \del_1(W, \del_0 W, \pi^*(P, \nabla))$. 

\begin{figure}[h]
\centering
\[
\begin{tikzpicture}[yscale=1.3]
\node (0) at (-2,0) {};
\node (1) at (+2,0) {};
\node (B) at (0,-1) {};
\draw[black] (-2,0) node[above]{$0$}  -- (+2,0) node[above]{$1$} ;
\draw[blue] (-2,0) --node [left, yshift=-2.6] {$N'=\partial_0 W$ }  (0,-1) node[below] {$M=\partial_{01}W$};
\draw[red] (+2,0) --node[right, yshift=-2.6]{$N'=\partial_1 W$} (0,-1) ;
\node (W) at (0,-.4) {$W$};
\end{tikzpicture}
\]
\caption{The null bordism $W$ of $N'=\partial_1 W$ (drawn in red).
This $W$ is obtained by first considering $N'\times [0,1]$ and collapsing $M\times [0,1]$ where $M=\partial N'$ into $M$.
We have a differential twisted string structure on $N'=\partial_0 W$ (drawn in blue) lifting the differential spin structure on $W$ restricted there.
\label{fig:corner}}
\end{figure}
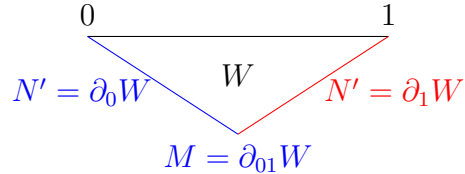

We have the exact sequence for differential homology \eqref{eq_exact_MTB}, 
\begin{equation}
\begin{aligned}
    \H_{d-20}((\MSpin/\MString) \wedge BG_+; \R) &\xrightarrow[]{a} (\reallywidehat{\MSpin/\MString) \wedge_{t}BG_+})_{d-21} \\
    & \xrightarrow[]{I} \MSpin/\MString_{d-21+t}(BG) \to 0. 
\end{aligned}
\end{equation}
It is convenient to view the first term as
\begin{equation}
\begin{aligned}
    &\H_{d-20}((\MSpin/\MString) \wedge BG_+; \R) \\
    &\simeq \Hom_\R(\H^{d-20}((\MSpin/\MString) \wedge BG_+; \R ), \R)  \\
    &\simeq \Hom_\R\left(\left(\left(\H^\bullet(BG; \R)\otimes_\R \ker(\H^*(\MSpin; \R) \to \H^*(\MString; \R)\right)\right)^{d-20}, \R\right), \label{eq_bordism_homology_R}
\end{aligned}
\end{equation}
where for the second isomorphism we used the assumption that $d$ is even. 
We have
\begin{align}
    a\left(\cw([W, \del_0 W, \pi^*(P, \nabla)])\right) = [N', M, P, \nabla]_\diff, 
\end{align}
where $\cw([W, \del_0 W, \pi^*(P, \nabla)])$ is the homomorphism in \eqref{eq_bordism_homology_R} which evaluates characteristic polynomials on $(W, \pi^*(P, \nabla))$, explicitly given by the integration of the relative closed form \eqref{eq_cw_MSpin/MString2}. 

Lemma \ref{lem_diff_generalized_pairing_compatibility} in this case implies the compatibility of the two pairings in the following diagram, 
\begin{equation}
    \hskip-5em\vcenter{\xymatrixcolsep{0pc}
    \xymatrix{
     {\TMF}^{d+t}(BG) \ar[d]^-{\alpha_\text{spin/string}^\R}&\times& (\reallywidehat{\MSpin/\MString) \wedge_{t}BG_+})_{d-21}\ar[rrr]^-{\alpha_\text{spin/string}}  &&& \R/\Z  \\
     \H^{d-20}((\MSpin/\MString) \wedge BG_+; \R )&\times& \H_{d-20}((\MSpin/\MString) \wedge BG_+; \R) \ar[rrr]^-{\langle \cdot , \cdot\rangle} \ar[u]^-{a} &&& \R \ar[u]^{\mod \Z}. 
    }}
\end{equation}
Thus we have
\begin{align}\label{eq_lift_BG_pairing_3}
\begin{aligned}
    \alpha_\text{spin/string}(\hat{x},[{N'}, M, P, \nabla]_\diff)
    &=\left\langle\alpha_\text{spin/string}^\R(\hat{x}) , \cw([W, \del_0 W, \pi^*(P, \nabla)])\right\rangle \\
    &= \int_W  \cw(X) - \int_{\del_0 W \simeq N'} H \wedge \cw(Y) \pmod \Z  \\
    &= -\int_{N'} H \wedge \cw(Y) , 
\end{aligned}
\end{align}
where the first equality uses the formula \eqref{eq_cw_MSpin/MString2} and the last equality follows from the fact that $W$ has the cylindrical structure. 
Combining \eqref{eq_lift_BG_pairing_1}, \eqref{eq_lift_BG_pairing_2} and \eqref{eq_lift_BG_pairing_3}, we get the desired result. This completes the proof of Proposition \ref{prop:pairing}. 
\end{proof}

With this we can proceed to the computation of our pairing. 
For that, we need an enhancement of Proposition~\ref{prop_TEJF2},
where we refine $SU(2)$ equivariance to $SU(2)\times SU(2)$ equivariance.
Below, we sometimes put primes to objects belonging to the second $SU(2)$.

Recall that $SU(2)$-equivariant $\TMF$ had a map to the ring of integral Jacobi forms \eqref{toJF}.
With $SU(2)\times SU(2)'$-equivariance, we have \begin{equation}
e_{\JF\times \JF}\colon \TMF^{m+k\overline V_{SU(2)}+k'\overline V_{SU(2)}'}_{SU(2)\times SU(2)} \to \mathrm{JJF}^{\text{ind}=(k,k')}_{ \text{wt}=\frac m2},
\label{eq_char_}
\end{equation}
where $\mathrm{JJF}^{\text{ind}=(k,k')}_{ \text{wt}=\frac m2}$ is 
the space of what can be called as \emph{double Jacobi forms $f(\tau;z,z')$ of index $(k,k')$ and weight $w$},
which are functions such that it is a Jacobi form index-$k$ weight-$w$ Jacobi form with respect to the variables $(\tau, z)$ and that with index-$k'$ weight-$w$ with respect to the variables $(\tau, z')$, and moreover has integral coefficients when expanded in the variable $(q = e^{2\pi i \tau}, y = e^{2\pi i z}, y' = e^{2\pi i z'})$.
% $f(\tau;z,0)$ is an index-$i$ weight-$w$ Jacobi form,
% $f(\tau;0,z')$ is an index-$i'$ weight-$w$ Jacobi form,
% and $f(\tau;z,z)$ is an index-$(i+i')$ weight-$w$ Jacobi form.
\def\mathcalEdefinition#1{
Let us consider the following element in $\mathrm{JJF}^{\ind=(1,1')}_{\wt = 4}$, 
\begin{equation}
\begin{aligned}
	\mathcal{E}_4(\tau;z, z') &=\frac{1}{144} \Bigl( c_4 \phi_{0,1} \phi_{0, 1}' - c_6 \phi_{-2,1} \phi_{0,1}'
	 - c_6 \phi_{0,1} \phi_{-2,1}' + c_4^2 \phi_{-2,1} \phi_{-2,1}\Bigr)  \\
	 &= 1+ (58+(y  ^2+1+\frac1{y ^{2}}) + 32(y  +\frac1y )+\\
	 & \qquad\qquad (y '{}^2+1+\frac{1}{y '{}^{2}})+32(y ' +\frac1{y '})+12(y  +\frac1{y })(y ' +\frac1{y '{}^{-1}}) )+
	 O(q)^2 
\end{aligned}
\label{#1}
\end{equation}
where $c_4$, $c_6$ are modular forms in $q=e^{2\pi i\tau}$,
$\phi_{0,1}$, $\phi_{-2,1}$ are Jacobi forms in $q=e^{2\pi i\tau}$, $y =e^{2\pi i z}$,
and 
$\phi_{0,1}'$, $\phi_{-2,1}'$ are Jacobi forms in $q=e^{2\pi i\tau},y '=e^{2\pi i z'}$,
where we again follow the notation of \cite{Gritsenko2020382}.
We can show that the right hand side has integral formal coefficients. Indeed, the $\mathcal{E}_4$ is characterized by the property that
\begin{align}\label{eq_cE4}
		\mathcal{E}_4(\tau;z, 0) = E_{4,1}(z), \quad \mathcal{E}_4(\tau;z, z) = E_{4,2}(z), \quad \mathcal{E}_4(\tau;z_1, z_2 ) = \mathcal{E}_4(\tau;z_2, z_1 ) .
\end{align}
}
\mathcalEdefinition{calE4}

We will use the following proposition, which will be proved in Appendix~\ref{app:equivariant}:
\begin{prop}\label{prop_TEJFTEJF}
	There is an element $\check{\check{e}}_8 \in \TMF^{16+\overline V_{SU(2)} + \overline V_{SU(2)}'}_{SU(2)\times SU(2)}(\pt)$ which maps to $\mathcal{E}_{4}/\Delta$ under the map \eqref{eq_char_}. 
	Moreover, this element pulls back to $\check e_8$ via both inclusions $(\mathrm{id},1)$, $(1,\mathrm{id}) \colon SU(2)\to SU(2)\times SU(2)$.
\end{prop}

\begin{proof}[Proof of Proposition~\ref{prop:computation}]
\parindent15pt
The element $x_{-28}$ was obtained by pulling back an element 
\begin{equation}
\check{\check{e}}_8\check{\check{e}}_8'\in \TMF^{32+\overline V_{SU(2)}+\overline V_{SU(2)}+\overline V_{SU(2)}'+\overline V_{SU(2)}'}_{(SU(2) \times SU(2))\times (SU(2) \times SU(2))}(\pt)
\end{equation}
using the map 
\begin{equation}
S^4 \to BSU(2)\times BSU(2)\xrightarrow{(\id \times 1) \times (\id \times 1)} B(SU(2) \times SU(2))\times B(SU(2) \times SU(2)).
\end{equation}
%Using \begin{equation}
% SU(2)\times E_7 \twoheadrightarrow [SU(2)\times E_7]/\{\pm1\} \hookrightarrow E_8,
%\end{equation}
%we can regard the pulled-back class as an element of $\TMF^{32+\tau+\tau'}(BE_7\times BE_7\times S^4)$.
Therefore, $x_{-28}$ can be lifted to an element $\check x_{-28}$ of $\TMF^{28+\overline V_{SU(2)}+\overline V_{SU(2)}'}_{SU(2)\times SU(2)}(\pt)$.

Its image $\sigma(\check{x}_{-28})$ in  $\KO^{28}_{SU(2)\times SU(2)}((q))(\pt)$
can be  computed by the KO-theory pushforward  using \eqref{calE4} and Proposition~\ref{prop_TEJFTEJF}.
The result is 
\begin{equation}
\sigma(\check{x}_{-28}) = (12\, \mathfrak{A}-12\, \mathfrak{A}') q + O(q^2),
\label{sigma-hat-x}
\end{equation}
where $\mathfrak{A}$ and $\mathfrak{A}'$ are the defining two-dimensional representations of $SU(2)$ and $SU(2)'$, respectively.
%and $c:\KO^{28}_{SU(2)\times SU(2)}((q))(\pt)\to \KO^{28}((q))(BSU(2)\times BSU(2))$ is the Atiyah-Segal completion map.
$\alpha_\text{spin}(\check x_{-28})$ can be computed from the $q$-expansion given in \eqref{sigma-hat-x}, and is given after rationalization by \begin{equation}
\alpha_\text{spin}(\check x_{-28}) = (\frac{p_1(T)}2 - \tau-\tau')\frac12 (\tau-\tau').
\end{equation}

Let us move on to the analysis of $\nu^2 \in \pi_6\MString$.
Let us focus on $\nu\in\pi_3\MString$ first.
As $\nu$ is a boundary of $B^4$,
we can take $[B^4, \nu] \in \pi_4\MSpin/\MString$.
Next we consider $B^4$ equipped with an $SU(2)$ bundle $P\to B^4$
supported away from the boundary
such that $\int_{B^4} \tau(P)=+1$.
We can pick a 3-form $H$ satisfying $dH=p_1(TN')/2 - \tau(P)$ on $B^4$.
For simplicity we pick a metric which makes $B^4$ into a completely round hemisphere
and $S^3=\partial B^4$ a completely round sphere.
This in particular makes $\int_{S^3} H=1$ on the boundary.
These data make $B^4$ into an $(SU(2),\tau)$-twisted string null bordism of $\nu$.
%We can further regard it as an $E_7$-twisted string structure via $SU(2)\to E_7$,
%under which {$\tau \colon BE_7 \to K(\Z, 4)$ } pulls back to {$-c_2\colon BSU(2) \to K(\Z, 4)$}.

With these information we deal with $\nu^2\in \pi_6\MString$ itself.
Its spin null bordism is given by $N=B^4 \nu$,
and consider $[N,\nu^2] \in \pi_7 \MSpin/\MString$.
We also need to use $N'=B^4 \nu$ equipped with the $E_7$-twisted string structure constructed as above on the side of $B^4$.
Then $\tilde N= \overline{N'}\sqcup N$ is $S^4\times S^3$
which is equipped with an $SU(2)$ bundle $P\to \tilde N$
which is pulled back from an $SU(2)$ bundle $P\to S^4$
such that $\int_{S^4} \tau(P)=1$.
The class $V\in \KO^{28}(S^4\times S^3)$
as defined in \eqref{defV}
is pulled back from $V\in \KO^{28}(S^4)$
which is given by $P\times_{SU(2)} (12\,\mathfrak{A}- 24) $.

We can finally invoke Proposition~\ref{prop:pairing}.
The first term on the right hand side of \eqref{eq:pairing} is \begin{equation}
\frac12\eta_{\tilde N}( V  ) = \frac12\mathrm{ind}_{S^4} (V)  \eta_{S^3} = 0.
\end{equation} 
Here the first equality uses the product formula of $\eta$ on a product manifold
and the second equality uses the fact that $\eta$ of round $S^3$ is zero.

In contrast, the second term of  \eqref{eq:pairing} is \begin{equation}
\int_{B^4 \times S^3} H \frac12\tau = \frac12 \int_{B^4}\tau \int_{S^3} H = \frac12.
\end{equation}
We therefore conclude that \begin{equation}
\alpha_\text{spin/string}(x,\nu^2) = \frac12.
\end{equation}
This is what we wanted to determine.
\end{proof}

%\newpage

\appendix

% !TEX root = paper.tex
\section{Conjectures on vertex operator algebras and $\TMF$}
\label{app:VOA}

\subsection{Generalities}
The Stolz-Teichner proposal \cite{StolzTeichner1,StolzTeichner2}
says the following:
\begin{proposal}
\label{conj:SST}
Let $\SQFT_{n}$ be the `space' of 2d unitary spin \Nequals{(0,1)} supersymmetric quantum field theories (SQFT) {with degree $n \in \Z$. }
% with anomaly $n\in (I_\bZ\Omega^\spin)^4(\pt) \simeq \bZ$. 
The sequence $\{\SQFT_\bullet\}$ forms an $\Omega$-spectrum, and agrees with $\TMF$,
the spectrum of  topological modular forms. 
\end{proposal}
Before proceeding, we need some comments: \begin{itemize}
\item  The adjective `spin' in the statement stresses the condition that the SQFT in question is formulated in the category of spin manifolds.
 \item With suitable regularity assumptions,
 unitary quantum field theories defined on spacetimes with Lorentzian signature are believed to be in one-to-one correspondence
 with reflection positive quantum field theories defined on spacetimes with Euclidean signature.
With this expectation,  we always consider unitary Lorentzian theories or equivalently reflection-positive Euclidean theories
in this paper. 
\item The {\it degree} specifies the {\it anomaly} of an SQFT. 
This point requires a more detailed remark given below.
\end{itemize}

{
We take the point of view that an anomalous $d$-dimensional QFT is a boundary theory of a $(d+1)$-dimensional invertible theory. 
%In this context, a  theory of degree $n$ is by definition a boundary theory of the $3$-dimensional theory $\mathrm{Fer}^{\otimes n}$ as defined in \cite{berwickevans2023field}.
In view of the conjecture by Freed and Hopkins of the classification of possibly non-topological invertible QFT \cite{FreedHopkins2021}, the spectrum $\Sigma^{4}I_\Z \MSpin$ should classify the deformation classes of invertible three-dimensional spin QFTs.
Then, a theory of degree $n$ is by definition a boundary theory of the three-dimensional invertible phase whose deformation class is specified by 
\begin{align}
    n \in \Z \simeq (I_\Z \MSpin)^4(\pt). %\simeq [S^0, \Z \times BO \langle 0, \cdots, 4 \rangle]_+. 
\end{align}
%It is important for us to identify the degree $n$ of an SQFT with the element 
We note here that a three-dimensional invertible theory whose deformation class corresponds to the generator $1 \in \Z \simeq (I_\Z \MSpin)^4(\pt)$ was given as an element of the differential Anderson dual of the spin bordism group in \cite[Example 4.69]{Yamashita:2021cao}. }

{Moreover, there is a map
\begin{align}
    \Z \times BO\langle 0, \ldots, 4 \rangle \to \Sigma^4 I_\Z \MSpin
\end{align}
which exhibits the left hand side as a connective cover of the right hand side \cite[Proposition C.5]{TachikawaYamashita}. This means that the space $\Z \times BO \langle 0, \ldots, 4 \rangle$, which mathematically classifies the degree and the twist of $\TMF$, can be used as a classifying {\it space} of the anomaly of SQFTs in our context. }
Before proceeding, we note that there have been mathematical works that seek to implement this proposal to various degrees, e.g.~\cite{Cheung,berwickevans2023field}.

We now consider a family of SQFTs with symmetry $G$.
In view of the comments above, an anomaly of such a theory is specified by a map
\begin{align}\label{eq_G_twist}
    (n, k) \colon BG \to \Z \times BO\langle 0, \cdots, 4 \rangle.  
\end{align}
% Physically, the anomaly of such a family is believed to be {classified} by \cite{FreedHopkins2021}
% \begin{equation}
% n\oplus k\in (I_\bZ\Omega^\text{spin})^4(BG)=(I_\bZ\Omega^\spin)^4(\pt)\oplus (\widetilde{I_\bZ\Omega^\text{spin}})^4(BG)
% \end{equation}
% As discussed in Appendix~C of \cite{TachikawaYamashita},
% we have 
% \begin{equation}
% (\widetilde{I_\bZ\Omega^\spin})^4(X) \simeq [X, \BO\langle0,\ldots,4\rangle]=\pi_0 \mathrm{Map}(X, \BO \langle 0, \cdots, 4 \rangle),
% \end{equation}
% where the mapping space $\mathrm{Map}(X, \BO \langle 0, \cdots, 4 \rangle)$ is known to be usable as a parametrizing space for twists of $\TMF$-cohomology on $X$\cite{ABG}.
% Here $\BO\langle 0,\ldots,n\rangle$ is the Whitehead truncation of $\BO$ keeping only its $\pi_0$ up to $\pi_n$.
%\textcolor{blue}{MY is not sure how to fix the subtlety of the notion of twists appearing in this section. Let us discuss further on mattermost. }
Mathematically, we expect to have a genuinely equivariant version of $\TMF$ twisted by a map of the form \eqref{eq_G_twist},
although 
the theory of twisted genuinely-equivariant topological modular forms has not been fully developed yet.
In this appendix we assume the existence of such a theory, together with an appropriate generalization 
of the Atiyah-Segal completion map $c$, such that we have
the following commuting square:  \begin{equation}
\vcenter{\xymatrix{
\TMF_G^{d+k}(\pt) \ar[d]^-{c} \ar[r]^-{\sigma} &
\KO_G^{d+\tilde k}(\pt) \ar[d]^-{c} \\
\TMF^{d+k}(BG)  \ar[r]^-{\sigma} & 
\KO^{d+\tilde k}((q))(BG)
}},
\end{equation}
where {$\tilde k$ is the composition of $k$ and the natural projection $BO\langle0,1,2,3,4\rangle\to BO \langle0,1,2\rangle$. }
We note that we do use $\RO(G)$-graded equivariant $\TMF$ in this paper,
whose theory is recalled in Appendix~\ref{app:equivariant}.
The $\RO(G)$-graded version only allows certain subsets of the twists
and not the completely general ones given in \eqref{eq_G_twist}.

% $\tilde k\in (\widetilde{I_\bZ\Omega^\text{pin$-$}})^3(BG) = [BG,BO\langle0,1,2\rangle]$ 
% is determined by $k$  via a natural projection \begin{equation}
% BO\langle0,1,2,3,4\rangle\to BO \langle0,1,2\rangle.
% \end{equation}
With this caveat, we have the following generalized version of the Stolz-Teichner proposal \cite{Johnson-Freyd:2020itv,TachikawaYamashita}:
\begin{proposal}
The deformation classes of 2d unitary spin \Nequals{(0,1)} supersymmetric quantum field theories with symmetry group $G$, whose anomaly is specified by the element
{$(n, k) \colon BG \to \Z \times BO \langle 0, \cdots, 4 \rangle$,}
% $n\oplus k \in (I_\bZ\Omega^\text{spin})^4(BG)$, 
form the group  $
\TMF^{n+k}_G(\pt)
$.
\end{proposal}

A subclass of 2d quantum field theories is given by 2d superconformal field theories (CFTs),
which have commuting actions of two Virasoro algebras.
Two copies of Virasoro algebras are distinguished by calling them as left-moving
and right-moving\footnote{%
On a two-dimensional flat spacetime with Lorentzian signature, 
operators depending on $t+x$ but not on $t-x$ are called left-moving 
and those depending on $t-x$ but not on $t+x$ are called right-moving.
After the Wick rotation to put it on a two-dimensional flat spacetime with Euclidean signature,
one set of operators become holomorphic, depending on $z=x+iy$ but not on $\bar z$,
and the other set of operators become anti-holomorphic, depending on $\bar z$ but not on $z$.
}, respectively,
We denote their central charges by $c_L$ and $c_R$.
In this case the degree $n\in (I_\bZ\Omega^\text{spin})^4(\pt)\simeq \bZ$ is known to be given by $n=2(c_L-c_R)$.
In a unitary theory, the central charges $c_{L,R}$ are non-negative,
and when $c_L$ (or $c_R$) is zero,
the left-moving (or the right-moving) side is empty.
The theories where $c_L>0$, $c_R=0$ are known as purely left-moving
and those for which $c_L=0$, $c_R>0$ are known as purely right-moving.

CFTs can be considered either in the category of oriented manifolds
or in the category of spin manifolds.
Physicists often call the former bosonic CFTs,
and the latter spin (or fermionic) CFTs.

For a CFT to be \Nequals{(0,1)} supersymmetric,
the action of the right-moving Virasoro algebra has to be extended to the action of the super-Virasoro algebra,
while the left-moving Virasoro algebra does not have to be extended.
This in particular means that a purely left-moving not-necessarily-supersymmetric CFT is actually \Nequals{(0,1)} supersymmetric,
where the action of the right-moving super  Virasoro algebra is trivial.

Therefore, as tractable subcases of the Stolz-Teichner proposal,
we can consider either 
a purely right-moving supersymmetric spin CFT,
or a purely left-moving spin CFT which is not necessarily supersymmetric.
The former was the topic of \cite{Gaiotto:2018ypj},
and the aim of this appendix is to discuss the latter.

A vertex operator algebra (VOA) axiomatizes the properties of subalgebras of left-moving operators of a bosonic CFT.
For a spin CFT,  subalgebras of left-moving operators form a vertex operator superalgebra  (VOSA) instead.
Here the adjective `super' means that the grading is in $\frac12\bZ$,
and we do not require that it has a super-Virasoro subalgebra.
We are mainly interested in unitary quantum field theories,
and therefore all VOAs and VOSAs of our interest should carry a unitary structure,
for whose details we refer the reader e.g.~to \cite[Sec.~5]{Kawahigashi} and \cite[Sec.~3]{carpi2023vosa}.

The entirety of operators of a purely left-moving spin CFT $T$ forms a VOSA,
which we denote by $\V(T)$. 
If $T$ has a symmetry $G$ which is a simply-connected simple compact Lie group,
{the deformation class of the anomaly is specified by an element in $ (\widetilde{I_\bZ\Omega^\text{spin}})^4(BG)=\H^4(BG;\bZ)\simeq\bZ$},
which is usually called its level.
Then $\V(T)$ should contain the affine Lie algebra VOA $\mathfrak{g}_k$  at level $k$ as a subalgebra.

It is not yet known exactly which properties of a VOSA $\V$ guarantees that it arises as $\V(T)$ for some $T$. \if0
One very strong constraint comes from the modular invariance, which states that the character \begin{equation}
\chi_\V(q) = \tr\nolimits_\V q^{L_0-c/24}
\end{equation}
gives a projective one-dimensional representation 
%of the modular group $SL(2,\bZ)$ in the case of bosonic CFTs,
and of the group $\Gamma_0(2)$ in the case of spin CFTs.
\fi
A strong candidate for such a condition is that the VOSA in question is
\emph{self-dual} (or equivalently \emph{holomorphic}) in the sense of the VOA theory,
i.e.~that any reasonable representation is the direct sum of copies of $\V$ itself \cite{Zhu}.
The most optimistic formulation is then
\begin{proposal}
A self-dual unitary  VOSA $\V$ containing
 the affine Lie algebra VOA $\mathfrak{g}_k$  comes from 
a unique purely left-moving unitary spin CFT $T$ with $G$ symmetry at level $k$,
so that $\V=\V(T)$.
\end{proposal}

\subsection{VOA and $\TMF$ classes}
\label{sec:vosa}
\subsubsection{Self-dual VOSAs}
Let us recall some of the structures associated to  self-dual VOSAs.
For details, we refer the reader to a recent paper \cite{HoehnMoeller}.
One unfortunate situation for us is that, in most of the mathematical literature on VOAs  and VOSAs,
the unitarity condition is not imposed and a different set of `regularity conditions' are used instead,
see e.g.~\cite[Sec.2.3]{HoehnMoeller}.
In this appendix we tacitly assume that these regularity conditions follow from the unitarity condition.

The underlying vector space of a VOSA $\V=\V(T)$ is
the Hilbert space of the theory $T$ on $S^1$ with trivial spin structure.
The Hilbert space of the theory $T$ on $S^1$ with nontrivial spin structure
defines a \emph{canonically $\bZ/2$-twisted} module of $\V$,
which we denote by $\W$.
$\V$ itself has a decomposition $\V=\V_0\oplus \V_1$ 
so that $\V_0$ has integer gradings and is an ordinary VOA,
whereas $\V_1$ has half-integer gradings.
When $\V_1$ is empty, $\V$ is actually a self-dual VOA,
and $\W\simeq \V$.
Otherwise,
$\W$ is either irreducible or a sum of two irreducibles as a representation of $\V_0$.
Then  $\V_0$, $\V_1$, and the irreducible components of $\W$
generate the representation category $\cM(\V_0)$ of $\V_0$,
and therefore $\V_0$ is rational in the sense of the VOA theory.
$\cM(\V_0)$ is known to have a natural structure as a modular tensor category.

A basic series of examples of such VOSAs are given by Clifford super VOAs $\V(n)$, see e.g.~\cite{KacWang,DuncanMackCrane}.
They are  VOSAs constructed from $\bR^n$ with positive inner product,
generalizing the construction of the Clifford algebra $\mathrm{Cliff}(n)$ from the same starting point.
It has central charge $n/2$
and $\V(n)_0$ is equal to the affine Lie VOA $\mathfrak{so}(n)_1$ for $n\ge 5$.
We denote its canonically $\bZ/2$-twisted module by $\W(n)$,
whose highest weight vectors form the Clifford module for $\mathrm{Cliff}(n)$.

Given a self-dual VOSA $\V$, % containing $\mathfrak{g}_k$,
its central charge is of the form $c=n/2$.
The unitarity implies  a certain reality condition on the underlying vector space of $\V$ and $\W$,
extending the reality condition for a unitary VOA given in \cite{Kawahigashi}.
The reality condition on $\V$ was recently detailed in \cite{carpi2023vosa}.
To state the reality condition on $\W$ expected from physics,
let $8i$ be the smallest multiple of eight above or equal to $n$.
Then $\V\otimes \V(8i-n)$ is a VOSA with $c=4i$,
and its canonical $\bZ/2$-twisted module $\W\otimes \W(8i-n)$ 
should be a complexification of a real $\bZ/2$-graded vector space
with an action of $\mathrm{Cliff}(8i-n)$.

%\subsubsection{When $\V$ contains an affine algebra VOA}
%\label{app:voa:even}
Let us now describe how the symmetries of VOA enter the story.
When $\V$ contains $\mathfrak{g}_k$,
there should be an action of $G$ (the compact simply-connected group corresponding to $\mathfrak{g}$) on $\W\otimes \W(8i-n)$,
compatible with the $\bZ/2$-grading, the real structure, and the action of $\mathrm{Cliff}(8i-n)$.
This means that each graded piece of $\W$ naturally determines an element of $\KO_G^n(\pt)$.
The grading is provided by $L_0-c/24=L_0-n/48$, and the eigenvalues of $L_0$ on $\W$
is of the form $n/16+\Z$.
Then, $\W$ naturally determines an element  $[\W]\in q^{n/24}\KO_G^n((q))(\pt)$.

Combining the considerations thus far, we can now formulate the following conjecture:
\begin{conj}
\label{conj:affinevoa}
Let $\V$ be a unitary self-dual VOSA
whose central charge is $c=n/2$ such that
it contains the affine Lie VOA $\mathfrak{g}_k$.
Let $\W$ be its canonically $\bZ/2$-twisted module.
Then there is a class $[\V]\in \TMF_G^{n+k\tau}(\pt)$\footnote{
{Mathematically, for each integer $k$ we have a preferred choice of the map $k\tau \colon BG \to K(\Z, 4)$ \cite[Example 5.1.5]{FSSdifferentialtwistedstring} representing $k$ times the generator $\tau$ of $\H^4(BG,\bZ)\simeq \bZ$. This allows us to talk about $\TMF_G^{n+k\tau}$. }},
such that its image $\sigma([\V])\in \KO_G^n((q))(\pt)$ is given by \[
\sigma([\V]) = \eta(q)^{-n} [\W] \in \KO_G^n((q))(\pt).
\]
\end{conj}

It is known that we can start from a (possibly odd) self-dual lattice $L$
and form a self-dual VOSA $\V(L)$.
$\V(L)$ is unitary when $L$ is positive definite,
and $\V(L)$ is VOA when $L$ is even.
If $L$ contains a root lattice $\Lambda_{\mathfrak{g}}$ for a simple Lie algebra $\mathfrak{g}$,
$\V(L)$ contains an affine Lie VOA $\mathfrak{g}_1$ at level 1 as a subalgebra.
This level corresponds to the generator $\tau$ of $\H^4(BG;\bZ)\simeq \bZ$.
Combined with the conjecture above,
we now have a plethora of $\TMF$ classes
associated to unimodular lattices.

For example, the root lattice of $E_8$ is the unique even self-dual lattice of rank 8,
and there are two even self-dual lattices of rank 16, one is $E_8\times E_8$
and another containing the root lattice of $D_{16}$.
We denote the corresponding VOAs by $\V_{E_8}$,
$\V_{E_8\times E_8}$, $\V_{D_{16}}$, respectively.

Combining these analyses, we have the following conjecture:
\begin{conj}\label{conj_e8}
The class $e_8\in \pi_{-16}\TMF$, specified by its image $\sigma(e_8)=c_4/\Delta\in \pi_{-16}\KO((q))$, lifts to an element \[
\breve e_8 = [\V_{E_8}]\in \TMF_{E_8}^{16+\tau}(\pt),
\] such that its image in $\KO_{E_8}^{16}((q))(\pt)$ is given by 
\[
%\label{sigma-e8}
\sigma(\breve e_8) = \eta(q)^{-16} \chi_{(E_8)_1}
\]
where $\chi_{(E_8)_1}$ is the unique integrable representation of the affine Lie algebra $(\mathfrak{e}_8)_1$
regarded as an element of $q^{2/3}\KO_{\mathfrak{e}_8}^0((q))(\pt)$, so that \[
\chi_{(E_8)_1} = q^{-1/3}(1+ \mathfrak{e}_8 q + \cdots ).
\]
\end{conj}

\subsubsection{VOA and $\TMF$ classes in the range $c\le 16$}
Self-dual VOAs of central charges $c\le 16$ were classified by \cite{DongMason} without the unitarity assumption.
In general, the central charge $c$ of a self-dual VOA is a multiple of eight,
and it was found in \cite{DongMason} that all such VOAs with $c\le 16$ 
are actually lattice VOAs associated to even self-dual lattices of rank $8$ and $16$.
%which we briefly recalled at the end of Sec.~\ref{app:voa:even}.
%For $c=8$ it is the root lattice $\Lambda_{E_8}$ of $E_8$ itself,
%and for $c=16$ it is either $(\Lambda_{E_8})^2$ or $L_{D_{16}}$ which contains 
%the root lattice $\Lambda_{D_{16}}$ as an index-2 subgroup.

Self-dual  VOSAs with central charge $c\le 24$ were recently rigorously classified in \cite{HoehnMoeller}.\footnote{%
See also a physics-level discussion up to $c\le 16$ in \cite{ManyPeople} and up to $c<23$ in \cite{Rayhaun}.
}
In the subrange $c\le 16$ 
they are given by the products of the basic self-dual VOSAs tabulated in Table~\ref{table:voa}.
\begin{table}[h]
\begin{tabular}{|l|l|l|l|l|}
\hline
name & $c$ &  affine Lie sub VOA &  in $\pi_{-2c}\KO((q))$ &  comment \\
\hline
\hline
$\V(1)$ & $c=1/2$ & none & 0&   Clifford VOSA \\
$\V_{E_8}$ & $c=8$ & $(\mathfrak{e}_8)_1$ & $c_4/\Delta$ & VOA \\
$\V_{D_{12}}$ & $c=12$ & $(\mathfrak{d}_{12})_1$ & $24/\Delta$ &  VOSA of \cite{DuncanMackCrane} \\
$\V_{E_7\times E_7}$ &  $c=14$ & $(\mathfrak{e}_7)_1\oplus (\mathfrak{e}_7)_1$& 0&\\ 
$\V_{A_{15}}$  &  $c=15$ &  $(\mathfrak{a}_{15})_1$ &0 &\\
$\V_{E_8,2}$ & $c=31/2$ &  $(\mathfrak{e}_{8})_2$&0&  \\
$\V_{D_{8}\times D_8}$ & $c=16$ & $(\mathfrak{d}_8)_1 \oplus (\mathfrak{d}_8)_1$& 0 &\\
$\V_{D_{16}}$ & $c=16$ & $(\mathfrak{d}_{16})_1$ & $(c_4/\Delta)^2$ & VOA\\
\hline
\end{tabular}
\caption{List of holomorphic non-product super VOAs with $c\le 16$. \label{table:voa}}
\end{table}

These basic  VOSAs $\V$ should give rise to classes $[\V]\in \TMF^{2c+\tau}_G(\pt)$ which can be reduced to $[\V] \in \pi_{-2c}\TMF$.
By passing to $\KO((q))$ via $\sigma$,
we also have elements $\sigma([\V])$ in $\KO^{2c}_G((q))(\pt)$
and in $\pi_{-2c}\KO((q))(\pt)$.
$\sigma([\V])$ as an element in $\KO^{2c+\tilde k}_G((q))(\pt)$ can be found using Conjecture~\ref{conj:affinevoa}.
It turns out to be non-zero when $c>1/2$,
where we take $G$ to be %$\bZ/2$ for $\V=\V(1)$
the simply-connected simple group for the affine subalgebra $\mathfrak{g}_k$. % otherwise.
As an element in $\KO^{2c}((q))(\pt)$, however, $\sigma([\V])$ is often zero.
For such a $\V$, 
$[\V]\in \pi_{-2c}\TMF$ is in the kernel $A_{-2c}$ of $\sigma:\pi_{-2c}\TMF\to \pi_{-2c}\KO((q))$.

It is tantalizing that the values $c=14$, $15$, $31/2$ and $32$ we see in Table~\ref{table:voa} are exactly where
the kernel $A_{-2c}$ of $\sigma:\pi_{-2c}\TMF\to \pi_{-2c}\KO((q))$ is nontrivial
in the range $0\le 2c \le 32$,
according the data gathered in Appendix~\ref{app:TMF}.
The kernel is isomorphic to $\bZ/2$, $\bZ/2$, $\bZ/2$ and $\bZ/3$, respectively.

There are physics reasons to believe that the elements $x_{-31}\in \pi_{-31}\TMF$ and $x_{-28}\in \pi_{-28}\TMF$ defined in Definitions~\ref{defn:31} and \ref{defn:28}
agree with the classes $[\V_{E_8,2}]$  and $[\V_{E_7\times E_7}]$.
One indication is the following. 
The construction of $x_{-31}$ and $x_{-28}$, given in our Definitions just referred to,
starts from $e_8\in \pi_{-16}\TMF$,
which can be lifted to $\breve e_8\in\TMF^{16+\tau}_{E_8}(\pt)$ by Conjecture~\ref{conj:above}.
This then implies that $x_{-31}$ can be lifted to $\TMF^{31+2\tau}_{E_8}(\pt)$
and that $x_{-28}$ can be lifted to $\TMF^{28+\tau+\tau'}_{E_7\times E_7}(\pt)$,
respectively. 

Now, our Theorem~\ref{thm:31} % and \ref{thm:28}
 in the main text
says that $x_{-31}$ and $x_{-28}$ are indeed the generators 
of the kernel $A_{-2c}$ of $\sigma:\pi_{-2c}\TMF\to \pi_{-2c}\KO((q))$ at $c=31/2$ and $c=14$,
respectively.
These circumstantial pieces of evidence strongly suggest the following general conjecture:
\begin{conj}
The generator of the kernel $A_{-2c}$ of $\sigma:\pi_{-2c}\TMF\to \pi_{-2c}\KO((q))$
at $c=32$, $c=31/2$, $c=15$ and $c=14$,
isomorphic to $\bZ/3$, $\bZ/2$, $\bZ/2$, $\bZ/2$ respectively,
comes from the corresponding self-dual  VOSA $\V$ given in Table~\ref{table:voa}.
In particular, the generator of the kernel 
can be lifted to $\TMF^{2c+t}_G(\pt)$ for an appropriate choice of $G$ and $t\colon BG \to K(\Z, 4)$,
whose image in $\KO^{2c}_G((q))(\pt)$ is nonzero and is given by 
the class determined by $\eta(q)^{-2c} [\W]$,
where $\W$ is the canonically $\bZ/2$-twisted module of $\V$.
\end{conj}

We note that a physics piece of evidence that $\V_{D_{16}}$ generates $A_{-32}=\bZ/3$ was given in \cite{Tachikawa:2024ucm},
and a physics piece of evidence that $\V_{E_8,2}$ gives the generator $x_{-31}$ of $A_{-31}=\bZ_2$
was given in \cite{Saxena:2024eil}.
We also note that  the list of self-dual VOAs in the wider range up to $c\le 24$ is given in \cite{HoehnMoeller} and \cite{Rayhaun}.
It would be interesting to compare this list against the $\TMF$ classes in the corresponding range.

%\input{firstproof}

%\input{secondproof}

% !TEX root = paper.tex
\section{Two proofs of Theorem~\ref{main}}
\label{app:proofs}
Here we provide two different proofs of our theorem~\ref{main}, which states that 
the secondary morphism 
    \begin{align}\label{eq_main_thm_x}
        \alpha_{\KO((q))/\TMF} \colon \KO((q))/\TMF \to \Sigma^{-20}I_\Z \TMF
    \end{align}
    is an equivalence. 
Our first proof uses the Anderson duality of $\tmf$ as formulated in \cite[Chapters 10, 13]{BrunerRognes}
and is rather computational,
whereas the second proof uses the Anderson self-duality of $\Tmf$ as originally proved by \cite{Sto1,Sto2}
and is more conceptual.
This second proof was kindly provided by Sanath Devalapurkar.

% !TEX root = paper.tex
%\section{A computational proof of Theorem~\ref{main}}
\subsection{A computational proof}
\label{app:original-proof}
\if0
Here we provide a proof of our theorem~\ref{main}, that 
the secondary morphism 
    \begin{align}\label{eq_main_thm_x}
        \alpha_{\KO((q))/\TMF} \colon \KO((q))/\TMF \to \Sigma^{-20}I_\Z \TMF
    \end{align}
    is an equivalence. 
%As we will see, this equivalence describes a type of the Anderson self-duality of the spectrum $\TMF$.
Originally the Anderson duality of topological modular forms was formulated and proved 
as the Anderson self-duality of $\Tmf$ in the form $\Tmf \simeq \Sigma^{-21} I_\bZ \Tmf$ by \cite{Sto1,Sto2}.
This was expressed in terms of $\tmf$ in \cite[Chapter 10 and 13]{BrunerRognes}.
We recall the necessary results, based on \cite{BrunerRognes},
and reformulate the Anderson duality in terms of $\TMF$.
This will lead us to \eqref{eq_main_thm_x}.
\fi

First we introduce notations. 
We always use tensor product over $\Z$ so that we use the notation $\otimes := \otimes_{\Z}$. 
For an $E_\infty$ ring spectrum $R$, an element $x \in \pi_d(R)$ and an $R$-module spectrum $M$, we define the $R$-module spectrum $M[x^{-1}]$ by
\begin{align}
    M[x^{-1}] := \hocolim \left(M \xrightarrow{x\cdot} \Sigma^{-d}M \xrightarrow{x\cdot} \Sigma^{-2d}M \xrightarrow{x\cdot} \cdots \right). 
\end{align}
Denote by $M/x^\infty$ the homotopy cofiber of the structure map $M \to M[x^{-1}]$. 
We have  the relation $\pi_\bullet(M[x^{-1}]) = \pi_\bullet(M) [x^{-1}]$, and an exact sequence
\begin{align}
    0 \to \pi_\bullet(M)/x^\infty \to \pi_\bullet (M/x^\infty) \to \Gamma_x \pi_{\bullet-1}(M)\to 0, 
    \label{short-exact-seq}
\end{align}
where $\Gamma_x\pi_{\bullet-1}(M)$ denotes the $x$-power torsion submodule of $\pi_{\bullet-1}(M)$, and $\mathfrak{M}/\mathfrak{x}^\infty$ for an element $\mathfrak{x} \in \mathfrak{R}$ in an ordinary graded ring $\mathfrak{R}$ and an $\mathfrak{R}$-module $\mathfrak{M}$ is defined just as $\mathfrak{M}/\mathfrak{x}^\infty:= \mathrm{coker}\left( \mathfrak{M} \to \mathfrak{M}[\mathfrak{x}^{-1}]\right)$. 
For a pair of elements $x, y \in \pi_\bullet(R)$, we iterate the above construction to define 
\begin{align}
    R[x^{-1}, y^{-1}] := (R[x^{-1}])[y^{-1}], \quad R/(x^\infty, y^\infty):= (R/x^\infty)/y^\infty. 
\end{align}
It is known that $R/(x^\infty, y^\infty) \simeq R/(y^\infty, x^\infty)$.

We turn to $\TMF$ and $\tmf$. 
%Following the notations of \cite{BrunerRognes}, let $B \in \pi_{8}\tmf \subset \pi_8 \TMF$ be the Bott element, which is characterized by its modular form image $c_4$. 
 Let $B\in \pi_8\tmf$ be a class introduced in \cite{BrunerRognes} whose image by $\sigma$ is $\sigma(B) =c_4$.\footnote{
 There are two  elements in $\pi_8\tmf$, $B$ and $\tilde B:= B +\epsilonelem$, whose image by $\sigma$ is $c_4$.
 $B$ and $\tilde B$ are distinguished by their Adams filtrations.
}
Let $X \in \pi_{576}\tmf$ denote the element whose image in $\mf$ is $\Delta^{24}$. $X$ is the periodicity element of $\TMF = \tmf[X^{-1}]$. 
%The goal of this subsection is to prove that $\alpha_{KO((q))/\TMF}$ is essentially the Anderson self-duality isomorphism. 
%For this, we first recall the Anderson self-duality statement given in \cite{BrunerRognes}. 
We use $\TMF/B^\infty$ defined by  the cofiber sequence \begin{equation}
\xymatrix{
\TMF  \ar[r]^-{\phi} &
\TMF[B^{-1}]  \ar[r]^-{C\phi} &
\TMF/B^\infty \\
}.
\end{equation}
We also define 
\begin{align}
    \tmf' := \tmf / (  B^\infty, X^\infty)
\end{align}
\if0
Its homotopy groups fit into the short exact sequence % (\cite[Theorem 10.15]{BrunerRognes})
\begin{align}
    0 \to \pi_\bullet(\tmf)/(X^\infty,B^\infty ) \to \pi_\bullet \tmf' \to \Gamma_B \pi_{\bullet-1}(\tmf)/X^\infty \to 0. 
\end{align}
\fi
which fits in the cofiber sequence \begin{equation}
\tmf/B^\infty 
\xrightarrow{\rho}
(\tmf /B^\infty)[X^{-1}] \simeq 
\TMF/B^\infty 
\xrightarrow{C\rho}
(\tmf/B^\infty)/X^\infty \simeq  \tmf'.
\label{Crho}
\end{equation}

\begin{fact}[{\cite[Proposition 10.12 and Proposition 13.21]{BrunerRognes}}]\label{fact_duality_tmf'}
We have a $\tmf$-module equivalence
\begin{align}\label{eq_isom_tmf'}
    \alpha_{\tmf'} \colon \tmf' \simeq \Sigma^{-20}I_{\Z} \tmf, 
\end{align}
which is specified by the element 
\begin{align}
       \alpha_{\tmf'}  \in \pi_{20}( I_{\Z} \tmf') \simeq \mathrm{Hom}(\pi_{-20}\tmf', \Z) \simeq \mathrm{Hom}(\Z, \Z). 
\end{align}
Here we use $\pi_{-21}\tmf' = 0$ and $\pi_{-20}\tmf' \simeq \Z$, with the generator whose modular form image is $2c_6/(c_4 \Delta)$; see \cite[p.382]{BrunerRognes} and \cite[p.590]{BrunerRognes} for the $2$ and $3$-completed results, respectively. 
\end{fact}

For our purposes, we need to understand the $\bZ[X]$-module structure of the homotopy groups involved in the Anderson duality given above.
For this purpose, we use the following two facts:
\begin{fact}[{\cite[Definition 10.4, Definition 13.15 and Lemma 13.16]{BrunerRognes}}]
\label{tmfBX}
Let $N_\bullet\subset\pi_\bullet\tmf$ be the $\bZ[B]$-submodule
generated by elements of degree less than $576$.
We then have  \begin{equation}
\pi_\bullet  \tmf \simeq N_\bullet \otimes \bZ[X],\label{def_N}
\end{equation} as $\bZ[B, X]$-modules,
where the isomorphism is given by the multiplication.
With this $N$, we have
 we have \begin{equation}\label{444}
\pi_\bullet  \TMF \simeq N_\bullet \otimes \bZ[X^{-1},X]
\end{equation}
as $\bZ[B, X]$-modules.
\end{fact}

\begin{fact}[{\cite[Lemma10.16 and Lemma 13.23]{BrunerRognes}}]
    There exists a unique, up to isomorphism, $\Z[B]$-module extension 
    \begin{align}\label{eq_def_N'}
        0 \to N_\bullet / B^\infty \to N'_\bullet \to \Gamma_B N_{\bullet - 1} \to 0
    \end{align}
    such that we have a commutative diagram of $\Z[B, X]$-modules,
    \begin{align}
        \xymatrix{
        0 \ar[r] & (\pi_\bullet\tmf)/B^\infty \ar[r] \ar[d]^-{\simeq}_-{\eqref{def_N}}& \pi_\bullet (\tmf/B^\infty) \ar[r] \ar[d]^-{\simeq} &  \Gamma_B \pi_{\bullet-1}\tmf \ar[d]^-{\simeq}_-{\eqref{def_N}}\ar[r] &  0 \\
        0 \ar[r] & N_\bullet / B^\infty \otimes \Z[X] \ar[r] & N'_\bullet \otimes \Z[X] \ar[r] &  \Gamma_B N_{\bullet - 1} \otimes \Z[X] \ar[r] &  0 . 
        }
    \end{align}
    Here the upper row is the exact sequence \eqref{short-exact-seq} and the lower row is the exact sequence obtained by tensoring the flat $\Z$-module $\Z[X]$ to \eqref{eq_def_N'}. 
    With this $N'$, we have canonical isomorphisms of $\Z[B, X]$-modules,
    \begin{align}
\pi_\bullet  (\tmf/B^\infty) &\simeq N'_\bullet \otimes \bZ[X],\label{111}\\
\pi_\bullet  (\TMF/B^\infty) &\simeq N'_\bullet \otimes \bZ[X^{-1},X],\label{222}\\
\pi_\bullet\tmf' &\simeq N'_\bullet \otimes  \bZ[X^{-1},X]/\bZ[X]\label{333}.
\end{align}
\end{fact}

We are going to define a $\TMF$-module map  $\alpha_{\TMF/B^\infty} :\TMF/B^\infty \to \Sigma^{-20} I_\bZ \TMF$.
We start with a lemma:
\begin{lem}
     There is an isomorphism \begin{equation}\label{eq_IZ_TMF/Binfty}
     \pi_{20}(I_\bZ (\TMF/B^\infty) ) \simeq \Hom(\pi_{-20}(\TMF/B^\infty), \bZ).
     \end{equation}
\end{lem}

\begin{proof}
This follows from $\pi_{-21}\TMF/B^\infty=0$,
which in turn follows from the long exact sequence of the homotopy groups
associated to the cofiber sequence above: \begin{equation}
\pi_{-21} \TMF[B^{-1}] \xrightarrow{C\phi} \pi_{-21} \TMF/B^\infty \xrightarrow{\delta} \pi_{-22} \TMF \xrightarrow{\phi} \pi_{-22} \TMF[B^{-1}].
\end{equation}
Indeed, $\pi_{-21}\TMF[B^{-1}]$ is zero, since $\pi_{-21}\TMF$ is purely of $B$-power torsion. 
Furthermore, the kernel of the morphism $\phi$ on the right is $A_{-22}$ from Proposition~\ref{kernel=Btorsion}, which is trivial.
Therefore, $\pi_{-21}\TMF/B^\infty=0$.
\end{proof}

\begin{defn}\label{def_a}
     We define \[
     \alpha_{\TMF/B^\infty} :\TMF/B^\infty \to \Sigma^{-20} I_\bZ \TMF
     \] in terms of 
     $\alpha_{\TMF/B^\infty}\in \pi_{20} I_\bZ (\TMF/B^\infty)$
     using the isomorphism \eqref{eq_IZ_TMF/Binfty} via 
     \begin{equation}
     \label{eq_duality_formula}
     \begin{aligned}
         \alpha_{\TMF/B^\infty} \colon (\TMF/B^\infty)_{-20} \hookrightarrow (\MF/c_4^\infty)_{-10} &\to \bZ, \\
         [\phi(q)] &\mapsto \frac{1}{2}\Delta(q)\phi(q)\big|_{q^0}.  
     \end{aligned}
\end{equation}
\end{defn}

\begin{rem}
Note that $\MF/c_4^\infty$ is the cokernel of the map $\MF \to \MF[c_4^{-1}]$.
Then the well-definedness of the formula \eqref{eq_duality_formula} follows by the fact that $\Delta(q)f(q)|_{q^0} = 0$ for all $f\in \MF_{-10}$, thanks to Fact~\ref{fact:mf}.
\end{rem}

\begin{prop}
The following diagram commutes: 
\begin{equation}
 \vcenter{\xymatrix{
 \TMF/B^\infty 
\ar[rr]^-{\alpha_{\TMF/B^\infty}}  
 \ar[d]^-{C\rho}  
 && \Sigma^{-20} I_\bZ \TMF 
 \ar[d]
 \\
 \tmf' 
 \ar[rr]^-{\alpha_{\tmf'}}_-{\sim} && \Sigma^{-20} I_\bZ \tmf \\
 }}
 \label{TMFtmf}
 \end{equation}
where $C\rho$ is the morphism given in \eqref{Crho}.
\end{prop}
\begin{proof}
We need to show that 
the element $\alpha_{\tmf'}\in \pi_{-20}I_\bZ \tmf'$
is sent to $\alpha_{\TMF/B^\infty} \in \pi_{-20}I_\bZ \TMF/B^\infty$
via $I_\bZ C\rho$.
This follows by the characterization of $\alpha_{\tmf'}$ in Fact \ref{fact_duality_tmf'} and 
the fact that the formula \eqref{eq_duality_formula} maps $2c_6/(c_4 \Delta) \in (\MF/c_4^\infty)_{-10}$ to $1 \in \Z$. 
\end{proof}

We have $\TMF \simeq \tmf[X^{-1}] = \hocolim \left( \tmf \xrightarrow{X \cdot} \Sigma^{-576}\tmf \xrightarrow{X \cdot} \cdots \right)$. Taking the Anderson dual, we get a canonical identification
\begin{align}\label{eq_IZTMF}
    I_\Z \TMF \simeq \mathrm{holim} \left(\cdots \xrightarrow{X \cdot}\Sigma^{2 \times 576} I_\Z \tmf  \xrightarrow{X \cdot}\Sigma^{576} I_\Z \tmf \xrightarrow{X
    \cdot}  I_\Z \tmf  \right). 
\end{align}

\begin{prop}\label{prop_N'TMFtmf}
    The commutative diagram \eqref{TMFtmf}, the isomorphisms \eqref{222} and \eqref{333} combine and naturally extend to form the following commutative diagram of $\Z[X]$-modules, 
\begin{align*}\label{diag_N'TMFtmf}
    \vcenter{\hbox{\xymatrixcolsep{.5cm}\xymatrix{ 
    N'_\bullet \otimes \Z[X^{-1}, X]  \ar[ddd]^-{{\rm mod } \ \Z[X]} \ar@{^{(}-_>}[rrrr] &&&& N'_\bullet \otimes  \Z[X^{-1}, X]] \ar[ddd]^-{{\rm mod } \ \Z[[X]]}\\ 
& \pi_\bullet \TMF/B^\infty \ar[lu]^-{\simeq}_-{\eqref{222}}
\ar[rr]^-{\alpha_{\TMF/B^\infty}}  
 \ar[d]^-{C\rho}  
 && \pi_\bullet \Sigma^{-20} I_\bZ \TMF 
 \ar[d] \ar[ru]^-{\simeq}&
 \\
 &\pi_\bullet \tmf' \ar[ld]^-{\simeq}_-{\eqref{333}}
 \ar[rr]^-{\alpha_{\tmf'}}_-{\simeq} && \pi_\bullet \Sigma^{-20} I_\bZ \tmf \ar[rd]^-{\simeq}  &\\
 N'_\bullet \otimes \Z[X^{-1}, X] /\Z[X] \ar@{=}[rrrr] &&&& N'_\bullet \otimes  \Z[X^{-1}, X]/\Z[X] \\ 
 }}}
\end{align*}
where the upper trapezium is also commutative as a diagram of $\Z[X^{-1}, X]$-modules. 
\end{prop}
\begin{proof}
We already know the commutaivity of the left, middle and bottom trapezium. 
The isomorphism $\pi_\bullet \Sigma^{-20}I_\Z \TMF \simeq N'_\bullet \otimes \Z[X^{-1}, X]]$ is induced by the sequence \eqref{eq_IZTMF} as follows. Since \eqref{eq_IZTMF} satisfies the Mittag-Leffler condition, we have
\begin{align}
    \pi_\bullet \Sigma^{-20}I_\Z \TMF &\simeq \lim \left( \cdots \xrightarrow{X \cdot}\Sigma^{2 \times 576-20} \pi_\bullet I_\Z \tmf  \xrightarrow{X \cdot}\pi_\bullet\Sigma^{576-20} I_\Z \tmf \xrightarrow{X
    \cdot} \pi_\bullet \Sigma^{-20} I_\Z \tmf  \right) \\
    &\simeq \lim \left( \cdots \xrightarrow{X \cdot}N'_{\bullet-576} \otimes \Z[X^{-1}, X] /\Z[X] \xrightarrow{X
    \cdot} N'_\bullet \otimes \Z[X^{-1}, X] /\Z[X]  \right) \\
    &\simeq N'_\bullet \otimes \Z[X^{-1}, X]. 
\end{align}
Then the commutativity of the right trapezium in the diagram % in \eqref{diag_N'TMFtmf} 
is automatic. 
Then we are left to prove that the upper trapezium is commutative as a diagram of $\Z[X^{-1}, X]$-modules. But this follows by the fact that the inclusion $N'_\bullet \otimes \Z[X^{-1}, X] \hookrightarrow N'_\bullet \otimes \Z[X^{-1}, X]]$ is the unique $\Z[X^{-1}, X]$-module homomorphism between the two modules which lifts the isomorphism $N' \otimes \Z[X^{-1}, X]/\Z[X] \simeq N' \otimes \Z[X^{-1}, X]]/\Z[[X]]$. 
\end{proof}

Consider now the commutative diagram of fiber sequences of $\TMF$-modules
    \begin{align}\label{diag_completion_TMF/B}
\vcenter{\xymatrix{
\TMF \ar@{=}[d] \ar[r]^-{\phi} &
\TMF[B^{-1}] \ar[d]^-{p} \ar[r]^-{C\phi} &
\TMF/B^\infty \ar[d]^-{\varrho} \\
\TMF \ar[r]^-{\sigma}& 
\KO((q))  \ar[r]^-{C\sigma} & 
\KO((q))/\TMF
}}
\end{align} where $p$ is induced by the invertibility of $\sigma(B)=c_4$ in $\pi_\bullet \KO((q))$ and we fix a $\TMF$-module morphism $\varrho$ to make the diagram commute. 
%The following is the first main result of this subsection. 

\begin{prop}
\label{prop:relation-secondary-anderson}
The following diagram commutes: \begin{equation}
\vcenter{\xymatrix{
\TMF/B^\infty 
\ar[rr]^-{\alpha_{\TMF/B^\infty}}  
\ar[d]_-{\varrho}  
&& \Sigma^{-20} I_\bZ \TMF 
\ar@{=}[d]
\\
\KO((q))/\TMF \ar[rr]^-{\alpha_{\KO((q))/\TMF}} && \Sigma^{-20} I_\bZ \TMF \\
}}
\label{bigdiagram}
\end{equation}
\end{prop}

\begin{proof}

We  need to show that 
the element $\alpha_{\KO((q))/\TMF} \in \pi_{-20} I_\bZ \KO((q))/\TMF$
is sent to $\alpha_{\TMF/B^\infty}$
via $I_\bZ\varrho$.
This  follows 
by comparing \eqref{alpha-KO/TMF} and \eqref{eq_duality_formula}.
\if0
The factor of $1/2$ is due to the fact that $\MF \hookrightarrow \bZ((q))$ 
is done via $\MF_{2n} \hookrightarrow \pi_{4n}\K((q))\simeq \bZ((q))$
whereas $\pi_{8k+4}\KO((q)) \simeq \bZ((q)) \to \pi_{8k+4}\K((q)) \simeq \bZ((q))$ is a multiplication by 2.
\fi
\end{proof}

Our aim was to show that the second row of the diagram \eqref{bigdiagram} is an equivalence.
We already know $\pi_\bullet (\TMF/B^\infty)\simeq N'_\bullet\otimes \bZ[X^{-1},X]$
and $\pi_\bullet\Sigma^{-20}I_\bZ\TMF\simeq N'_\bullet \otimes \bZ[X^{-1},X]]$.
What we need to do then is to show that the effect of $\varrho$ on the homotopy group is the $X$-adic completion.

\begin{lem}\label{lem_TMF_completion}
There is an isomorphism $N_\bullet[B^{-1}] \otimes \bZ[X^{-1},X]]\simeq \pi_\bullet \KO((q))$ of $\Z[X^{-1},X]$-modules which makes the following diagram commute:
\begin{align}
    \vcenter{\hbox{\xymatrix{
    \pi_\bullet \TMF[B^{-1}] \ar[r]^-{p} & \pi_\bullet \KO((q)) \\
    N_\bullet[B^{-1}] \otimes \bZ[X^{-1},X] \ar@{^{(}-_>}[r] \ar[u]^-{\simeq}_-{\eqref{444}}& N_\bullet[B^{-1}] \otimes \bZ[X^{-1},X]] \ar[u]^-{\simeq}
    }}}. 
\end{align}

\end{lem}
\begin{proof}
We note first that, for each degree $d$, the middle vertical arrow $p$ in \eqref{diag_completion_TMF/B} induces the inclusion of $\pi_0 \TMF = \Z[J]$-modules, 
\begin{align}
    p \colon \pi_d \TMF[B^{-1}] &\xrightarrow{\simeq}
    \begin{cases}
        c_4^k \cdot  \pi_{d} \KO[J, J^{-1}], & d = 8k, 8k+1, 8k+2, \\
        c_4^{k-1}c_6 \cdot  \pi_{d} \KO[J, J^{-1}], & d = 8k+4, \\
        0 & \mbox{otherwise}
    \end{cases} \label{eq_isom_TMF_KO[JJ-1]} \\
    & \hookrightarrow \pi_d \KO((q)). \notag 
\end{align}
Indeed, the injectivity of $p \colon \pi_d \TMF[B^{-1}] \to \pi_d \KO((q))$ follows from Proposition \ref{kernel=Btorsion}, and the statement about the image is verified by the fact that $p$ is a $\pi_0\TMF[B^{-1}] \simeq \Z[J, J^{-1}]$-module map and the fact that the generators $c_4^k$ or $c_4^{k-1}c_6$ of the modules of the right hand side are contained in the image. 
The statement then follows  by noticing that, at each fixed degree, the $X$-adic completion is the $J^{-1}$-adic completion.
\end{proof}

\begin{prop}\label{prop_completion_TMF/B}
There is an isomorphism $N'_\bullet \otimes \bZ [X^{-1},X] \simeq \pi_\bullet \KO((q))/\TMF $ of $\Z[X^{-1}, X]$-modules which makes the following diagram commute:
\begin{align}
    \vcenter{\hbox{\xymatrix{
    \pi_\bullet \TMF/B^\infty \ar[r]^-{\varrho } & \pi_\bullet \KO((q))/\TMF  \\
    N'_\bullet \otimes \bZ [X^{-1},X] \ar[u]^-{\simeq}_{\eqref{222}} \ar@{^{(}-_>}[r] & N'_\bullet \otimes \bZ [X^{-1},X]]\ar[u]^-{\simeq}
    }}}.
\end{align}

\end{prop}
\begin{proof}
We have the commutative diagram of long exact sequences, 
\begin{align}
\hskip-1cm\vcenter{\xymatrix{
\cdots \ar[r]& \pi_d \TMF \ar[r]^-{\phi} \ar@{=}[d]& \pi_d \TMF[B^{-1}] \ar[r]^-{C\phi} \ar[d]^-{p} & \pi_d \TMF/B^\infty \ar[r] \ar[d]^-{\varrho} & \pi_{d-1}\TMF \ar[r] \ar@{=}[d] & \cdots \\
\cdots \ar[r]&
\pi_d\TMF \ar[r]^-{\sigma}& 
\pi_d \KO((q))  \ar[r]^-{C\sigma} & 
\pi_d \KO((q))/\TMF \ar[r]& \pi_{d-1}\TMF \ar[r]& \cdots 
}}.
\end{align}
We know $\pi_\bullet \TMF\simeq N_\bullet \otimes \bZ[X]$,
$\pi_\bullet \TMF[B^{-1}]\simeq N_\bullet[B^{-1}] \otimes \bZ[X^{-1},X]$,
$\pi_\bullet( \TMF/B^\infty) \simeq N'_\bullet \otimes \bZ[X^{-1},X]$
and
$\pi_\bullet \KO((q)) \simeq N_\bullet[B^{-1}] \otimes \bZ[X^{-1},X]]$
from previous propositions.
Combining, we get an isomorphism $\pi_\bullet\KO((q))/\TMF \simeq N'_\bullet \otimes \bZ[X^{-1},X]]$ with the desired property.
\end{proof}

The main result of this subsection is the following:
\begin{thm}[=Theorem~\ref{main}]
The morphism \[
\alpha_{\KO((q))/\TMF} \colon  \KO((q))/\TMF \to \Sigma^{-20} I_\bZ \TMF
\] is an equivalence of $\TMF$-modules. 
\end{thm}
\begin{proof}
This immediately follows by combining Proposition~\ref{prop_N'TMFtmf}, Proposition~\ref{prop:relation-secondary-anderson} and Proposition~\ref{prop_completion_TMF/B}.
\end{proof}

% !TEX root = paper.tex

%\section{A more conceptual proof of Theorem~\ref{main}}
\subsection{A more conceptual proof}
\label{app:sanath-proof}
Here we provide another proof of our theorem~\ref{main},
that the morphism \eqref{eq_main_thm_x} is an equivalence.
\if0
Here we provide another proof of our theorem~\ref{main}, that 
the secondary morphism 
    \begin{align}%\label{eq_main_thm}
        \alpha_{\KO((q))/\TMF} \colon \KO((q))/\TMF \to \Sigma^{-20}I_\Z \TMF
    \end{align}
    is an equivalence. 
This section was  kindly provided by Sanath Devalapurkar.
\fi
Recall that originally the Anderson duality of topological modular forms was formulated and proved by Stojanoska \cite{Sto1,Sto2} in terms of the Anderson self-duality of $\Tmf$ as follows.

\begin{fact}[\cite{Sto1, Sto2}]\label{fact_duality_Tmf}
    We have a $\Tmf$-module equivalence
    \begin{align}
        \alpha_{\Tmf} \colon \Tmf \simeq \Sigma^{-21} I_\bZ \Tmf. 
    \end{align}
    This equivalence is given by the multiplication by a generator of
    \begin{align}
        \pi_{21}I_\Z \Tmf \simeq \Hom(\pi_{-21}\Tmf, \Z) \simeq \Z, 
    \end{align}
    where we use $\pi_{-21}\Tmf \simeq \Z$ and $\pi_{-22}\Tmf =0$. 
We  also denote this generator by $\alpha_\Tmf$.
\end{fact}

\begin{rem}
$\alpha_\Tmf$ is determined only up to a sign at this point. 
This ambiguity will be fixed along the way.
\end{rem}

We are going to relate this self-duality of $\Tmf$ with our secondary morphism $\alpha_{\KO((q))/\TMF}$. 
% For the techniques used below, the reader is recommended to consult \cite{BehrensSurveyTMF,TMFBook}.
The following lemma follows from the construction of $\Tmf$ in \cite{HillLawson}:
\begin{lem}\label{lem_Tmf_TMF}
    We have a homotopy pullback square
    \begin{align}\label{diag_Tmf_TMF}
        \vcenter{\xymatrix{
\Tmf \ar[r] \ar[d]^-{\sigma'} & \TMF \ar[d]^{\sigma} \\
\KO[[q]] \ar[r] & \KO((q))
        }}.
    \end{align}
\end{lem}

\begin{proof}
We follow \cite{HillLawson}. 
Let $\mathcal{M}_{\mathrm{ell}}$ be the moduli stack of smooth elliptic curves and  $\overline{\mathcal{M}}_{\mathrm{ell}}$ be its compactification which classifies elliptic curves with possible nodal singularities.
The formal neighborhood of the complement of $\mathcal{M}_{\mathrm{ell}}$ in $\overline{\mathcal{M}}_{\mathrm{ell}}$ is  the moduli 
stack of the Tate curves ${\mathcal{M}}_{\mathrm{Tate}} \simeq \left[\mathrm{Spf}(\Z[[q]])//\Z/2\right]$, see \cite[Subsection 3.5]{HillLawson}. 
We have presheaves $\mathcal{O}^{\mathrm{smooth}}$ and $\mathcal{O}^{\mathrm{Tate}}$ of elliptic spectra on $\mathcal{M}_{\mathrm{ell}}$ and ${\mathcal{M}}_{\mathrm{Tate}}$, respectively. 
The structure presheaf $\mathcal{O}$ on $\overline{\mathcal{M}}_{\mathrm{ell}}$ is by definition the homotopy pullback of the presheaves $\mathcal{O}^{\mathrm{smooth}}$ and $\mathcal{O}^{\mathrm{Tate}}$ over $\Delta^{-1}\mathcal{O}^{\mathrm{Tate}}$, see \cite[Section 5]{HillLawson}. 
We have $\Tmf := \Gamma(\overline{\mathcal{M}}_{\mathrm{ell}}; \mathcal{O})$, $\TMF := \Gamma(\mathcal{M}_{\mathrm{ell}}; \mathcal{O}^{\mathrm{smooth}})$ and $\KO[[q]] \simeq \Gamma({\mathcal{M}}_{\mathrm{Tate}}; \mathcal{O}^{\mathrm{Tate}})$, and the diagram \eqref{diag_Tmf_TMF} is the homotopy pullback square induced on the global sections. 
\end{proof}

Denoting the cofiber of $\KO[[q]] \to \KO((q))$ by $\KO((q))/\KO[[q]]$, Lemma \ref{lem_Tmf_TMF} implies that we have a fiber sequence
\begin{align}\label{eq_Tmf_TMF_fiberseq}
    \Sigma^{-1}\left(\KO((q))/\KO[[q]]\right) \to \Tmf \to \TMF. 
\end{align}
Note that $\KO((q))/\KO[[q]]$ is naturally equipped with a $\KO[[q]]$-module structure. 

\begin{lem}\label{lem_duality_KO[[q]]}
    Let $\alpha_{\KO((q))/\KO[[q]]} \in \pi_{20} \left( \KO((q))/\KO[[q]]\right) \simeq \Hom\left(\pi_{20}\KO((q))/\pi_{20}\KO[[q]], \Z\right) $ be the element specified as follows:
    \begin{equation}
    \label{eq_duality_KO[[q]]_formula}
    \begin{aligned}
       \alpha_{\KO((q))/\KO[[q]]} \colon \pi_{20}\KO((q))/\pi_{20}\KO[[q]] \simeq %q^{-1}\Z[q^{-1}]\simeq 
       {2\Z((q))/2\Z[[q]]} &\to \Z, \\
       \phi &\mapsto {\frac{1}{2}} \Delta \phi |_{q^0}. 
    \end{aligned}
\end{equation}
    Let us also denote by 
    \begin{align}\label{eq_duality_KO[[q]]}
        \alpha_{\KO((q))/\KO[[q]]} \colon \KO[[q]] \to \Sigma^{-20}I_\Z \left(\KO((q))/\KO[[q]]\right)
    \end{align}
    the $\KO[[q]]$-module morphism given by the multiplication by the above element. 
    This \eqref{eq_duality_KO[[q]]} is an equivalence of $\KO[[q]]$-modules. 
\end{lem}
\begin{proof}
    This follows by the Anderson self-duality of $\KO$ and the Bott periodicity. 
\end{proof}

\begin{lem}\label{penultimate_step}
	We can choose the so far undetermined sign of $\alpha_{\Tmf}$ so that 
    the following diagram commutes:
    \begin{align}\label{diag_duality_relation}
        \vcenter{\xymatrix{
         \Tmf \ar[r] \ar[d]_-{\simeq }^-{\alpha_{\Tmf}} & \KO[[q]] \ar[d]_-{\simeq}^-{\alpha_{\KO((q))/\KO[[q]]}} \ar[r] & \KO((q))/\TMF \ar[d]^-{\alpha_{\KO((q))/\TMF}}  \\
         \Sigma^{-21}I_\Z \Tmf \ar[r] & \Sigma^{-20}I_\Z \left(\KO((q))/\KO[[q]]\right) \ar[r] &\Sigma^{-20}I_\Z \TMF 
        }}.
    \end{align}
\end{lem}
\begin{proof}
    First we show the commutativity of the  square on the left of \eqref{diag_duality_relation}. 
    Recall that the vertical arrows are $\Tmf$ and $\KO[[q]]$-module homomorphisms induced by the elements \begin{equation}
    \alpha_{\Tmf} \in \pi_{21} I_\Z \Tmf \simeq \Hom(\pi_{-21}\Tmf, \Z)
    \end{equation}
    and \begin{equation}
    \alpha_{\KO((q))/\KO[[q]]} \in \pi_{20} I_\Z \left(\KO((q))/\KO[[q]]\right) \simeq \Hom(\pi_{-20}\left(\KO((q))/\KO[[q]] \right), \Z),\end{equation} respectively. By the explicit formulas for those homomorphisms, it is enough to show that the following diagram can be made to commute by choosing the sign of $\alpha_{\Tmf}$: 
    \begin{align}\label{diag_duality_compatibility_Tmf_KO}
        \vcenter{\xymatrix{
        \pi_{-20} \left( \KO((q))/\KO[[q]]\right) \ar[rd]_-{{\frac{1}{2}}\Delta \cdot - |_{q^{0}}} \ar[r]^-{\del} & \pi_{-21} \Tmf \simeq \Z \ar[d]^-{\alpha_{\Tmf}} \\
        & \Z
        }},
    \end{align}
    where the horizontal arrow is a part of the exact sequence
    \begin{align}
        \pi_{-20}\TMF \to \pi_{-20}\KO((q))/\KO[[q]] \xrightarrow{\del} \pi_{-21}\Tmf \to \pi_{-21} \TMF =0. 
    \end{align}
    Recall that  $\Delta\phi|_{q^0} = 0 $ for $\phi \in \MF_{-10}$ from Fact~\ref{fact:mf}. 
    Therefore  the diagonal arrow in \eqref{diag_duality_compatibility_Tmf_KO} defines a map from $\mathrm{coker}\left(\pi_{-20}\TMF \to \pi_{-20}\KO((q))/\KO[[q]] \right) \simeq \pi_{-21}\Tmf$. 
    Then, by the fact that $\pi_{-21}\Tmf \simeq \Z$ and that the diagonal map is a surjection, the induced map should be a generator in $\Hom(\pi_{-21}\Tmf, \Z)$, which indeed characterizes the element $\alpha_{\Tmf}$ as desired.

    Next we show the commutativity of the  square on the right of \eqref{diag_duality_relation}. 
    But this follows directly from the agreement of the formulas for $\alpha_{\KO((q))/\TMF}$ in Proposition \ref{rem:rat} and for $\alpha_{\KO((q))/\KO[[q]]}$ in Lemma \ref{lem_duality_KO[[q]]}. This completes the proof. 
\end{proof}

We can now give the proof to our fundamental theorem:
\begin{thm}[=Theorem~\ref{main}]
 The secondary morphism 
    \begin{align}%\label{eq_main_thm}
        \alpha_{\KO((q))/\TMF} \colon \KO((q))/\TMF \to \Sigma^{-20}I_\Z \TMF
    \end{align}
    is an equivalence. 
\end{thm}
\begin{proof}
We examine the commutative diagram given in Lemma \ref{penultimate_step}.
    The upper row is a fiber sequence by Lemma \ref{lem_Tmf_TMF} and the lower row is the Anderson dual of the fiber sequence \eqref{eq_Tmf_TMF_fiberseq}, and in particular is also a fiber sequence. 
    The vertical arrow on the left is an equivalence by Fact \ref{fact_duality_Tmf},
    and  the middle vertical arrow is also an equivalence by  Lemma \ref{lem_duality_KO[[q]]}.
    Therefore the  vertical arrow on the right is also an equivalence.
\end{proof}

% !TEX root = paper.tex
\section{The structure of $\pi_\bullet\TMF$ and its Anderson duality}
\label{app:TMF}

Here we provide a summary of the structure of $\pi_\bullet\TMF$ as  Abelian groups,
which we mainly take from \cite{TMFBook,BrunerRognes}.
We also discuss how the Anderson duality affects them.

\subsection{$\pi_\bullet \TMF$ as Abelian groups}
We start with a definition.

\begin{defn}
\label{defn:A}
We let $A_d$, $U_d$  be the kernel and the image of $\sigma:\pi_d\TMF \to \pi_d\KO((q))$, so that 
we have the short exact sequence \begin{equation}
\label{short} 0 \to A_d \to \pi_d\TMF\to U_d\to 0.
\end{equation}
\end{defn}
The image $U_d$ of $\sigma$ can be nonzero only when $d=4k$, $d=8k+1$ or $d=8k+2$.
Let us first discuss the image when $d=4k$.
We first recall that the composition $\pi_{4k}\TMF\to\pi_{4k}\KO((q))\to \pi_{4k}\K((q))$ is factored by
$\pi_{4k}\TMF\to \MF_{2k}$, where $\MF_{2k}\to \pi_{4k}\K((q))\simeq \bZ((q))$ 
is given by the $q$-expansion.
Then we have:
\begin{fact}[= \cite{Hopkins2002}, Proposition 4.6]
\label{fact:imageZ}
The image of the ring homomorphism $\sigma:\pi_{4\bullet}\TMF\to \MF_{2\bullet}$ 
has a $\bZ$-basis given by \begin{equation}
a_{i,j,k} c_4^i c_6^j \Delta^k, \qquad \text{$i\ge 0$; $j=0,1$; $k\in\bZ$}
\end{equation}where \begin{equation}
a_{i,j,k} = \begin{cases}
24/\gcd(24,k) & \text{if $i=j=0$},\\
2 & \text{if $j=1$},\\
1 & \text{otherwise}.
\end{cases}\label{hop}
\end{equation}
\end{fact}

\begin{rem}
The factor $2$ when $j=1$ in \eqref{hop} is due to the multiplication by 2 in $\pi_{8k+4}\KO\simeq \bZ \to \pi_{8k+4} \K\simeq \bZ$ and is not very surprising. 
The nontrivial information is therefore mainly in the $i=j=0$ case.
\end{rem}

Let us next discuss the images of $\sigma$ when $d=8k+1$ or $d=8k+2$. 
In \cite{TYY} the following statement was proved, which was an easy corollary of the information contained in \cite{BrunerRognes}:
\begin{fact}[=\cite{TYY}, Sec.~4.1]
\label{fact:imageZ/2}
The image of $\sigma:\pi_{d}\TMF\to \pi_{d}\KO((q))$ when $d=8k+1$ and $8k+2$
is contained in $\eta \MF_{4k}$ and $\eta^2 \MF_{4k}$, respectively,
where $\eta\in \pi_{1}\KO$ is the generator. 

Furthermore, the cokernel of $\sigma$ as a homomorphism into $\eta \MF_{4k}$ and $\eta^2\MF_{4k}$ 
is spanned by $[\eta\Delta^{8t+i}]$ with $i=2,3,5,6,7$ 
and $[\eta^2\Delta^{8t+j}]$ with $j=3,6,7$.
\end{fact}

\begin{rem}
\label{rem:KO2}
Complimentarily, $\eta\Delta^{8t+\hat\imath}$ is hit by $\sigma$ if and only if $\hat\imath=0,1,4$
and $\eta^2\Delta^{8t+\hat\jmath}$ is hit if and only if $\hat\jmath=0,1,2,4,5$,
exhibiting the duality between the cokernel and the kernel of $\sigma$ under $i+\hat\jmath = \hat\imath + j =7$.
This is a manifestation of the Anderson duality given below.
\end{rem}

Let us move on to the description of $A_d$:

\begin{prop}
\label{A=torsion}
$A_d$ are torsion.
\end{prop}
\begin{proof}
All non-torsion elements of $\pi_\bullet\TMF$ have a nonzero image in $\MF_{\bullet/2}$.
\end{proof}

%We first have the basic result:
\begin{prop}
\label{kernel=Btorsion}
The subgroup $A_d=\Ker(\sigma)\subset \pi_\bullet\TMF $ agrees with 
the subgroup $\Gamma_B \pi_\bullet\TMF$ of $B$-power-torsion elements.
\end{prop}
\begin{proof}
Let $x$ be annihilated by $B^k$.
Then $\sigma(B^kx)=\sigma(B)^k \sigma(x)=0$. As $\sigma(B)=c_4$ is invertible, we find $\sigma(x)=0$.
Therefore $\Gamma_B \pi_\bullet\TMF \subset \Ker\sigma$.
The structure of $\pi_\bullet\TMF/\Gamma_B\pi_\bullet\TMF$
was explicitly given for $p=2$ and $p=3$ in Theorem 9.26  and  Theorem 13.18 of \cite{BrunerRognes}.
By inspection, all the generators there can be seen to map to nonzero elements in $\pi_\bullet\KO((q))$, completing the proof.
\end{proof}

\begin{prop}
\label{fact:split}
The short exact sequence \eqref{short} splits.
\end{prop}
\begin{proof}
As $A_d$ are all torsion, the short exact sequence \eqref{short} can only possibly be non-split
when $\pi_d\KO((q))$ is torsion, i.e.~when $d=8k+1$ or $d=8k+2$.
In these degrees all nonzero elements of $\pi_d\KO((q))$ have order 2,
so what matters is the $p=2$ case.
That the sequence splits at this prime is \cite[Theorem~9.26]{BrunerRognes}.
\end{proof}

\begin{table}
\[
\hskip-1em\begin{array}{|c|ccccccccccccc|}
\hline
d & -11 & -10 & -9 & -8 & -7 & -6 & -5 & -4 & -3 & -2 & -1 & 0 & \\
A_d & & & & & & & & & & & & & \\
A_d & & & & & & & & & & & \cellcolor{xgray}& & \\
d & -11 & -12 & -13 & -14 & -15 & -16 & -17 & -18 & -19 & -20 & \cellcolor{xgray}-21 & -22 & \\
\hline
d & 1 & 2 & \cellcolor{xgray} 3 & 4 & 5 & 6 & 7 & 8 & 9 & 10 & 11 & 12 & \\
A_d & & & \cellcolor{xgray} 8& & &2 & & 2& 2& & & & \\
A_d & & &  & & &2 & & 2 & 2& & & & \\
d & -23 & -24 & -25 & -26 & -27 & -28 & -29 & -30 & -31 & -32 & -33 & -34 & \\
\hline
d & 13 & 14 & 15 & 16 & 17 & 18 & 19 & 20 & 21 & 22 & 23 & 24 & \\
A_d & & 2& 2& & 2  & & & 8 & 2& 2 & & & \\
A_d & & 2& 2& & 2& & & 8 & 2& 2& \cellcolor{xgray}8& & \\
d & -35 & -36 & -37 & -38 & -39 & -40 & -41 & -42 & -43 & -44 &\cellcolor{xgray} -45 & -46 & \\
\hline
d & 25 & 26 &\cellcolor{xgray} 27 & 28 & 29 & 30 & 31 & 32 & 33 & 34 & 35 & 36 & \\
A_d & & & \cellcolor{xgray}4& 2& & & & 2&2 & 2& 2 & & \\
A_d & & & & 2& & & & 2&2 & 2& 2 & & \\
d & -47 & -48 & -49 & -50 & -51 & -52 & -53 & -54 & -55 & -56 & -57 & -58 & \\
\hline
d & 37 & 38 & 39 & 40 & 41 & 42 & 43 & 44 & 45 & 46 & 47 & 48 & \\
A_d & & & 2&4 &2 &2 & & &2 & 2 & & & \\
A_d & & & 2&4 &2 &2 & & &2 & 2& \cellcolor{xgray}4& & \\
d & -59 & -60 & -61 & -62 & -63 & -64 & -65 & -66 & -67 & -68 & \cellcolor{xgray}-69 & -70 & \\
\hline
d & 49 & 50 & \cellcolor{xgray}51 & 52 & 53 & 54 & 55 & 56 & 57 & 58 & 59 & 60 & \\
A_d & & & \cellcolor{xgray}8& 2& 2& 4& & &2 & &2 & 4 & \\
A_d & & & & 2& 2& 4& & &2 & &2 & 4 & \\
d & -71 & -72 & -73 & -74 & -75 & -76 & -77 & -78 & -79 & -80 & -81 & -82 & \\
\hline
d & 61 & 62 & 63 & 64 & 65 & 66 & 67 & 68 & 69 & 70 & 71 & 72 & \\
A_d & & & & &2\times2 &2 & & 2 & &2 & & & \\
A_d & & & & &2\times2 &2 & & 2 & &2 & \cellcolor{xgray}8 & & \\
d & -83 & -84 & -85 & -86 & -87 & -88 & -89 & -90 & -91 & -92 & \cellcolor{xgray}-93 & -94 & \\
\hline
d & 73 & 74 &\cellcolor{xgray} 75 & 76 & 77 & 78 & 79 & 80 & 81 & 82 & 83 & 84  & 85\\
A_d & & & \cellcolor{xgray}2& & & & & 2 & & & & &2\\
A_d & & & & & & & &2  & & & & &2\\
d & -95 & -96 & -97 & -98 & -99 & -100 & -101 & -102 & -103 & -104 & -105 & -106  & -107\\ 
\hline
\end{array}
\]
\caption{2-primary components of $A_d=\Ker \sigma = \Gamma_B \pi_d \TMF$.  
It is 192-periodic and $A_d$ and $A_{-d-22}$ are Pontryagin dual
if $d\neq 24k+3$ and $24k -1$. 
We abbreviate $\bZ/n$ simply as $n$, and $\bZ/2 \times \bZ/2$ as $2\times2$.
The blank entries denote trivial groups.
The entries for $d=24k+3$ are shaded; the entries for $d=24k-1$ are always empty.
\label{tab:p=2}}
\end{table}

\begin{table}
\[
\begin{array}{|c|ccccccccccccc|}
\hline
d & -11 & -10 & -9 & -8 & -7 & -6 & -5 & -4 & -3 & -2 & -1 & 0 & \\
A_d & & & & & & & & & & & & & \\
A_d & & & & & & & & & & & \cellcolor{xgray}& & \\
d & -11 & -12 & -13 & -14 & -15 & -16 & -17 & -18 & -19 & -20 & \cellcolor{xgray} -21 & -22 & \\
\hline
d & 1 & 2 & \cellcolor{xgray}3 & 4 & 5 & 6 & 7 & 8 & 9 & 10 & 11 & 12 & \\
A_d & & & \cellcolor{xgray}3& & & & & & & 3& & & \\
A_d & & &  & & & & &  & & 3 & & & \\
d & -23 & -24 & -25 & -26 & -27 & -28 & -29 & -30 & -31 & -32 & -33 & -34 & \\
\hline
d & 13 & 14 & 15 & 16 & 17 & 18 & 19 & 20 & 21 & 22 & 23 & 24 & 25\\
A_d & 3& & & &   & & &3   & &  & & & \\
A_d & 3& & & & & & &  3&  & &\cellcolor{xgray} 3& & \\
d & -35 & -36 & -37 & -38 & -39 & -40 & -41 & -42 & -43 & -44 & \cellcolor{xgray}-45 & -46 & -47\\
\hline
\end{array}
\]
\caption{3-primary components of $A_d=\Ker \sigma = \Gamma_B \pi_d \TMF$.  
It is 72-periodic and $A_d$ and $A_{-d-22}$ are Pontryagin dual
unless $d=24k+3$ or $d=24k-1$. 
We abbreviate $\bZ/n$ simply as $n$.
The blank entries denote trivial groups.
The entries for $d=24k+3$ are shaded;
the entries for $d=24k-1$ are always empty.
\label{tab:p=3}}
\end{table}

The combination of Definition \ref{defn:A}, the short exact sequence \eqref{short},
Facts \ref{fact:imageZ} and \ref{fact:imageZ/2}
and Proposition \ref{fact:split} determines $\pi_d\TMF$ as abstract Abelian groups,
once $A_d$ is given.
We tabulate this important group $A_d$ in Table~\ref{tab:p=2} and Table~\ref{tab:p=3},
for $p=2$ and $p=3$, respectively,
taking information from  Figure 10.1 (for $p=2$) and Lemma 13.17 (for $p=3$) of \cite{BrunerRognes}.
The 2-primary and 3-primary components of $A_d$ have periodicity $192$ and $72$ respectively,
and we arranged them to make its duality manifest.
For more details including the generators and the structure as a graded ring,
the readers should consult the original descriptions in \cite{BrunerRognes}.

\subsection{Pairings on $\pi_\bullet\TMF$ induced by the Anderson duality}
\label{app:duality}
Let us start by recalling the case of the Anderson self-duality of $\KO$-theory, which has the form $\zeta: \KO\to \Sigma^{4}I_\bZ \KO$.
In this case the $\bZ$-valued pairing  is \begin{equation}
\pi_d \KO \times \pi_{4-d}\KO \to \bZ
\end{equation} and the $\bQ/\bZ$-valued pairing  is \begin{equation}
(\pi_d \KO)_\text{torsion} \times (\pi_{3-d} \KO)_\text{torsion}\to \Q/\bZ.
\end{equation}
Using the 8-periodicity of $\KO$, we see that these provide perfect pairings \begin{equation}
{\pi_0\KO}\times {\pi_4 \KO}  \to \bZ
\end{equation} and \begin{equation}
{\pi_1\KO} \times {\pi_2 \KO} \to \bZ/2
\end{equation} on the nonzero entries of $\pi_\bullet \KO$, which are \begin{equation}
\pi_0\KO\simeq \bZ,\quad
\pi_1\KO\simeq \bZ/2,\quad
\pi_2\KO\simeq \bZ/2,\quad
\pi_4\KO\simeq {2}\bZ.
\end{equation}

In the case of the Anderson duality of $\TMF$, the situation is more intricate,
because we have a pairing between $\pi_\bullet \TMF$ and $\pi_\bullet\KO((q))/\TMF$,
where the latter is the $X$-adic completion of $\pi_\bullet(\TMF/B^\infty)$,
with $X$ being the periodicity element in $\pi_{576}\TMF$.
We decompose these homotopy groups as follows.

First, for $\pi_\bullet(\TMF)$, we have
\begin{equation}
\vcenter{\xymatrix{
0\ar[r] & A_\bullet \ar[r] %\ar@{=}[d] 
& \pi_\bullet(\TMF) \ar[r] %\ar[d] 
&  U_{\bullet} \ar[r]%\ar[d]
 & 0
%0\ar[r] & A_\bullet \ar[r]  & \pi_\bullet(\overline{\TMF}) \ar[r]  &  \overline{U_{\bullet}} \ar[r] & 0
}}, \label{D17}
\end{equation}
where \begin{equation}
A_\bullet =\mathrm{ker}\ \phi = \mathrm{ker}\ \sigma,\qquad
U_\bullet  =\mathrm{image}\ \phi = \mathrm{image}\ \sigma
\end{equation}
for  the morphisms $\phi$ and $\sigma$ 
\begin{equation}
\phi : \pi_\bullet\TMF \to (\pi_\bullet\TMF)[B^{-1}], \qquad
\sigma : \pi_\bullet\TMF \to \pi\KO((q)),
\end{equation}
as discussed in Proposition~\ref{kernel=Btorsion}.
The sequence \eqref{D17} always splits according to Proposition~\ref{fact:split}.

Second, for $\pi_\bullet\TMF/B^\infty$ and $\pi_\bullet \KO((q))/\TMF$, we have
\begin{equation}
\vcenter{\xymatrix{
0\ar[r] & C_\bullet \ar[r] \ar[d] & \pi_\bullet(\TMF/B^\infty) \ar[r] \ar[d] &  A_{\bullet-1} \ar[r]\ar@{=}[d] & 0\\
0\ar[r] & \overline{C_\bullet} \ar[r]  & \pi_\bullet(\KO((q))/\TMF) \ar[r]  &  A_{\bullet-1} \ar[r] & 0
}},
\label{9999}
\end{equation}
where 
\begin{equation}
C_\bullet = \mathrm{coker}\ \phi, \qquad 
\overline{C_\bullet} = \mathrm{coker}\ \sigma
\end{equation}
as discussed in Proposition~\ref{prop_completion_TMF/B}.
In \eqref{9999},  the second row is the $X$-adic completion of the first row, as we saw  in %Appendix~\ref{app:original-proof}.
Lemma~\ref{lem_TMF_completion} and Proposition~\ref{prop_completion_TMF/B}.
We also recall that these two sequences in \eqref{9999} split unless $\bullet\equiv 4$ mod $24$, as we saw in Sec.~\ref{subsubsec_BM_tmf}.
%The Abelian groups there are  and

%The overlines are the completion with respect to the topologies of $\pi_\bullet\overline\TMF$ and $\pi_\bullet(\KO((q))/\TMF)$ given in Appendix~\ref{app:original-proof}.

More explicitly, the Abelian groups appearing above are all $\pi_0\TMF=\bZ[J]$-modules
where $J$ is the unique inverse image of  $c_4^3/\Delta$. Concretely, as $\bZ[J]$-modules, \begin{itemize}
\item $A_d$ were given in Tables~\ref{tab:p=2} and \ref{tab:p=3},
where $J$ acts trivially.
\item For $d=4\ell$,
\begin{itemize}
\item $U_d$ is of the form $a_d \bZ[J] + J \bZ[J] \subset \bZ[J]$  for an integer $a_d$.
This integer $a_d$ is $1$ unless $d=24k$, for which $a_{d}=24/\gcd(24,k)$. This follows from \eqref{hop} in Proposition~\ref{fact:imageZ}.
%\item $\overline{U_d}$ is its $J$-adic completion, and of the form $a_d \bZ[[J]] + J \bZ[[J]] \subset \bZ[[J]]$.
\item  $C_d$ is $\bZ[J,J^{-1}]/U_d$ and $\overline{C_d}$ is  its completion in $\pi_d(\KO((q))/\TMF)$.
\end{itemize}
\item For $d=8\ell+1$ or $8\ell+2$,
\begin{itemize}
\item $U_d=\bZ/2[J]$, %and $\overline {U_d}=\bZ/2[[J]]$, 
\item  $C_d=\bZ/2[J,J^{-1}]/ U_d$ and $\overline{C_d}$ is its completion in $\pi_d(\KO((q))/\TMF)$.
\end{itemize}
\end{itemize}

This leads to the following four-way decomposition of the Anderson duality pairings, as was found by \cite{BrunerRognes} for $\tmf$.
The perfectness of the pairing between $U_d$ and $C_{-d-20}$ below
is with respect to the $J$-adic topology on $U_d$
and the $J^{-1}$-adic topology on $C_{-d-20}$.

\begin{itemize}
\item First, there is a $\bZ$-valued pairing for $d\equiv 0$ mod $4$ \begin{equation}
U_d \times C_{-d-20} \to \bZ,\label{firstpairing}
\end{equation} 
which can be detected at the level of their images in $\pi_d\KO((q))$;
the rationalized expression of the pairing was given in Proposition~\ref{rem:rat}.
This $\Z$-pairing is 
\begin{itemize}
\item perfect when $d\not\equiv 0$ and $4$ mod $24$,
\item not perfect when $d\equiv 0$ mod $24$, in which case $C_{-d-20}$ is a nontrivial subgroup of $\pi_{-d-20}(\TMF/B^\infty)$ and the sequence \eqref{9999} does not split, and
\item not  perfect when $d\equiv 4$ mod $24$, in which case $C_{-d-20}$ contains torsion,
because of the factor $24/\gcd(24,k)$ in the $i=j=0$ case of \eqref{hop}.
\end{itemize}
In the last two cases, the perfectness is restored by taking into account the third pairing \eqref{thirdpairing} given below.
\item Second, there is the perfect $\bZ/2$-valued pairing for $d\equiv 1$ and $2$ mod $8$, \begin{equation}
U_d \times C_{-d-21} \to \bZ/2,
\end{equation} 
see also Remark~\ref{rem:KO2}.
%which can be captured by their images in $\pi_d \KO((q))$ and $\pi_{-d-21} \KO((q))//\sigma(\pi_{-d-21}\TMF)$.
This pairing can also be detected by their images in $\pi_d\KO((q))$.
\item Third, there is a perfect pairing \begin{equation}
A_d  \times (C_{-d-21})_\text{torsion} \to \Q/\Z\label{thirdpairing}
\end{equation}
when $d\equiv 3$ mod $24$,
associated to the fact that the homomorphism $\pi_{-d-21}\TMF \to \MF_{(-d-21)/2}$ has a nontrivial cokernel
precisely in these degrees.
This pairing also accounts for the non-perfectness of the first pairing \eqref{firstpairing} given above,
and is detected by the invariants of Bunke and Naumann.
\item Fourth and finally, there is a perfect pairing \begin{equation}
A_d \times A_{-d-22}  \to \Q/\Z
\end{equation}
when $d\not\equiv 3,-1$ mod $24$. This pairing is tabulated in Tables~\ref{tab:p=2} and~\ref{tab:p=3}.
We note that the method to determine the  pairings between concrete elements from the information in \cite{BrunerRognes} is outlined in Remark~\ref{rem:AApairing}.
\end{itemize}

% !TEX root = paper.tex
\section{On equivariant $\TMF$}
\label{app:equivariant}

Here we review the theory of equivariant topological modular forms to 
a minimal degree necessary for the purposes of this paper,
and prove two theorems concerning the equivariant lift of our favorite element $e_8\in \pi_{-16}\TMF$.
We start with the generalities. 

\subsection{Generalities}
The content of this subsection is reviewd in more detail in \cite[Section~2]{lin2024topologicalellipticgenerai}, so we refer there and briefly sketch the necessary part here. We always assume a group $G$ to be a compact Lie group. 
Genuine $G$-equivariant topological modular forms were constructed in \cite{GepnerMeier},
although only the non-twisted version was treated there.
It is widely believed that the $G$-equivariant topological modular forms can be twisted by 
any map $\tau:BG\to BO\langle 0,\ldots,4\rangle$,
but the corresponding general theory of twisted equivariant topological modular forms is not yet available.
That said,
the $\RO(G)$-graded theory already follows by the construction in \cite{GepnerMeier},
which provides us with a way of dealing with subset of twists, coming from (virtual) orthogonal representations of $G$.
Let us work in the $\RO(G)$-graded language.

Let us first fix the notations. 
Given a genuine $G$-equivariant spectrum $E$ and $V \in \RO(G)$ of $G$,  
we set $E[V]:=E\otimes S^V$. 
For $V$ an orthogonal {\it representation} (not merely a virtual representation), we have an associated Euler class $\chi(V): E\to E\otimes S^V$. We also regard it as an element $\chi(V) \in \pi_0 E[V]$. We use the notation
\begin{align}
	\overline V := V - \underline \R^{\dim V} \in \mathrm{RO}(G).
\end{align}
We denote the genuine $G$-fixed point spectrum of $E$ by $E^G$. 
Also, for $G$ one of $U(n), \SU(n)$ and $\Spin(n)$, we denote by $V_G$ its vector representation (i.e., the orthogonal representation pulled back from the defining representations of $\SO(2n)$ or $\SO(n)$). 
For any compact Lie group $G$, we denote by $\mathrm{Ad}(G)$ its adjoint representation.  

\paragraph{Equivariant $\TMF$ and character maps}
Genuine equivariant $\TMF$ is constructed by making sense of the `moduli space of elliptic curves with flat $G$-connections in spectral algebraic geometry', constructing the quasi-coherent sheaves over it and taking the global sections. 
In particular, there is a canonical map functorially for any $V \in \RO(G)$ given by the base change to $\bC$, 
\begin{align}\label{eq_evC}
	\mathrm{red}_\bC \colon \pi_\bullet \TMF[V]^{G} \to \Gamma(\moduli_\bC^G; \cL_\bC(-V) \otimes p^*\omega^{\bullet/2}). 
\end{align}
where $\moduli_\bC^G$ is the moduli stack of complex-analytic elliptic curves with flat $G$-bundles, $\cL_\bC(-V) \in \Pic(\moduli_\bC^G)$ is the line bundle associated to $V\in \RO(G)$. 

Furthermore, if $G$ is connected and $\pi_1 G$ is torsion-free, choosing a maximal torus $T \subset G$ with the Weyl group $W$, we have a canonical identification
\begin{align}\label{eq_moduliG=moduliT/W}
	\moduli_\bC^G \simeq \moduli_\bC^{T}/W  , \quad  \moduli_\bC^{T} \simeq \mathcal{E}_\bC \times_\Z \Hom(S^1, T)
\end{align}
where $\cE_\bC$ is the total space of the universal elliptic fibration on $\moduli_\bC$. This allows us to identify sections of sheaves over $\moduli^G_\bC$ in terms of those over $\moduli_\bC^{T}$.
For those $G$, we can make sense of the $\Z$-graded $\MF$-module of {\it integral $V$-twisted $G$-equivariant modular forms} denoted by $\MF[V]^G$ (see \cite[Section~2.2.2]{lin2024topologicalellipticgenerai}. Importantly, it is {\it canonically} associated to the homotopy class $[V] \in \Z \times [BG, BO\langle 0,1, \cdots, 4 \rangle]$). We have
\begin{align}
	\MF[k \overline V_{U(1)}]^{U(1)}|_{\deg = \bullet} = \JF^{\ind = \frac{k}{2}}_{\wt = \bullet/2}, 
\end{align}
where $\JF$ is the ring of integral, weakly-holomorphic Jacobi forms,
and $\ind$ and $\wt$ specify the index and the weight, respectively.
More generally, $\MF[V]^T$ for compact connected abelian Lie group $T$ is defined in terms of {\it multi-variable} integral weakly holomorphic Jacobi forms. 
For $G$ connected and $\pi_1 G$ torsion-free, we define
\begin{align}
	\MF[V]^G := \MF[\res_G^T V]^W
\end{align}
where $W$ and $T$ are as before. 
The map \eqref{eq_evC} factors through integral $G$-equivariant modular forms as
\begin{align}\label{eq_character}
	\mathrm{red}_\bC \colon \pi_\bullet \TMF[V]^{G} \xrightarrow{e_G} \left( \MF[V]^G\right)  |_{\deg = \bullet} \hookrightarrow
	\Gamma(\moduli_\bC^G; \cL_\bC(-V) \otimes p^*\omega^{\bullet/2}). 
\end{align}
We call the first map $e_G$ as the {\it $G$-equivariant character map}. 

The character of the equivariant Euler class $\chi(V) \in \pi_0 \TMF[V]^G$ for the orthogonal representation $V_G$ for $G=SO(2r)$  is given by the {\it theta function}, 
\begin{align}\label{eq_Phi_V=chiV}
	e_G(\chi(V_G)) = \Theta_V = \eta(q)^{-3r} \prod_{i=1}^r \theta_{11}(\tau;z_i),
\end{align}
where, in the last term, we have identified the maximal torus $ T \subset G$ with $U(1)^r$.
Our convention for the theta function is \begin{equation}
\theta_{11}(\tau;z)=q^{1/8}(y^{1/2}-y^{-1/2})\prod_{n\ge 1}(1-q^ny)(1-q^ny^{-1})(1-q^n),
\end{equation}
where $q=e^{2\pi i\tau}$ and $y=e^{2\pi i z}$.

\begin{rem}
In the conjectural framework in Sec.~\ref{sec:vosa} of the correspondence between $\TMF$ classes and VOSAs,
$\chi(V_{SO(2r)})$ corresponds to the VOSA $\V(2r)$  discussed there,
and $e_G(\chi(V_G))$ above agrees with the character of $\eta(q)^{-2r}[\W(2r)]$,
which we discussed in Conjecture~\ref{conj:affinevoa}.
\end{rem}

An important case for us is when $G=\SU(2)$. Using the standard maximal torus $U(1) \subset \SU(2)$ with Weyl group $\Z/2$, the $\SU(2)$-equivariant modular forms are identified with Jacobi forms which are {\it even} in the elliptic coordinate $z$. We have
\begin{align}
	\MF[k\overline V_{\SU(2)}]^{\SU(2)}|_{\deg = m} = \begin{cases}
		\JF^{\ind = k}_{\wt =m/2}  & m \equiv 0 \pmod 4, \\
		0 & \mbox{otherwise}. 
	\end{cases}. 
\end{align}
So we may regard the character map as a map
\begin{align}\label{eq_char_SU2}
	e_{\SU(2)} \colon \pi_{2m} \TMF[k\overline V_{\SU(2)}]^{\SU(2)} \to \JF^{\ind=k}_{\wt = m}. 
\end{align}

We will also need a few facts concerning  $\JF$. For more details, see \cite{Gritsenko2020382}. 

\begin{prop}
	\label{prop:JF1}
	$\JF^{\ind=1}_\bQ= \phi_{0,1} \MF_\bQ \oplus \phi_{-2,1}\MF_\bQ$,  where \[
	\phi_{0,1}=y +y ^{-1}+10+O(q),\quad
	\phi_{-2,1}= y +y ^{-1}-2+O(q)
	\] are certain Jacobi forms of index $0$, weight $1$ and of index $-2$, weight $1$, respectively.
	Furthermore, the kernel of $\JF^{\ind=1}_\bQ\to \MF_\bQ$ is $\phi_{-2,1}\cdot \MF_\bQ$.
	Here and in the following, we follow the notation of \cite{Gritsenko2020382}, except we use $y$ instead of their $\zeta$.
\end{prop}

\begin{prop}[\cite{Gritsenko2020382}]
	\label{prop:JF2}
	$\JF^{\ind=1}$ is a $\bZ$-module within $\JF^{\ind=1}_\bQ$ generated by \[
	\phi_{0,1},\quad
	\phi_{-2,1},\quad
	E_{4,1}=\frac{c_4\phi_{0,1}-c_6\phi_{-2,1}}{12},\quad
	E_{6,1}=\frac{c_6\phi_{0,1}-c_4^2\phi_{-2,1}}{12}.
	\]
	Furthermore, the kernel of $\res_{U(1)}^e = \mathrm{ev}_{z = 0} \colon \JF^{\ind=1}\to \MF$ is $\phi_{-2,1}\cdot \MF$.
\end{prop}

Another important group for us is $\Spin(4) \simeq \SU(2) \times \SU(2)$, and in the notations below we put primes ($'$) to objects associated to the latter component of the product.
Let us focus on the case where the $\RO(G)$-grading is given by $V_{\Spin(4)} \in \RO(\Spin(4))$. 
The character map becomes
\begin{align}\label{eq_char}
	e_{\Spin(4)} \colon	\pi_{2m} \TMF[\overline V_{\Spin(4)}]^{\Spin(4)}\to 
	\mathrm{JJF}^{\ind=(1,1')}_{\wt=m},
\end{align}
where the domain is the two-variable integral Jacobi forms with level $(1,1)$,
i.e.~those weakly holomorphic formal power series $\phi(z, z', q)$ with integral coefficients, 
which transform as Jacobi forms with index $1$ for each variable $z$ and $z'$. 

\mathcalEdefinition{calE4app}
This is the function introduced in \eqref{calE4}.

\paragraph{The transfer}
An important feature of the genuine equivariant $\TMF$ is the existence of {\it equivariant transfer maps} along {\it any} homomorphism $f \colon G \to H$ of compact Lie groups. It is a map in the covariant direction, 
\begin{align}\label{eq_tr_f}
	\tr_f \colon	\TMF[-\mathrm{Ad}(G) + \res_G^H V]^G \to \TMF[-\mathrm{Ad}(H) + V]^H 
\end{align}
for $V \in \RO(H)$. 
This extends the transfer map along inclusions $G \hookrightarrow H$, which exists for any genuine $H$-equivariant spectra. The existence of this general transfer is a special feature of the genuine equivariant $\TMF$. 

\begin{rem}
As will be explained in an upcoming paper by the second author with Y.~Lin \cite{LinYamashitaNew} there is a nice physical interpretation in terms of {\it gauging}. 
In the language of the VOSA under the correspondence we discussed in Sec.~\ref{sec:vosa},
the transfer should correspond to {\it orbifolding}.
\end{rem}

From now on, let us focus on the situation where $K \subset G$ is a finite-index normal subgroup of $G$ and $f$ is the quotient homomorphism $f \colon G \to H=G/K$. 
Then we have $\res_G^H \mathrm{Ad}(H) \simeq \mathrm{Ad}(G)$ so that the transfer simplifies to 
\begin{align}\label{eq_tr_without_ad}
		\tr_f \colon	\TMF[ \res_G^H V]^G \to \TMF[ V]^H . 
\end{align}
Morally speaking, the map \eqref{eq_tr_without_ad} is given by considering the covering map
\begin{align}
f \colon	\moduli^G \to \moduli^H
\end{align} 
and taking the weighted sum of the sections of quasi-coherent sheaves along the fibers. 
By the base change to $\bC$, the transfer map is compatible with the corresponding operation on the equivariant modular forms. 
We have a commutative diagram, 
\begin{align}\label{diag_tr}
	\vcenter{\xymatrix{
	\TMF[\res_G^H V]^{G} \ar[r]^-{\mathrm{red}_\bC} \ar[d]^-{\tr_f}& \Gamma(\moduli_\bC^G; f^* \cL_\bC(-V) \otimes p^*\omega^{\bullet/2}) \ar[d]^-{\tr_f} \\
		\TMF[V]^{H} \ar[r]^-{\mathrm{red}_\bC}& \Gamma(\moduli_\bC^H; \cL_\bC(-V) \otimes p^*\omega^{\bullet/2})
	}}
\end{align}
where the right vertical arrow is that, for each complex point $q \colon \mathrm{Spec}(\bC) \to \moduli_\bC^H $, we have
\begin{align}
	(\tr_f \phi) (q) := \sum_{\tilde q \in f_\bC^{-1}(q) } \frac{\phi(\tilde q)}{|\Aut_{f_\bC^{-1}(q)}(\tilde q)|}. 
\end{align}
where $f_\bC \colon \moduli_\bC^G \to \moduli_\bC^H$ is associated to $f$.

\subsection{Equivariant lifts}
Equipped with these preliminaries,
we move on to the proofs of two results required in the main part of the paper.
Recall that  our class $e_8\in \pi_{-16}\TMF$ is a unique element in $\pi_{-16}\TMF$
whose modular form image is $c_4/\Delta$.
The goal of this section is to construct canonical $\SU(2)$ and $\Spin(4)$-equivariant lifts of the element $e_8$. 

First consider the group $\Spin(16)$. Recall that there is a central subgroup $\Z/2 \subset \Spin(16)$ 
such that the quotient is different from $\SO(16)$. The quotient is called the {\it Semispin group}
and we denote it by  $\Semispin(16)$. 
Explicitly, if we take the standard realization of the $\Spin(16)$ as a subgroup of the even part of the real Clifford algebra $\Cl_{16}$ with Clifford generators $\{e_i\}_{1 \le i \le 16}$, the central $\Z/2$ is realized by the central order-two element $\omega := e_1e_2 \cdots e_{16} \in \Spin(16)$. 
On the other hand, let us also consider the standard inclusion $\Spin(4)\subset Spin(16)$. Since $\omega \notin \Spin(4)$, the inclusion descends to the inclusion $i \colon \Spin(4) \hookrightarrow \Semispin(16) = \Spin(16)/\{1, \omega\}$. 

Define the subgroup $\widetilde{\Spin(4)} \stackrel{\tilde i}{\hookrightarrow} \Spin(16)$ by the following pullback, 
\begin{align}
	\vcenter{\xymatrix{
		\widetilde{\Spin(4)} \ar@{^{(}->}[r]^-{\tilde{i}} \ar[d]^-{q} & \Spin(16) \ar[d] \\
		\Spin(4) \ar@{^{(}->}[r]^-{i} & \Semispin(16)
	}}
	\label{spin4pull}
\end{align}
As a group, we have $\widetilde{\Spin(4)} \simeq \Spin(4) \times \Z/2$. 

\begin{lem}\label{lem_stringstr_ROSpin4}
	The element $\res_{\tilde i} (V_{\Spin(16)}) -\res_q (V_{\Spin(4)}) \in \RO(\widetilde{\Spin(4)})$ admits a string structure (in the sense of \cite[Section 2.3]{lin2024topologicalellipticgenerai}), i.e., the associated map
\begin{align}\label{eq_stringstr_Spin4}
	B\widetilde{\Spin(4)} \xrightarrow{\res_{\tilde i} (\overline V_{\Spin(16)}) -\res_q (\overline V_{\Spin(4)})} \BO \xrightarrow{} \BO\langle 0, \cdots, 4 \rangle
\end{align}
is nullhomotopic. 
\end{lem}
\begin{proof}
	Recall that we have $\widetilde{\Spin(4)} \simeq \Z/2 \times \Spin(4)$ as a group. If we denote the two inclusions of the factors by $j_1 \colon \Z/2 \hookrightarrow \widetilde{\Spin(4)}$ and $j_2 \colon \Spin(4) \hookrightarrow
	\widetilde{\Spin(4)}$, we have
	\begin{multline}
	j_1^* \oplus j_2^* \colon	[B\widetilde{\Spin(4)}, \BO\langle 0, \cdots, 4 \rangle] \\
	{\simeq}  [B\Z/2, \BO\langle 0, \cdots, 4 \rangle] \oplus [B\Spin(4), \BO\langle 0, \cdots, 4 \rangle]. 
	\end{multline}
	Thus it is enough to show that the homotopy class of \eqref{eq_stringstr_Spin4} pulls back to zero by both $j_1$ and $j_2$. 
	
	For $j_1$, we have the following equality in $\RO(\Z/2) $, 
	\begin{align}\label{eq_stringstr_proof}
		\res_{j_1} \circ \left( \res_{\tilde i} (V_{\Spin(16)}) -\res_q (V_{\Spin(4)})\right)
		= 16V_{\Z/2} - 4\underline{\R}
	\end{align}
	where $V_{\Z/2}$ denotes the real $1$-dimensional sign representation of $\Z/2$. On the other hand, we have
	\begin{align}\label{eq_Z2sigma}
		[B\Z/2, BO\langle 0, \cdots, 4\rangle] \simeq \Z[\overline{V}_{\Z/2}]/(8[\overline{V}_{\Z/2}])
	\end{align}
	so that the homotopy class of \eqref{eq_stringstr_Spin4} pulls back to zero by $j_1$. 
	
	For $j_2$, we have
	\begin{align}
			\res_{j_2} \circ \left( \res_{\tilde i} (V_{\Spin(16)}) -\res_q (V_{\Spin(4)})\right)
		= V_{\Spin(4)} + 12\underline{\R} - V_{\Spin(4)}  = 12\underline{\R}
	\end{align}
	in $\RO(\Spin(4))$. So the desired statement also holds for $j_2$. This completes the proof. 
\end{proof}
Choose a string structure $\mathfrak{s} $ for the virtual representation in Lemma \ref{lem_stringstr_ROSpin4}. 
The choice is not unique up to homotopy, 
but we may choose any one of them.\footnote{The choice of string structure here a priori affects the results of the construction from now on. But since the choice of string structure does not affect the result of the character maps, we can achieve Proposition \ref{prop_char_check}. But actually with a little more effort we can show that $\pi_{-16}\TMF[16 + \overline V_{\Spin(4)}]^{\Spin(4)}$ is torsion-free, which implies that the resulting elements $\check{\check{e}}_8$ and $\check{e}_8$ do not depend on the choice. } 
Since the equivariant Thom isomorphism for string-oriented representations is valid for the group $\widetilde{\Spin(4)} \simeq \Z/2 \times \SU(2) \times \SU(2)$ by a slight extension of \cite[Fact~2.84]{lin2024topologicalellipticgenerai}\footnote{
	In \cite[Fact~2.84]{lin2024topologicalellipticgenerai}, the equivariant sigma orientation is stated for a class of groups including $\SU(n)$ and $U(n)$ for all $n$, and closed under taking finite product. 
We can easily generalize that class to include the group $\Z/2\times \SU(2) \times \SU(2)$. The argument is as follows. 
The equivariant Thom isomorphism for string oriented representation of $\Z/2$ is deduced from the case for $U(1)$ which is already valid: indeed, since we have $\RO(\Z/2) \simeq \Z[\underline{\R}] \oplus \Z[V_{\Z/2}]$ and \eqref{eq_Z2sigma}, we only need to show the following equivalence of $\Z/2$-equivariant $\TMF$-module spectra, 
\begin{align}
\TMF[8\overline{V}_{\Z/2}] \simeq \TMF. 
\end{align}
This follows by restricting the following equivalence of $U(1)$-equivariant $\TMF$-module spectra, 
\begin{align}
	\TMF[4 V_{U(1)}-3] \simeq \TMF[ V_{U(1)}\otimes_\bC V_{U(1)}]. 
\end{align}
which follows by the sigma orientation of $4V_{U(1)} - V_{U(1)}\otimes_\bC V_{U(1)}$. 
Then the case for $\Z/2 \times \SU(2) \times \SU(2)$ follows by the same argument showing the closed-under-product statement in \cite[Fact~2.84]{lin2024topologicalellipticgenerai}. 
}, we get the Thom isomorphism
\begin{align}\label{eq_sigma_spin4}
	\oldsigma(\mathfrak{s})\colon \TMF[\res_{\tilde i} (V_{\Spin(16)})]^{\widetilde{\Spin(4)}} \simeq \TMF[\res_q (V_{\Spin(4)})+ 12]^{\widetilde{\Spin(4)}} 
\end{align}
We can further compose it with the eqivariant transfer map \eqref{eq_tr_without_ad} in genuine equivariant $\TMF$, 
\begin{align}
	\mathrm{tr}_{q} \colon \TMF[\res_q (V_{\Spin(4)})+ 12]^{\widetilde{\Spin(4)}} \to \TMF[\overline V_{\Spin(4)}+16]^{\Spin(4)} . 
\end{align}

Now we are ready to define the $\Spin(4)$ and $\SU(2)$-equivariant $\TMF$-elements of our interest. 

\begin{defn}
	\begin{enumerate}
		\item 	Define the element $\check{\check{e}}_8 \in \pi_0 \TMF[\overline V_{\Spin(4)}+16]^{\Spin(4)} = \TMF_{\Spin(4)}^{16 +\overline{V}_{\Spin(4)} }(\pt)$ to be the image of the equivariant Euler class $\chi(V_{\Spin(16)}) \in \pi_0 \TMF[V_{\Spin(16)}]^{\Spin(16)} $ under the composition
		\begin{multline}\label{eq_def_checkchecke8}
			\TMF[V_{\Spin(16)}]^{\Spin(16)} \xrightarrow{\res_{\tilde i}} \TMF[\res_{\tilde i} (V_{\Spin(16)})]^{\widetilde{\Spin(4)}}  \\
			\xrightarrow{\mathrm{tr}_{q} \circ \oldsigma(\mathfrak{s}) }
			\TMF[\overline V_{\Spin(4)}+16]^{\Spin(4)}. 
		\end{multline}
		\item Furthermore, define the element $\check{e}_8 \in \pi_0 \TMF[\overline V_{\SU(2)} +16 ]^{\SU(2)} = \TMF_{\SU(2)}^{16 + \overline{V}_{\SU(2)}}(\pt)$ by the image of $\check{\check{e}}_8$ under the restriction map along the inclusion $\SU(2) \stackrel{(\id, 1)}{\hookrightarrow} \SU(2) \times \SU(2) \simeq \Spin(4)$, 
		\begin{align}
			\res_{\Spin(4)}^{\SU(2)}  \colon 	\TMF[\overline V_{\Spin(4)}+16]^{\Spin(4)} \to \TMF[\overline V_{\SU(2)} +16 ]^{\SU(2)}. 
		\end{align}
		Here we have used that $\res_{(\id, 1)}V_{\Spin(4)} = V_{\SU(2)}$ in $\RO(\SU(2))$. 
	\end{enumerate}
\end{defn}

\begin{prop}\label{prop_char_check}
	We have
	\begin{align}\label{eq_prop_char_check_1}
		\res_{\Spin(4)}^e \left(\check{\check{e}}_8\right) = \res_{\SU(2)}^e (\check{e}_8) = e_8 \in \pi_{-16}\TMF. 
	\end{align}
	Moreover, the images under the character maps \eqref{eq_char} and \eqref{eq_char_SU2} are
	\begin{align}\label{eq_char_hathate}
		e_{\Spin(4)} \left(\check{\check{e}}_8\right) = \mathcal{E}_4/\Delta \in \mathrm{JJF}^{\ind=(1,1')}_{\wt=-8},
	\end{align}
	where $\cE_4$ is introduced in \eqref{calE4} and
	\begin{align}\label{eq_E41Delta_1}
		e_{\SU(2)} (\check e_8) = E_{4, 1} / \Delta \in \JF^{\ind = 1}_{\wt = -8}. 
	\end{align}
	where $E_{4,1}$ appears in Proposition \ref{prop:JF2}, respectively.
\end{prop}

\begin{proof}
	First notice that it is enough to show \eqref{eq_char_hathate}. Indeed, the first equality follows by definition of $\check{e}_8$, and the rest of the statements follows by $\cE_4(\tau; z, 0) = E_{4,1}(z)$ \eqref{eq_cE4}, $E_{4,1}(z=0) = c_4$ (Proposition \ref{prop:JF2}) and the definition of $e_8$. 
	
	Let us show \eqref{eq_char_hathate} by applying the character formula \eqref{eq_Phi_V=chiV} for the Euler class and the commutative diagram \eqref{diag_tr} for the transfer maps. 
	We know, via Proposition~\ref{eq_Phi_V=chiV}, that 
	\begin{align}
		e_{\Spin(16)}\left( \chi(V_{\Spin(16)}) \right) = \Theta_{V_{\Spin(16)}} =  \Delta^{-1} \prod_{i=1}^8 \theta_{11}(\tau; z_i),
	\end{align}
	where we slightly abused the notation and used the coordinates $z_i$ of the Cartan torus of $SO(16)$
	instead of that of $SemiSpin(16)$.
	The composition \eqref{eq_def_checkchecke8} is translated into the following formula for the character,
	\begin{equation}
	\begin{aligned}
	e_{\Spin(4)}  \left(\check{\check{e}}_8\right) &=\frac12 \sum_{w = 0, \frac12, \frac{\tau}{2}, \frac{1+\tau}{2}}  \Theta_{V_{\Spin(16)}} (\tau; (z+z') + w, (z-z') + w, w, w, w, w, w, w) \\
%	&= \frac{1}{2\Delta}\left(\theta_{11}\left( \tau; z+\frac12\right) \theta_{11}\left( \tau; z' + \frac12 \right) \theta_{11} \left( \tau; \frac12\right) ^6 \right. \\
%	& \quad\quad + \theta_{11}\left( \tau: z+\frac{\tau}2\right) \theta_{11}\left( \tau; z' + \frac{\tau}2\right)  \theta_{11} \left( \tau; \frac{\tau}2\right) ^6  \\
%&	\quad\quad+  \left. \theta_{11}\left( \tau: z+\frac{1+\tau}2\right) \theta_{11}\left( \tau; z' + \frac{1+\tau}2\right)  \theta_{11} \left( \tau; \frac{1+\tau}2\right) ^6  \right) . 
&= \frac1{2\Delta}\sum_{w=\frac12, \frac{\tau}{2}, \frac{1+\tau}{2}} \theta_{11}(\tau;z+w)\theta_{11}(\tau;z'+w) \theta_{11}(\tau;w)^6.
	\end{aligned}
	\label{eeeee}
	\end{equation}
	We can show that this formula agrees with $\cE_{4,1}/\Delta$, for example by checking  the characterizing properties \eqref{eq_cE4}. This completes the proof. 
\end{proof}

\begin{rem}
If a genuine equivariant topological modular forms were available for general twists,
we would have been able to perform the transfer for $Spin(16)\to SemiSpin(16)$,
instead of  pulling back to the transfer for $\widetilde{Spin(4)}\to Spin(4)$ as in \eqref{spin4pull}.
The result would be an element $\dot e_8 \in \TMF^{16+\tau}_{SemiSpin(16)}(\pt)$ for a suitable twist,
whose character should be given by \begin{equation}
e_{SemiSpin(16)} (\dot e_8)= \frac1{2\Delta}\sum_{w=\frac12, \frac{\tau}{2}, \frac{1+\tau}{2}} \prod_{i=1}^8 \theta_{11}(\tau;z_i+w)
\label{char-ss}
\end{equation}
as a natural generalization of \eqref{eeeee}.
\end{rem}

\begin{rem}
The element $\dot e_8\in \TMF^{16+\tau}_{SemiSpin(16)}(\pt)$ should be the restriction of 
$\tilde e_8\in\TMF^{16+\tau}_{E_8}(\pt)$ discussed in Conjecture~\ref{conj_e8} via the inclusion $SemiSpin(16)\subset E_8$.
One piece of evidence is that  $e_{SemiSpin(16)} (\dot e_8)$ above equals $\eta(q)^{-16}$ times 
the character of the unique integral representation of the affine Lie algebra $(\mathfrak{e}_8)_1$.
In the framework of the conjectural correspondence between $\TMF$ classes and VOSAs in Sec.~\ref{sec:vosa},
the construction of $\dot e_8$ or $\tilde e_8$ via the transfer corresponds to the construction of the VOA $\V_{E_8}$ 
by a $\bZ/2$ orbifolding of $\V(n)$.
\end{rem}

%\input{equivariant}
%\input{e8}

% !TEX root = paper.tex
\section{On the differential bordism groups and their Anderson duals}
\label{app:diff_bordism}
In this section we summarize the notions of differential bordism groups and differential Anderson duality pairings, which are used throughout the paper. 
We use a model of differential extension of Anderson dual cohomology groups given in \cite{YamashitaDifferentialIE}, which extends the construction of \cite{Yamashita:2021cao}. 
In that paper, the model of a differential extension $\widehat{I_\Z E}{}^\bullet$ of $I_\Z E^\bullet$ 
for a spectrum $E$
is defined provided that a differential extension $\widehat{E}_\bullet$ of $E$-{\it homology} is given. 
Since we need to deal with $I_\Z \MString$ and their twisted versions, we start from introducing the corresponding differential homology groups.

\subsection{Twisted differential string structures}\label{app_subsec_diff_string}
We first review the notion of {\it twisted differential string structures} here.
Our presentation is essentially based on \cite{Sati:2009ic,FSSdifferentialtwistedstring} which were in turn based on \cite{Waldorf},
in terms of $\infty$-sheaves of $\infty$-groupoids on the category of manifolds, i.e.~$\infty$-stacks.
See also a recent discussion in \cite{debray2023bordism}.

We start from the untwisted case. 
Differential principal $\mathrm{String}(d)$-bundles are classified by the $\infty$-stack $\mathbf{B}\mathrm{String}(d)_\nabla$ defined as the $\infty$-pullback
\begin{equation}
    \vcenter{\xymatrix{
    \mathbf{B}\mathrm{String}(d)_\nabla \ar[r]\ar[d] & \Omega^3 \ar[d]^-{a}\\
    \mathbf{B}\mathrm{Spin}(d)_\nabla \ar[r]^-{\widehat{\mathbf{p_1/2}}} & \mathbf{B}^3\mathrm{U}(1)_\nabla
    }},
\end{equation}
where 
\begin{itemize}
    \item $\mathbf{B}\mathrm{Spin}(d)_\nabla$ is the $\infty$-stack of principal $\Spin(d)$-bundles with connections, i.e., for each manifold $X$, $\mathbf{B}\mathrm{Spin}(d)_\nabla(X)$ is the groupoid of principal $\Spin(d)$-bundles with connections over $X$. 
    \item $\mathbf{B}^3\mathrm{U}(1)_\nabla$ is the $\infty$-stack of {\it circle $3$-bundles with connections}, which is a model of fourth differential ordinary cohomology in the sense that we have $\pi_0(\mathbf{B}^3\mathrm{U}(1)_\nabla(X)) = \widehat{H}^4(X; \Z)$. $\mathbf{B}^3\mathrm{U}(1)_\nabla$ is obtained by the homotopy sheafification of the nerve of the presheaf of cochain complexes
    \begin{align}\label{eq_B3U(1)}
       X \mapsto (C^\infty(X; U(1)) \xrightarrow{d\log} \Omega^1 (X)\xrightarrow{d} \Omega^2(X) \xrightarrow{d} \Omega^3(X) ). 
    \end{align}
    \item The map $\widehat{\mathbf{p_1/2}} \colon \mathbf{B}\mathrm{Spin}(d)_\nabla \to \mathbf{B}^3\mathrm{U}(1)_\nabla$ is the canonical differential refinement of the first Pontryagin class $p_1/2 \in H^4(\mathrm{BSpin}; \Z)$, which was obtained by the $\infty$-Chern-Weil theory \cite{FSSdifferentialtwistedstring}. 
    \item The map $a \colon \Omega^3 \to \mathbf{B}^3\mathrm{U}(1)_\nabla$ is induced by the obvious inclusion in terms of \eqref{eq_B3U(1)}, which induces the structure map $a \colon \Omega^3 \to \widehat{H}^4(-; \Z)$ in differential cohomology. 
\end{itemize}
A differential string structure on a manifold is a differential string structure on its tangent bundle. 
Unwinding the definition, the data specifying a differential string structure on a manifold $M$ consists of
\begin{itemize}
    \item a differential spin structure $g^\spin$, i.e., a spin structure with connection, on its tangent bundle, 
    \item a $3$-form $H \in \Omega^3(M)$ and
    \item an isomorphism $i \colon \widehat{\mathbf{p_1/2}}(g^\spin) \simeq a(H)$ in $\mathbf{B}^3\mathrm{U}(1)_\nabla(M)$. 
\end{itemize}
In particular, the last condition implies that, 
\begin{itemize}
    \item mapping by the forgetful map $\mathbf{B}^3\mathrm{U}(1)_\nabla(M) \to \mathbf{B}^3\mathrm{U}(1)(M)$, we have a string structure in the topological sense, 
    \item mapping by the curvature map $R\colon \mathbf{B}^3\mathrm{U}(1)_\nabla(M) \to \Omega^4_\clo(M)$, we have $\cw_{g^\spin}(p_1/2) = dH$ in $\Omega_\clo^4(M)$, where $\cw_{g^\spin}$ is the Chern-Weil construction with respect to the spin connection. 
\end{itemize}
This gives the physical explanation of the differential string structure used in the Appendix~\ref{sec:introphys} for physicists.

Now we turn to twisted differential string structures. 
In particular we use the twists coming from principal $G$-bundles with connections, which we call {\it $(G, \tau)$-twisted differential string structures}. 
We fix a compact Lie group $G$ and {a map $\tau \colon BG \to K(\Z, 4)$. }
We define the $\infty$-stack of $(G, \tau)$-twisted differential principal $\mathrm{String}(d)$-bundles as the following $\infty$-pullback
\begin{align}
    \vcenter{\xymatrix{
    \mathbf{B}\mathrm{String}(d)_\nabla \times_\tau \mathbf{B}G_\nabla \ar[r]\ar[d] & \Omega^3 \ar[d]^-{a}\\
    \mathbf{B}\mathrm{Spin}(d)_\nabla \times \mathbf{B}G_\nabla \ar[r]^-{\widehat{\mathbf{p_1/2}} - \widehat{\mathbf{\tau}}} & \mathbf{B}^3\mathrm{U}(1)_\nabla
    }},
\end{align}
where $\widehat{\mathbf{\tau}} \colon \mathbf{B}G_\nabla \to \mathbf{B}^3\mathrm{U}(1)_\nabla$ is the canonical differential refinement of $\tau$ obtained by the $\infty$-Chern-Weil theory \cite{FSSdifferentialtwistedstring}. 

A $(G, \tau)$-twisted differential string structure on a manifold $M$ is a $(G, \tau)$-twisted differential principal $\mathrm{String}(d)$-bundle whose underlying $\mathrm{GL}(d)$-bundle is identified with the frame bundle of the tangent bundle. 
Unwinding the definitions, the  data specifying a $(G, \tau)$-twisted differential string structure on a manifold $M$ consists of
\begin{itemize}
    \item a differential spin structure $g^\spin$, %i.e., a spin structure with connection, on its tangent bundle, 
    \item a principal $G$-bundle with connection $(P, \nabla)$ on $M$,
    \item a $3$-form $H \in \Omega^3(M)$ and
    \item an isomorphism $i \colon \widehat{\mathbf{p_1/2}}(g) - \widehat{\tau}(P, \nabla) \simeq a(H)$ in $\mathbf{B}^3\mathrm{U}(1)_\nabla(M)$. 
\end{itemize}
In particular, the last condition implies that, 
\begin{itemize}
    \item mapping by the forgetful map $\mathbf{B}^3\mathrm{U}(1)_\nabla(M) \to \mathbf{B}^3\mathrm{U}(1)(M)$, 
    we have a $\tau(P)$-twisted string structure in the topological sense, 
    \item mapping by the curvature map $R\colon \mathbf{B}^3\mathrm{U}(1)_\nabla(M) \to \Omega^4_\clo(M)$, we have $\cw_{g^\spin}(p_1/2) - \cw_\nabla(\tau) = dH$ in $\Omega_\clo^4(M)$. 
\end{itemize}

\subsection{Differential bordism groups}\label{app_subsec_diff_bordism}
Our model of the differential extension of the Anderson duals to bordism homologies are based on {\it differential bordism groups}. 
Our discussions here are based on \cite{Yamashita:2021cao}, where the case of tangential structures given by sequences of Lie groups (such as Spin, SO) are treated in detail. 
In this paper we need the following generalizations.
\begin{itemize}
    \item We need to treat differential string structures and their twisted versions. 
    \item Moreover, we need to treat \emph{relative} differential bordism groups, namely the differential version of $\MSpin/\MString$. 
\end{itemize}

We first explain the non-relative non-twisted case where we allow the tangential structure $\mathcal{B}$ to be either coming from Lie groups or string structure. 
The differential tangential $\mathcal{B}$-bordism group $\widehat{MT\mathcal{B}}_n = \widehat{MT\mathcal{B}}_n(\pt)$ is constructed in terms of {\it differential stable tangential $\mathcal{B}$-cycles}, as follows. 
An $n$-dimensional differential stable tangential $\mathcal{B}$-cycle 
over $X$ is a pair $(M, g)$, consisting of an $n$-dimensional closed manifold $M$ and a {\it differential} stable $\mathcal{B}$-structure $g$ (i.e., equipped with connection in the case of $\mathcal{B}$ given by Lie groups, and the string case recalled in Subsection \ref{app_subsec_diff_string}). 

We use the {\it Bordism Picard groupoid} $\hBord^{\mathcal{B}_\nabla}_n$ defined in \cite[Definition 3.8]{Yamashita:2021cao}. 
The objects are differential stable tangential $\mathcal{B}$-cycles $(M, g)$ of dimension $n$, and 
the morphisms are {\it bordism classes} $[W, g_W]$ of bordisms of differential stable tangential $\mathcal{B}$-cycles. 

% We need the Chern-Weil construction in this setting (\cite[Subection 4.1.1]{YamashitaYonekura2021}) We set\footnote{In the general notation introduced in \eqref{eq_def_NE} below, we have $N^\bullet_G = N^\bullet_{MTG}$. We abbreviate the notation. } 
% \begin{align}\label{eq_inv_poly}
%     N_G^\bullet
%     := H^\bullet(MTG; \R) = \varprojlim_{d} H^\bullet(BG_d; \R_{G_d}) = \varprojlim_{d}(\mathrm{Sym}^{\bullet/2}\mathfrak{g}_d^* \otimes_\R \R_{G_d})^{G_d}. 
% \end{align}
% In the case where $G$ is oriented, i.e., the image of $\rho_d$ lies in $\mathrm{SO}(d, \R)$ for each $d$, the $G_d$-module $\R_{G_d}$ is trivial and $N_G^\bullet$ is the projective limit of invariant polynomials on $\mathfrak{g}_d$. 
% In general cases, $N_G^\bullet$ can be regarded as the projective limit of polynomials on $\mathfrak{g}_d$ which change the sign by the action of $G_d$.

A differential stable tangential $\mathcal{B}$-structure $g$ on a manifold $M$ defines a homomorphism,
\begin{align}\label{eq_cw}
    \mathrm{cw}_g  \colon \H^\bullet(MT\mathcal{B}; \R) \to \Omega_\clo^\bullet(M; \mathrm{Ori}(M)), 
\end{align}
where $\mathrm{Ori}(M)$ is the orientation line bundle of $M$ and $\Omega^\bullet_\clo$ denotes the group of closed differential forms. 
The homomorphism \eqref{eq_cw} is given by the Chern-Weil construction in the case where $\mathcal{B}$ is given by Lie groups
as in \cite[Definition 4.4]{Yamashita:2021cao}, and in the string case it is just the Chern-Weil construction by the underlying differential spin structure. 

An object $(M, g)$ in $\hBord^{\mathcal{B}_\nabla}_n$ gives an element
\begin{align}
    \cw(M, g) \in \H_n(MT\mathcal{B}; \R) \simeq \Hom_\R (\H^n(MT\mathcal{B}; \R), \R), 
\end{align}
by
\begin{align}
    \cw(M, g)(\omega) := \int_M \cw_g(\omega)
\end{align}
for $\omega \in \H^n(MT\mathcal{B}; \R)$.
Similarly, for a morphism $[W, g_W]$ in $\hBord^{\mathcal{B}_\nabla}_n$, we get a well-defined element
\begin{align}\label{eq_cw_B}
    \cw[W, g_W] \in \H_{n+1}(MT\mathcal{B}; \R) \simeq \Hom_{\R} (\H^{n+1}(MT\mathcal{B}; \R), \R). 
\end{align}

\begin{defn}[of {$\widehat{MT\mathcal{B}}_n$}]\label{def_diff_bordism}
Let $n$ be an integer. 
\begin{enumerate}
    \item 
We set
\[
    \widehat{MT\mathcal{B}}_n := \{(M, g, \eta)\} / \sim, 
\]
where $(M, g)$ is an object in $\hBord^{\mathcal{B}_\nabla}_n$ and $\eta \in H_{n+1}(MT\mathcal{B}; \R) $. 
The relation $\sim$ is the equivalence relation generated by the relations
\[
    (M_-, g_-, \eta) \sim (M_+, g_+, \eta - \cw[W,g_W]) 
\]
for each morphism $[W, g_W] \colon (M_-, g_-) \to (M_+, g_+)$ in $\hBord^{\mathcal{B}_\nabla}_n$. 
\item
We define maps $R$, $I$ and $a$ by
\[
\begin{array}{rr@{\,}l@{\quad}r@{\,}l}
    R \colon& \widehat{MT\mathcal{B}}_n&\to H_n(MT\mathcal{B}; \R) ,&
    [M, g, \eta] &\mapsto \cw(M, g) , \\
    I \colon& \widehat{MT\mathcal{B}}_n&\to \pi_n MT\mathcal{B},&
    [M, g,  \eta] &\mapsto [M, g], \\
    a  \colon& H_{n+1}(MT\mathcal{B}; \R) &\to \widehat{MT\mathcal{B}}_n , &
    \eta &\mapsto [\varnothing, \eta]. 
\end{array}
\]
\end{enumerate}
\end{defn}

We can extend the above definition to that of $\widehat{MT\mathcal{B}}_\bullet(X)$ for manifolds $X$, and verify that the data $(\widehat{MT\mathcal{B}}_\bullet, R, I, a)$ gives a differential $MT\mathcal{B}$-homology in the sense of \cite{YamashitaDifferentialIE}. 
In particular the axiom of differential homology includes the exactness of the following sequence, 
\begin{align}\label{eq_exact_MTB}
    MT\mathcal{B}_{n+1} \xrightarrow{\mathrm{ch}} H_{n+1}(MT\mathcal{B}; \R)\xrightarrow{a} \widehat{MT\mathcal{B}}_n \xrightarrow{I} MT\mathcal{B}_n \to 0.
\end{align}

We have a canonical symmetric monoidal functor
\begin{align}\label{eq_hBord_to_MTB}
    \hBord^{\mathcal{B}_\nabla}_n \to  {\left( H_{n+1}(MT\mathcal{B}; \R) \xrightarrow{a}
    \widehat{MT\mathcal{B}}_n, \right)}, 
\end{align}
{where the right hand side is the Picard groupoid associated to the morphism of abelian groups},
which maps objects as $(M, g ) \mapsto [M, g, 0]$ and morphisms as $[W, g_W] \mapsto \cw[W, g_W]$. 
Using this functor we regard an object in $\hBord^{\mathcal{B}_\nabla}_n$ as an element in $\widehat{MT\mathcal{B}}_n$ and use the notation $[M, g]:=[M, g, 0]$.

Now we explain the variations mentioned above. 
The definition of corresponding differential bordism groups are given by suitable modifications as follows. 
\begin{itemize}
    \item $\reallywidehat{(\MString \wedge_\tau BG_+)}_n$ with {$\tau \colon BG \to K(\Z, 4)$}, which classifies the $(G, \tau)$-twisted differential string manifolds. 
    This is a differential homology version of the $\MString$-module Thom spectrum, which we denote by $\MString \wedge_\tau BG_+$, associated to the map $BG \xrightarrow{\tau} K(\Z, 4) \to BGL_1\MString$. 

    In this case, a differential cycle is a pair $(M, g)$, where $g$ is a $(G, \tau)$-twisted differential string structure on $M$. 
    Explicitly, $g$ consists of data $g = (g^\spin, (P, \nabla), H, i)$ described in Subsection \ref{app_subsec_diff_string}. 
    On realification the twist becomes trivial, and we have the isomorphism $\mathrm{MString}^{n+\tau}(BG; \R) \simeq \H^{n}(\MString \wedge BG_+; \R)$. 
    The Chern-Weil construction in this case is the map
    \begin{multline}\label{eq_cw_twisted_String}
        \mathrm{cw}_{g}=\cw_{g^\spin} \otimes \cw_\nabla \colon H^*(\MString \wedge BG_+; \R) \\ 
        \simeq H^*(B\mathrm{String}; \R) \otimes_\R H^*(BG; \R) \to \Omega_\clo^*(M), 
    \end{multline}
    The rest of the construction works by replacing $H(MT\mathcal{B}; \R)$ with $H(\MString \wedge BG_+; \R)$. 
    
    \item $\reallywidehat{\left((\MSpin/\MString)\wedge_\tau BG_+ \right)}_n$ with {a map $\tau \colon BG \to K(\Z, 4)$}, which classifies the relative $(G, \tau)$-twisted differential string/spin manifolds. 
    This is a differential homology version of the homotopy cofiber of the map $\MString \wedge_\tau BG_+ \to \MSpin \wedge BG_+$ which we denote by $(\MSpin/\MString)\wedge_\tau BG_+$. 
    
    In this case, a relative differential cycle consists of data $(N, M, g^\spin_N, g^{\text{$\tau$-string}}_M, (P, \nabla))$, where $N$ is a compact $n$-dimensional manifold with boundary $\del N = M$, $(P, \nabla)$ is a principal $G$-bundle with connection on $N$, $g^\spin_N$ is a differential spin structure on $N$ and $g^{\text{$\tau$-string}}_M$ is a $(G, \tau)$-twisted differential string structure (Subsection \ref{app_subsec_diff_string}) on $M$ whose underlying differential spin structure and $G$-bundle with connection are identified with $g^\spin_N|_M$ and $(P, \nabla)|_M$, respectively. 
    
    In this case, the corresponding bordism Picard groupoid involves manifolds with corners, for which we use the notion of {\it $\langle k \rangle$-manifolds} \cite{Janich1968}, also recalled in \cite[Subsection 2.3]{Yamashita:2021cao}. 
    The construction is similar to the relative bordism Picard groupoid $\hBord^{G_\nabla}_n(X, Y)$ in \cite{Yamashita:2021cao}. 

    The realification of $\left((\MSpin/\MString)\wedge_\tau BG_+\right)^n$ is computed by the homotopy fiber long exact sequence. The result is 
    \begin{align}
        &\left((\MSpin/\MString)\wedge_\tau BG_+\right)^n_\R\notag\\
        &    \label{eq_MSpin/MString_R}
        \quad \simeq \ker 
        \left(\H^n(\MSpin \wedge BG_+; \R) \to \Hom(\MSpin/\MString_{n+\tau}(BG), \R)
        \right)\\
        &\quad = \left( \frac{p_1}{2} -\tau \right) \H^{n-4}(\MSpin \wedge BG_+; \R) \notag
    \end{align} 
    In terms of the identification \eqref{eq_MSpin/MString_R}, the Chern-Weil construction with respect to a relative differential cycle $(N, M, g^\spin_N,g^{\text{$\tau$-string}}_M,  P, \nabla)$ takes the form
    \begin{align}\label{eq_cw_MSpin/MString}
        \cw_{(g^\spin_N, g^{\text{$\tau$-string}}_M, \nabla)} \colon \left( \frac{p_1}{2} -\tau \right) \H^{*-4}(\MSpin \wedge BG_+; \R) \to \Omega_\clo^*(N, M), 
    \end{align}
    where $\Omega_\clo^*(N, M)$ is the group of {\it relative closed differential forms}, consisting of pairs $(\omega, \eta) \in \Omega^*(N) \oplus\Omega^{*-1}(M)$ with $\omega = d\eta$. 
    The map \eqref{eq_cw_MSpin/MString} maps an element $\psi = \left( \frac{p_1}{2} -\tau \right)\tilde{\psi}$ 
    to the relative closed differential form
    \begin{align}\label{eq_cw_MSpin/MString2}
      \left( \cw_{g_N^\spin, \nabla}(\psi) , \ H \wedge \cw_{g^{\text{$\tau$-string}}_M}(\tilde{\psi})\right). 
    \end{align}
    Here $H \in \Omega^3(M)$ is the $3$-form in the data of $g^{\text{$\tau$-string}}_M$. 
\end{itemize}

\subsection{Differential Anderson duals to bordism groups}

% \begin{defn}[{$(\widehat{I_\Z E})^*$, \cite[Definition 5.3]{YamashitaDifferentialIE}}]\label{def_diff_IE}
% Let $E$ be a spectrum and assume we are given a differ $\widehat{E}_*$ be a 
% We denote by $E$ the corresponding spectra representing the topological bordism homology theories. 
% For a manifold $X$ and integer $n$, we set
% \begin{align}
%     (\widehat{I_\Z E})^n := \{(\omega, h) \}, 
% \end{align}
% where pairs $(\omega, h)$ consists of
% \begin{itemize}
%     \item $\omega \in H^n(E; \R)$. 
%     \item $h \colon \widehat{E}_{n-1} \to \R/\Z$ is a group homomorphism, 
%     \end{itemize}
%    for which the following diagram commutes. 
%    \begin{align}\label{diag_compatibility}
%         \xymatrix{
%         H_n(E; \R) \ar[r]^-{a} \ar[d]^-{ \langle -,  \omega \rangle} &\widehat{E}_{n-1}(X) \ar[d]^-{h} \\
%         \R\ar[r]^-{\mod \Z} &\R/\Z
%         }. 
%     \end{align}
%     Here $a$ is one of the structure map in Definition \ref{def_diff_bordism} (2). 
% In the casees $\widehat{E}_* = \widehat{\mathrm{MString}}_{*+\tau} (BG)$ and $\widehat{\mathrm{MSpin}/\mathrm{MString}}_{*+\tau}(BG)$, we use the nontations $(\widehat{I_\Z E})^n=(\widehat{I_\Z\mathrm{MString}})^{n+\tau}(BG)$ and $(\widehat{I_\Z\mathrm{MSpin}/\mathrm{MString}})^{n+\tau}(BG)$, respectively. 
% \end{defn}

Applying the definition of the differential extension $(\widehat{I_\Z E})^n$ given in \cite[Definition 5.3]{YamashitaDifferentialIE} to our differential bordism homology theories in the last subsection, we obtain the following model. 
Let $E$ be one of the bordism spectra $MT\mathcal{B}$, $\MString \wedge_\tau BG_+$ and $(\MSpin/\MString) \wedge_\tau BG_+$ in the last subsection. 
We have
\begin{align}
    (\widehat{I_\Z E})^n = (\widehat{I_\Z E})^n(\pt) := \{(\omega, h) \}, 
\end{align}
where 
\begin{itemize}
    \item $\omega \in H^n(E; \R)$, regarded as a characteristic polynomial for differential $E$-manifolds, 
    \item $h \colon \widehat{E}_{n-1} \to \R/\Z$ is a group homomorphism,
\end{itemize}
for which the following diagram commutes. 
   \begin{align}\label{diag_compatibility}
       \vcenter{\xymatrix{
         H_n(E; \R) \ar[r]^-{a} \ar[d]^-{ \langle -,  \omega \rangle} &\widehat{E}_{n-1} \ar[d]^-{h} \\
       \R\ar[r]^-{\mod \Z} &\R/\Z
         }}. 
     \end{align}
     Here $a$ is one of the structure map in Definition \ref{def_diff_bordism} (2). 
In particular, we have a canonical pairing, 
\begin{align}\label{eq_IE_diff_pairing}
   \langle \cdot, \cdot \rangle_{\widehat{I_\Z}} \colon  (\widehat{I_\Z E})^n \times \widehat{E}_{n-1} \to \R/\Z, 
\end{align}
which induces the differential pairing in Subsection \ref{subsec_diff_pairing}. 
     
In terms of Definition \ref{def_diff_bordism}, we can make the above definition more explicit: given any morphism $[W, g_W] \in \Hom_{\hBord_{n-1}^\bullet}((M_-, g_-), (M_+, g_+))$ in the differential bordism Picard groupoid in question, we have the following compatibility condition,
\begin{align}
    h(M_+, g_+) - h(M_-, g_-) \equiv \int_W\cw_{g_W} (\omega) \pmod \Z. 
\end{align}

One feature of the model $(\widehat{I_\Z E})^n$ is that we have the following commutative diagram with exact rows\footnote{
In the language of differential cohomology, this inclusion comes from the fact that the flat part (the kernel of the curvature homomorphism) of $(\widehat{I_\Z E})^{*+1}$ is identified with  ``$I_\Z E$ with $\R/\Z$-coefficient'', which is $I_{\R/\Z}E$. 
}, 
\begin{align}\label{diag_exact_diff_IE}
    \vcenter{\xymatrix{
    0 \ar[r] & (I_{\R/\Z}E)^{n-1} \ar[r]^{\iota_{\text{flat}}} \ar@{=}[d]  & (\widehat{I_\Z E})^n \ar[r]^-{R} \ar[d]^-{I}& H^n(E; \R)\ar@{=}[d] & \\
    H^{n-1}(E; \R) \ar[r] & (I_{\R/\Z}E)^{n-1} \ar[r] & (I_\Z E)^n \ar[r]& H^n(E; \R)\ar[r] & (I_{\R/\Z}E)^{n} ,
    }}
\end{align}
where the bottom sequence is the homotopy fiber exact sequence. 
The middle arrow $\iota_{\text{flat}}$ of the first row maps $\tilde{h} \in (I_{\R/\Z}E)^{n-1}= \Hom( \pi_{n-1}E, \R/\Z)$ to $(0, \tilde{h}\circ I)$. 
The last arrow $R$ of the first row maps $(\omega, h)$ to $\omega$. 
By the commutativity of \eqref{diag_exact_diff_IE}, we get the following relation between the differential pairing \eqref{eq_IE_diff_pairing} and the topological pairings. 

\begin{lem}\label{lem_pairing_compatibility}
    In the following diagram, 
\[
\vcenter{\xymatrixcolsep{0pc}
    \xymatrix{
   (I_{\R/\Z}E)^{n-1}\ar[d]^-{\iota_{\text{flat}}}  &\times& \pi_{n-1}E \ar[rrrrrr]^-{\langle \cdot, \cdot\rangle_{I_{\R/\Z}}} &&&&&& \R/\Z \ar@{=}[d]\\
    (\widehat{I_\Z E})^n \ar[d]^-{R}&\times& \widehat{E}_{n-1} \ar[u]^-{I}\ar[rrrrrr]^-{\langle \cdot, \cdot \rangle_{\widehat{I_\Z}}}  &&&&&& \R/\Z \\
    H^n(E; \R) &\times& H_n(E; \R)\ar[rrrrrr]^-{\langle \cdot, \cdot \rangle_{H\R}} \ar[u]^-{a}&&&&&& \R \ar[u]_-{\text{mod} \ \Z}
    }},
\]
the three pairings are compatible, in the sense that we have
\begin{align*}
    \langle x, I(y) \rangle_{I\R/\Z} &= \langle \iota_{\text{flat}}(x), y \rangle_{\widehat{I_\Z}}, \\
     \langle z, a(w) \rangle_{\widehat{I_\Z}} &\equiv \langle R(z), w \rangle_{\widehat{H\R}} \pmod \Z, 
\end{align*}
for any $x \in (I_{\R/\Z}E)^{n-1}$, $y \in \widehat{E}_{n-1}$, $z \in (\widehat{I_\Z E})^n$ and $w \in H_n(E; \R)$. 
\end{lem} 
Similarly, we also have the following compatibility. 
\begin{lem}\label{lem_pairing_compatibility_2}
    In the following diagram, 
\begin{align*}
\vcenter{\xymatrixcolsep{0pc}
    \xymatrix{
   H^{n-1}(E; \R)\ar[d]^-{a}  &\times& H_{n-1}(E; \R)\ar[rrrrrr]^-{\langle \cdot, \cdot\rangle_{H\R}} &&&&&& \R \ar[d]^{\mod \Z}\\
    (\widehat{I_\Z E})^n &\times& \widehat{E}_{n-1} \ar[u]^-{R}\ar[rrrrrr]^-{\langle \cdot, \cdot \rangle_{\widehat{I_\Z}}}  &&&&&& \R/\Z
    }}
\end{align*}
the two pairings are compatible. 
\end{lem}

% !TEX root = paper.tex
\section{An analogue  for spin/spin$^c$ manifolds and $\KO$-theory}
\label{app:toy}
%In this section we explain a baby version of the story developed in the main body of this paper. 
In the main part of this paper, we constructed a morphism
\begin{align}
    \alpha_{\stri/\spin} \colon \TMF \to \Sigma^{-20}I_\Z \MString/\MSpin, 
\end{align}
related it to the Anderson duality of topological modular forms,
and studied the induced pairings between $\pi_\bullet\TMF$ and $\pi_\bullet \MString/\MSpin$. 
In this section, we explain an analogous story for much simpler settings. We construct a morphism
\begin{align}
    \gamma_{\spin^c/\spin} \colon \KO \to \Sigma^{-2}I_\Z\MSpin^c/\MSpin, 
\end{align}
relate it to the Anderson self-duality of $\KO$-theory,
and study the induced pairings between $\pi_\bullet\KO$ and $\pi_\bullet \MSpin^c/\MSpin$. 
This section is logically independent from the main part of the paper, but we recommend the reader to refer to this section because it illustrates our basic strategy in a much simpler manner. 

In this section, we use the following notations for morphisms between spectra appearing in the commutative diagram in the $(\infty, 1)$-category of $\MSpin$-module spectra, 
\begin{align}\label{diag_ABS_rel}
    \vcenter{\xymatrix{
\MSpin \ar[d]^-{\ABS_{\spin}} \ar[r]^-{\iota'} &
\MSpin^c \ar[d]^-{\ABS_{\spin^c}} \ar[r]^-{C\iota'} &
\MSpin^c/\MSpin \ar[d]^-{\ABS_{\spin^c/\spin}} \\
\KO \ar[r]^-{c}& 
\K \ar[r]^-{R} & 
\K/\KO  %\simeq \Sigma^2\KO. 
}}.
\end{align}
where each row is a homotopy fiber sequence, the bottom horizontal arrow is the complexification, the left and middle vertical arrows are Atiyah-Bott-Shapiro orientations, and the right vertical arrow is the morphism canonically induced on the homotopy cofibers.  
Here, we denoted the structure morphism $Cc$ in \eqref{diag_ABS_rel} by $R$
because 
 the complexification $c \colon \KO \to \K$ is a part of the {\it $\eta c R$-cofiber sequence} \cite{BrunerGreenlees}
 \begin{align}\label{eq_eta_c_R}
        \Sigma^{1}\KO \xrightarrow{\cdot \eta} \KO \xrightarrow{c} \K \xrightarrow{R} \Sigma^2 \KO,  
\end{align}
where $R := r \circ \beta^{-1}$ is the composition of the realification morphism $r \colon \K \to \KO$ and the Bott periodicity of $\K$,
so that we have the equivalence $\K/\KO \simeq \Sigma^2 \KO$.
Our convention is to denote the Bott element of $\K$ by $\beta\in\pi_2\K$.
We could have denoted $\K/\KO$ in \eqref{diag_ABS_rel} equivalently by $\Sigma^2 \KO$, but in order to make clear the analogy to $\KO((q))/\TMF$ in the main part of the paper, we intentionally denoted it by $\K/\KO$. 

\subsection{Secondary morphisms and the Anderson self-duality of $\KO$-theory}

Our starting point is the $\MSpin^c$-module morphism $\gamma_{\spin^c}: \K \to \Sigma^{-2}I_\bZ\MSpin^c$,
where we equip $\K$ with the $\MSpin^c$-module structure using the Atiyah-Bott-Shapiro orientation.
Due to our basic Lemma~\ref{lemma:basic}, it is specified by an element in $\pi_{2}I_\bZ \K$:
\begin{defn}
\label{eq_def_gamma_K}
    We let $\gamma_{\K} \in \pi_2 I_\Z \K$ to be the Anderson duality element of K-theory shifted by the Bott periodicity.
    In other words, via the isomorphism $\pi_2 I_\Z \K \simeq \Hom(\pi_{-2}K, \Z)\simeq \Hom(\Z \cdot \beta^{-1}, \Z)$, the element $\gamma_{\K}$ corresponds to the isomorphism sending $\beta^{-1}$ to $1$. 
\end{defn}

% \begin{prop}\label{prop_vanish_anomaly_SQM}
% Let $I_\Z \psi \colon I_\Z \K \to I_\Z \KO$ denote the Anderson dual to the middle vertical arrow in \eqref{diag_ABS_rel}. 
%     We have the following equality in $\pi_{2}I_\Z \KO$. 
%     \begin{align}
%         I_\Z \psi (\gamma_\K) = 0. 
%     \end{align}
% \end{prop}
% \begin{proof}
% This is simply because $\pi_2 I_\Z \KO=0$, which follows by $\pi_{-3}\KO=0$ and $\pi_{-2}\KO=0$. 
% \end{proof}

We now use the homotopy fiber sequence
\begin{align}\label{eq_K/KO_fiberseq}
    I_\Z \K/\KO \xrightarrow{I_\Z R} I_\Z \K \xrightarrow{I_\Z c} I_\Z \KO. 
\end{align}
We then have the following proposition:
\begin{prop}\label{prop_def_gamma_K/KO}
The morphism $I_\Z R$ induces an isomorphism
\begin{align}
    I_\Z R \colon \pi_{2}I_\Z \K/\KO \simeq \pi_{2}I_\Z \K. 
\end{align}
We define $\gamma_{ \K/\KO} \in \pi_{2}I_\Z \K/\KO $ to be the element which maps to $\gamma_{\K}$ under the isomorphism above. 
\end{prop}
\begin{proof}
The long exact sequence for the homotopy fibration \eqref{eq_K/KO_fiberseq} gives
\begin{align}\label{eq_pi2_IZK/KO}
     \pi_{3}I_\bZ \KO \to \pi_{2} I_\Z  \K/\KO \xrightarrow{I_\Z R}\pi_{2}I_\bZ \K \to \pi_{2}I_\bZ \KO. 
\end{align}
Since $\Ext(\pi_{-4}\KO, \Z)$, $\pi_{-3}\KO$ and $\pi_{-2}\KO$ are all zero, the first and the last terms of \eqref{eq_pi2_IZK/KO} are zero. Thus the middle arrow is an isomorphism, as desired. 
\end{proof}

We apply Definition \ref{defn:mor} for $R=\KO$, $M=\K/\KO$ and $\alpha = \gamma_{\K/\KO}$. Following the convention in Remark \ref{rem:mor}, we get morphisms
\begin{align}
    \gamma_{\K/\KO} &\colon \K/\KO \to \Sigma^{-2} I_\Z \KO, \\
    \gamma^{\vee}_{\K/\KO} &\colon \KO \to \Sigma^{-2} I_\Z \K/\KO .
\end{align}

\begin{defn}[{$\gamma_{\spin^c/\spin}$}]
        We define an $\MSpin$-module morphism $\gamma_{\spin^c/\spin}$ to be the following composition, 
    \begin{align}\label{eq_def_gamma_spinc/spin}
        \gamma_{\spin^c/\spin} \colon \KO \xrightarrow{\gamma^{\vee}_{\K/\KO}}\Sigma^{-2}I_\Z \K/\KO \xrightarrow{I_\Z \ABS_{\spin^c/\spin}} \Sigma^{-2}I_\Z \MSpin^c/\MSpin. 
    \end{align}
    where $\ABS_{\spin^c/\spin}$ is given in \eqref{diag_ABS_rel}. 
\end{defn}
We also use the $\MSpin^c$-module morphism $\gamma_{\spin^c} $ given by the following composition. 
\begin{align}
    \gamma_{\spin^c} \colon \K \xrightarrow{\gamma^{\vee}_{\K}}\Sigma^{-2}I_\Z \K \xrightarrow{I_\Z \ABS_\spin^c} \Sigma^{-2}I_\Z \MSpin^c. 
\end{align}
The following diagram commutes. 
\begin{align}
    \vcenter{\hbox{\xymatrix{
    \KO \ar[rrr]^-{\gamma_{\spin^c/\spin}} \ar[d]^-{c} &&& \Sigma^{-2}I_\Z \MSpin^c/\MSpin \ar[d]^-{I_\Z C\iota'} \\
    \K \ar[rrr]^-{\gamma_{\spin^c}} &&& \Sigma^{-2}I_\Z \MSpin^c
    }}}. 
\end{align}

\subsubsection{Anderson self-duality of the $\eta c R$-cofiber sequence in terms of $\gamma_{\K/\KO}$}
This subsubsection corresponds to Subsection \ref{sec:xxx} in the main part. 
In this toy model, it is very straightforward to relate the secondary morphism $\gamma_{\K/\KO}$ with the Anderson self-duality of $\KO$: 

\begin{prop}\label{prop_main_toymodel}
    Under the identification $\K/\KO \simeq \Sigma^2 \KO$ explained below \eqref{diag_ABS_rel}, the element $\gamma_{\K/\KO} \in \pi_{2}I_\Z \K/\KO \simeq \pi_4 I_\Z \KO$ coincides with the Anderson self-duality element of $\KO$. 
    In particular, the morphisms
    \begin{align}
        \gamma_{\K/\KO} \colon \K/\KO \to \Sigma^{-2}I_\Z \KO, \quad \gamma^{\vee}_{\K/\KO} &\colon \KO \to \Sigma^{-2} I_\Z \K/\KO 
    \end{align}
    are equivalences of $\KO$-modules. 
\end{prop}
\begin{proof}
    The self-duality element of $\KO$ is the generator of $\pi_4 I_\Z \KO \simeq \Z$, which is unique up to a sign. 
    Our definition of $\gamma_{\K/\KO}$ in Proposition \ref{prop_def_gamma_K/KO} is clearly equivalent to that. 
\end{proof}

Moreover, it is clear from the construction in the last subsubsection that $\gamma_{\K/\KO}$ and $\gamma_K$ show the Anderson self-duality of the $\eta c R$-cofiber sequence \eqref{eq_eta_c_R} as follows. 
\begin{prop}
    We have a commutative diagram
    \begin{align}
        \vcenter{\xymatrix{
        \KO \ar[r]^-{c} \ar[d]_-{\simeq}^-{\gamma_{\K/\KO}^\vee} & \K \ar[r]^-{R} \ar[d]_-{\simeq}^-{\gamma_\K} & \K/\KO \simeq \Sigma^{2}\KO \ar[d]_-{\simeq}^-{\gamma_{\K/\KO}} \\
        \Sigma^{-4}\KO \simeq \Sigma^{-2}I_\Z \K/\KO \ar[r]^-{I_\Z R}  & \Sigma^{-2}I_\Z \K \ar[r]^-{I_\Z c} &  \Sigma^{-2}I_\Z \KO
        }}
    \end{align}
    where the rows are $\eta c R$-cofiber sequence \cite{BrunerGreenlees} and its Anderson dual, and the vertical arrows are equivalences. 
\end{prop}

\subsection{Pairings induced by secondary morphisms}
This subsection corresponds to Section \ref{sec:pairings} in the main body of the paper. 
We study of the pairings induced by our secondary morphisms on the homotopy groups of $\KO$, $\K/\KO$ and $\MSpin^c/\MSpin$.
Of course our secondary morphism is a re-interpretation of the familiar Anderson self-duality of $\KO$ by Proposition \ref{prop_main_toymodel}, but our interpretation naturally connects it with the relative bordisms in an interesting way. 
Here we record the information of bordism groups which are relevant for us in this section in Table~\ref{table_toy}.
There, all the $S^1$'s appearing in the table are equipped with the nonbounding spin structure.\footnote{%
Most of the data there are either well-known or easy to determine. 
For example, the map $\iota':\pi_4(\MSpin)\to \pi_4(\MSpin^c)$
can be found by  tabulating the characteristic numbers: \[
\begin{array}{c|ccc}
& [K3] & [\mathbb{CP}^2] &  [(\mathbb{CP}^1)^{\times 2} ]  \\
\hline
\frac{p_1}{24}-\frac{(c_1)^2}{8} & -2 & 0 & -1\\
(c_1)^2  & 0 & 1 & 8
\end{array}
\].
It is also known that $\MSpin^c/\MSpin \simeq \Sigma^{2} \MSpin \wedge BU(1)_+$ \cite[Example 7.26]{many},
which can also be used to find $\pi_n(M\Spin^c/\MSpin)$.
}

\begin{table}[h]
\[
\begin{array}{c|@{\quad}cccccc}
n & \stackrel{\partial}{\longrightarrow }&\pi_n(\MSpin)&\stackrel{\iota'}{\longrightarrow}& \pi_n(\MSpin^c) &\stackrel{C\iota'}{\longrightarrow}& \pi_n(\MSpin^c/\MSpin) \\
\hline
4 &\longrightarrow& \bZ \cdot [K3] &\stackrel{(16, -2)}{\longrightarrow}& 
\begin{array}{r@{\,}c@{\,}l}
\bZ &\cdot& [\mathbb{CP}^2]\vphantom{\text{\Huge H}} \\
\oplus \bZ &\cdot& [(\mathbb{CP}^1)^{\times 2} ] 
\end{array}
&{\longrightarrow}& 
\begin{array}{r@{\,}c@{\,}l}
\bZ/2 &\cdot&  [8\overline{\mathbb{CP}^2} \sharp (\mathbb{CP}^1)^{\times 2}, \varnothing] \vphantom{\text{\Huge H}}\\ 
\oplus \bZ& \cdot& [\mathbb{CP}^2 , \varnothing] 
\end{array}
\\
3 &\longrightarrow& 0 &\longrightarrow& 0 &\longrightarrow& \Z/2 \cdot [D^2 \times S^1, (S^1)^{\times 2} ] \\
2 &\stackrel{\simeq}{\longrightarrow}& \bZ/2  \cdot [(S^1)^{\times 2} ]  &\stackrel{0}{\longrightarrow}& \bZ\cdot[\mathbb{CP}^1]&\stackrel{\times 2}{\longrightarrow}& \bZ \cdot [D^2, S^1]  \\
1 &\longrightarrow& \bZ/2 \cdot [S^1]   &\longrightarrow& 0&\longrightarrow& 0\\
0 &\longrightarrow& \bZ &\longrightarrow& \bZ &\longrightarrow& 0
\end{array}
\]

\caption{Homotopy groups of $\MSpin^c$, $\MSpin$ and $\MSpin^c/\MSpin$. }
\label{table_toy}
\end{table}

\subsubsection{Induced non-torsion pairing%  and the Adams $e$-invariants
}
Applying Definition \ref{def_int_pairing} to $\gamma_{\K/\KO}$ and $\gamma_{\spin^c/\spin}$, we get the non-torsion pairings
\begin{align}
    \langle - , - \rangle_{\gamma_{\K/\KO}} &\colon \pi_{-d} \KO \times \pi_{d-2} \K/\KO \to \Z, \label{eq_pairing_free_K/KO}\\
    \langle - , - \rangle_{\gamma_{\spin^c/\spin}}&\colon \pi_{-d} \KO \times \pi_{d-2} \MSpin^c/\MSpin \to \Z. \label{eq_pairing_free_MSpinc/MSpin}
\end{align}
The result of the pairing \eqref{eq_pairing_free_K/KO} is simply as follows. 
\begin{prop}\label{prop_free_pairing_toy}
    The pairing \eqref{eq_pairing_free_K/KO} is nontrivial only for $d \equiv 0 \pmod 4$. For those cases, we have $\pi_{-d} \KO \simeq \Z$ and $\pi_{d-2} \K/\KO \simeq \Z$, and the pairing is a perfect pairing $\Z \times \Z \to \Z$. 
\end{prop}
\noindent We can prove Proposition \ref{prop_free_pairing_toy} easily by definition. We can also resort to Proposition \ref{prop_main_toymodel}. 

Next let us consider the pairing \eqref{eq_pairing_free_MSpinc/MSpin}.
As it is the composition of $\ABS_{\spin^c/\spin} $ %\colon \MSpin^c/\MSpin$ 
and the pairing \eqref{eq_pairing_free_K/KO}, we can understand the result from that aspect.
But here we give an example of a rather direct computation of the pairing, which is in the spirit of the main body of the paper, c.f., Subsubection \ref{sec_example_free}. 
\begin{ex}[The pairing with $\pi_2 \MSpin^c/\MSpin$]
    Let us consider $d=4$ case. We refer to Table \ref{table_toy} to see $\pi_2 \MSpin^c/\MSpin = \Z \cdot [D^2, S^1]$. 
In contrast, if we denote a generator of $\pi_{-4} \KO \simeq \Z$ by $A$, the complexification $c \colon \pi_{-4} \KO \to \pi_{-4} \K$ maps $A$ to $\pm 2\beta^{-2}$, where $\beta \in \pi_2 \K$ is the complex Bott element. Let us fix the sign of $A$ so that it satisfies $c(A) = 2\beta^{-2}$. 
Then we claim that
\begin{align}\label{eq_pairing_free_toy_ex}
    \langle A, [D^2, {S^1}] \rangle_{\gamma_{\spin^c/\spin}} = 1. 
\end{align}
This can be proved by  the following compatibility of the pairings: 
\begin{align}
    \vcenter{\xymatrixcolsep{0pc}\xymatrix{
    \pi_{-4}\KO \ar[d]^-{c} & \times & \pi_{2}\MSpin^c/\MSpin  \ar[rrrrrr]^-{\langle -, - \rangle_{\gamma_{\spin^c/\spin}}} &&&&&& \Z \ar@{=}[d]\\
   \pi_{-4}\K \ar@{=}[d]  & \times&\pi_{2}\MSpin^c \ar[rrrrrr]^-{\langle -, - \rangle_{\gamma_{\spin^c}}}\ar[u]^-{C\iota'} \ar[d]^-{\ABS_{\spin^c}} &&&&&& \Z  \ar@{=}[d] \\
   \pi_{-4}\K  & \times&\pi_{2}\K \ar[rrrrrr]^-{\langle -, - \rangle_{\gamma_{\K}}} &&&&&& \Z
    }}. 
\end{align}
We compute
\begin{align}
    2\langle A, [D^2, {S^1}] \rangle_{\gamma_{\spin^c/\spin}}
    &=\left\langle A, C\iota' \left([\mathbb{CP}^1] \right) \right\rangle_{\gamma_{\spin^c/\spin}} \\
    &= \left\langle c(A), [\mathbb{CP}^1]\right\rangle_{\gamma_{\spin^c}} \\
    &= 2\left\langle \beta^{-2}, \ABS_{\spin^c}([\mathbb{CP}^1])\right\rangle_{\gamma_{\K}} = 2\left\langle \beta^{-2}, \beta\right\rangle_{\gamma_{\K}}= 2, 
\end{align}
verifying \eqref{eq_pairing_free_toy_ex}. 
\end{ex}

\subsubsection{Induced torsion pairings}
Now we discuss the torsion pairings. Applying Definition \ref{def_torsion_pairing} to $\gamma_{\K/\KO}$ and $\gamma_{\spin^c/\spin}$, we get the torsion pairings
\begin{align}
    ( - , - )_{\gamma_{\K/\KO}} &\colon (\pi_{-d} \KO)_\tor \times (\pi_{d-3} \K/\KO)_\tor \to \Q/\Z, \label{eq_pairing_tor_K/KO}\\
    ( - , - )_{\gamma_{\spin^c/\spin}}&\colon (\pi_{-d} \KO)_\tor \times (\pi_{d-3} \MSpin^c/\MSpin)_\tor \to \Q/\Z. \label{eq_pairing_tor_MSpinc/MSpin}
\end{align}
The result of the pairing \eqref{eq_pairing_tor_K/KO} is simply as follows. 
\begin{prop}\label{prop_tor_pairing_toy}
    The pairing \eqref{eq_pairing_tor_K/KO} is nontrivial only for $d \equiv -1, -2 \pmod 8$. For those cases, we have $(\pi_{-d} \KO)_\tor \simeq \Z/2$ and $(\pi_{d-2} \K/\KO)_\tor \simeq \Z/2$, and the pairings are the perfect pairing $\Z/2 \times \Z/2 \to \frac{1}{2}\Z/\Z \subset \Q/\Z$. 
\end{prop}
Again, Proposition \ref{prop_tor_pairing_toy} directly follows from Proposition \ref{prop_main_toymodel}. 
But here, we want to point out that it relies on the knowledge of the Anderson self-duality of $\KO$. It is quite nontrivial to compute the result of the torsion pairing directly from the definition of $\gamma_{\K/\KO}$ and that of torsion pairings. 
This is exactly analogous to the challenges we faced in the main part of the paper, especially in Section \ref{sec:diff}. 
Our strategy is to compute the geometric pairing \eqref{eq_pairing_tor_MSpinc/MSpin} instead. We do it by refining the pairings to {\it differential} pairings explained in the next section. 
More concretely, we will establish Proposition \ref{prop_tor_pairing_toy} by proving the following two propositions using the differential pairing:

\begin{prop}\label{claim_toy_7}
    The torsion pairing \eqref{eq_pairing_tor_MSpinc/MSpin} for $d = 7$ gives
    \begin{align}\label{eq_toy_7}
        ( \eta B^{-1},[8\overline{\mathbb{CP}^2} \sharp (\mathbb{CP}^1)^{\times 2}, \varnothing] )_{\gamma_{\spin^c/\spin}} = \frac{1}{2}. 
    \end{align}
\end{prop}

\begin{prop}\label{claim_toy_6}
    The torsion pairing \eqref{eq_pairing_tor_MSpinc/MSpin} for $d = 6$ gives
    \begin{align}\label{eq_toy_6}
        ( \eta^2 B^{-1}, [D^2 \times S^1, (S^1)^{\times 2} ] )_{\gamma_{\spin^c/\spin}} = \frac{1}{2}. 
    \end{align}
\end{prop}

Here we denote the real Bott element by $B \in \pi_8 \KO$ and let $\eta$ be the generator of $\pi_1\KO\simeq \bZ/2$, so that $\pi_{-7} \KO = \Z/2 \cdot \eta B^{-1}$ and $\pi_{-6} \KO = \Z/2 \cdot \eta^2 B^{-1}$. 
Please also refer to Table \ref{table_toy} for the elements of relative Spin$^c$/Spin bordism groups.

\subsubsection{Induced differential pairing}\label{subsec_diff_pairing_toy}
In Subsection \ref{subsec_diff_pairing} we explained the Anderson duality non-torsion/torsion pairings are combined and refined to the {\it differential pairing} between differential cohomology and homology. 
Applying the general theory explained there to $\gamma_{\spin^c/\spin}$ here, we get the differential pairing
\begin{align}\label{eq_toy_diff_pairing}
    \widehat{\gamma}_{\spin^c/\spin } \colon \widehat{\KO}^{d} \times \reallywidehat{\MSpin^c/\MSpin}_{d-3} \to \R/\Z.  
\end{align}
It is related to the non-torsion pairing via Lemma \ref{lem_diff_generalized_pairing_compatibility} and to the torsion pairing via Lemma \ref{lem_diff_pairing_torsion}. 
The morphism $\gamma_{\spin^c} \colon \K \to \Sigma^{-2}I_\Z \MSpin^c$ also induces the corresponding differential pairing. They are compatible in the following way:
\begin{align}\label{diag_compatibility_KO_K_pairing}
    \vcenter{\xymatrixcolsep{0pc}\xymatrix{
    \widehat{\KO}^d \ar[d]^-{c} & \times & \reallywidehat{\MSpin^c/\MSpin}_{d-3}  \ar[rrrrrr]^-{\widehat{\gamma}_{\spin^c/\spin}} &&&&&& \R/\Z \ar@{=}[d]\\
   \widehat{\K}^d   & \times&\reallywidehat{\MSpin^c }_{d-3}\ar[rrrrrr]^-{\widehat{\gamma}_{\spin^c}}\ar[u]^-{C\iota'}  &&&&&& \R/\Z   
    }}.
\end{align}

At this point, we can prove Proposition \ref{claim_toy_7} as follows. 
\begin{proof}[Proof of Proposition \ref{claim_toy_7}]
    We use the compatibility \eqref{diag_compatibility_KO_K_pairing} for $d=7$. Following the axiom of differential cohomology, we have
    \begin{align}
        \widehat{\KO}^7 &\stackrel{I}{\simeq}  \KO^7(\pt) = \Z\cdot \eta B^{-1} /(2\eta B^{-1}) \simeq \Z/2, \\
       \widehat{\K}^7  &\stackrel{a}{\simeq}\K^{6}(\pt) \otimes \R/\Z = (\R\cdot \beta^{-3})/(\Z \cdot \beta^{-3}) \simeq \R/\Z. 
    \end{align}
    We claim that the differential refinement of the complexification map is the injection, i.e., given by
    \begin{align}\label{eq_diff_complexification}
        c \colon \widehat{\KO}^7 \to \widehat{\K}^7, \quad \eta B^{-1} \mapsto \frac{1}{2}\beta^{-3} . 
    \end{align}
    This is because of the following. Multiplying by the real periodicity element $B \in \pi_8\KO \simeq \widehat{\KO}^{-8}$ which induces the periodicities in $\widehat{\KO}$ and $\widehat{\K}$, it is enough to show the corresponding claim for $c \colon \widehat{\KO}^{-1} \simeq \Z\eta/(2\eta) \to \widehat{\K}^{-1} \simeq \R \beta/(\Z \beta)$. 
    We use the fact that elements in $\widehat{\K}^{-1}$ and $\widehat{\KO}^1$ is realized by the {\it differential integration}, or equivalently  the {\it differential pushforward}. For a general account, see \cite{YamashitaAndersondualPart2} for example. They can be regarded as pairings between differential generalized cohomology and differential bordism homology, and in our case consist of the following compatibility diagram. 
    \begin{align}\label{diag_compatibility_diff_int_K}
        \vcenter{\xymatrixcolsep{0pc}\xymatrix{
        \widehat{\KO}^N \ar[d]^-{c} & \times & \reallywidehat{\MSpin}_{n} \ar[d]^-{\iota'} \ar[rrrrrr]^-{\int^{\widehat{\KO}}_{(-)}(-)} &&&&&& \widehat{\KO}^{N-n} \ar[d]^-{c}\\
   \widehat{\K}^N   & \times&\reallywidehat{\MSpin^c }_{n}\ar[rrrrrr]^-{\int^{\widehat{\K}}_{(-)}(-)} &&&&&& \widehat{\K}^{N-n}   
        }}
    \end{align}
    We appply this to $N=0$ and $n=1$. In this case the upper row reduces to the topological $\KO$-integration, which is
    \begin{align}
        \int^{\widehat{\KO}}_{(S^1, \rm{periodic})}1_{\widehat{\KO}} = \eta \in \frac{\Z \cdot \eta}{2 \cdot \eta} = \widehat{\KO}^{-1}
    \end{align}
    On the other hand, in the lower row, the element $\iota'(S^1, \text{periodic}) \in \widehat{\MSpin^c}_n$ is $S^1$ equipped with the spin$^c$-structure whose $U(1)$-connection has holonomy $-1$. The $\widehat{\K}^{-1} \simeq \R/\Z$-valued integration takes this holonomy, so that
    \begin{align}
        \int^{\widehat{K}}_{\iota'(S^1, \rm{periodic})}1_{\widehat{\K}} = \frac{1}{2} \beta \in \frac{\R \cdot \beta}{\Z \cdot \beta} = \widehat{\K}^{-1}. 
    \end{align}
    The compatibility in \eqref{diag_compatibility_diff_int_K} and the fact that $c(1_{\widehat{\KO}}) = 1_{\widehat{\K}}$ implies that $c(\eta)=\frac{1}{2}\beta$, verifying \eqref{eq_diff_complexification}.

    Now we use the compatibility of the various Anderson duality pairings to compute the left hand side of \eqref{eq_toy_7}. We compute
    \begin{equation}
    \begin{aligned}
         &\langle \eta B^{-1}, [8\overline{\mathbb{CP}^2} \sharp (\mathbb{CP}^1)^{\times 2}, \varnothing]\rangle_{\gamma_{\spin^c/\spin}} \\
         &=
         \langle \eta B^{-1}, [8\overline{\mathbb{CP}^2} \sharp (\mathbb{CP}^1)^{\times 2}, \varnothing]_\diff \rangle_{\widehat{\gamma}_{\spin^c/\spin}} && \mbox{ by Lemma \ref{lem_diff_pairing_torsion}} \\
         &= \langle \frac{1}{2}\beta^{-3}, [8\overline{\mathbb{CP}^2} \sharp (\mathbb{CP}^1)^{\times 2}]_\diff \rangle_{\widehat{\gamma}_{\spin^c}} &&\mbox{ by \eqref{diag_compatibility_KO_K_pairing}}\\
         &=\frac{1}{2} \langle \beta^{-3}, [8\overline{\mathbb{CP}^2} \sharp (\mathbb{CP}^1)^{\times 2}]_\diff\rangle_{(\gamma_{\spin^c})_\R} && \mbox{ by Lemma \ref{lem_diff_generalized_pairing_compatibility}} \\
         &= \frac{1}{2} \mod 1, 
    \end{aligned}
    \end{equation}
    where the last equality follows from the fact that $\mathrm{ABS} \colon \pi_4 \MSpin^c \to \pi_4 \K$ sends both $[\mathbb{CP}^2]$ and $[(\mathbb{CP}^1)^{\times 2}]$ to $\beta^2$ and the definition of $\gamma_\K$. This completes the proof of Proposition \ref{claim_toy_7}. 
\end{proof}

Now we are left to prove Proposition \ref{claim_toy_6}. 
Here we give a proof which uses a simplified version of the geometric trick used in the proof of Proposition \ref{prop:pairing}. 
We use the following variant of the differential pairing \eqref{eq_toy_diff_pairing}.\footnote{%
See footnote \ref{footnote_pairing} for details. Here we are allowed to use genuine twisted equivariant $\KO$, not the Borel equivariant one. The pairing factors through the Atiyah-Segal completion $\KO_G^{d+\tau} \to \KO^{d+\tau}(BG)$. In the main text we could not use the genuine twisted equivariant $\TMF$ since it has not been established currently, but we would have preferred to use the genuine one, if one were available. See also footnote \ref{footnote_genuine}.} 
If $G$ is a compact Lie group and {a map $\tau \colon BG \to BO\langle 0, 1, 2\rangle$} specifies a twist of $G$-equivariant $\KO$-theory, we get the differential pairing
\begin{align}\label{eq_diff_pairing_toy_equiv}
    {\KO}^{d+\tau}_G \times \left((\reallywidehat{\MSpin^c/\MSpin) \wedge_{\tau}BG_+}\right)_{d-3}  \to \R/\Z
\end{align}
provided that we have $\KO_\R^{(d-1)+\tau}(BG) =0$. This condition is satisfied if $d$ is even. 
We use this refinement to prove Proposition \ref{claim_toy_6} below.

For our purposes, we pick $G=U(1)$ and choose $\tau \colon BU(1) \to BO \langle 0, 1, 2 \rangle$ so that the homotopy class
\begin{equation}
[\tau] \in [BU(1), BO \langle 0, 1, 2 \rangle] \simeq H^2(BU(1); \Z/2)\simeq \Z/2
\end{equation}
is the unique nontrivial element. 
This classifies the extension $0\to \Z/2\to  U(1) \xrightarrow{(-)^2} U(1) \to 0$.

We can use $\tau$ to twist both $\KO$ and $\K$,
and similarly both $\MSpin$ and $\MSpin^c$.
In particular, 
we have a canonical equivalence \begin{equation}
f: \MSpin^c \xrightarrow{\sim} \MSpin\wedge_\tau BU(1)_+ ,
\label{ffff}
\end{equation}
i.e.~a $(U(1), \tau)$-twisted spin structure is equivalent to a spin$^c$ structure.
We will also need the morphism
\begin{equation}
\iota' := \iota' \wedge_\tau \id_{BU(1)_+}: \MSpin\wedge_\tau BU(1)_+ \to \MSpin^c \wedge_\tau BU(1)_+.
\label{iotaiota}
\end{equation}
\if0
There is a further equivalence \begin{equation}
\MSpin^c \wedge_\tau BU(1) \simeq \MSpin^c \wedge BU(1)
\end{equation}
\footnote{%
This is further equivalent to the pair of a spin$^c$ structure and a $U(1)$ bundle,
corresponding to the fact that $\tau$ is null-homotopic as a twist of $\K$.
We opt not to use this fact, as this null homotopy involves an unnecessary choice which makes the computation more confusing.
}
\fi

By fixing a null-homotopy $2\tau \sim 0$,
we can choose an element $V\in \K_{U(1)}^\tau(\pt)$ such that $\K_{U(1)}^0(\pt) \simeq \K_{U(1)}^{2\tau}(\pt)=\Z[V^2,V^{-2}]$ and $\K_{U(1)}^{\tau}(\pt)=V \cdot \K_{U(1)}^0(\pt)$.
By an explicit computation using Clifford modules, one finds that there exists an element $W_{2n-1} \in \KO_{U(1)}^{-2 + \tau}(\pt)$ for each $n \ge 1$ such that \begin{equation}
c(W_m) = \beta(V^m - V^{-m}). \label{cccc}
\end{equation}
so that we have \begin{equation}
\KO^{-2+\tau}_{U(1)}(\pt) =\bigoplus_{n\ge 1}  \Z  W_{2n-1} . 
\end{equation}
We fix an inclusion $i \colon \pt \to BU(1)$, with which we have $i^*: \KO^{d+\tau}_{U(1)}(\pt) \to \KO^d(\pt)$.
We find that \begin{equation}
i^* W_{2n-1} = \eta^2 \label{dddd}
\end{equation}
for all $n \ge 1$. 
After these preparations, we can proceed to the proof.

\begin{proof}[Proof of Proposition \ref{claim_toy_6}]
We apply the pairing \eqref{eq_diff_pairing_toy_equiv} for $d=6$, $G=U(1)$.
We have the following compatibility between differential pairings. 
\begin{align}
    \vcenter{\xymatrixcolsep{0pc}
    \xymatrix{
    \widehat{\KO}^6 \simeq \KO^6  & \times & \reallywidehat{\MSpin^c/\MSpin}_{3}  \ar[rrrrrr]^-{\widehat{\gamma}_{\spin^c/\spin}} \ar[d]^-{i_*} &&&&&& \R/\Z \ar@{=}[d]\\
  {\KO}^{6+\tau}_{U(1)} \ar[u]^-{i^*} \ar[d]^{c} & \times&\left((\reallywidehat{\MSpin^c/\MSpin) \wedge_{\tau}BU(1)_+}\right)_{3} \ar[rrrrrr]^-{\widehat{\gamma}_{\spin^c/\spin}}  &&&&&& \R/\Z \ar@{=}[d] \\
  {\K}^{6+\tau}_{U(1)} &\times& (\reallywidehat{\MSpin^c \wedge_\tau BU(1)_+})_{3} \ar[rrrrrr]^-{\widehat{\gamma}_{\spin^c}} \ar[u]^-{C\iota'}&&&&&& \R/\Z 
    }}\label{u1twisted}
\end{align}
%Then we have $i_*[(S^1)^{\times 2}] = 0 \in \MSpin_{2+\tau}(BU(1))\simeq \pi_2\MSpin^c$,
%since $[S^1] = 0\in \pi_1\MSpin^c=0$. 
%The pair $(D^2, P_{D^2})$ gives a $(U(1), \tau)$-twisted spin nullbordism of $S^1$. Here $P_{D^2}$ is a principal $U(1)$-bundle on $D^2$ with a trivialization on the boundary $\partial D^2 \simeq S^1$, which represents the fundamental class in $H^2(D^2, S^1; \Z)$. 
%Let us take a connection $\nabla_{D^2} $ on $P_{D^2}$ which is compatible with the trivialization on $S^1$. 
Recall we use  a relative spin$^c$/spin manifold $(D^2,S^1)$.
By using $f$ \eqref{ffff}, we have a relative \{$(U(1),\tau)$-twisted spin\} / spin manifold $f(D^2,S^1)$.
By further applying $\iota' $ \eqref{iotaiota}, we have a relative \{$(U(1),\tau)$-twisted spin$^c$\}/ spin manifold
$\iota'\circ f (D^2,S^1)$. 
We can now glue the orientation reversal of $\iota'\circ f (D^2,S^1)$
and the original $(D^2,S^1)$ to have an $S^2$ equipped a $(U(1),\tau)$-twisted spin$^c$ structure. 
We denote its bordism class by $[S^2,P_{S^2}] \in \pi_2(\MSpin^c \wedge_\tau BU(1)_+)$, where $P_{S^2}$ symbolically stands for the entire data involved.
We can perform this construction at the level of differential relative cycles. We typically denote by $[-]_\diff$ the corresponding element in differential bordism homology groups. 

From this observation, we have the following equality in $\left((\reallywidehat{\MSpin^c/\MSpin) \wedge_{\tau}BU(1)_+}\right)_{2}$,
\begin{align}
    i_* [D^2, S^1]_\diff = 
    C\iota'([S^2 ,P_{S^2}]_\diff) + \iota'\circ f([D^2, S^1]_\diff) .
\end{align}
With this, we can evaluate the pairing as follows:
\begin{align}
         &\langle \eta^2 B^{-1}, [D^2\times S^1,(S^1)^{\times 2}] \rangle_{\gamma_{\spin^c/\spin}}  \\
         &=
         \langle \eta^2 B^{-1}, [D^2\times S^1,(S^1)^{\times 2}]_\diff \rangle_{\widehat{\gamma}_{\spin^c/\spin}} && \mbox{by Lemma \ref{lem_diff_pairing_torsion}}\\
         &= \langle W_1B^{-1},  C\iota'(X_\diff)\rangle_{\widehat{\gamma}_{\spin^c/\spin}} 
         +\langle  W_1B^{-1} ,  Y_\diff\rangle_{\widehat{\gamma}_{\spin^c/\spin}} 
         && \mbox{by \eqref{dddd}}
          \label{YYYY}
\end{align}
where \begin{align}
X_\diff &=[S^2 ,P_{S^2}]_\diff \times [S^1]_\diff \in \reallywidehat{\MSpin^c\wedge_\tau BU(1)_+}_{3}, \\
Y_\diff &=\iota'\circ f([D^2, S^1]_\diff)\times [S^1]_\diff \in \left((\reallywidehat{\MSpin^c/\MSpin) \wedge_{\tau}BU(1)_+}\right)_{3} .
\end{align}
The proof concludes by showing that the first term of \eqref{YYYY} is $1/2$ 
and that the second term of \eqref{YYYY} is $0$.

As for the first term, 
we have \begin{align}
\langle  W_1B^{-1} ,  C\iota'(X_\diff)\rangle_{\widehat\gamma_{\spin^c/\spin}}
&=\langle c(W_1B^{-1}) ,  X_\diff \rangle_{\widehat\gamma_{\spin^c}}
&&\text{by \eqref{u1twisted}} \\
&=\langle   (V-V^{-1})\beta^{-3} ,  X_\diff \rangle_{\widehat\gamma_{\spin^c}} 
&&\text{by \eqref{cccc}}.
\label{pppp}
\end{align}
The pairing \eqref{pppp} is computed in terms of the differential pushforward
as the eta invariant of a Dirac operator on the $(\mathrm{Spin}^c \times U(1))/(\Z/2)$-bundle tensored with the vector bundle associated to $V-V^{-1}$ on $X_\diff$ \cite{YamashitaAndersondualPart2}. This can be evaluated using the product formula for the eta invariants as \begin{align}
\langle   (V-V^{-1})\beta^{-3}  ,  X_\diff \rangle_{\widehat\gamma_{\spin^c}} &= \eta((S^2 ,P_{S^2}, \nabla_{S^2})\times S^1) \mod \Z \\
&= -\mathrm{ind}((S^2 ,P_{S^2}, \nabla_{S^2})) \cdot \eta(S^1) \mod \Z \\
&=  -1\cdot \frac 12 \mod \Z.
\end{align}

As for the second term,
%We start from the pairing of $R_{U(1)}(\beta^{-2})$ and $Y_\diff$.
we actually have
\begin{align}\label{eq_Ydiff=0}
    Y_\diff = 0 \in \left((\reallywidehat{\MSpin^c/\MSpin) \wedge_{\tau}BU(1)_+}\right)_{3}. 
\end{align}
This can be deduced from the following general lemma about differential (co)homology. 
\begin{lem}
\label{lem:above}
Let $\widehat{E}_*$ be a generalized differential homology theory which is a module over a generalized multiplicative differential homology theory $\widehat{F}_*$,
with  a multiplication map of the form
\begin{align}
    -\cdot- \colon \widehat{F}_m(\pt) \times \widehat{E}_n(Z) \to \widehat{E}_{n+m}(Z)
\end{align}
for a manifold $Z$. 
Let us take elements $\widehat{e} \in \widehat{E}_n(Z)$ and $\widehat{f} \in \widehat{F}_m(\pt)$, and suppose that $I(\widehat{e})=0$ and $I(\widehat{f}) \in \pi_m F$ is a torsion element.
Then $\widehat{f} \cdot \widehat{e} =0$.

Here we use the notation for structure maps of differential homology as in \cite{YamashitaDifferentialIE},
\begin{align}
    \xymatrix{
    \Omega_{n+1}(Z ; (E_\R)_\bullet) / \mathrm{im}(\del) \ar[r]^-{a} & \widehat{E}_n(Z) \ar[r]^-{I} \ar[d]^-{R} & E_n(Z) \ar[r] & 0 \\
    &\Omega_{n}^{\mathrm{clo}}(Z ; (E_\R)_\bullet)&&
    },
\end{align}
where the horizontal sequence is exact. 
For $\widehat{F}$ the input is $\pt$ and it becomes
\begin{align}\label{diag_diff_homology_structure_pt}
    \xymatrix{
    \pi_{m+1}F_\R \ar[r]^-{a} & \widehat{F}_m(\pt) \ar[r]^-{I} \ar[d]^-{R} & \pi_m F \ar[r] \ar[ld]^-{\otimes \R} & 0 \\
    &\pi_m F_\R&&,
    }
\end{align}
where the triangle commutes. 
\end{lem}
\begin{proof}
Take a differential current $\eta \in \Omega_{n+1}(Z ; V_\bullet^E) / \mathrm{im}(\del) $ so that $a(\eta) = \widehat{e}$. 
In this setting, we have the following equation in $\widehat{E}_{n+m}(X)$, 
\begin{align}\label{eq_diff_hom_torsion_multiplication}
   \widehat{f} \cdot \widehat{e} = a\left(R(\widehat{f}) \cdot {\eta}\right) . 
\end{align}
Here the multiplication on the right hand side is the map 
\begin{equation}
\pi_mF_\R \times \Omega_{n+1}(Z; (E_\R)_\bullet) \to \Omega_{n+m+1}(Z; (E_\R)_\bullet)
\end{equation}
 induced by the multiplication $\pi_mF_\R \times \pi_* E_\R \to \pi_{*+m} E_\R$ on the coefficient. 
The proof of \eqref{eq_diff_hom_torsion_multiplication} is exactly analogous to the corresponding well-known statement for differential cohomology. 
In particular, when $I(\widehat{f}) \in \pi_m F$ is a torsion element, by the commutativity of the triangle in \eqref{diag_diff_homology_structure_pt} we have $R(\widehat{f})=0$ so that 
%\begin{align}\label{eq_diff_hom_torsion_multiplication_0}
   $ \widehat{f} \cdot \widehat{e} = a\left(0\cdot {\eta}\right) =0. $
%\end{align}
\end{proof}

We apply this general lemma to our setting, where we set $E = (\MSpin^c/\MSpin) \wedge BU(1)_+$, $F = \MSpin$, $\widehat{e} = \iota'\circ f([D^2, S^1]_\diff)$ and $\widehat{f} = [S^1]_\diff$. 
We have
\begin{align}
    I(\iota'\circ f([D^2, S^1]_\diff)) = [D^2, S^1, P_{D^2}] = 0 \in (\MSpin^c/\MSpin)_{2+\tau}(BU(1)) 
\end{align}
by the same reason as the corresponding statement in the proof of Proposition \ref{prop:pairing} (in particular see Figure \ref{fig:corner}). 
Since we know that $I([S^1]_\diff) = \eta \in \pi_1 \MSpin$ is a torsion element, we can apply
the lemma above %\eqref{eq_diff_hom_torsion_multiplication_0}
 and get \eqref{eq_Ydiff=0} as desired. 
In particular, the second term in \eqref{YYYY} is zero. 

% We start from the pairing of $R_{U(1)}(\beta^{-2})$ and $Y_\diff$. We use
% Lemma \ref{lem_diff_generalized_pairing_compatibility}. In the setting here, the lemma implies the compatibility of the two pairings in the following diagram, 
% \begin{equation}
%     \hskip-5em\vcenter{\xymatrixcolsep{0pc}
%     \xymatrix{
%      {\KO}^{d+\tau}_{U(1)} \ar[d]^-{\gamma_{\spin^c/\spin}^\R}&\times& (\reallywidehat{\MSpin^c/\MSpin) \wedge_{\tau}BU(1)_+})_{d-3}\ar[rrr]^-{\gamma_{\spin^c/\spin}}  &&& \R/\Z  \\
%      \H^{d-2}((\MSpin^c/\MSpin) \wedge BU(1)_+; \R )&\times& \H_{d-2}((\MSpin^c/\MSpin) \wedge BU(1)_+; \R) \ar[rrr]^-{\langle \cdot , \cdot\rangle} \ar[u]^-{a} &&& \R \ar[u]^{\mod \Z}. 
%     }}
% \end{equation}

% Note that we have $[D^2, S^1, P] = 0 \in (\MSpin^c/\MSpin)_{2+\tau}(BU(1))$, which also implies that we have $[Y_\diff] = 0 \in (\MSpin^c/\MSpin)_{3+\tau}(BU(1))$. 
% Take a nullbordism $W_\diff$ of $(D^2, S^1, P_{D^2}, \nabla_{D^2})$ with $(U(1), \tau)$-twisted differential $\spin^c/\spin$-structure. 
% This means that the differential pairing is given by an integral of a relative differential form obtained by the Chern-Weil construction applied to $R_{U(1)}(\beta^{-2})$ on the $\langle 2 \rangle$-manifold which exhibits $Y_\diff$ as a boundary of relative $(U(1), \tau)$-twisted $\spin^c/\spin$-manifold.
% The relative differential form simply vanishes \textcolor{blue}{because },  and therefore this part of the pairing is zero.

As already explained above, this concludes the proof of this proposition.
\end{proof}

% !TEX root = paper.tex
\section{A rough physics translation}
\label{sec:introphys}

%This paper is written mostly for mathematicians,
%and all the sections except this appendix for physicists are written in the mathematical style.
We will provide the physics translation of salient points of this paper
in this appendix.
More applications written in the style of physics will be given elsewhere.
%Mathematicians can skip this section entirely and read the rest of the paper without problems.

In a previous paper \cite{TachikawaYamashita}, 
the authors showed the vanishing of all anomalies, both global and local, 
of any heterotic compactifications using arbitrary internal worldsheet CFTs.
This was done by translating this question using the Stolz-Teichner proposal
to a property of $\TMF$,  the topological modular forms.
This proposal says that, among others,
a two-dimensional \Nequals{(0,1)}-supersymmetric conformal field theory $T$ with a Lie group symmetry $G$ with level $k$
gives a class $[T]\in \TMF^{2(c_L-c_R)+ k\tau}_G(\pt)$ of the group of
the twisted equivariant topological modular forms.\footnote{%
In the untwisted case, we have $\TMF^n(\pt)=\TMF_{-n}(\pt)=\pi_{-n}\TMF$.
In the mathematical part of this paper, we used the notation $\pi_{-n}\TMF$ most often.}
This allowed an application of various techniques of algebraic topology to study this question.

In this sequel, we find that a refinement of the method employed in \cite{TachikawaYamashita}
allows us to determine the Green-Schwarz coupling of the $B$-field, in such a way that
we know 
not only the standard local part given by $B\wedge X_{d-2}$ where $X_{d-2}$ is a gauge-invariant $(d-2)$-form,
but also the discrete part responsible for canceling the global anomalies of the fermions. 
We will also explain that the computation can often be done
effectively using the Anderson duality of $\TMF$.
As examples, we find that certain heterotic backgrounds considered in \cite{KaidiOhmoriTachikawaYonekura,Kaidi:2024cbx}
have nonzero discrete `gravitational' theta angles.

\subsection{Green-Schwarz anomaly cancellation from a modern perspective}
\label{subsec:GS}
\subsubsection{Traditional account}
The standard textbook account of the Green-Schwarz anomaly cancellation goes as follows.
Suppose that a set of $d$-dimensional fermions has an anomaly described by an anomaly polynomial $P_{d+2}$, a gauge-invariant differential form of degree $d+2$ constructed
from the gauge and gravitational curvatures. 
Suppose further that $P_{d+2}$ admits a factorization of the form \begin{equation}
P_{d+2}=(\frac {p_1(R)}2 - \sum_i k_i \tau(F_i)) X_{d-2}
\label{factorized-anomaly}
\end{equation}
where $p_1(R)$ is the first Pontryagin class of the gravitational curvature,
$k_i$ are integers,
and $\tau(F_i)$ are the instanton densities of  the $i$-th gauge group.

To cancel this anomaly, we introduce a $B$-field such that its invariant gauge field strength $H$ satisfies \begin{equation}
dH=\frac {p_1(R)}2 - \sum_i k_i \tau(F_i),\label{GS}
\end{equation} and a spacetime coupling \begin{equation}
2\pi i \int BX_{d-2}.\label{BX}
\end{equation}
The standard argument then shows that the $B$-field with this property has an anomaly
which cancels the factorized anomaly given in \eqref{factorized-anomaly}.

This description is not perfectly satisfactory, however, because of a couple of reasons.
One is that we know that there can be global anomalies of fermions
not captured by an anomaly polynomial.
We therefore need a way to characterize which global fermion anomalies can be cancelled 
by a Green-Schwarz mechanism, and if so, how.
Another is to ensure the well-defined-ness of the coupling \eqref{BX}.
It requires us to understand the integrality properties of $X_{d-2}$.
The way to resolve these points has not been explicitly discussed  in the literature to the authors' knowledge,
but the necessary ingredients for this task can be found scattered in literature, see e.g.~\cite{Wang:2017loc,Kobayashi:2019lep,Hsieh:2020jpj,Yonekura:2020upo}.\footnote{%
For other related approaches to discrete Green-Schwarz terms,
see e.g.~\cite{Garcia-Etxebarria:2017crf,Lee:2022spd,Dierigl:2022zll}.
}

Very briefly and schematically, the way the perturbative cancellation is generalized is as follows.
The anomaly polynomial vanished when restricted on manifolds satisfying \eqref{GS}.
This meant that the anomaly polynomial $P_{d+2}$ contained a factor of $\frac12{p_1(R)}-\sum_i k_i\tau(F_i)$ as in \eqref{factorized-anomaly}.
From this, a non-vanishing factor $X_{d-2}$ can be derived by dividing $P_{d+2}$ by $\frac12{p_1(R)}-\sum_i k_i\tau(F_i)$,
and the resulting quotient $X_{d-2}$ was the main ingredient of the Green-Schwarz coupling \eqref{BX}.

In the previous paper \cite{TachikawaYamashita},
it was shown that the anomaly including the global part vanished when restricted on manifolds satisfying \eqref{GS}.
Using algebraic topology, we can make sense of the phrase `dividing the anomaly by $\frac12{p_1(R})-\sum_i k_i\tau(F_i)$' so that it applies to global anomalies.
We can furthermore show that the resulting quotient gives the Green-Schwarz coupling, again including the subtle global part.
Let us give an outline of how this is actually achieved, leaving a more detailed account to be presented elsewhere.\footnote{
Before proceeding, the authors thank Kazuya Yonekura for his help in formulating this presentation.
The authors also thank N.~Kawazumi for pointing out after the seminar by YM on \cite{TachikawaYamashita} that the vanishing of the anomaly as formulated in algebraic topology as a morphism 
should provide us with a nonzero secondary morphism and asking for its physics interpretation.
This eventually led to the realization by the authors, during a collaboration leading to \cite{KaidiOhmoriTachikawaYonekura,Kaidi:2024cbx}, 
that the secondary morphism has a natural interpretation as encoding the Green-Schwarz coupling in the language of algebraic topology.
}

\subsubsection{Modern account}
We first discuss the case when there are only gravitational anomalies. 
The anomaly of a $d$-dimensional fermion system $\Psi$ is given by a spin invertible phase $\cI$ in $(d+1)$ dimensions.
The deformation class $[\cI]$ of such invertible phases takes values in 
$(I_\bZ\Omega^\text{spin})^{d+2}(\pt)$,
the Anderson dual to the spin bordism group \cite{FreedHopkins2021}.
For a leisurely account for this mathematical object for physicists explaining its necessity, 
the readers are referred to \cite[Sec.~2]{Lee:2020ojw}.

The invertible theory $\cI$ then assigns: \begin{itemize}
\item  a one-dimensional Hilbert space $\cH_\cI(M_d)$ for a closed $d$-dimensional spin manifold $M_d$ with metric,
\item and a unitary transformation $Z_{\cI}(N_{d+1}): \cH_\cI(M_d)\to \cH_\cI(M_d')$  
for a  $(d+1)$-dimensional spin manifold $N_{d+1}$ with metric such that 
$\partial N_{d+1}$ consists of an incoming boundary $M_d$ and an outgoing boundary $M_d'$.
\end{itemize}
In particular, for an empty boundary $M_d=\varnothing$, we have $\cH_\cI(\varnothing)=\bC$.
Then for $N_{d+1}$ with a single outgoing boundary, 
we have a Hartle-Hawking wavefunction $Z_\cI(N_{d+1}) \in \cH_\cI(\partial M_d)$,
and for $N_d$ without boundary, we have the partition function $Z_\cI(N_{d+1})\in U(1)$.

With the help of the bulk invertible phase $\cI$, we can now formulate 
the partition function of the fermion system $\Psi$ on $M_d$ as 
$Z_\Psi(M_d)\in \cH_\cI(M_d)$.
Although $\cH_\cI(M_d)$ is one dimensional,
there is no canonical basis vector,
therefore we cannot convert $Z_\Psi(M)$ into a number.

We now consider extending the structure of the spacetime by introducing the $B$-field such that its gauge invariant field strength $H$ satisfies \begin{equation}
dH=\frac{p_1(R)}2.\label{dH}
\end{equation}
This is an expression at the level of differential forms,
but it can be refined to include discrete topological data.
The entirety of such data is known as the differential string structure \cite{Sati:2009ic,FSSdifferentialtwistedstring}, which is reviewed mathematically in Appendix \ref{app_subsec_diff_string}.
We denote a string manifold by $M^H$
where $M$ is a spin manifold and $H$ symbolically denotes the $B$-field.
We sometimes drop the superscript $H$ when it can be understood from the context.

In many cases of interest to physicists, we often need to consider fermion systems coupled to gauge fields with gauge group $G$. 
In such cases, the deformation class of the fermion anomaly is specified by an element 
$[\cI]\in (I_\bZ\Omega^\text{spin})^{d+2}(BG)$.
For brevity let us assume $G$ is a simply-connected simple Lie group.
The $B$-field in such cases satisfies the relation \begin{equation}
dH=\frac{p_1(R)}2-k \tau(F)\label{HRF}
\end{equation} where $k$ is an integer and $\tau(F)$ is the instanton number density of the $G$ gauge field $F$.
Mathematically, such a structure is known as a $(G,k\tau)$-twisted string structure,
where $k\tau$ is now regarded as an element of $\H^4(BG;\bZ)$.

Suppose now that this anomaly spin invertible theory $\cI$ can be trivialized
when restricted on string manifolds.
At the level of the deformation class,
this means that the image of $[\cI]$ in $ (I_\bZ\Omega^\text{string})^{d+2+ k\tau}(BG)$ vanishes.
This implies, among others, that
the anomaly polynomial $P_{d+2}$ of $\cI$
vanishes when the condition \eqref{HRF} is satisfied.
This in turn means that the anomaly polynomial factorizes as $P_{d+2}=(\frac{p_1}2-k\tau ) X_{d-2}$.

That $\cI$ is trivialized means that the one-dimensional Hilbert spaces $\cH_\cI(M_d^H)$ have canonical basis vectors $v(M_d^H)$ 
and that the evolution by a string manifold $W_{d+1}^H$ preserves these canonical basis vectors.
This latter condition can be phrased using $X_{d-2}$ as 
\begin{equation}
v(\tilde M_d^H) = \exp(-2\pi i \int_{W_{d+1} }H X_{d-2}) Z_\cI(W_{d+1}) v(M_d^H)
\label{consistency-condition}
\end{equation}
where $M_d^H$ and $\tilde M_d^H$ are incoming and outgoing boundaries of $N_{d+1}^H$.
As a degenerate situation, when $W_{d+1}^H$ has no boundary,
we have \begin{equation}
Z_\cI(W_{d+1})=\exp(2\pi i \int_{W_{d+1} }H X_{d-2})
\end{equation}
which indeed has the variation given by $P_{d+2}=(\frac12p_1(R)-k\tau(F)) X_{d-2}$ thanks to \eqref{HRF}.

With such basis vectors we can make sense of partition functions of $\Psi$ on string manifolds $M_d^H$
as numbers via the formula\begin{equation}
\langle v(M_d^H), Z_\Psi(M_d)\rangle. \label{vZ}
\end{equation}
This can be considered as the partition function of the fermion system $\Psi$ and a fixed chosen $B$-field background.

We can also discuss the action of the $B$-field in isolation in the following sense.
For this purpose we need to assume that the string manifold  $M_{d}^H$
has a bulk extension $N_{d+1}$ given as a spin manifold
such that $M_{d}=\partial N_{d+1}$.
Note that we do not require that $N_{d+1}$ is equipped with the string structure,
and only require that $N_{d+1}$ is a spin manifold.
After all, the $B$-field in isolation is supposed to have an anomaly as a theory on spin manifolds
which cancels that of the fermions,
and therefore needs to be formulated on the boundary of a spin manifold.
With this caveat, we can rewrite \eqref{vZ} as \begin{equation}
\langle v(M_d^H), Z_\Psi(M_d)\rangle
=\langle v(M_d^H), Z_\cI(N_{d+1})\rangle
\langle Z_\cI(N_{d+1}), Z_\Psi(M_d)\rangle
\end{equation}
and 
 call the factor \begin{equation}
e^{-S_\text{top}[H]}:=\langle v(M_d^H), Z_\cI(N_{d+1})\rangle
\label{B-field-coupling}
\end{equation}
as the topological part of the $B$-field action in the Euclidean signature;
the part \begin{equation}
\langle Z_\cI(N_{d+1}), Z_\Psi(M_d)\rangle
\end{equation} is then the fermion partition function
which now depends on the choice of the bulk $N_{d+1}$.

This formalism allows us to determine not only the traditional continuous part 
but also the discrete part of the Green-Schwarz coupling.
Let us first see that it  encodes the standard continuous part of the Green-Schwarz coupling given in \eqref{BX}.
For this, take $M_d^{H_0}$ and $M_d^{H_1}$ defined on the same spin manifold $M_d$,
such that the $B$-fields are different by a globally-defined 2-form $B$ so that $H_1-H_0=dB$.
A string bordism $W_{d+1}^H$ between $M_d^{H_0}$ and $M_d^{H_1}$ is given by $W_{d+1}=M_d\times [0,1]$ with $H=H_0+d(sB)$, where $s\in [0,1]$.
Then the relation \eqref{consistency-condition} together with \eqref{B-field-coupling} 
indeed shows that \begin{equation}
e^{-S_\text{top}[H_1]} / e^{-S_\text{top}[H_0]} = \exp(-2\pi i\int_{W_{d+1}} H X_{d-2})
= \exp(-2\pi i \int_{M_d} BX_{d-2}).
\end{equation}

\subsubsection{Relative bordism groups and relative invertible phases:}

We now recognize that the topological part of the $B$-field action \eqref{B-field-coupling} is based on a somewhat unfamiliar type of invertible phases,
which assigns a complex number not to a closed manifold $N_{d+1}$ with a structure $\cS$ (such as orientation, spin structure, etc.)
but to a manifold $N_{d+1}$ with boundary $M_d=\partial N_{d+1}$
such that the boundary $M_d$ is equipped with a structure $\cS'$ related to but different from $\cS$.
Such invertible phases have actually appeared implicitly as the action of a boundary theory for a given anomalous theory obtained by a group extension method in \cite{Wang:2017loc,Tachikawa:2017gyf,Kobayashi:2019lep},
and the approach here is similar.

In our case $\cS$ is the spin structure and $\cS'$ is the string structure.
Then, $\cS$-manifolds whose boundaries are $\cS'$-manifolds 
have corresponding bordism groups, known as relative bordism groups,
which we denote by $\Omega^{\cS/\cS'}_{d+1}(\pt)$.
Then an $\cS$-manifold $N_{d+1}$ whose boundary $M_d=\partial N_{d+1}$ is equipped with an $\cS'$ structure gives a class $[N_{d+1},M_d]\in \Omega^{\cS/\cS'}_{d+1}(\pt)$.

These groups $\Omega^{\cS/\cS'}_\bullet(-)$ are generalized cohomology theories,
and represented by a spectrum $M\cS/M\cS'$ which is defined by the cofiber sequence \begin{equation}
M\cS' \to  M\cS \to M\cS/M\cS' \label{SSS}
\end{equation}
where $M\cS$, $M\cS'$ are Thom spectra representing the $\cS$- and $\cS'$- bordism groups, respectively \cite[p.25--26]{StongTextbook}.
In particular, $M\cS_d(\pt)=\Omega^\cS_d(\pt)$ and $(M\cS/M\cS')_d(\pt)=\Omega^{\cS/\cS'}(\pt)$.
The deformation classes of invertible phases defined on such $(d+1)$-dimensional manifolds 
with $d$-dimensional boundaries are then classified 
by its Anderson dual. %$(I_\bZ(M\cS/M\cS'))^{d+2}(\pt)$.

The  sequence \eqref{SSS} leads to the following long exact sequence \begin{multline}
\label{eq:crucial}
\cdots \to (I_\bZ M\cS)^{d+1}(\pt)
\stackrel{(4)}{\to}
(I_\bZ M\cS')^{d+1}(\pt) \\
\stackrel{(3)}{\to}
(I_\bZ(M\cS/M\cS'))^{d+2}(\pt)
\stackrel{(1)}{\to}
\\
(I_\bZ M\cS)^{d+2}(\pt)
\stackrel{(2)}{\to}
(I_\bZ M\cS')^{d+2}(\pt) \to \cdots
\end{multline}
where each of the steps can be given the following physics interpretations:
\begin{itemize}
\item The step $(1)$ means that an invertible phase defined on $N_{d+1}$ with $\cS$ structure
whose boundary $M_d$ has $\cS'$ structure
tautologically defines an invertible phase on closed manifolds $N_{d+1}$ with $\cS$ structure.
\item The step $(2)$ reinterprets an $\cS$-invertible phase as an $\cS'$-invertible phase.
That the composition $(2)\circ(1)$ vanishes means that 
the middle term $(I_\bZ(M\cS/M\cS'))^{d+2}(\pt)$
is for the anomalies of $\cS$ structure 
such that it becomes trivial when restricted on $\cS'$ structure.
\item The step $(3)$ modifies an  invertible phase defined on $N_{d+1}$ with $\cS$ structure
whose boundary $M_d$ has $\cS'$ structure
by adding an $\cS'$-invertible phase purely supported on its boundary $M_d$ with $\cS'$ structure.
\item The step $(4)$ reinterprets an $\cS$-invertible phase as an $\cS'$-invertible phase.
That the composition $(3)\circ(4)$ vanishes means that
the modification of an invertible phase given by the middle term $(I_\bZ(M\cS/M\cS'))^{d+2}(\pt)$ by an $\cS$-invertible phase purely supported on its boundary
does not change it.
This is as it should be, since the boundary is null-bordant as $\cS$ manifolds by definition.
\end{itemize}

Note that the torsion part of $(I_\bZ M \cS')^{d+1}(\pt)$
is dual to the torsion part of $ (M\cS')_{d}(\pt)=\Omega^{\cS'}_d(\pt)$,
i.e.~they are the discrete gravitational theta angles
detecting $\cS'$-bordism classes.
When $\cS$ is spin structure and $\cS'$ is string structure,
they are discrete gravitational theta angle 
detecting torsion string bordism classes.
The discussion above means that the $B$-field topological coupling
described by $(I_\bZ(\MSpin/\MString))^{d+2}(\pt)$
includes the information of such discrete gravitational theta angles.

\subsection{Green-Schwarz couplings for general heterotic compactifications}

After all these preparations, we are finally in a position to explain what the mathematical content of this paper allows us to do in physics terms.

\subsubsection{Summary of the authors' previous work}
In a previous paper \cite{TachikawaYamashita}, we considered 
 general compactifications of heterotic string theory to $d$ dimensional spacetimes
using an arbitrary worldsheet conformal theory $T$ of central charge $(c_L,c_R)=(26-d,15-\frac32 d)$,
and showed that both the  perturbative anomaly and the global anomaly vanished, irrespective of the choice of $T$.
This was done as follows.

First, the anomaly $\cI$ of spacetime fermions was computed.
This was done first by enumerating the massless spacetime fermions from the data of $T$,
and then summing over their individual contributions to the anomaly.
This process was formulated in \cite{TachikawaYamashita} in the three steps given below: \begin{itemize}
\item  We first form the deformation class $[T]$ of $T$,
which takes values in $\TMF^{22+d}(\pt)$ using the Stolz-Teichner proposal.
\item We then apply the standard homomorphism $\sigma:\TMF^{22+d}(\pt)\to \KO^{22+d}((q))(\pt)$. 
This physically corresponds to taking the right-moving ground states of $T$ in the R-sector,
from which the spacetime massless fermion spectra can be easily read off.
\item Finally, the massless fermion spectra are converted to the anomaly by a further application of a certain homomorphism $\alpha_\text{spin}:\KO^{22+d}((q))(\pt)\to (I_\bZ\Omega^\text{spin})^{d+2}(\pt)$.
This encodes the anomalies of both spin-$1/2$ fermions and gravitinos in a single formula.
\end{itemize}
Then the anomaly $\cI\in (I_\bZ\Omega^\text{spin})^{d+2}(\pt)$ 
is given as the composition $\alpha_\text{spin}(\sigma([T]))$.
The main point of \cite{TachikawaYamashita} was that 
this anomaly vanished when restricted on string manifolds,
thereby showing  the absence of all anomalies of heterotic string theories,
compactified using arbitrary worldsheet CFTs.

\subsubsection{Green-Schwarz couplings and Anderson duality of $\TMF$}
The next question for us is to determine the Green-Schwarz coupling 
responsible for the cancellation of the anomalies.
According to the general discussion in Sec.~\ref{subsec:GS},
it is given by an element in the middle term $(I_\bZ\Omega^\text{spin/string})^{d+2}(\pt)$ of \eqref{eq:crucial},
which should map to $\alpha_\text{spin}(\sigma([T]))$ 
under the arrow $(1)$.
That there is at least \emph{an} element satisfying this condition is
guaranteed by the fact that $\alpha_\text{spin}(\sigma([T]))$ is sent to zero by the arrow $(2)$.
But this condition alone does not uniquely determine it.
One of the main results in this paper is that
there is a \emph{unique} choice of such elements $\alpha_\text{spin/string}([T])\in (I_\bZ\Omega^\text{spin/string})^{d+2}(\pt)$
compatible with successive compactifications. 
This is Proposition~\ref{lem:a_spin/string},
and determines the Green-Schwarz coupling of the compactification using the worldsheet CFT $T$, not only in its differential form part, but also the discrete global part.

How do we actually compute the Green-Schwarz coupling $\alpha_\text{spin/string}([T])$ for a given $T$?
For example, which discrete gravitational theta angle is induced by which worldsheet CFT?
This can be given a fairly complete answer using the Anderson duality of $\TMF$.
A particularly nice subcase is the following.

The Anderson duality of $\TMF$ is known to imply the following \cite{BrunerRognes}. 
Let $A^n$ be the kernel of $\sigma:\TMF^n(\pt)\to \KO^n((q))(\pt)$.\footnote{%
In the main part of the paper, we use the notation $A_{-n}:=A^n$ instead.
}
Physically this operation $\sigma$ corresponds to the extraction of the ordinary elliptic genus originally introduced by Witten \cite{Witten:1986bf} or its mod-2 version  discussed in \cite{TYY}.
Therefore the elements in $A^n$ correspond to SQFTs whose ordinary or mod-2 elliptic genera vanish,
and are necessarily very subtle. 
These elements are known to be necessarily torsion.
Furthermore, when $n\not\equiv -3$ or $1$ mod $24$, $A^{n}$ and $A^{22-n}$ are known to be naturally Pontryagin dual,
i.e.~there is a perfect pairing $A^n \times A^{22-n} \to U(1)$.

Let us now suppose that we choose the worldsheet CFT $T$ such that $[T]\in A^{22+d}$
for a compactification to $d$ dimensions.
Such $[T]$ is necessarily torsion, and therefore there is no term of the form $B\wedge X_{d-2}$ in the Green-Schwarz coupling.
However, the discrete part of the Green-Schwarz coupling can be nonvanishing.
For example, let $M_d^H$ to be a string manifold,
such that it is null as a spin manifold.
We can then pick an auxiliary bulk spin manifold $N_{d+1}$ such that $\partial N_{d+1}=M_d$.
What is the Green-Schwarz coupling of this $d$-dimensional compactification 
on this spacetime?
More mathematically, what is the value of the invertible phase $\alpha_\text{spin/string}([T])$ on the class determined by $[N_{d+1},M_d^H]\in \Omega^\text{spin/string}_{d+1}(\pt)$,
assuming that it is torsion?

Our main result, Theorem~\ref{main} combined with Proposition~\ref{pont-dual}, is that it is given by the pairing we quoted above.
By an abuse of notation, 
let us denote the \Nequals{(0,1)} sigma model on $M_d^H$ also as $M_d^H$.
Then it determines a class $[M_d^H]\in \TMF^{-d}(\pt)$,
which vanishes when sent by $\sigma$ to $\KO^{-d}((q))(\pt)$,
since $M_d$ is null-bordant as spin manifolds.
Therefore $[M_d^H]\in A^{-d}$.
Then, our analysis shows that the value of the invertible phase is given by the natural pairing between $[T]\in A^{22+d}$ and $[M_d^H]\in A^{-d}$,
when $d\neq 24k+3$.
In other words, the compactification using $T$ 
has a discrete gravitational theta angle detecting the string bordism class $[M_d^H]$.

\subsection{Examples}
The ranges of immediate interest to string theorists are $0\le d\le 10$.
The groups $A^{-d}$ are known to be nonzero 
only when $d=3$, $d=6$, $d=8$, $d=9$ and $d=10$, respectively,
and given by \begin{equation}
A^{-3}=\bZ_{24},\quad
A^{-6}=\bZ_2,\quad
A^{-8}=\bZ_2,\quad
A^{-9}=\bZ_2,\quad
A^{-10}=\bZ_3.
\end{equation}
Their generators are known to be given respectively by  the group manifold of $SU(2)$, $SU(2)^2$, $SU(3)$, $SU(2)^3$  and $Sp(2)$ equipped with its natural $H$ flux, cf.~\cite{Hopkins2002}.\footnote{%
$U(3)$ and $U(1)\times SU(3)$ are also known to be bordant to $SU(2)^3$.
}

The case $d=3$ shows a somewhat different behavior 
because it does not satisfy $d\neq 24k+3$;
but this case was essentially already discussed in \cite{Tachikawa:2021mvw} and \cite[Sec.~3.2]{TachikawaYamashita}. 
In this paper we discuss the case $d=9$ and $d=6$ in detail,
and make some conjectures in Appendix~\ref{app:VOA} concerning $d=10$ and $d=8$.
%for which more details will be given elsewhere.

\subsubsection{$d=9$}
As the first example, $A^{31}$ and $A^{-9}$ are both $\bZ_2$ and are dual to each other.
$A^{-9}\subset \TMF^{-9}(\pt)$  is known to be generated by 
the \Nequals{(0,1)} sigma model on $\nu\times \nu\times \nu$,
where $\nu=(S^3)^H$ is a round $S^3$ equipped with the unit $H$ flux on it.
%where $S^1$ has a periodic spin structure and $SU(3)$ is equipped with a unit $H$ flux.
We showed in Sec.~\ref{app:power} that the generator of $A^{31}$ 
corresponds to the fibration of the $E_8\times E_8$ worldsheet current algebra
over the spacetime $S^1$ with anti-periodic spin structure
such that two $E_8$ factors are interchanged when we go around $S^1$.
This means that this 9-dimensional compactification 
has a non-zero gravitational theta angle detecting $\nu^3$. %$S^1\times SU(3)$.
We note that this fibration over $S^1$ describes the angular part of the non-supersymmetric 7-brane recently discussed in \cite{KaidiOhmoriTachikawaYonekura,Kaidi:2024cbx}.

\subsubsection{$d=6$}

As another example, $A^{28}$ and $A^{-6}$ are both $\bZ_2$ and are dual to each other.
$A^{-6}$ is known to be generated by $\nu\times \nu$.
Our general argument then implies that the worldsheet CFT realizing the generator of $A^{28}=\bZ_2$ has a nonzero discrete gravitational theta angle detecting this configuration. 

Let us now describe this generator of $A^{28}$ in more concrete terms.
Take $S^4$, and regard its tangent bundle as an $SU(2)\times SU(2)$ bundle.
This can be further embedded into $E_8\times E_8$,
which can be used to fiber the $E_8\times E_8$ current algebra over the \Nequals{(0,1)} sigma model on $S^4$.
This specifies a worldsheet theory $T$ for this particular compactification to six dimensions.
We note that this fibration over $S^4$ describes the angular part of the non-supersymmetric 4-brane recently discussed in \cite{KaidiOhmoriTachikawaYonekura,Kaidi:2024cbx}.

We will show that this theory $T$ is the generator of $A^{28}=\bZ_2$ by computing the Green-Schwarz coupling on $M_6^H=\nu \times \nu$ and showing it to be $-1\in U(1)$.
This will be done as follows.
According to \eqref{B-field-coupling}, the Green-Schwarz coupling is given by \begin{equation}
e^{-S_\text{top}[H]}= \langle v(M_6^H),Z_\cI(N_7) \rangle
\label{explicit-pairing}
\end{equation}
where $\cI$ is the anomaly theory of the massless fermions contained in this 6-dimensional  compactification.
As $SU(2)\times SU(2)$ breaks $E_8\times E_8$ down to $E_7\times E_7$,
the massless fermions transform under $E_7\times E_7$.
A standard computation reveals that 
there is one symplectic-Majorana Weyl fermions in $\mathbf{56}$ of one $E_7$ on one chirality, 
and another symplectic-Majorana Weyl fermions in $\mathbf{56}$ of another $E_7$ of the other chirality. The anomaly polynomial is \begin{equation}
\frac12(\frac{p_1}2 -\tau-\tau') (\tau-\tau')
\end{equation} where $\tau$, $\tau'$ are the instanton number density for the two $E_7$ factors.

We take $N_7=S^3\times B^4$, where the $E_7\times E_7$ gauge fields are trivial.
We also use another copy of the same space, $N_7'=S^3\times B^4$,
this time adding an instanton of one of the two $E_7$'s at the center of $B^4$,
solving $dH=\frac{p_1}2-\tau-\tau'=-\tau$.
This gives a unit $H$ flux on $S^3=\partial B^4$.
Using \eqref{consistency-condition}, we find that \begin{equation}
v(M_6^H)=\exp\bigl(-2\pi i \int_{N_7'}\frac H2(\frac{p_1}2-\tau-\tau')\bigr)Z_\cI(N_7')
\end{equation}
where $Z_\cI(N_7')\in \cH_\cI(M_6)$ is the Hartle-Hawking wavefunction 
associated to the spin manifold $N_7'$ equipped with the $E_7$ bundle.
Then the pairing \eqref{explicit-pairing} is \begin{align}
e^{-S_\text{top}[H]}&= \exp\bigl(-2\pi i \int_{N_7'}\frac H2(\frac{p_1}2-\tau-\tau')\bigr) 
\langle Z_{\cI}(N_7), Z_\cI(N_7')\rangle \\
&= \exp\bigl(-2\pi i \int_{N_7'}\frac H2(\frac{p_1}2-\tau-\tau')\bigr)  Z_{\cI}(\overline{N_7} \cup_{M_6} N_7').
\label{almost-done}
\end{align}
The first factor on the right hand side of \eqref{almost-done} is \begin{equation}
\exp\bigl(-2\pi i \int_{N_7'}\frac H2(\frac{p_1}2-\tau-\tau')\bigr)
= \exp\bigl(-2\pi i \frac12  \int_{S^3} H  \int_{B^4} (-\tau)\bigr) = -1.
\end{equation}
The second factor on the right hand side of \eqref{almost-done} is the
exponentiated $\eta$ invariant of the symplectic Majorana Weyl fermions in $\mathbf{56}$ of $E_7$ on
$\overline{N_7} \cup_{M_6} N_7' = S^3\times S^4$
where $S^4$ has a unit instanton of $E_7$.
The eta invariant on the product space $M_\text{odd} \times M_\text{even}'$
satisfies $\eta(M_\text{odd} \times M_\text{even}')=
\eta(M_\text{odd}) \mathrm{index}(M_\text{even}')$
and the eta invariant on the round $S^3$ is $0$. Therefore \begin{equation}
Z_{A}(\overline{N_7} \cup N_7') =+1.
\end{equation}
Combining, we find that \begin{equation}
e^{-S_\text{top}[H]} = -1,
\end{equation} which is what we wanted to show.
A mathematically rigorous version of this computation is provided in Sec.~\ref{sec:computation},
which uses the techniques of differential bordism groups and their Anderson duals developed in 
\cite{Yamashita:2021cao,YamashitaAndersondualPart2,YamashitaDifferentialIE}
and recalled in Appendix~\ref{app:diff_bordism}.

\subsubsection{Comments}
We saw above that the non-trivial string bordism class $\nu^3$ in 9 dimensions,
leading to a class in $A^{-9}$
is detected by the $E_8\times E_8$ heterotic string on $S^1$ with antiperiodic spin structure with the exchange of two $E_8$ factors,
which gives the generator of $A^{31}$.
Similarly, we saw that the non-trivial string bordism class $\nu^2$ in 6 dimensions,
leading to a class in $A^{-6}$,
is detected by the $E_8\times E_8$ heterotic string on $S^4$ where the $SU(2)\times SU(2)$ curvature of the tangent bundle is embedded into $E_8\times E_8$,
which gives the generator of $A^{28}$.

We note that these are the angular parts of the two of the exotic heterotic brane configurations
recently discussed in \cite{KaidiOhmoriTachikawaYonekura,Kaidi:2024cbx}.
There, it was shown that the worldsheet CFTs corresponding to these configurations
can be continuously deformed to purely left-moving fermionic CFTs
based on the $(\mathfrak{e}_8)_2$ affine algebra and the $(\mathfrak{e}_7)_1\times (\mathfrak{e}_7)_1$ affine algebra,
with central charges $c=31/2$ and $c=14$, respectively.
Using the Stolz-Teichner proposal,
our finding translates to the statement that
these two purely left-moving fermionic CFTs
give rise to the generators of $A^{31}\subset \TMF^{31}(\pt)$
and $A^{28}\subset \TMF^{28}(\pt)$, respectively.
The relation between purely left-moving fermionic CFTs
and the $\TMF$ classes are discussed more fully in Appendix~\ref{app:VOA}.

We also note that another exotic heterotic brane configuration in 8 dimensions
is also discussed in \cite{KaidiOhmoriTachikawaYonekura,Kaidi:2024cbx}.
We have some physics evidence that it gives the generator of $A^{30}$
and detects the generator  of $A^{-8}$.
The details will be discussed elsewhere;
here it will suffice to mention that this configuration can be continuously deformed 
to the purely left-moving fermionic CFT based on the affine Lie algebra $\mathfrak{su}(16)_1$
of central charge $c=15$.
This therefore should be the generator of $A^{30}\subset \TMF^{30}(\pt)$.

\if0
%% This is for lawyers at CMP
\section*{Declarations}
\noindent\emph{Funding:} 
\funding
The authors have no other competing interests to declare that are relevant to the content of this article.
\fi

\def\arxivfont{\rm}
\bibliographystyle{ytamsalpha}
\bibliography{ref}

\end{document}